%% file: thesis_v2.tex
\documentclass[a4paper,12pt]{report}

\input{preamble}

\begin{document}

\include{title_toc}
\pagenumbering{arabic}
\include{chapter1}
\include{chapter2}
\include{chapter3}

\include{chapter4}
\include{chapter5}
\include{chapter6}
\include{chapter7}
\include{chapter8}
\include{chapter9}

\begin{appendices}
	\include{Appendix}
	\include{AppendixB}
\end{appendices}

%\doublespacing
%\bibliographystyle{plain}
%\bibliography{refs}{}
%\nocite{*}

% \bib, bibdiv, biblist are defined by the amsrefs package.

\end{document}

%% file: preamble.tex
\usepackage{etoolbox}
\usepackage{latexsym}
\usepackage{amsfonts}
\usepackage{amsmath, amssymb, amsthm,amscd,epsfig,mathtools}
\usepackage{mathrsfs}
\usepackage{graphicx}
\usepackage{setspace}
\usepackage{xcolor}
\usepackage{dsfont}
\author{Quinn Patterson}
\usepackage{amsrefs}
\usepackage{pst-node}
\usepackage{tikz-cd} 
\usepackage[all]{xy}
\usepackage{appendix}
\usetikzlibrary{decorations.markings}
\usetikzlibrary{arrows}
\usetikzlibrary{shapes, decorations.text}
\allowdisplaybreaks

\makeatletter
\def\@tocline#1#2#3#4#5#6#7{\relax
	\ifnum #1>\c@tocdepth % then omit
	\else
	\par \addpenalty\@secpenalty\addvspace{#2}%
	\begingroup \hyphenpenalty\@M
	\@ifempty{#4}{%
		\@tempdima\csname r@tocindent\number#1\endcsname\relax
	}{%
		\@tempdima#4\relax
	}%
	\parindent\z@ \leftskip#3\relax \advance\leftskip\@tempdima\relax
	\rightskip\@pnumwidth plus4em \parfillskip-\@pnumwidth
	#5\leavevmode\hskip-\@tempdima
	\ifcase #1
	\or\or \hskip 1em \or \hskip 2em \else \hskip 3em \fi%
	#6\nobreak\relax
	\dotfill\hbox to\@pnumwidth{\@tocpagenum{#7}}\par
	\nobreak
	\endgroup
	\fi}
\makeatother

\numberwithin{equation}{section} %Equation numbers reset in each section
\newcommand{\PP}{\mathcal{P}}
\newcommand{\HH}{\mathcal{H}}
\newcommand{\FF}{\mathcal{F}}

\newcommand{\TT}{\mathcal{T}}
\newcommand{\BB}{\mathcal{B}}
\newcommand{\LL}{\mathcal{L}}
\newcommand{\MM}{\mathcal{M}}
\newcommand{\UU}{\mathcal{U}}
\newcommand{\KK}{\mathcal{K}}
\newcommand{\OO}{\mathcal{O}}

\newcommand{\GG}{\mathcal{G}}

\newcommand{\N}{\mathbb{N}}
\newcommand{\Z}{\mathbb{Z}}
\newcommand{\R}{\mathbb{R}}
\newcommand{\C}{\mathbb{C}}
\newcommand{\T}{\mathbb{T}}

\newcommand{\Q}{\mathbb{Q}}
\newcommand{\h}{\mathbb{H}}

\newcommand{\ep}{\epsilon}
\newcommand{\dis}{\displaystyle}
\newcommand{\bs}{\backslash}

\newcommand{\Image}{\operatorname{Image}}
\newcommand{\Tr}{\operatorname{Tr}}
\newcommand{\diag}{\operatorname{diag}}
\newcommand{\Span}{\operatorname{span}}
\newcommand{\Cl}{\operatorname{Cl}}

\newcommand{\Ker}{\operatorname{Kernel}}
\newcommand{\Coker}{\operatorname{Cokernel}}

\newcommand{\gr}{\operatorname{gr}}
\newcommand{\lt}{\operatorname{lt}}
\newcommand{\Cliff}{\mathbb{C}\!\operatorname{liff}}
\newcommand{\Otimes}{\operatorname{\widehat{\otimes}}}
\newcommand{\bigOtimes}{\operatorname{\widehat{\bigotimes}}}

\DeclarePairedDelimiter\IP{\langle}{\rangle}

\DeclarePairedDelimiter\BigIP{\Big\langle}{\Big\rangle}

\DeclarePairedDelimiter\normof{\Vert}{\Vert}
\DeclarePairedDelimiter\bignormof{\big\Vert}{\big\Vert}
\DeclarePairedDelimiter\Bignormof{\Big\Vert}{\Big\Vert}
\DeclarePairedDelimiter\Biggnormof{\Bigg\Vert}{\Bigg\Vert}

\theoremstyle{plain}
\newtheorem{theorem}{Theorem}[section]
\newtheorem{lemma}[theorem]{Lemma}
\newtheorem{proposition}[theorem]{Proposition}
\newtheorem{corollary}[theorem]{Corollary}

\theoremstyle{definition}
\newtheorem{definition}[theorem]{Definition}

\newtheorem{example}[theorem]{Example}

\newtheorem{notation}[theorem]{Notation}

\theoremstyle{remark}
\newtheorem*{remark}{Remark}
\newtheorem*{sketch}{Sketch of proof}

%% file: title_toc.tex
\title{\vspace{-2.5cm}
{\centering\includegraphics[width=0.55\textwidth]{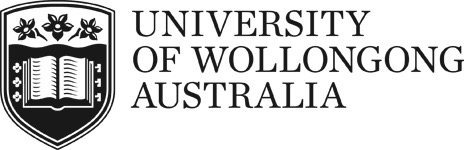}}\\~\\
\singlespacing{\Huge{\textsc{Exact sequences in graded $KK$-theory for Cuntz-Pimsner algebras}}}}

\author{By \\~\\
\Huge{Quinn Patterson} \\~\\ ~\\ ~\\
\normalsize{Supervisors: Professor Aidan Sims and Doctor Adam Sierakowski} \\ \\ \\
\large{\textsc{Bachelor of Mathematics Advanced (Honours)}} \\
\large{School of Mathematics and Applied Statistics}}

\date{\large{\textsc{2018}}}

\maketitle

\begin{abstract}
	In this thesis we generalise the six-term exact sequence in graded $KK$-theory of \cite[Theorem~4.4]{ThePaper} to allow correspondences with non-compact left action. In particular, this allows us to compute the graded $KK$-theory of row-infinite graphs. We develop the theory necessary for following the arguments of \cite{ThePaper} and \cite{FreeProbTheory}, with detailed sections on Hilbert modules, $C^*$-correspondences, Crossed products, Toeplitz algebras, Cuntz-Pimsner algebras and $KK$-theory.
\end{abstract}

\newpage
\pagenumbering{gobble}

\thispagestyle{plain}
\vspace*{0.05\textheight}
\centerline{\bf\Large Acknowledgements}
\vspace{0.5em}

\begin{center}
	{\setlength{\parindent}{0cm}
	First and foremost I would like to thank my family and my partner Jedda for supporting me throughout this thesis and always. I haven't exactly made that an easy task and I appreciate the patience you all have had with me.\\
	I would like to express my deepest gratitude to my supervisors Aidan Sims and Adam Sierakowski for taking me on as their honours student. I could not have asked for better supervisors; their guidance, support and stream of knowledge has been unending. I thank them for sharing their knowledge, for answering my emails at two in the morning, dedicating so much of their time to my thesis, and for teaching me how to be a mathematician. I can only assume that I am single-handedly responsible for wiping out what would have otherwise been several years worth of red pens with my long-winded, convoluted and often completely incorrect writing style.\\
	Besides my supervisors I am incredibly grateful for the support of Adam Rennie who has been a constant presence throughout my entire university experience and always has a book for every occasion. Adam's $KK$-theory lectures gave me new insight which I hope has seeped into the pages of this thesis. I would like to thank Jono Taylor, for proofreading my thesis and always being around to bounce ideas off during my thesis and my entire undergraduate degree and Alexander Mundey for years of being an awesome maths senpai by showing me the bigger context in a way that makes so much sense. Thank you to Angus, Tess, Laura, Jacob and Jess who have each made my time at UOW such an amazing experience with their friendship and support.
	} 
\end{center}

{\small \tableofcontents}\normalsize

%% file: chapter1.tex
\chapter{Introduction}
\label{Lit review}

$K$-theory is a useful tool for understanding the structure of $C^*$-algebras, in particular $K$-theory gives a method for proving when two $C^*$-algebras are not isomorphic. To each $C^*$-algebra $A$ we associate a pair of ordered abelian groups $K_0(A)$ and $K_1(A)$. Isomorphic $C^*$-algebras have isomorphic $K$-groups, so two $C^*$-algebras with different $K$-groups are not isomorphic.\\
A grading on a $C^*$-algebra $A$ is a self inverse automorphism $\alpha_A$. Such an operator necessarily gives a decomposition of $A$ as the direct sum of two subspaces $A_0$ and $A_1$ on which $\alpha_A$ acts as the identity and by multiplication by $-1$ respectively. We have $A_iA_j\subseteq A_{i+j}$ mod 2, so a grading can be viewed as a representation of $\Z_2$ on $A$. As such, a grading is an algebraic way of splitting a $C^*$-algebra into two natural pieces. A $C^*$-algebra $A$ may admit multiple gradings, so we consider pairs $(A,\alpha_A)$ and refer to $A$ as a graded $C^*$-algebra. Graded $C^*$-algebras are of interest in theoretical physics where they are often known as `superalgebras'.\\ 
Just as $K$-theory can be computed to distinguish $C^*$-algebras, we would like a version of $K$-theory that can be used to distinguish \emph{graded} $C^*$-algebras. As is acknowledged in \cite[Chapter~3]{ThePaper}, there does not appear to be a universally-accepted definition of graded $K$-theory for $C^*$-algebras. In this paper we follow the definitions of \cite{ThePaper} and \cite{Haag} by using Kasparov's $KK$-theory to define graded $K$-theory.\\
 $K$-theory has a sister theory called $K$-homology which gained popularity in the 70s. $K$-homology was the subject of the seminal 1977 paper \cite{BrownDouglasFillmore} of Brown, Douglas and Fillmore and had been used in 1968 by Atiyah and Singer in proving the Atiyah-Singer index theorem \cite{AtiyahSinger}. $KK$-theory was introduced by Gennadi Kasparov as a unifying theory that encompassed both $K$-theory and $K$-homology. More specifically, to each pair of graded $C^*$-algebras $A$ and $B$ we may assign abelian groups $KK_0(A,B)$ and $KK_1(A,B)$ for which when $A$ is trivially graded, $KK_i(\C,A)$ are isomorphic to the $K$-theory groups $K_i(A)$ and $KK_i(A,\C)$ are isomorphic to the $K$-homology groups $K^i(A)$. $KK$-theory however accounts for graded information where $K$-theory and $K$-homology do not. As such, it is natural to define the graded $K$-theory groups of a graded $C^*$-algebra $A$ to be $KK_i(\C,A)$ (and one could make a similar definition of graded $K$-homology).\\ 
 In this thesis we are particularly interested in computing the graded $K$-theory of \emph{Cuntz-Pimsner} algebras. In \cite{FreeProbTheory} Michael Pimsner defined a $C^*$-algebra $\OO_X$ constructed from a correspondence $X$ over a $C^*$-algebra $A$ (see Definition~\ref{def: correspondence}) using representations of $X$ on the Fock space $\FF_X$ (see Example~\ref{Fock this}). These $C^*$-algebras have since been called \emph{Cuntz-Pimsner} algebras and include graph algebras for graphs with no sinks (see \cite{CuntzKreigerIsCuntzPimsner}) and crossed products by $\Z$ (see \ref{thm: Crossed by Z is Cuntz Pimsner}). Thus their study has been of great interest. Takeshi Katsura provided an alternative definition of $\OO_X$ in \cite{Katsura} based on universal properties of $\OO_X$ that diverges from Pimsner's definition when the left action of the correspondence $X$ is not injective. In \cite[Proposition~3.10]{Katsura}, Katsura showed that the graph algebra of graphs with sinks are isomorphic to $\OO_X$ for a correspondence $X$ with non-injective left action, which has provided a strong argument for using Katsura's definition when $X$ has non-injective left action, as is stated in \cite[Chapter~1]{AddingTails}.\\ 
 In \cite[Theorem~2.4]{PimsnerVoiculescu}, Pimsner and Voiculescu proved the existence of an exact sequence in $K$-theory known as the Pimsner-Voiculescu exact sequence which can be used to compute the $K$-theory of crossed products by $\Z$. This gave a powerful tool for understanding the internal structure of crossed products and as pointed out in \cite[10.2]{Blackadar}, was one of the first non-trivial applications of $K$-theory to the structure of $C^*$-algebras. In \cite[Theorem~4.9]{FreeProbTheory}, Pimsner proved the existence of a much more general exact sequence in $KK$-theory involving Cuntz-Pimsner algebras which include crossed products by $\Z$. However, Pimsner made the assumption that $C^*$-algebras involved were trivially graded. In \cite[Theorem~4.4]{ThePaper}, Kumjian, Pask and Sims modified the arguments of Pimsner to show that Pimsner's exact sequence holds true for graded $C^*$-algebras, giving a powerful tool for the computation of graded $K$-theory of $C^*$-algebras.\\ 
 In \cite{ThePaper} Kumjian, Pask and Sims made the assumption that the left action of the correspondence $X$ be injective and by compact operators, whereas Pimsner did not require that the left action be by compacts in \cite{FreeProbTheory} (throughout \cite{FreeProbTheory} the assumption is made that the left action be injective, but it is noted in Remark~1.2~(1) that this is just for simplicity). In particular, when $\OO_X$ is a graph algebra, these two conditions correspond to row-finite graphs with no sources, which is an undesirable restriction.\\ 
 In the context of graph $C^*$-algebras there is a simple technique for dealing with sources called `adding tails' and in \cite{AddingTails} it is shown how this process can be carried over to correspondences where the left action is not injective. In Chapter~\ref{ch: closing remarks} we give some indication how the work of \cite{AddingTails} could be used to find a graded exact sequence from Pimsner's exact sequence for correspondences with non-injective left actions. On the other hand, row-finiteness is an undesirable assumption, and the closest analogue to adding tails for this situation is called `desingularisation' (see \cite{Desingularisation}) and it is quite complicated, particularly if we must also keep track of gradings.\\ 
 Our aim in this thesis is to show that the modifications of Kumjian, Pask and Sims of Pimsner's exact sequence can be generalised to correspondences with non-compact left actions by following the arguments of Pimsner in \cite{FreeProbTheory}. This allows us to compute the graded $K$-theory of graph algebras for row-infinite graphs which we carry out in Chapter~\ref{Graph algebras}. Our results in this area mirror the known results of \cite[Theorem~6.1]{ArbitraryGraphKTheory} for the regular $K$-theory of row-infinite graph algebras and generalises \cite[Corollary~8.3]{ThePaper} in which the graded $K$-theory of row-finite graph algebras is computed.

In Chapter~\ref{C^*-Correspondences} we draw upon \cite{RaeburnWilliams} and \cite{Lance} to develop the theory of Hilbert modules and $C^*$-correspondences with reference to the lecture notes \cite{Crossedandunitizations} of Aidan Sims. We discuss Kasparov's stabilisation theorem in Section \ref{section: Kasparovs stabilisation theorem} and use it to prove the existence of frames for countably generated Hilbert modules, which are needed to prove key structure results (see \ref{comapcts inclusion in Toeplitz}) about the Toeplitz algebra $\TT_X$ in Chapter~\ref{Special Funky C^*-algebras}. We then introduce $C^*$-correspondences and devote a significant portion of the chapter to proving results about the balanced tensor product and giving examples.

In Chapter~\ref{C^*-algebras} we introduce the notion of gradings, assembling some important facts that we will need in Chapter~\ref{KK-theory}, and we develop some further $C^*$-algebraic theory, carefully developing theory on unitizations following \cite[Chapter~2]{RaeburnWilliams} and outlining some facts about tensor products of $C^*$-algebras.

Chapter~\ref{Special Funky C^*-algebras} details the construction of special $C^*$-algebras that we are interested in computing the graded $K$-theory of: Clifford algebras, graph algebras, crossed products, Toeplitz algebras and Cuntz-Pimsner algebras. In particular for Crossed products, Toeplitz algebras and Cuntz-Pimsner algebras we provide careful detail, outlining several important results about the structure of these $C^*$-algebras.

We outline the definition of regular $K$-theory providing some examples in Chapter~\ref{K-theory} so that we may get a taste for how $K$-theory is computed before diving into the world of $KK$-theory. This chapter is shorter than the rest and mostly proof-free, we have included it mostly for illustrative purposes.

In Chapter~\ref{KK-theory} we give a selection of detailed proofs of elementary results in $KK$-theory. We have chosen to state some results without proof with heavy reference to \cite{Blackadar} to avoid building machinery that we would not otherwise be using. We also introduce graded $K$-theory and Morita equivalence, proving that Morita equivalent $C^*$-algebras are also $KK$-equivalent (see Theorem~\ref{Morita equiv KK}).

Chapter~\ref{The Exact Sequence!} is the main original content of this thesis, where we follow the arguments of \cite{FreeProbTheory} and \cite{ThePaper} to prove the existence of an exact sequence in $KK$-theory for $\OO_X$ when the left action on $X$ is not by compacts. We also note the interesting result Theorem~\ref{KK equivalence} that for a correspondence $X$ over $A$ with injective left action, the $C^*$-algebras $A$ and $\TT_X$ are $KK$-equivalent.

In Chapter~\ref{Graph algebras} we use the results of Chapter~\ref{The Exact Sequence!} to compute graded $K$-theory of row-infinite graph algebras, giving specific examples.

We conclude this thesis with closing remarks in Chapter~\ref{ch: closing remarks} by comparing our results in Chapter~\ref{Graph algebras} to the analogous results for the regular $K$-theory of graph algebras, and giving an indication of how the work of \cite{AddingTails} may be used to compute the graded $K$-theory of $\OO_X$ for $X$ with non-injective left action.

%% file: chapter2.tex
\chapter{$C^*$-Correspondences} 
\label{C^*-Correspondences}

In order to understand $KK$-theory we will need to build some background on $C^*$-correspond\-ences. $C^*$-correspond\-ences can be thought of as `generalised homomorphisms' between $C^*$-algebras in an appropriate sense. Specifically, every homomorphism $\phi$ between $C^*$-algebras determines a $C^*$-correspondence $_{\phi}A$ (see Example~\ref{Marnie}), and there is a natural isomorphism between the correspondence $_{\phi\circ\psi}A$ and the balanced tensor product $_{\phi}A\otimes_{\psi}A$ of the individual correspondences (see Proposition~\ref{Trang} and Example~\ref{Gen hom balance}). So we regard isomorphism classes of correspondences as morphisms between $C^*$-algebras with composition induced by balanced tensor products.
The category whose objects are $C^*$-algebras are somewhat better behaved when the morphisms are taken to be isomorphism classes of $C^*$-correspondences rather than homomorphisms. We need to use isomorphism classes for our morphisms because despite being isomorphic, strictly speaking $(X\otimes_A Y)\otimes_B Z$ is not equal to $X\otimes_A(Y\otimes_BZ)$, so $\otimes$ is not associative.
\vspace{0.2cm}	
\paragraph{\textbf{Conventions:}}
Throughout this document by an \underline{ideal} in a $C^*$-algebra we always mean a closed two sided ideal, and by a \underline{homomorphism} between $C^*$-algebras we always mean a $*$-homomorphism. %We will use the notation $f:A\hookrightarrow B$ to denote that $f$ is an injective function and $f:A\twoheadrightarrow B$ to denote that $f$ is surjective.
\section{Hilbert modules}

We take the following definitions and conventions from \cite[Chapter~2]{RaeburnWilliams}.
\begin{definition}
	Let $A$ be a $C^*$-algebra. A \emph{right $A$-module} $X$ is a complex vector space together with a bilinear right multiplication $(x,a)\mapsto x\cdot a:X\times A\to X$ such that for all $x,y\in X$, $a,b\in A$ and $\lambda\in \C$
	\begin{align*}
	(x+y)\cdot a&=x\cdot a+y\cdot a,\\
	x\cdot (a+b)&=x\cdot a+x\cdot b,\\
	(x\cdot a)\cdot b&=x\cdot (ab),\text{ and}\\
	(\lambda x)\cdot a&=\lambda (x\cdot a).\\
	\end{align*}
	We sometimes write $X_A$ to denote the module $X$ if we wish to emphasise the right action of $A$ on $X$, and we will write $xa$ in place of $x\cdot a$ unless we wish to emphasise the right action.
\end{definition}
\begin{definition}
	A \emph{right inner-product $A$-module} $X$ is a right $A$-module with a pairing $\IP{\cdot,\cdot}_A:X\times X\to A$ satisfying for all $x,y,z\in X$, $\mu,\lambda\in\C$ and $a\in A$
	\begin{align}
	\label{IP1}\IP{x,\lambda y+\mu z}_A&=\lambda\IP{x,y}_A+\mu\IP{x,z}_A,\\
	\label{IP2}\IP{x,ya}_A&=\IP{x,y}_Aa,\\
	\label{IP3}\IP{x,y}^*_A&=\IP{y,x}_A,\\
	\label{IP4}\IP{x,x}_A&\quad\text{ is a positive element of }A,\text{ and}\\
	\label{IP5}\IP{x,x}_A&=0\implies x=0.\\\nonumber
	\end{align}
	In context where it is clear that the inner-product is taking values in $A$, we will write $\IP{\cdot,\cdot}$ instead of $\IP{\cdot,\cdot}_A$.
\end{definition}
\begin{lemma}
	If $X$ is a right inner-product $A$-module, then $\IP{\cdot ,\cdot}_A$ is conjugate linear in the first variable, $\IP{xa,y}=a^*\IP{x,y}$ for all $a\in A$ and $x,y\in X$, and 
	\begin{equation*}
	I=\overline{\Span}\{\IP{x,y}_A:x,y\in X\}
	\end{equation*}
	is an ideal of $A$.
\end{lemma}
\begin{proof}
	Fix $x,y,z\in \C$, $\mu,\lambda\in\C$ and $a,b\in A$. Using \eqref{IP1} and \eqref{IP3}, we compute
	\begin{align*}
	\IP{\lambda x+\mu y,z}&=\IP{z,\lambda x+\mu y}^*=\left(\lambda \IP{z,x}+\mu \IP{z,y}\right)^*\\
	&=\overline{\lambda} \IP{z,x}^*+\overline{\mu} \IP{z,y}^*
	=\overline{\lambda} \IP{x,z}+\overline{\mu} \IP{y,z}.\\
	\end{align*}
	Likewise, using \eqref{IP2} and \eqref{IP3} we have
	\begin{equation}\label{Nifty}
	\IP{xa,y}=\left(\IP{y,xa}\right)^*=\left(\IP{y,x}a\right)^*=a^*\IP{y,x}^*=a^*\IP{x,y}.
	\end{equation}
	The formulas \eqref{Nifty} and \eqref{IP2} show that for $a\in A$, $c_i\in\C$ and $x_i,y_i\in X$, both
	\begin{equation*}
	a\left(\sum_i c_i\IP{x_i,y_i}\right)\qquad\text{ and }\qquad \left(\sum_i c_i\IP{x_i,y_i}\right) a
	\end{equation*}
	belong to $I$. Since $I$ is closed by construction, $I$ is an ideal in $A$.
\end{proof}
\begin{definition}
	A \emph{right Hilbert $A$-module} $X$ is a right inner-product $A$-module which is complete in the norm
	\begin{equation}
	\normof{x}_X\coloneqq \normof{\IP{x,x}_A}^{1/2},\label{act}
	\end{equation}
	where the right-hand side is the $C^*$-norm on $A$. We say that $X$ is \textit{full} if
	\begin{equation*}
	\{\IP{x,y}:x,y\in X\}
	\end{equation*}
	is dense in $A$. We will refer to a right Hilbert $A$-module as simply a Hilbert $A$-module.
\end{definition}
That Equation \eqref{act} defines a norm on $X$ follows from the Cauchy-Schwarz inequality
\begin{equation}
\IP{x,y}^*_A\IP{x,y}_A\leq \normof{\IP{x,x}_A}\IP{y,y}_A\label{C-S}
\end{equation}
for right-inner-product $A$-modules.
A proof of this inequality can be found in \cite[Lemma~2.5]{Raeburn Williams}.
Lots of familiar friends are Hilbert modules.
\begin{example}\label{dang}
	Any $C^*$-algebra $A$ is a Hilbert $A$-module under the right action $a\cdot b=ab$ given by multiplication in $A$, and inner-product
	\begin{equation*}
	\IP{a,b}=a^*b.
	\end{equation*}
	We denote this module by $A_A$. Fix an approximate identity $\{e_{\lambda}\}$ for $A$. For $a\in A$ we have
	\begin{equation*}
	a=\lim_{\lambda\in\Lambda}e^*_{\lambda}a=\lim_{\lambda\in\Lambda}\IP{e_{\lambda},a}\in\overline{\Span}\{\IP{b,c}:b,c\in A_A\},
	\end{equation*}
	so $A_A$ is full. The norm on $A_A$ defined by Equation \eqref{act} coincides with the $C^*$-norm on $A$ because
	\begin{equation*}
	\normof{a}_{A_A}=\normof{\IP{a,a}}^{1/2}_A=\normof{a^*a}^{1/2}_A=\normof{a}_A.
	\end{equation*}
\end{example}
\begin{example}
	This example is from \cite{Kasnotes}. A \emph{complex vector bundle} $V$ over a locally compact topological space $X$ is a topological space $V$ together with a continuous surjection $\pi:V\to X$ such that for each $x\in X$, the pre-image $\pi^{-1}(x)$ is a finite-dimensional complex vector space. We call the vector space $\pi^{-1}(x)$ the \textit{fibre} over $x$. A vector bundle $V$ over $X$ is \textit{locally trivial} if for each $x\in X$ there is an open neighbourhood $U$ of $x$ and a homeomorphism $\Phi_U:\pi^{-1}(U)\to U\times \C^k$ for which, if $p:U\times \C^k\to U$ denotes the projection map, we have $\pi|_U=p\circ\Phi_U$. The set of sections $\Gamma(V)$ is the set of continuous functions $\sigma:X\to V$ such that $\sigma(x)\in\pi^{-1}(x)$ for each $x\in X$. A section $\sigma$ \textit{vanishes at infinity} if for every $\ep>0$ there exists a compact set $K\subseteq X$ such that
	\begin{equation*}
	\normof{\sigma|_{X\bs K}}_{\infty}<\ep.
	\end{equation*}
	We denote the set of sections that vanish at infinity $\Gamma_0(V)$.
	A vector bundle $V$ over $X$ is \textit{Hermitian} if there exists a family of Hermitian inner-products $\{\IP{\cdot,\cdot}_x\}_{x\in X}$ on the fibres $\pi^{-1}(x)$ that varies continuously in the sense that for all sections $\sigma_1,\sigma_2$, the function $x\mapsto \IP{\sigma_1(x),\sigma_2(x)}_x$ is a continuous function from $X$ to $\C$.\\
	If $X$ is a locally compact Hausdorff space and $V\to X$ is a locally trivial hermitian vector bundle, then the set $\Gamma_0(V)$ of sections of $V$ that vanish at infinity is a Hilbert $C_0(X)$-module. The right action is defined by
	\begin{equation*}
	(\sigma\cdot f)(x)=f(x)\sigma(x).
	\end{equation*}
	The $C_0(X)$-valued inner-product is defined by
	\begin{equation*}
	\IP{\sigma_1,\sigma_2}(x)=\IP{\sigma_1(x),\sigma_2(x)}_x.
	\end{equation*}
	This inner-product is full. To see this, we aim to apply the Stone-Weierstrass Theorem. So fix $x,y\in X$ with $x\neq y$. Let $U\ni x$ be an open set in $X$ over which we have trivialisation $\Phi:V|_U\to U\times\C^n$. Let $V\subset U$ be an open set containing $x$ and not containing $y$. Let $v\in\C^n$ be a unit vector. By Urysohn's Lemma there exists a continuous function $\phi:X\to \R$ vanishing outside of $V$ such that $\phi(x)=1$. Define $\sigma:X\to V$ by $\sigma(z)=\phi(z)\Phi^{-1}(z,v)$ for $z\in X$. Then $\sigma\in\Gamma_0(V)$ and we have $\IP{\sigma,\sigma}(x)=\phi(x)^2\IP{v,v}=1$ and $\IP{\sigma,\sigma}(y)=0$. Hence
	\begin{equation*}
	B=\{\IP{f,g}:f,g\in\Gamma_0(X)\}\subseteq C_0(X)
	\end{equation*}
	separates points in $X$ and vanishes nowhere. This $B$ is a sub-algebra of $C_0(X)$ because $\IP{\cdot,\cdot}$ is $C_0(X)$-linear. Hence the Stone-Weierstrass theorem states that $\overline{B}=C_0(X)$. Thus $\IP{\cdot,\cdot}_{C_0(X)}$ is full.
\end{example}

\begin{example}\label{direct sum}
	Given two Hilbert $A$-modules $X$ and $Y$, their vector space direct sum $X\oplus Y$ is a Hilbert $A$-module under the right action
	\begin{equation*}
	(x,y)\cdot a=(x\cdot a,y\cdot a)
	\end{equation*}
	and inner-product
	\begin{equation*}
	\IP{(x,y),(z,w)}=\IP{x,z}+\IP{y,w}.
	\end{equation*}
	In particular, $X^n=\bigoplus_{i=1}^nX$ is a Hilbert module.\\
\end{example}

Recognising an internal direct sum turns out to be more difficult. If we were to write a Hilbert $A$-module $X$ as a direct sum $X=X_0\oplus X_1$ we would like each $X_i$ to be a module in its own right. This leads us to the following definition:
\begin{definition}\label{def: Direct sum of Hilbert modules}
	Let $X$ be a Hilbert $A$-module. Given two subspaces $X_0$ and $X_1$, we say that $X$ is the \emph{direct sum of $X_0$ and $X_1$} if $X=X_0\oplus X_1$ as vector spaces and $\IP{x_0,x_1}=0$ for all $x_0\in X_0$ and $x_1\in X_1$. We write $X=X_0\oplus X_1$ and denote this by the Hilbert module direct sum.
\end{definition}
\begin{remark}
	Note that requiring $X_0$ and $X_1$ to be orthogonal forces $X_iA\subseteq X_i$, since for $a\in A$, $x_0\in X_0$ and $x_1\in X_1$,
	\begin{equation*}
	\IP{x_0,x_1a}=\IP{x_0,x_1}a=0=a^*\IP{x_0,x_1}=\IP{x_0a,x_1}.
	\end{equation*}
	Thus each $X_i$ is a sub-module. A first definition may be to require only that each $X_i$ be a Hilbert sub-module without any orthogonality requirements. Under this definition however it is not necessarily the case that $X_0\oplus X_1$ and $X$ are isomorphic (see Definition~\ref{Heidi Klum}).
\end{remark}
\begin{example}
	For any $C^*$-algebra $A$, we have just seen in Example~\ref{direct sum} that the tuple $A^n=\bigoplus_{i=1}^nA_A$ is a Hilbert $A$-module. We call these \emph{free modules}. Consider a unitization (see section~\ref{Multiplier}) $A^{\sim}$ of $A$, and let $p\in M_n(A^{\sim})\cong M_n(\C)\otimes A^{\sim}$ be a projection. Then $pA^n$ is a Hilbert $A$-module with right action
	\begin{equation*}
	\begin{pmatrix}
	a_1\\
	\vdots\\
	a_n\\
	\end{pmatrix}\cdot a=
	\begin{pmatrix}
	a_1a\\
	\vdots\\
	a_na\\
	\end{pmatrix}
	\end{equation*}
	and inner-product
	\begin{equation*}
	\left\langle\begin{pmatrix}
	a_1\\
	\vdots\\
	a_n
	\end{pmatrix},\begin{pmatrix}
	b_1\\
	\vdots\\
	b_n
	\end{pmatrix}\right\rangle=\sum_{i,j=1}^na_i^*b_j.
	\end{equation*}
\end{example}

\begin{definition}\label{def: standard moduel}
	Let $A$ be a $C^*$-algebra and let $\HH_A$ be the set of sequences $(a_n)\in A^{\infty}=\prod_{i=1}^{\infty}A$ such that $\sum_{n=1}^{\infty}a^*_na_n$ converges in $A$. If $\sum_{i=1}^{\infty}a_i^*a_i$ and $\sum_{i=1}^{\infty}b_i^*b_i$ both converge then for any finite sum, the Cauchy Schwartz inequality \eqref{C-S} on the Hilbert module $A_A^n$ shows
	\begin{align*}
	\Bigg\|\sum^n_{i=1}a_i^*b_i\Bigg\|^2&=\Bigg\|\Bigg\langle\begin{pmatrix}
	a_1\\
	\vdots\\
	a_n\\
	\end{pmatrix},
	\begin{pmatrix}
	b_1\\
	\vdots\\
	b_n\\
	\end{pmatrix}
	\Bigg\rangle\Bigg\|^2_{A_A^n}\\
	&=\Bigg\|\Bigg\langle\begin{pmatrix}
	a_1\\
	\vdots\\
	a_n\\
	\end{pmatrix},
	\begin{pmatrix}
	b_1\\
	\vdots\\
	b_n\\
	\end{pmatrix}
	\Bigg\rangle^*
	\Bigg\langle\begin{pmatrix}
	a_1\\
	\vdots\\
	a_n\\
	\end{pmatrix},
	\begin{pmatrix}
	b_1\\
	\vdots\\
	b_n\\
	\end{pmatrix}
	\Bigg\rangle
	\Bigg\|_{A_A^n}\\
	&\leq \Bigg\| \begin{pmatrix}
	a_1\\
	\vdots\\
	a_n\\
	\end{pmatrix}\Bigg\|_{A_A^n}\Bigg\|
	\begin{pmatrix}
	b_1\\
	\vdots\\
	b_n\\
	\end{pmatrix}\Bigg\|_{A_A^n}=\Bigg\|\sum_{i=1}^na_i^*a_i\Bigg\|\Bigg\|\sum_{i=1}^nb_i^*b_i\Bigg\|\\
	\end{align*}
	which shows that $\sum_{i=1}^{\infty}a_i^*b_i$ converges. If $\sum_{i=1}^{\infty}a_i^*a_i$ converges then for $a\in A$
	\begin{align*}
	\sum_{i=1}^n(a_ia)^*a_ia=a^*\Big(\sum_{i=1}^na_i^*a_i\Big)a.
	\end{align*}
	Thus, since the sum converges, we have $(a_ia)$ belonging to $\HH_A$.
	We then define the $A$-valued inner-product on $\HH_A$ by
	\begin{equation*}
	\IP{(a_n),(b_n)}=\sum_{n=1}^{\infty}a^*_nb_n
	\end{equation*}
	and a diagonal right action defined by $(a_n)\cdot a=(a_na)$. We have just shown that $\HH_A$ is the countable Hilbert space direct sum of the Hilbert module $A_A$, the details of which are discussed in Appendix~\ref{Direct sums}.
	It can be shown that $\HH_A$ is complete with respect to the inner-product norm, so $\HH_A$ is a Hilbert $A$-module. This is not the only way of constructing $\HH_A$. Let $\HH$ be any infinite dimensional separable Hilbert space with an orthonormal basis $(e_j)$. Then there is an isomorphism (see Definition~\ref{Heidi Klum}) from the completion of the vector space tensor product $\HH\odot A$ in the norm induced by the inner-product
	\begin{equation*}
	\Big\langle\xi\otimes a,\eta\otimes b\Big\rangle=\IP{\xi,\eta}a^*b
	\end{equation*}
	to $\HH_A$ that acts on elementary tensors by $h\otimes a\mapsto (\IP{h,e_i}a)_{i=1}^{\infty}$. This point of view is taken in \cite[Eaxmple~1.2.6]{Kasnotes}.
\end{definition}

\section{Adjointable operators}

We think of Hilbert $A$-modules as analogues of Hilbert spaces in which the scalar field $\C$ is replaced by a $C^*$-algebra $A$. In particular, a Hilbert $\C$-module is simply a Hilbert space. One of the first things we do with Hilbert modules is to investigate analogues of the usual Hilbert space machinery.
\begin{definition}
	If $A$ is a $C^*$-algebra and we have Hilbert $A$-modules $X$ and $Y$, we say that a map $T:X\to Y$ is \emph{adjointable} if there exists a map $S:Y\to X$ such that
	\begin{equation*}
	\IP{Tx,y}_A=\IP{x,Sy}_A
	\end{equation*}
	for all $x\in X$ and $y\in Y$.
\end{definition}
Our first deviation from the Hilbert space case is that not all bounded linear operators are adjointable.
\begin{example}\label{Wang}
	This example is taken from \cite[page 8]{Lance}. Let $X$ be a compact space and $U$ be a dense, proper, open subset (for example, $X=[0,1]$ and $U=(0,1)$). Set $Y=C(X)$ and $Z=\{f\in C(X):f|_{X\bs U}=0\}\cong C_0(U)$. Then $Y$ and $Z$ are both Hilbert $C(X)$-modules as in Example~\ref{dang}. Consider the inclusion mapping $\iota:Z\to Y$. Clearly $\iota$ is linear and bounded with norm 1. For the function $1\in Y$ and any $f\in Z$ we have
	\begin{equation*}
	\IP{\iota(f),1}=\overline{f}.
	\end{equation*}
	Suppose for contradiction that $S:Y\to Z$ is an adjoint for $\iota$. Then
	\begin{equation*}
	\overline{f}=\IP{f,S(1)}=\overline{f}S(1)
	\end{equation*}
	for all $f\in Z$. That is, $\overline{f}(x)=\overline{f}(x)S(1)(x)$ for all $x\in U$ and $f\in Z$, forcing $S(1)(x)=1$ for all $x\in U$. By definition of $Z$ we have $S(1)(x)=0$ for $x\in X\bs U$, and since $U$ is dense in $X$, we deduce that $S(1)$ is not continuous, contradicting $S(1)\in Z$. Hence $\iota$ is not adjointable.
\end{example}
This leads us to our next definition: that of the space of adjointable operators $\LL(X)$. The adjointable operators are the Hilbert-module equivalent of $\BB(\HH)$ for Hilbert spaces.
\begin{definition}
	Given Hilbert modules $X_A$ and $Y_A$, the set $\LL(X;Y)$ is the set 
	\begin{equation*}
	\LL(X;Y)=\{T:X\to Y:T\text{ is adjointable}\}
	\end{equation*}
	of adjointable operators from $X$ to $Y$. If $X=Y$ then we write $\LL(X)$ for $\LL(X;X)$.
\end{definition}
It turns out that adjointable operators are automatically linear, $A$-linear, bounded and have unique adjoints. First we prove a useful lemma.
\begin{lemma}\label{lemma: norm is sup of IPs}
	Let $X$ be a Hilbert $A$-module. Then for all $x,y\in X$ and $a\in A$ we have $\normof{\IP{x,y}}\leq \normof{x}\normof{y}$ and $\normof{x\cdot a}\leq \normof{x}\normof{a}$. We also have
	\begin{equation*}
	\normof{x}=\sup\{\normof{\IP{x,y}}:y\in X,\normof{y}\leq 1\}.
	\end{equation*}
	In particular, if $\IP{x,z}=\IP{y,z}$ for all $z\in X$ then
	\begin{equation*}
		\normof{x-y}=\sup\{\normof{\IP{x-y,z}}:z\in X, \normof{z}\leq 1\}=0,
	\end{equation*}
	so $x=y$.
\end{lemma}
\begin{proof}
	This proof follows that of \cite[Corollary~1.15,Corollary~1.16]{Crossedandunitizations}.
	The first assertion follows from taking norms in the Cauchy-Schwartz inequality \eqref{C-S}. For the second assertion, fix $x\in X$ and $a\in A$. We compute
	\begin{align*}
		\normof{xa}^2&=\normof{\IP{xa,xa}}=\normof{a^*\IP{x,x}a}\leq \normof{a^*}\normof{\IP{x,x}}\normof{a}=\normof{x}\normof{a}^2.\\
	\end{align*}
	Now the proof of the final assertion is the same argument as for Hilbert spaces. If $x=0$ the result is immediate. So fix $x\in X$ non-zero. Then $x/\normof{x}$ has norm equal to one and so
	\begin{equation*}
		\sup_{\normof{y}\leq 1}\normof{\IP{x,y}}\geq \normof{\IP{x,x/\normof{x}}}=\normof{x}.
	\end{equation*}
	For the reverse inequality we compute
	\begin{equation*}
		\sup_{\normof{y}\leq 1}\normof{\IP{x,y}}\leq\sup_{\normof{y}\leq 1}\normof{x}\normof{y}=\normof{x}.
	\end{equation*}
\end{proof}
\begin{lemma}\label{Rang}
	Let $X$ and $Y$ be Hilbert $A$-modules. Every adjointable operator $T:X\to Y$ is linear, $A$-linear and bounded, and there is a unique operator $T^*:Y\to X$ such that $\IP{Tx,y}=\IP{x,T^*y}$ for all $x\in X$ and $y\in Y$.
\end{lemma}
\begin{proof}
	This proof is taken from \cite[Lemma 2.18]{RaeburnWilliams}. If $T$ is adjointable, let $T^*$ be an adjoint for $T$. Then for all $x,y\in X$, $z\in Y$ and $a,b\in A$
	\begin{align*}
	\IP{T(x\cdot a+y\cdot b),z}&=\IP{x\cdot a+y\cdot b,T^*z}=\IP{x\cdot a,T^*z}+\IP{y\cdot b,T^*z}\\
	&=a^*\IP{x,T^*z}+b^*\IP{y,T^*z}=\IP{(Tx)\cdot a+(Ty)\cdot b,z}.\\
	\end{align*}
	By Lemma~\ref{lemma: norm is sup of IPs} we deduce that $T$ is $A$-linear. A similar calculation shows that $T$ is complex-linear. To see uniqueness, suppose that $R,S:Y\to X$ satisfy 
	\begin{equation*}
		\IP{x,Ry}=\IP{Tx,y}=\IP{x,Sy}
	\end{equation*}
	for all $x\in X$ and $y\in Y$. Then we have
	\begin{equation*}
	0=\IP{Tx,y}-\IP{Tx,y}=\IP{x,Ry}-\IP{x,Sy}=\IP{x,Ry-Sy}.
	\end{equation*}
	So we determine $Ry=Sy$ for all $y\in Y$, hence $T^*=S$. So the adjoint of $T$ is unique. Now we wish to show that $T$ is bounded. If $x_n\to x$ and $T(x_n)\to z$, then $\IP{T(x_n),y}\to \IP{z,y}$ and 
	\begin{equation*}
		\IP{T(x_n),y}=\IP{x_n,T^*(y)}\to \IP{x,T^*(y)}=\IP{T(x),y}.
	\end{equation*} 
	Hence we have $T(x)=z$. Thus, $T$ has closed graph, and so by the Closed Graph Theorem \cite[I.2.1.5]{Encyclopedia}, $T$ is bounded.
\end{proof}

It can be shown that the usual operator norm on $\LL(X)$ obeys the $C^*$-identity and that $\LL(X)$ is complete with respect to the operator norm, hence $\LL(X)$ is a $C^*$-algebra. It may be tempting to think that the right action of $A$ on $X_A$ defines adjointable operators $R_a(x)=xa$ on $X_A$ with adjoint $(R_a)^*=R_{a^*}$. But this is not the case: Let $A$ be any non-commutative, unital $C^*$-algebra and fix $a,b\in A$ with $ab\neq ba$. Define $R_b:A_A\to A_A$ by $R_b(x)=xb$. Then $R_b(1a)=(1a)b=ab$, whereas $R_b(1)a=1ba=ba$. So $R_a$ is not $A$-linear, and it follows from Lemma~\ref{Rang} that $R_a$ is not adjointable.
However, the next example shows that left multiplication is adjointable.\\
\begin{example}\label{left multn is adjointable}
	Consider the Hilbert $A$-module $A_A$ of Example~\ref{dang}. For each $a\in A$, define $L_a:A_A\to A_A$ by $L_a(b)=ab$. Then $L_a$ is adjointable with adjoint $L_a^*=L_{a^*}$. To see this we calculate
	\begin{equation*}
		\IP{L_a(b),c}=\IP{ab,c}=(ab)^*c=b^*(a^*c)=\IP{b,a^*c}=\IP{b,L_{a^*}c}.
	\end{equation*}
	We claim that the operator norm
	\begin{equation}\label{norm}
		\normof{L_a}=\sup_{\normof{b}\leq 1}\normof{ab}
	\end{equation}
	of $L_a$ is the $C^*$-norm $\normof{a}$. We saw in Example~\ref{dang} that the Hilbert module norm on $A_A$ is the $C^*$-norm on $A$, so sub-multiplicativity applied to Equation~\eqref{norm} shows that $\normof{L_a}\leq \normof{a}$. Letting $b$ equal $a^*/\normof{a}$ in Equation~\eqref{norm} then gives the reverse inequality.\\
	More generally, if $B$ is a $C^*$-algebra and $\phi:B\to A$ is a homomorphism, then $\phi$ can be considered a homomorphism $\widetilde{\phi}:B\to\LL(A_A)$ that takes $b$ to the left multiplication operator $L_{\phi(a)}$. We then obtain $\normof{\widetilde{\phi}(b)}\leq\normof{b}$ since $\widetilde{\phi}$ is a $C^*$-homomorphism.
\end{example}
	Recall that an element $p$ of a $C^*$-algebra $A$ is a \emph{projection} if $p^2=p=p^*$, and an element $u$ of a unital $C^*$-algebra $A$ is \emph{unitary} if $u$ is invertible with $u^*=u^{-1}$. The following proposition gives us useful criteria for checking whether elements of $\LL(X,Y)$ are unitary.
\begin{proposition}\label{Jeddabub}
	Let $X$ and $Y$ be Hilbert $A$-modules and let $U:X\to Y$ be complex-linear. Then $U$ is unitary if and only if it is surjective and satisfies
	\begin{equation}
	\IP{Ux,Ux}=\IP{x,x}\label{Jedda}
	\end{equation}
	for all $x\in X$.
\end{proposition}
\begin{proof}
	Suppose that $U$ is surjective and satisfies Equation \eqref{Jedda}. We first claim that Equation \eqref{Jedda} implies $\IP{Ux,Uy}=\IP{x,y}$ for all $x,y\in X$. Indeed, using the polarisation identity, we have
	\begin{align*}
	\IP{Ux,Uy}&=\frac{i-1}{2}\IP{Uy,Uy}+\frac{i-1}{2}\IP{Ux,Ux}-\frac{i}{2}\IP{U(x+iy),U(x+iy)}\\
	&\qquad\qquad+\frac{1}{2}\IP{U(x+y),U(x+y)}\\
	&=\frac{i-1}{2}\IP{y,y}+\frac{i-1}{2}\IP{x,x}-\frac{i}{2}\IP{x+iy,x+iy}+\frac{1}{2}\IP{x+y,x+y}\\
	&=\IP{x,y}.\\
	\end{align*}
	If $Ux=0$ then
	\begin{equation*}
	0=\normof{\IP{Ux,Ux}}=\normof{\IP{x,x}}=\normof{x}^2
	\end{equation*}
	and so $x=0$. Hence $\Ker(U)=\{0\}$. Since $U$ is linear by Lemma~\ref{Rang}, we deduce that $U$ is injective. As $U$ is surjective by hypothesis, $U$ has an inverse $U^{-1}$. To see that $U^*=U^{-1}$ it suffices by Lemma~\ref{Rang} to show that
	\begin{equation*}
	\IP{Ux,y}=\IP{x,U^{-1}y}
	\end{equation*}
	for all $x,y\in X$, which follows from the calculation
	\begin{align*}
	\IP{Ux,y}&=\IP{Ux,UU^{-1}y}\\
	&=\IP{x,U^{-1}y}\qquad\text{by hypothesis \eqref{Jedda}.}\\
	\end{align*}
	So we deduce that $U$ is unitary. Conversely suppose that $U$ is unitary. Then $U$ must be surjective since it has an inverse and we may compute
	\begin{equation*}
	\IP{Ux,Ux}=\IP{x,U^*Ux}=\IP{x,U^{-1}Ux}=\IP{x,x}
	\end{equation*}
	and so $U$ satisfies Equation \eqref{Jedda} and bob's your uncle.
\end{proof}

\section{Compact operators}

Next we would like to extend our notion of compact operators to Hilbert modules. Whilst it is true for Hilbert spaces that the compact operators can be equivalently defined both as the ideal generated by rank one operators and as the operators under which the image of the closed unit ball is compact, this does not hold in the Hilbert module setting. Our definition is that of the closure of the finite rank operators.
\begin{definition}
	Let $X_A$ be a Hilbert $A$-module. For each $x,y\in X$, define the operator\\ $\Theta_{x,y}:X\to X$ by
	\begin{equation*}
	\Theta_{x,y}(z)=x\cdot \IP{y,z}.
	\end{equation*}
	For $x,y\in X$, we call $\Theta_{x,y}$ a \emph{rank one operator}. 
	We define
	\begin{equation*}
	\KK(X)=\overline{\Span}\{\Theta_{x,y}:x,y\in X\}
	\end{equation*}
	in $\LL(X)$ and refer to $\KK(X)$ as the \emph{compact operators} on $X$. It is important to note that this closure is taken with respect to the operator norm on $\LL(X)$. We call the subspace
	\begin{equation*}
	\Span\{\Theta_{x,y}:x,y\in X\}
	\end{equation*}
	of $\LL(X)$ the \emph{finite rank} operators.
\end{definition}
\begin{remark}
	In some texts the operators $\Theta_{x,y}$ are called generalised rank one operators. We have chosen to simply refer to them as rank one operators for brevity.
\end{remark}
\begin{lemma}
	The operators $\Theta_{x,y}$ are adjointable with adjoint $\Theta_{y,x}$. For all $T\in\LL(X)$ we have 
	\begin{equation*}
		T\Theta_{x,y}=\Theta_{Tx,y}\qquad\text{ and }\qquad \Theta_{x,y}T=\Theta_{x,T^*y}.
	\end{equation*}
	In particular, the compact operators $\KK(X)$ are in ideal in $\LL(X)$.
\end{lemma}
\begin{proof}
	For $x,y,w,z\in X$ the first assertion follows from the computation
	\begin{align*}
	\IP{\Theta_{x,y}w,z}&=\IP{x\IP{y,w},z}=\IP{y,w}^*\IP{x,z}\\
	&=\IP{w,y}\IP{x,z}=\IP{w,y\IP{x,z}}=\IP{w,\Theta_{y,x}z}.\\
	\end{align*}
	Now we wish to show that $\KK(X)$ is an ideal.
	For any $T\in\LL(X)$ we compute
	\begin{equation*}
		\Theta_{x,y}Tz=x\IP{y,Tz}=\Theta_{x,T^*y}z\qquad\text{ and }\qquad T\Theta_{x,y}z=Tx\IP{y,z}=\Theta_{Tx,y}z,
	\end{equation*}
	which shows that $\Span\{\Theta_{x,y}:x,y\in X\}$ is a (not closed) ideal in $\LL(X)$. If $T=0$ then clearly $TK\in\KK(X)$ for all $K\in\KK(X)$. Fix $\ep>0$ and $T\in\LL(X)$ with $T\neq 0$. If $S$ is a norm limit of finite rank operators $K_n=\sum_{i=1}^{k_n}\Theta_{x^i_n,y^i_n}$, then choosing $n$ large enough such that $\normof{S-K_n}<\ep/\normof{T}$ we have
	\begin{align*}
	\normof{TS-TK_n}&=\Bignormof{TS-\sum_{i=1}^n\Theta_{Tx_{n}^i,y^i_n}}\\
	&=\Bignormof{T(S-\sum_{i=1}^n\Theta_{x_{n}^i,y^i_n})}\\
	&\leq\normof{T}\normof{S-K_n}<\ep.\\
	\end{align*}
	Hence $TS$ is the norm limit of finite rank operators $TK_n$ and $\KK(X)$ is a left ideal. Taking adjoints shows that $\KK(X)$ is a two sided ideal of $\LL(X)$.
\end{proof}
\begin{example}\label{Fang}
	Consider the Hilbert $A$-module $A_A$ of Example~\ref{dang}. The rank one operators $\Theta_{a,b}$ look like
	\begin{equation*}
	\Theta_{a,b}c=ab^*c,
	\end{equation*}
	which is the left multiplication operator $(ab^*)\cdot$ from Example~\ref{left multn is adjointable}. We show that the left multiplication operators by $A$ on $A_A$ are exactly the compact operators $\KK(A_A)$. In Example~\ref{dang} we saw that the Hilbert module norm of $a$ is the $C^*$-norm of $a$, and in Example~\ref{left multn is adjointable} we saw that the operator norm of $a\cdot$ is also the $C^*$-norm of $A$. That these norms are all the same can hide what we are doing in the following calculations, so we will denote the operator norm of $a\cdot$ and the $C^*$-norm of $a$ by $\normof{a}_{\LL(A_A)}$ and $\normof{a}_A$. We compute that
	\begin{equation}\label{isom}
	\Bignormof{\sum_i\Theta_{a_i,b_i}}_{\LL(A_A)}=\Bignormof{\sum_i(a_ib_i^*)\cdot}_{\LL(A_A)}=\Bignormof{\sum_ia_ib_i^*}_{A},
	\end{equation}
	so if $\sum_i\Theta_{a_i,b_i}=\sum_j\Theta_{c_j,d_j}$ then $\sum_ia_ib_i^*=\sum_jc_jd_j^*$. There is then a well defined map 
	\begin{equation*}
	\psi:\Span\{\Theta_{a,b}:a,b\in A\}\to A
	\end{equation*} 
	such that $\psi(\sum_i\Theta_{a_i,b_i})=\sum_ia_ib_i^*$. This map is isometric by Equation \eqref{isom}, and it is a homomorphism since 
	\begin{align*}
	\psi\Big(\big(\sum_i\Theta_{a_i,b_i}\big)\big(\sum_j\Theta_{c_j,d_j}\big)\Big)&=\psi\big(\sum_{i,j=0}\Theta_{a_ib^*_ic_j,d_j}\big)\\
	&=\sum_{i,j}a_ib^*_ic_jd_j^*=\psi\big(\sum_i\Theta_{a_i,b_i}\big)\psi\big(\sum_j\Theta_{c_j,d_j}\big)\\
	\end{align*}
	and
	\begin{align*}
	\psi\Big(\big(\sum_i\Theta_{a_i,b_i}\big)^*\Big)&=\psi\Big(\sum_i\Theta_{b_i,a_i}\Big)\\
	&=\sum_ib_ia_i^*=\psi\big(\sum_i\Theta_{a_i,b_i}\big)^*.\\
	\end{align*}
	By Theorem~\ref{Sang}, $\psi$ extends as an injective homomorphism from the compacts $\KK(A_A)$ to $A$ such that if $K_n$ is a convergent sequence of finite rank operators then
	\begin{equation*}
	\psi(\lim_{n\to\infty}K_n)=\lim_{n\to\infty}\psi(K_n).
	\end{equation*}
	Since products af the form $ab^*$ are dense in $A$ (take $b$ to be an approximate identity) we deduce that $\psi$ is surjective, and hence we have an isomorphism $\KK(A_A)\cong A$.  We will see in Section~\ref{Multiplier} that $\LL(A_A)$ is the `maximal unitization' of $A$.
\end{example}
Similar to the orthogonal compliment of a subspace of a Hilbert space, given a sub-module $Y\subset X$, we define
\begin{equation*}
Y^{\perp}=\{x\in X:\IP{x,y}=0\;\forall\; y\in Y\}.
\end{equation*}
The set $Y^{\perp}$ is always a closed sub-module of $X$, but unlike for Hilbert spaces, even if $Y$ is closed we need not have $(Y^{\perp})^{\perp}=Y$, nor do we always have $X=Y\oplus Y^{\perp}$. For example, if $Y$ and $Z$ are as in Example~\ref{Wang}, then $Z^{\perp}$ is the trivial module $\{0\}$, so $(Z^{\perp})^{\perp}=Y\neq Z$, and $Z^{\perp}\oplus Z=Z\neq Y$. This example is also due to \cite{Lance}.\\
\vskip 2pt
We have been working solely with the operator norm
\begin{equation*}
\normof{T}=\sup_{\normof{x}\leq 1}\normof{Tx}
\end{equation*}
on $\LL(X)$, under which it is a $C^*$-algebra. There are other topologies on $\LL(X)$ which can be useful in various situations.
\begin{definition}
	Let $X$ be a Hilbert $A$-module and $T\in\LL(X)$. We say a net $(T_{\lambda})_{\lambda\in\Lambda}$ of operators in $\LL(X)$ converges to $T$ \emph{strongly} if
	\begin{equation*}
	\lim_{\lambda\in\Lambda}\normof{T_{\lambda}x-Tx}\to 0
	\end{equation*}
	for every $x\in X$. We say that $(T_{\lambda})$ converges to $T$ \emph{weakly} if
	\begin{equation*}
	\lim_{\lambda\in\Lambda}\normof{\IP{T_{\lambda}x,y}-\IP{Tx,y}}\to 0
	\end{equation*}
	for every $x,y\in X$.
	We say that $(T_{\lambda})$ converges to $T$ \emph{strictly} if
	\begin{equation*}
	\lim_{\lambda\in\Lambda}\normof{T_{\lambda}S-TS}= 0
	\end{equation*}
	for all $S\in\KK(X)$.
\end{definition}
The inequality $\normof{Tx}\leq \normof{T}\normof{x}$ shows that norm convergence implies strong convergence and the Cauchy-Schwartz inequality shows that strong convergence implies weak convergence.
By \cite[Proposition 2.31]{RaeburnWilliams} every $x\in X$ can be written in the form $x=y\IP{y,y}=\Theta_{y,y}y$ for some $y\in X$. This is called Cohen Factorisation. If $T_{\lambda}\to T$ strictly then $T_{\lambda}\Theta_{y,y}\to T\Theta_{y,y}$ in norm. Since norm convergence implies strong convergence, $T_{\lambda}\Theta_{y,y}(y)\to T\Theta_{y,y}(y)$, so
\begin{equation*}
\normof{T_{\lambda}x}=\normof{T_{\lambda}\Theta_{y,y}(y)}\to \normof{T\Theta_{y,y}(y)}=\normof{Tx},
\end{equation*}
so strict convergence implies strong convergence. Thus we have the hierarchy
\begin{align*}
\text{Norm convergence}&\implies\text{Strict convergence}\\
&\implies\text{Strong convergence}\\
&\implies\text{Weak convergence}.
\end{align*}
\begin{example}
	Let $A$ be a $C^*$-algebra and consider the Hilbert module $A_A$. We saw in Example~\ref{Fang} that $\KK(A_A)\cong A$. If we instead consider the strong closure of the finite rank operators, then for $\{e_{\lambda}\}$ an approximate identity, the we have
	\begin{equation*}
	\lim_{\lambda\to\infty}\Theta_{e_{\lambda},e_{\lambda}}a=a
	\end{equation*}
	for all $a\in A$, so $\Theta_{e_{\lambda},e_{\lambda}}\to 1$ strongly. Hence the strong closure of the finite rank operators is an ideal of $\LL(A_A)$ containing 1, and so is all of $\LL(A_A)$.
\end{example}
\begin{example}\label{ralph}
	If $\HH$ is a Hilbert space (that is, a Hilbert $\C$-module) then for an orthonormal basis $\{e_i\}$ the sums $\sum_{i=1}^N\Theta_{e_i,e_i}$ converge strongly to the identity. Thus the strong (and hence weak) closure of the finite rank operators is an ideal containing 1 so is all of $\BB(\HH)$. It is not difficult to show that the closure of the finite ranks in strict topology is $\BB(\HH)$ too.
\end{example}
Indeed in any Hilbert module $X$, the strict closure of the finite rank operators is $\LL(X)$. We'd like to pick an `orthonormal basis' for $X$ and show that we can approximate the identity with projections onto this basis as in Example~\ref{ralph}, but the concept of orthonormal basis doesn't quite carry over to Hilbert modules. It is still possible to do this using frames, which we will only very briefly touch on, but is described in detail in \cite[1.4.2]{Kasnotes}.

\section{Kasparov's stabilisation theorem}\label{section: Kasparovs stabilisation theorem}

\begin{definition}\label{Heidi Klum}
	We say that two Hilbert modules $X_A$ and $Y_B$ are \textit{isomorphic} if there exists a linear isomorphism $t:X\to Y$ and an isomorphism $\pi:A\to B$ such that
	\begin{enumerate}
		\item[(1)]
		$t(xa)=t(x)\pi(a)$\\
		\item[(2)]
		$\normof{t(x)}=\normof{x}$\\
	\end{enumerate}
	for all $x\in X$ and $a\in A$. We say that $X$ is \textit{countably generated} if there exists a countable set of elements $\{x_j\}_{j\geq 1}$ in $X$ such that
	\begin{equation*}
		X=\overline{\Span}\{x_ja:j\geq 1, a\in A\}.
	\end{equation*}
\end{definition}
\begin{remark}
	Given two right $A$-modules $X_A$ and $Y_A$ over the same algebra it suffices to show that there exists a unitary operator between $X$ and $Y$ to show that $X$ and $Y$ are isomorphic.
\end{remark}
\begin{example}
	Let $A$ be a $C^*$-algebra with $I\triangleleft A$ a non-trivial ideal. Then we may consider $I$ as a Hilbert $I$-module $I_I$ with $I$ acting on the right by multiplication and we may consider $I$ as a right $A$-module $I_A$ with $A$ acting on the right by multiplication. The Hilbert module $I_I$ is full whilst $I_A$ is not. According to Definition~\ref{Heidi Klum} these modules cannot be isomorphic since $I$ is not isomorphic to $A$, however the identity map $\text{id}:I_I\to I_A$ satisfies all other conditions of Definition~\ref{Heidi Klum}.
\end{example}
\begin{theorem}[Kasparov's Stabilisation theorem]\label{stab theorem}
	Let $A$ be a $C^*$-algebra and $X_A$ a countably generated Hilbert $A$-module. There exists an injective adjointable map $V:X_A\hookrightarrow\HH_A$ such that $V^*V=1_X$, and in particular $X_A\oplus\HH_A\cong \HH_A$.
\end{theorem}
The proof does not require too much more machinery than we have built, and can be found in \cite[Theorem 5.49]{RaeburnWilliams}. Kasparov's stabilisation theorem tells us that $\HH_A$ is in some sense a `universal' Hilbert $A$-module, in that every right $A$-module may be embedded in $\HH_A$. This allows to pull back a frame on $\HH_A$ to a frame on any Hilbert $A$ module.
\begin{definition}
	Let $X$ be a Hilbert $A$-module. A \emph{frame} for $X$ is a sequence $(x_n)$ of elements in $X$ such that
	\begin{equation*}
	\sum_{i=1}^n\Theta_{x_i,x_i}
	\end{equation*}
	is an approximate unit for $\KK(X)$. We say that a $C^*$-algebra $A$ is \emph{$\sigma$-unital} if there exists a sequential approximate identity.
\end{definition}
\begin{lemma}\label{frames exist}
	Let $X$ be a countably generated Hilbert $A$ module with $A$ $\sigma$-unital, and let $V:X\to\HH_A$ be the stabilisation map of Kasparov's Stabilisation Theorem~\ref{stab theorem}. Then there is frame $e_{ij}$ of $\HH_A$ such that $V^*e_{ij}$ is a frame for $X$.
\end{lemma}
\begin{proof}
	This proof is adapted from \cite[1.4.2]{Kasnotes}.
	We begin by constructing a frame $e_{ij}$ of $\HH_A$. Let $u_i$ be an approximate identity for $A$, and define $v_i=\sqrt{u_i-u_{i-1}}$ for $i>1$ and $v_1=\sqrt{u_1}$ so that $\sum_{i=1}^nv_i^2=u_n$. For each $i$ and $j$ define $(e_{ij})\in \HH_A$ to be the sequence
	\begin{equation*}
	(e_{ij})_n=\begin{cases}
	v_i&n=j\\
	0&n\neq j\\
	\end{cases}.
	\end{equation*}
	Fix $x\in\HH_A$. We compute
	\begin{align*}
	\sum_{i,j=1}^n\Theta_{e_{ij},e_{ij}}x&=\sum_{i,j=1}^ne_{ij}\IP{e_{ij},x}\\
	&=\sum_{i,j=1}^ne_{ij}\sum_{k=1}^{\infty}(e_{ij})^*_kx_k=\sum_{i,j=1}^ne_{ij}v_ix_j.\\
	\end{align*}
	Thus taking the $k$th entry for $k\leq n$, we find
	\begin{align*}
	\left(\sum_{i,j}^n\Theta_{e_{ij},e_{ij}}x\right)_k=\left(\sum_{i,j=1}^ne_{ij}v_ix_j\right)_k=\sum_{i,j=1}^n(e_{ij})_kv_ix_j=\sum_{i=1}^nv_i^2x_k=u_nx_k,\\
	\end{align*}
	and for $k>n$ we have
	\begin{equation*}
	\left(\sum_{i,j}^n\Theta_{e_{ij},e_{ij}}x\right)_k=0.
	\end{equation*}
	Now fix $\ep>0$. 
	Since $x-\sum_{i,j=1}^n\Theta_{e_{ij},e_{ij}}x\in\HH_A$, there is a $n_0$ such that for $n>n_0$ we have
	\begin{equation*}
		\Bignormof{\sum_{k=n_0+1}^{\infty}(x-\sum_{i,j=1}^n\Theta_{e_{ij},e_{ij}}x)^*(x-\sum_{i,j=1}^n\Theta_{e_{ij},e_{ij}}x)}<\ep/2.
	\end{equation*}
	Since $u_i$ is an approximate identity we may choose $m_0$ such that for $n>m_0$
	\begin{equation*}
	\Bignormof{\sum_{i=1}^{n_0}(x_i-u_nx_i)^*(x_i-u_nx_i)}<\ep/2.
	\end{equation*}
	Then for $n\geq N_0=\max\{m_0,n_0\}$ we see that
	\begin{align*}
	\Bignormof{x-\sum_{i,j=1}^n\Theta_{e_{ij},e_{ij}}x}&=\Bignormof{\sum_{k=1}^{\infty}(x-\sum_{i=1}^n\Theta_{e_{ij},e_{ij}}x)_k^*(x-\sum_{i=1}^n\Theta_{e_{ij},e_{ij}}x)_k}\\
	&\leq\Bignormof{\sum_{k=n_0+1}^{\infty}(x-\sum_{i=1}^n\Theta_{e_{ij},e_{ij}}x)_k^*(x-\sum_{i=1}^n\Theta_{e_{ij},e_{ij}}x)_k}\\
	&\quad\;\;+\Bignormof{\sum_{k=1}^{n_0}(x-\sum_{i=1}^n\Theta_{e_{ij},e_{ij}}x)_k^*(x-\sum_{i=1}^n\Theta_{e_{ij},e_{ij}}x)_k}\\
	&<\ep/2+\Bignormof{\sum_{k=1}^{n_0}(x-u_nx)_k^*(x-u_nx)_k}<\ep.\\
	\end{align*}
	Thus we deduce that $\sum_{i=1}^n\Theta_{e_{ij},e_{ij}}$ converges strongly to the identity. We now wish to show that this convergence is in fact strict. Fix a rank one operator $\Theta_{x,y}$. Since we have
	\begin{align*}
	\normof{\Theta_{x,y}-\sum_{i,j=1}^n\Theta_{e_{ij},e_{ij}}\Theta_{x,y}}&=\sup_{\normof{z}\leq 1}\normof{x\IP{y,z}-\sum_{i,j=1}^n\Theta_{e_{ij},e_{ij}}x\IP{x,z}}\\
	&\leq\sup_{\normof{z}\leq 1}\normof{x-\sum_{i,j=1}^n\Theta_{e_{ij},e_{ij}}x}\normof{\IP{y,z}}\\
	&=\normof{x-\sum_{i,j=1}^n\Theta_{e_{ij},e_{ij}}x}\normof{y},\\
	\end{align*}
	and since $\sum_{i,j=1}^n\Theta_{e_{ij},e_{ij}}x$ converges to $x$, we see that $\sum_{i,j=1}^n\Theta_{e_{ij},e_{ij}}\Theta_{x,y}$ converges to $\Theta_{x,y}$. Since $\Theta_{x,y}$ generate $\KK(\HH_A)$ we deduce that $\sum_{i,j=1}^n\Theta_{e_{ij},e_{ij}}$ converges strictly to the identity, that is, $\sum_{i,j=1}^n\Theta_{e_{ij},e_{ij}}$ is an approximate unit for $\KK(\HH_A)$.\\
	Now lastly, since $V$ is an isometry, defining $x_{ij}=V^*e_{ij}$ gives a frame for $X$. We compute for $\Theta_{u,v}$ a rank one operator in $\KK(X)$
	\begin{align*}
	\sum_{i,j=1}^n\Theta_{x_{ij},x_{ij}}\Theta_{u,v}&=V^*\sum_{i,j=1}^n\Theta_{e_{ij},e_{ij}}V\Theta_{u,v}\rightarrow V^*V\Theta_{u,v}=\Theta_{u,v}.\\
	\end{align*}
	Since the rank one operators generate $\KK(X)$ we deduce that $\sum_{i,j=1}^n\Theta_{x_{ij},x_{ij}}$ is an approximate identity for $\KK(X)$.
\end{proof}
\section{$C^*$-correspondences}

\begin{definition}\label{def: correspondence}
	Let $A$ and $B$ be $C^*$-algebras. An $A$--$B$ correspondence is a Hilbert $B$-module $X$ together with a homomorphism $\phi:A\to \LL(X)$ of $A$ into the adjointable operators on $X$. We can think of $\phi$ as a left action of $A$ on $X$ by $a\cdot x=\phi(a)x$, and will often refer to $\phi$ as a left action. We often write $_{\phi}X$, $_\phi X_B$ or $_AX_B$ to emphasise $\phi$, $A$ or $B$.
\end{definition}
\begin{remark}
	We will usually refer to an $A$--$B$ correspondence just as a $C^*$-\emph{correspondence} when we have no reason to emphasise the algebras $A$ and $B$.
\end{remark}
\begin{example}\label{Marnie}
	Given any homomorphism $\phi:B\to A$ of a $C^*$-algebra $A$ onto a $C^*$-algebra $B$, we have already seen in Example~\ref{left multn is adjointable} that $\phi$ determines a homomorphism $\widetilde{\phi}:B\to\LL(A_A)$ by $\widetilde{\phi}(b)=L_{\phi(b)}$. This left action $\widetilde{\phi}$ makes $A_A$ into a $B$--$A$ correspondence. We will stop distinguishing between $\widetilde{\phi}$ and $\phi$ when referring to this example later, and will call both maps $\phi$.
\end{example}
\begin{example}\label{Wahay!}
	Every Hilbert space $\HH$ is a $\KK(\HH)$--$\C$ correspondence with left action given by the inclusion of $\KK(\HH)$ into $\BB(\HH)=\LL(\HH)$. Every Hilbert $A$-module $X$ is a $\KK(X)$--$A$ correspondence in the same way.
\end{example}
\begin{example}\label{Marns}
	Recall that a directed graph $(E^0,E^1,r,s)$ is a quadruple consisting of two countable sets $E^0$ and $E^1$ with two functions $r,s:E^1\to E^0$. For $v,w\in E^0$ we write
	\begin{equation*}
		vE^1w=\{e\in E^1:r(e)=v,s(e)=w\}
	\end{equation*}  
	and we write $E^1w$ and $vE^1$ for the sets $\{e\in E^1:s(e)=w\}$ and $\{e\in E^1:r(e)=v\}$ respectively.
	Let $(E^0,E^1,r,s)$ be a directed graph and give $E^0$ and $E^1$ the discrete topologies. Give $C_c(E^1)$ the $C_0(E^0)$ valued inner-product
	\begin{equation*}
	\IP{f,g}(v)=\sum_{e\in E^1v}\overline{f(e)}g(e)
	\end{equation*}
	with right action of $C_0(E^0)$ given by
	\begin{equation*}
	(f\cdot \xi)(e)=f(e)\xi(r(e))
	\end{equation*}
	and left action $\phi:C_0(E^0)\to \LL(C_c(E^1))$
	\begin{equation*}
	(\phi(\xi)f)(e)=\xi(s(e))f(e).
	\end{equation*}
	The left action $\phi(\xi)$ is adjointable with adjoint $\phi(\bar{\xi})$ because for $f,g\in C_c(E^1)$
	\begin{equation*}
		\IP{\phi(\xi)f,g}(v)=\sum_{e\in E^1v}\overline{\phi(s(e))}\bar{f}(e)g(e)=\IP{f,\phi(\bar{\xi})g}(v).
	\end{equation*}
	Let $X$ be the completion of $C_c(E^1)$ with respect to the norm coming from the inner-product. Then $X$ is a $C_0(E^0)$-$C_0(E^0)$ correspondence. We call $X$ the \emph{graph correspondence} of $E$.
\end{example}
At this stage since we have introduced more structure to go from Hilbert modules to correspondences we must say what it means for correspondences to be isomorphic. 
\begin{definition}\label{def:isomorphic and direct sum as correspondences}
	Let $_{\phi}X$ and $_{\psi}Y$ be $A$--$B$ correspondences. We say that $X$ and $Y$ are isomorphic if there exists a unitary operator $U:X\to Y$ such that $U\phi(a)=\psi(a)U$ for all $a\in A$. If $X=X_0\oplus X_1$ as Hilbert modules (see Definition~\ref{def: Direct sum of Hilbert modules}), then we say that $X$ is the direct sum of $X_0$ and $X_1$ if $\phi(A)X_i\subseteq X_i$.
\end{definition}
\begin{remark}\label{Sam}
Note that this definition of isomorphism is saying that $X$ and $Y$ must be isomorphic as Hilbert modules and that the isomorphism must intertwine the left actions.
It is not at first obvious how the left and right actions differ on a $C^*$-correspondence. We have seen that one difference is that the left action must be by adjointable operators, whilst the right action need not be. Cohen factorisation tells us that if $X$ is an $A$--$B$ correspondence then $X\cdot B=X$. 
Given the left action $\phi:A\to \LL(X)$ there is no reason for elements of the form $\phi(a)x$ to be dense in $X$. If products of this form are dense in $X$, then we call $\phi$ \textit{non-degenerate}. We will refrain from calling a left action that is not non-degenerate a `degenerate' action because the term degenerate will be used later in a different context. A map $\phi$ between $C^*$-algebras $A$ and $B$ is said to be non-degenerate if $\phi$ maps an approximate identity for $A$ to an approximate identity for $B$. It can be shown that this definition is consistent with our definition for homomorphisms between $A$ and $\LL(X)$.
\end{remark}
\begin{example}
	Let $A$ be a $C^*$-algebra and $I\triangleleft A$ a non-trivial ideal in $A$. Then we can define a left action of $I$ on $A_A$ by left multiplication. We then have $\phi(I)A\subseteq I$ which is not dense in $A$, thus we do not have a non-degenerate action.
\end{example}
\begin{example}\label{Fish}
	Since Cohen factorisation implies that every $x\in X$ is of the form $x=y\IP{y,y}=\Theta_{y,y}(y)$ for some $y\in X$, we see that the $\KK(X)$--$A$ correspondence of Example~\ref{Wahay!} is non-degenerate. In particular, when $X=A_A$ and $\phi:A\to\KK(A_A)$ is the left multiplication representation as in Example~\ref{Fang}, then $\phi$ is non-degenerate.
\end{example}

\section{Tensor products}

At the beginning of this section we said that $C^*$-correspondences can be considered as `generalised homomorphisms', so we will need a way to compose them. Given an $A$--$B$ correspondence and a $B$--$C$ correspondence, we have different ways of forming their tensor product, the \textit{internal} (or `balanced') tensor product and the \textit{external} tensor product. We will focus only on the internal tensor product.
\begin{proposition}\label{Trang}
	Let $A$, $B$ and $C$ be $C^*$-algebras and let $_AX_B$ and $_BY_C$ be $C^*$-correspondences with $\phi:A\to\LL(X)$ and $\psi:B\to \LL(Y)$ the left actions of $A$ and $B$ respectively. There exists a sesquilinear map $\IP{\cdot,\cdot}:X\odot Y\times X\odot Y\to C$ on the vector space tensor product $X\odot Y$ such that
	\begin{equation*}
	\IP{x\otimes y,z\otimes w}_C=\IP{y,\psi(\IP{x,z}_B)w}_C
	\end{equation*}
	on simple tensors. Let $N$ be the subspace
	\begin{equation*}
	N=\Span\{xb\otimes y-x\otimes\psi(b)y:x\in X, y\in Y, b\in B\}.
	\end{equation*}
	Then $\IP{\cdot,\cdot}$ descends to a positive definite inner-product on $(X\odot Y)/N$. We define the internal tensor product $X\otimes_BY$ to be the completion of $(X\odot Y)/N$ in the norm coming from the inner-product. Then $X\otimes_BY$ is an $A$--$C$ correspondence with right action
	\begin{equation}\label{right action}
		(x\otimes y)\cdot c=x\otimes (y\cdot c)
	\end{equation}
	and left action $\widetilde{\phi}:A\to\LL(X\otimes_BY)$
	\begin{equation}\label{left action}
		\widetilde{\phi}(a)(x\otimes y)=(\phi(a)x)\otimes y.
	\end{equation}
\end{proposition}
Before presenting a proof we will need some technical lemmas about positivity.
\begin{lemma}\label{Panang}
	Let $X$ be a Hilbert $A$-module and $T\in\LL(X)$. Then $T\geq 0$ in $\LL(X)$ if and only if $\IP{x,Tx}\geq 0$ in $A$ for all $x\in X$.
\end{lemma}
\begin{proof}
	This proof is taken from \cite[Proposition 4.2]{Lance}.
	If $T\geq 0$ in $\LL(X)$ then $T=S^*S$ for some $S\in\LL(X)$ and
	\begin{equation*}
		\IP{x,Tx}=\IP{Sx,Sx}\geq 0
	\end{equation*}
	for all $x\in X$. Now for the reverse direction suppose that $\IP{x,Tx}\geq 0$ for all $\in X$. Then in particular $\IP{x,Tx}$ is a self adjoint element of $A$ and
	\begin{equation*}
	\IP{Tx,x}=\IP{x,Tx}^*=\IP{x,Tx}
	\end{equation*}
	for all $x\in X$. Polarising then gives
	\begin{equation*}
	\IP{x,Ty}=\IP{Tx,y}
	\end{equation*}
	for all $x,y\in X$ and so $T=T^*$. Since $T$ is self adjoint we may apply functional calculus to the functions $f(t)=\max(t,0)$ and $g(t)=\max(-t,0)$ to obtain $T=f(T)-g(T)$ with $f(T)g(T)=0$ and $f(T),g(T)\geq 0$. We know that
	\begin{equation*}
	\IP{g(T)x,Tg(T)x}=-\IP{x,g(T)^3x}\geq 0.
	\end{equation*}
	Since $g(T)$ is positive, so is $g(T)^3$ and so $\IP{x,g(T)^3x}\geq 0$. Hence we conclude that $\IP{x,g(T)x}=0$ for all $x\in X$. Polarisation again implies $\IP{x,g(T)y}=0$ for all $x,y\in X$, and so by Lemma~\ref{lemma: norm is sup of IPs} this determines $g(T)y=0$ for all $x\in X$ and so $g(T)=0$ whence $T=f(T)$ is positive.
\end{proof}
It is convenient at this stage to introduce the concept of completely positive maps.
\begin{definition}\label{Sprang}
	If $A$ and $B$ are $C^*$-algebras we say that a map $\rho:A\to B$ is \textit{completely positive} if the maps $\rho^{(n)}:M_n(A)\to M_n(B)$ defined by $\rho^{(n)}(S)_{ij}=\rho(S_{ij})$ are all positive maps for each $n\in\N$. That is, $\rho^{(n)}(S)\geq 0$ in $M_n(B)$ whenever $S\geq 0$ in $M_n(A)$.
\end{definition}
\begin{remark}
	If $\rho:A\to B$ is a homomorphism then for any $S,T\in M_n(A)$
	\begin{align*}
	\rho^{(n)}(ST)_{ij}=\rho\Big(\sum_{k=1}^nS_{ik}T_{kj}\Big)=\sum_{k=1}^n\rho(S_{ik})\rho(T_{kj})=(\rho(S)^{(n)}\rho(T)^{(n)})_{ij}
	\end{align*}
	and
	\begin{equation*}
	(\rho^{(n)}(S)^*)_{ij}=(\rho^{(n)}(S)_{ji})^*=\rho(S_{ji})^*=\rho^{(n)}(S^*)_{ji}
	\end{equation*}
	so the maps $\rho^{(n)}$ are also homomorphisms from $M_n(A)$ to $M_n(B)$. We see then that $\rho^{(n)}(S^*S)=\rho^{(n)}(S)^*\rho^{(n)}(S)$, so all homomorphisms are completely positive. In particular we will need to know that since it is a homomorphism, the left action for a $C^*$-correspondence is completely positive.
\end{remark}
We need one more technical lemma before proving Proposition~\ref{Trang}.
\begin{lemma}\label{Swang}
	If $X$ is a Hilbert $A$-module then for every $x\in X$ and $0<\alpha<1$ there exists some $w\in X$ such that
	\begin{equation*}
	x=w\IP{x,x}^{\alpha/2}.
	\end{equation*}
\end{lemma}
\begin{proof}
	Since $\IP{x,x}\geq 0$ the element $\IP{x,x}^{1/2}$ of $A$ is self adjoint and so the sequence of functions $g_n:\R\to\R$ defined by
	\begin{equation*}
	g_n(\lambda)=\begin{cases}
	n^{\alpha}&\lambda<1/n\\
	\lambda^{-\alpha}&\lambda\geq 1/n\\
	\end{cases}
	\end{equation*}
	consists of continuous functions on the spectrum $\sigma(\IP{x,x}^{1/2})$. For any continuous function $f\in C(\sigma(\IP{x,x}^{1/2}))$, we have the norm equality
	\begin{align}
	\normof{xf(\IP{x,x}^{1/2})}_X&=\normof{f^*(\IP{x,x}^{1/2})\IP{x,x}f(\IP{x,x}^{1/2})}_A^{1/2}\nonumber\\
	&=\normof{f^*(\IP{x,x}^{1/2})(\IP{x,x}^{1/2})^*\IP{x,x}^{1/2}f(\IP{x,x}^{1/2})}_A^{1/2}\nonumber\\
	&=\normof{\IP{x,x}^{1/2}f(\IP{x,x}^{1/2})}_A=\label{temp}\normof{\lambda f(\lambda)}_{\infty}.\\\nonumber
	\end{align}
	Let $h_{n}(\lambda)=\lambda g_n(\lambda)$. We have
	\begin{equation*}
		\normof{h_{n}-h_m}_{\infty}\leq \max\{n^{\alpha-1},(1-\alpha)^{1/\alpha}m^{\alpha-1}\}
	\end{equation*}
	for $n\geq m$ and for $0<\alpha<1$ we have
	\begin{equation*}
		\max\{n^{\alpha-1},(1-\alpha)^{1/\alpha}m^{\alpha-1}\}\to 0
	\end{equation*}
	as $n,m\to 0$. Hence $h_n$ is a Cauchy sequence in $C(\sigma(\IP{x,x}^{1/2}))$. Equation \eqref{temp} tells us that $xg_n(\IP{x,x}^{1/2})$ is a Cauchy sequence in $X$ and by completeness converges to some element $w\in X$. Note that $g_n(\lambda)\lambda^{\alpha}\to 1$ as $n\to \infty$. Using Equation \eqref{temp} again, we obtain
	\begin{align*}
	\normof{w\IP{x,x}^{\alpha/2}-x}&=\lim_{n\to\infty}\normof{xg_n(\IP{x,x}^{1/2})\IP{x,x}^{1/2}-x}\\
	&=\lim_{n\to\infty}\normof{x(g_n(\IP{x,x}^{1/2})\IP{x,x}^{\alpha/2}-1)}\\
	&=\lim_{n\to\infty}\normof{\lambda(g_n(\lambda)\lambda^{\alpha}-1)}_{\infty}=0,\\
	\end{align*}
	so $w\IP{x,x}^{\alpha/2}=x$ as required.
\end{proof}

\begin{proof}[Proof of Proposition~\ref{Trang}]
	This proof is taken from \cite[Proposition 4.5]{Lance}. 
	We first wish to that if we define
	\begin{equation*}
	M=\{z\in X\odot Y:\IP{z,z}=0\},
	\end{equation*}
	then $M\subseteq N$.
	We compute for any $b\in B$, $x\in X$ and $y\in Y$
	\begin{align*}
		&\IP{xb\otimes y-x\otimes\psi(b)y,xb\otimes y-x\otimes\psi(b)y}\\
		&=\IP{y,\psi(\IP{xb,xb})y}-\IP{\psi(b)y,\psi(\IP{x,xb})y}-\IP{y,\psi(\IP{xb,x})\psi(b)y}+\IP{\psi(b)y,\psi(\IP{x,x})\psi(b)y}\\
		&=\IP{y,\psi(b^*)\psi(\IP{x,x})\psi(b)y}-\IP{\psi(b)y,\psi(\IP{x,x})\psi(b)y}-\IP{y,\psi(b^*)\psi(\IP{x,x})\psi(b)y}\\
		&\quad+\IP{\psi(b)y,\psi(\IP{x,x})\psi(b)y}=0.\\
	\end{align*}
	So if 
	\begin{equation*}
		z\in N=\overline{\Span}\{xb\otimes y-x\otimes\psi(b)y:x\in X,y\in Y,b\in B\}
	\end{equation*}
	then $\IP{z,z}=0$. We deduce that $N\subseteq M$. This means that $\IP{\cdot,\cdot}$ preserves $N$ and so descends to a well defined bilinear map on $(X\odot Y)/N$. To show that $\IP{\cdot,\cdot}$ is positive definite on $(X\odot Y)/N$ we wish to show that if $\IP{z,z}=0$ then $z$ represents the zero element of $(X\odot Y)/N$, that is $z\in N$.
	Fix $z=\sum_{i=1}^nx_i\otimes y_i\in M$. 
	Let $U\in M_n(B^n)$ be the matrix with entries $U_{ij}=\IP{x_i,x_j}$, so that $\psi^{(n)}(U)$ is an element of $M_n(\LL(Y))$. We identify $M_n(\LL(Y))$ with $\LL(Y^n)$. 
	For all $b=(b_1,\dots,b_n)^T\in B^n$ we compute
	\begin{equation*}
		\IP{b,Ub}=\sum_{i,j=1}^nb_i^*,\IP{x_i,x_j}b_j=\sum_{i,j=1}^n\BigIP{x_ib_i,x_jb_j}=\BigIP{\sum_{i=1}^nx_ib_i,\sum_{i=1}^nx_ib_i}
	\end{equation*}
	which is precisely the inner-product of the vector $(x_1b_1,\dots x_nb_n)^T$ in $X^n$ with itself. Hence $\IP{b,Ub}\geq 0$ in $B$, and so by Lemma~\ref{Panang} we have $U\geq 0$ in $\LL(B^n)$. Since $\psi$ is a homomorphism it is completely positive so $\psi^{(n)}(U)\geq 0$ in $\LL(Y^n)$.
	We have
	\begin{align*}
		\IP{z,z}_{X\odot Y}&=\sum_{i,j=1}^n\IP{y_i,\psi(\IP{x_i,x_j})y_j}_{X\odot Y}\\
		&=\IP{y,\psi^{(n)}(U)y}_{Y^n}\\
	\end{align*}
	where $y=(y_1,\dots, y_n)$. Since $\psi^{(n)}(U)\geq 0$ in $\LL(Y^n)$, Lemma~\ref{Panang} tells us that
	\begin{equation*}
		\IP{z,z}_{X\odot Y}=\IP{y,\psi^{(n)}(U)y}_{Y^n}\geq 0.
	\end{equation*}
	Let $T=\psi^{(n)}(U)$. Since $T\geq 0$, we may apply functional calculus to obtain the self-adjoint operators $T^{1/2}$ and $T^{1/4}$. By positive definiteness of the inner-product on $Y^n$, we then have $T^{1/2}y=0$. Then
	\begin{equation*}
		0=\IP{y,T^{1/2}y}=\IP{T^{1/4}y,T^{1/4}y}
	\end{equation*}
	so $T^{1/4}y=0$. Now we may view $X^n$ as a Hilbert $M_n(B)$-module by defining for $x\in X^n$ and $M\in M_n(B)$
	\begin{equation*}
		x\cdot M=Mx\qquad\text{and}\qquad \IP{x,x}_{ij}=\IP{x_i,x_j}.
	\end{equation*}
	In the above notation we have $U^{1/2}=\IP{x,x}^{1/2}$, so applying Lemma~\ref{Swang} there exists $w=(w_1,\dots,w_n)$ such that $U^{1/4}w=x$. By properties of functional calculus we know that $\psi^{(n)}(U^{1/4})=T^{1/4}$, so if $U^{1/4}$ has matrix entries $c_{ij}$ then $T^{1/4}$ has entries $\psi(c_{ij})$. Since $T^{1/4}y=0$ we have $\sum_{i,j=1}^n\psi(c_{ij})y_j=0$, hence
	\begin{equation*}
		\sum_{i,j=1}^nw_jc_{ij}\otimes y_i-w_j\otimes\psi(c_{ij})y_i=\sum_{i=1}^nx_i\otimes y_i
	\end{equation*}
	is an element of $N$. Thus $M\subseteq N$ and we have a positive definite inner-product on $(X\odot Y)/N$.\\ 
	Lastly we wish to show that the right and left actions of $A$ and $C$ induce well defined actions on $X\otimes_BY$ defined by the formulas \eqref{right action} and \eqref{left action}. Clearly the left and right actions preserve $N$, so they descend to well defined actions on $(X\odot Y)/N$.
	For the right action, for $c\in C$ we compute
	\begin{align*}
	\normof{x\otimes (yc)}&=\normof{c^*\IP{y,\psi(\IP{x,x})y}c}^{1/2}\\
	&=\normof{(c^*\IP{y,\psi(\IP{x,x})y}^{1/2})^*\IP{y,\psi(\IP{x,x})y}^{1/2}c}^{1/2}\\
	&=\normof{\IP{y,\IP{x,x}y}^{1/2}c}\\
	&\leq \normof{\IP{y,\IP{x,x}y}}^{1/2}\normof{c}\\
	&=\normof{x\otimes y}\normof{c}.\\
	\end{align*}
	Thus the right action of $C$ on $(X\odot Y)/N$ is a uniformly continuous map for each $c$ and so extends to the completion $X\otimes_BY$. Now we wish to show that our formula \eqref{left action} defined a left action on $X\otimes_B Y$.
	For any $a\in A$, sub-multiplicativity for implies $\normof{\phi(a)x\otimes y}\leq \normof{\phi(a)}\normof{x\otimes y}$ so each $\phi(a)$ is a uniformly continuous operator and extends to the completion $X\otimes_BY$.
	We see that
	\begin{equation*}
	\IP{\phi(a)x\otimes y,w\otimes z}=\IP{y,\psi(\IP{\phi(a)x,w})z}=\IP{y,\psi(\IP{x,\phi(a^*)w})z}=\IP{x\otimes y,\phi(a^*)w\otimes z},
	\end{equation*}
	so $\phi$ is indeed a homomorphism into $\LL(X\otimes_B Y)$.
\end{proof}
\begin{remark}\label{remark: tensor product quotient N,M}
	We have just shown that the subspaces
	\begin{align*}
		N&=\Span\{x\otimes\psi(b)y-xb\otimes y:x\in X,y\in Y,b\in B\}\quad\text{ and }\\
		M&=\{x\otimes y\in X\odot Y:\IP{y,\phi(\IP{x,x}y)}=0\}\\
	\end{align*}
	of $X\odot Y$ are equal. Thus we may write $X\otimes_B Y$ interchangeably as the norm closure of $(X\odot Y)/N$ or $(X\odot Y)/M$. 
\end{remark}
In what follows we will compute examples which elucidate the nature of balanced tensor products. The following two lemmas will give us criteria for a formula to define a unitary operator on balanced tensor products.
\begin{lemma}\label{lemma: densely defined unitaries}
	Let $A$ be a $C^*$-algebra and let $X$ and $Y$ be Hilbert $A$-modules. Suppose that $S\subseteq X$ satisfies
	\begin{equation*}
		X=\overline{\Span}S.
	\end{equation*}
	Suppose that $U_S:\Span S\to Y$ is a function such that $\IP{U_Sx,U_Sy}=\IP{x,y}$ for all $x,y\in S$. Then there is a unique, continuous, linear, inner-product preserving operator $U:X\to Y$ such that $U\big|_S=U_S$. If in addition we have
	\begin{equation*}
		\overline{\Span}\{U_Sx:x\in S\}=Y,
	\end{equation*}
	then $U\in\LL(X\otimes_BY)$ with $U^*=U^{-1}$.
\end{lemma}
\begin{proof}
	Suppose that we have $\IP{U_Sx,U_Sy}=\IP{x,y}$ for all $x,y\in S$. We wish to prove a modified version of Lemma~\ref{lemma: norm is sup of IPs}: we wish to show that for all $x\in S$, we have 
	\begin{equation}\label{WOT}
		\normof{U_S(x)}=\sup_{\normof{U_S(z)}\leq 1}\normof{\IP{U_S(x),U_S(z)}}.
	\end{equation}
	By the same argument as in Lemma~\ref{lemma: norm is sup of IPs} we see that
	\begin{equation*}
		\sup_{\normof{U_S(z)}\leq 1}\normof{\IP{U_S(x),U_S(z)}}\leq\sup_{\normof{U_S(z)}}\normof{U_S(x)}\normof{U_S(z)}=\normof{U_S(x)}.
	\end{equation*}
	We see that
	\begin{equation*}
		\normof{U_S(x/\normof{x}),U_S(x/\normof{x})}=\IP{x/\normof{x},x/\normof{x}}=1,
	\end{equation*}
	so we may in particular choose $z=U_S(x/\normof{x})$ to obtain
	\begin{align*}
		\sup_{\normof{U_S(z)}\leq 1}\normof{\IP{U_S(x),U_S(z)}}&\geq\normof{\IP{U_S(x),U_S(x/\normof{x})}}=\normof{x,x/\normof{x}}\\
		&=\normof{\IP{x,x}}^{1/2}=\normof{\IP{U_S(x),U_S(x)}}^{1/2}=\normof{U_S(x)}.
	\end{align*}
	So we deduce that $\normof{U_S(x)}=\sup_{\normof{U_S(z)}\leq 1}\normof{\IP{U_S(x),U_S(z)}}$ for all $x\in S$.
	Then for $x,y,z\in \Span S$ and $a\in A$ we compute
	\begin{equation*}
		\IP{U_S(x+ya),U_Sz}=\IP{x+ya,z}=\IP{x,z}+a^*\IP{y,z}=\IP{U_Sx+U(y)a,U_Sz},
	\end{equation*}
	so by Equation~\ref{WOT} we deduce that $U_S$ is linear.
	The formula $\IP{U_Sx,U_Sy}=\IP{x,y}$ for all $x,y\in S$ shows that $\normof{U_Sx}=\normof{x}$ for all $x\in S$. Hence $U_S$ is uniformly continuous and so Theorem~\ref{Sang} gives a unique continuous, linear map $U:X\to Y$ such that $U\big|_S=U_S$. The map $U$ is defined such that if $s_n$ is a Cauchy sequence in $S$, then
	\begin{equation*}
		U(\lim_{n\to\infty}s_n)=\lim_{n\to\infty}U_S(s_n).
	\end{equation*}
	Continuity of the inner-product gives $\IP{Ux,Uy}=\IP{x,y}$ for all $x,y\in X$.
	Now suppose that we also have
	\begin{equation}
		\overline{\Span}\{U_Sx:x\in S\}=Y.\label{Lenny Henry}
	\end{equation}
	We wish to show that $U$ is surjective. Fix $y\in Y$. By Equation~\ref{Lenny Henry} we may pick a sequence $U_Sx_n$ converging to $Y$ such that $x_n\in S$. Since $U_S$ preserves inner-products, $x_n$ is a convergent sequence in $X$, so by definition of $U$ we have
	\begin{equation*}
		U(\lim_{n\to\infty}x_n)=\lim_{n\to\infty}U_S(x_n)=y.
	\end{equation*}
	Thus $U$ is surjective and satisfies the hypotheses of Proposition~\ref{Jeddabub}, so $U$ is unitary.
\end{proof}
\begin{lemma}\label{lemma: Well defined unitaries on tensor product}
	Let $A$ and $B$ be a $C^*$-algebras, let $X$ be a Hilbert $A$-module and let $Y$ be an $A$--$B$ correspondence with left action $\phi$. Let $Z$ be a Hilbert $B$-module and $U:X\odot Y\to Z$ be a linear map defined on the vector space tensor product satisfying
	\begin{equation}
		\IP{U(x\otimes y),U(u\otimes v)}=\IP{y,\phi(\IP{x,u})v}\label{bad liar}
	\end{equation}
	for all $x,u\in X$ and $y,v\in Y$. Then if $q:X\odot Y\to X\otimes_A Y$ is the quotient map, there is a unique, linear, continuous, inner-product preserving operator $\widetilde{U}\in\LL(X\otimes_A Y,Z)$ such that $\widetilde{U}(x)=q\circ U(x)$ for all $x\in X\odot Y$. If $U$ is has dense range then $\widetilde{U}$ is adjointable and unitary.
\end{lemma}
\begin{proof}
	If $\IP{y,\phi(\IP{x,x})y}=0$ then by Equation \eqref{bad liar} we have
	\begin{equation*}
		\IP{U(x\otimes y),U(x\otimes y)}=\IP{y,\phi(\IP{x,x})y}=0.
	\end{equation*}
	So Remark~\ref{remark: tensor product quotient N,M} shows that there is a well defined operator $\bar{U}:\Span\{q(x\otimes y):x\in X,y\in Y\}\subset X\otimes_B Y\to Z$ such that $\bar{U}(q(x\otimes y))=U(x\otimes y)$. Equation \eqref{bad liar} then gives
	\begin{align*}
		\IP{\bar{U}(q(x\otimes y)),\bar{U}(q(u\otimes v))}&=\IP{U(x\otimes y),U(u\otimes v)}\\
		&=\IP{y,\phi(\IP{x,u})v}=\IP{q(x\otimes y),q(u\otimes v)}\\
	\end{align*}
	for all $x,u\in X$ and $y,v\in Y$. Lemma~\ref{lemma: densely defined unitaries} gives a unique unitary operator $\widetilde{U}\in\LL(X\otimes_AY,Z)$ satisfying the required criteria. If $U$ is has dense range then
	\begin{equation*}
		\overline{\Span}\{\bar{U}(q(x\otimes y)):x\in X,y\in Y\}=\overline{\Span}\{U(x\otimes y):x\in X,y\in Y\}=Z,
	\end{equation*}
	so~\ref{lemma: densely defined unitaries} implies $\widetilde{U}$ is adjointable and unitary.
\end{proof}
\begin{remark}
	In the above proof we described spanning elements of $X\otimes_B Y$ as $q(x\otimes y)$ for $x\in X$ and $y\in Y$. When we do not need to distinguish between a simple tensor $x\otimes y\in X\odot Y$ and its image in $X\otimes_B Y$, we will simply write $x\otimes y$ for spanning elements of $X\otimes_BY$. Such elements then satisfy $xb\otimes y=x\otimes\phi(b)y$.
\end{remark}
\begin{example}\label{Gen hom balance}
	Let $A$, $B$ and $C$ be $C^*$-algebras. Consider the Hilbert modules $B_B$ and $C_C$ and suppose that $\phi:A\to \LL(B)$ and $\psi:B\to \LL(C)$ make $_{\phi}B_B$ and $_{\psi}C_C$ into $C^*$-correspondences. Suppose that $\psi$ is non-degenerate. Then we have an isomorphism $B\otimes_{\psi}C\cong\; _{\psi\circ\phi}C_C$. To see this we calculate
	\begin{align*}
	\IP{U(b\otimes c),U(d\otimes e)}_C&=\IP{\psi(b)c,\psi(d)e}_C\\
	&=\IP{c,\psi(b^*d)e}_C\\
	&=\IP{c,\psi(\IP{b,d}_B)e}_C,\\
	\end{align*}
	so Lemma~\ref{lemma: Well defined unitaries on tensor product} gives an inner-product preserving operator $U:B\otimes_BC\to C_C$ such that $U(b\otimes c)=\phi(b)c$.
	Since $\phi$ is non-degenerate, $U$ is surjective so Lemma~\ref{lemma: Well defined unitaries on tensor product} implies $U$ is unitary. That is, $U:B\otimes_{\psi}C\to\;_{\psi\circ\phi}C_C$ is an isomorphism of Hilbert $C$-modules, to check that they are isomorphic as $A$--$C$ correspondences we must show that $\psi\circ\phi(a)U=U\phi(a)$ for all $a\in A$. By linearity and continuity of $U$, it suffices to show $\psi\circ\phi(a)U(b\otimes c)=U(\phi(a)b\otimes c)$ for all $a\in A$, $b\in B$ and $c\in C$. We compute
	\begin{align*}
	U(\phi(a)b\otimes c))&=\psi(\phi(a)b)c=\psi\circ\phi(a)(\psi(b)c)=\psi\circ\phi(a)U(b\otimes c)\\
	\end{align*} 
	so we deduce that $_{\phi}(B\otimes_{\psi}C)\cong\; _{\psi\circ\phi}C_C$.
\end{example}
\begin{example}
	Let $E$ be a directed graph and consider the $C_0(E^0)$--$C_0(E^0)$ correspondence $X$ of Example~\ref{Marns}. So $X$ is the completion of $C_c(E^1)$ in the inner-product
	\begin{equation*}
	\IP{f,g}(v)=\sum_{e\in E^1v}\overline{f(e)}g(e)
	\end{equation*}
	with left action $\phi(\xi)f(e)=\xi(s(e))f(e)$ and right action $(f\cdot \xi)(e)=f(e)\xi(r(e))$.
	We define a $C_0(E^0)$--$C_0(E^0)$ correspondence $Y$ as follows.
	For $f\in C_c(E^2)$ and $\xi\in C_0(E^0)$, we equip $C_c(E^2)$ with the right action
	\begin{equation*}
	(f\cdot \xi)(e)=f(e)\xi(r(e))
	\end{equation*}
	and left action $\psi:C_0(E^0)\to\LL(C_c(E^2))$
	\begin{equation*}
	\psi(\xi)f(e)=\xi(s(e))f(e).
	\end{equation*}
	We give $C_c(E^2)$ the inner-product
	\begin{equation*}
	\IP{f,g}(v)=\sum_{e\in E^2v}\overline{f(e)}g(e)
	\end{equation*}
	and define $Y$ to be the completion of $C_c(E^2)$ in the norm coming from this inner-product. We claim that $X\otimes_{C_0(E^0)}X\cong Y$. To see this we compute
	\begin{align*}
	\sum_{e_1e_2\in E^2w}\overline{f(e_1)g(e_2)}u(e_1)v(e_2)&=\sum_{e_1\in E^1s(e_2)}\sum_{e_2\in E^1w}\overline{f(e_1)g(e_2)}u(e_1)v(e_2)\\
	&=\sum_{e_2\in E^1w}\overline{g(e_2)}\IP{f,u}_X\big(s(e_2)\big)v(e_2)\\
	&=\sum_{e_2\in E^1w}\overline{g(e_2)}(\phi(\IP{f,u}_X)v)(e_2)\\
	&=\IP{g,\phi(\IP{f,u}_X)v}_X\\
	&=\IP{f\otimes g,u\otimes v}_{Y},\\
	\end{align*}
	so by Lemma~\ref{lemma: Well defined unitaries on tensor product} there is an inner-product preserving operator $U:X\otimes_{C_0(E^0)}X\to Y$ such that $U(f\otimes g)(e_1e_2)=f(e_1)g(e_2)$. Since the point mass functions $\delta_{e_1e_2}$ span a dense subset in $Y$ and $U(\delta_{e_1}\otimes\delta_{e_2})=\delta_{e_1e_2}$, we deduce that $U$ is surjective. Hence Lemma~\ref{lemma: Well defined unitaries on tensor product} implies that $U$ is unitary and so $X\otimes_{C_0(E^0)}X$ and $Y$ are isomorphic as Hilbert $C_0(E^0)$ modules. Since $U$ is continuous and linear, for $\xi\in C_0(E^0)$ and $f,g\in X$ the computation
	\begin{align*}
		\psi(\xi)U(f\otimes g)(e_1e_2)&=\xi(s(e_1e_2))f(e_1)g(e_2)=\xi(s(e_1))f(e_1)g(e_2)=U((\phi(\xi)f)\otimes g)(e_1e_2)\\
	\end{align*}
	shows that $\psi(\xi)\circ U=U\circ\phi(\xi)$ for all $\xi\in C_0(E^0)$. Hence $X\otimes_{C_0(E^0)}X$ and $Y$ are isomorphic as $C_0(E^0)$--$C_0(E^0)$ correspondences.
\end{example}
\begin{example}\label{left and right}
	Let $_{\phi}X$ be an $A$--$B$ correspondence with $\phi$ non-degenerate. We claim that $A\otimes_{\phi} X\cong X\cong X\otimes_{\iota}B$ where $\iota:B\to\LL(B)$ is the inclusion of $B$ as left multiplication operators. We consider $A$ as an $A$--$A$ correspondence with left action $j:A\to\LL(A)$ left multiplication, so that $j\otimes 1$ is the left action on $A\otimes X$.
	First we show the the left-hand isomorphism. We see that 
	\begin{align*}
	\IP{\phi(a)x,\phi(b)y}_X=\IP{x,\phi(a^*b)y}_X=\IP{x,\phi(\IP{a,b}_A)y}_X=\IP{a\otimes x,b\otimes y}_{A\otimes X},\\
	\end{align*}
	so Lemma~\ref{lemma: Well defined unitaries on tensor product} gives an operator $U:A\otimes_{\phi} X\to X$ such that $U(a\otimes x)=\phi(a)x$. Since $\phi$ is non-degenerate $U$ is surjective. To see this, fix $x\in X$ and suppose that $a_n\in A$ and $x_n\in X$ satisfy $\phi(a_n)x_n\to x$. Then we calculate, using at the last step that $U$ preserves inner-products:
	\begin{align*}
	\normof{\phi(a_n)x_n-\phi(a_m)x_m}&=\normof{U(a_n\otimes x_n)-U(a_m\otimes x_m)}\\
	&=\normof{U(a_n\otimes x_n-a_m\otimes x_m)}\\
	&=\normof{a_n\otimes x_n-a_m\otimes x_m}.\\
	\end{align*}
	Hence $a_n\otimes x_n$ is a Cauchy sequence in $A\otimes_{\phi}X$. By completeness it has a limit $L$ and
	\begin{equation*}
	U(L)=U(\lim_{n\to\infty} a_n\otimes x_n)=\lim_{n\to\infty}U(a_n\otimes x_n)=x,
	\end{equation*}
	hence $U$ is surjective and so by Lemma~\ref{lemma: Well defined unitaries on tensor product} $U$ is unitary. Lastly we see that  
	\begin{equation*}
	U(j(a)\otimes 1(b\otimes x))=U(ab\otimes x)=\phi(ab)x=\phi(a)U(b\otimes x)
	\end{equation*}
	and so $A\otimes_{\phi} X\cong X$ as $A$--$B$ modules.\\
	\vskip 1pt
	Now for the second isomorphism the calculation
	\begin{align*}
	\IP{xb,yc}_X&=b^*\IP{x,y}_Xc=b^*\iota(\IP{x,y}_X)c=\IP{b,\iota(\IP{x,y}_X)c}_B=\IP{x\otimes b,y\otimes c}_{X\otimes_{\iota}B}\\
	\end{align*}
	shows that there is a well defined inner-product preserving operator $V:X\otimes_{\iota} B\to X$ such that $V(x\otimes b)=x\cdot b$. Then \cite[Proposition 2.31]{RaeburnWilliams} tells us that for every $x\in X$ there exists some $y\in X$ such that $x=y\cdot \IP{y,y}$, which shows that $V$ is surjective, and hence unitary. Lastly we see that
	\begin{equation*}
	V(\phi(b)\otimes 1(x\otimes b))=V(\phi(b)x\otimes c)=\phi(b)xc=\phi(b)V(x\otimes c),
	\end{equation*}
	hence $X\cong X\otimes_{\iota}B$ as $A$--$B$ correspondences.
\end{example}
\begin{remark}
	Note that we only required $\phi$ to be non-degenerate for the first isomorphism in Example~\ref{left and right}. It is interesting to see that for $A$--$A$ correspondences $X$ and $Y$, there can be a significant difference between $X\otimes_A Y$ and $Y\otimes_A X$. 
	For example, consider the $\C$--$\C$ correspondences $X=\C$ and $Y=\C^2$. We equip $X$ with the right action $x\cdot z=xz$, left action $\phi(z)x=zx$ and usual dot product $\IP{x,y}=\bar{x}y$. For $Y$ We define a left action of $\C$ by $\psi(z)= \begin{psmallmatrix}
	0&0\\
	0&z\\
	\end{psmallmatrix}$, a right action by $(x,y)\cdot z=(xz,yz)$ and we equip $Y$ with the usual dot product $\IP{x,y}=\bar{x}_1y_1+\bar{x}_2y_2$. Then it is easy to check using Lemma \eqref{lemma: Well defined unitaries on tensor product} that $X\otimes_{\psi}Y\cong X$ and $Y\otimes_{\phi}X\cong Y$.
\end{remark}
\begin{lemma}\label{ONE}
	Let $A$ be a $C^*$-algebra with $I\triangleleft A$ an ideal. The set
	\begin{equation*}
	X_I\coloneqq\{x\in X:\IP{x,x}\in I\}
	\end{equation*}
	is a Hilbert $I$-module under the same operations as $X$, equal to the set
	\begin{equation*}
	X\cdot I=\{xi:x\in X, i\in I\}.
	\end{equation*}
	If $\iota:A\to\LL(I)$ denotes the left action of $A$ by left multiplication on $I$, then we have an isomorphism $X_I\cong X\otimes_{\iota}I$. If $X$ carries a (possibly degenerate) left action $\phi:A\to \LL(X)$, then $X_I\cong X\otimes_{\iota}I$ as $A$--$I$ correspondences, and we also have an isomorphism $\overline{\phi(I)X}\cong I\otimes_{\phi}X$ as Hilbert $A$ modules so that the left action $\phi$ of $A$ on $\overline{\phi(I)X}$ corresponds to left multiplication on $I\otimes_{\phi}X$.
\end{lemma}
\begin{proof}
	We begin by showing that $X_I$ is equal to $X\cdot I$. If $xi\in X\cdot I$ then
	\begin{equation*}
	\IP{xi,xi}=i^*\IP{x,x}i\in I,
	\end{equation*}
	so clearly $X\cdot I\subseteq X_I$. Now suppose that $x\in X_I$. Cohen factorisation allows us to find some $y\in X$ such that $x=y\IP{y,y}$, so if we can show that $\IP{y,y}\in I$ then we will be done. Since $x\in X_I$ we have
	\begin{equation*}
	I\ni \IP{x,x}=\IP{y,y}^3.
	\end{equation*}
	Since $\IP{y,y}$ is a positive element, so is $\IP{y,y}^3$, and applying functional calculus we have
	\begin{equation*}
	\IP{y,y}=(\IP{y,y}^3)^{1/3}\in C^*(y).
	\end{equation*}
	Since $C^*(y)$ is generated by $\IP{y,y}^3$ which is an element of $I$, $C^*(y)$ is a sub-algebra of $I$ whence we deduce that $\IP{y,y}\in I$, and so $X_I=X\cdot I$.\\
	Now we show that $X_I$ is a Hilbert $I$ module. If $i\in I$ then for any $x\in X$
	\begin{equation*}
	\IP{xi,xi}=i^*\IP{x,x}i\in I
	\end{equation*}
	so $I$ preserves $X_I$, and we have a well defined right multiplication. Any $x,y\in X_I$ we have just seen can be written as $x=x_0i$ and $y=y_0j$, so we see that
	\begin{equation*}
	\IP{x,y}=i^*\IP{x_0,y_0}j\in I,
	\end{equation*}
	and so the inner-product takes values in $I$.
	We need only then check that $X_I$ is complete in the Hilbert module norm. Let $x_i$ be a Cauchy sequence in $X_I$. Since $X_i\subseteq X$ and $X$ is complete, $x_i$ has a limit $x\in X$. We aim to show that $x\in X_I$. By continuity of inner-products and since $I$ is closed we see
	\begin{equation*}
	\IP{x,x}=\lim_{i\to\infty}\IP{x_i,x_i}\in I,
	\end{equation*}
	so $X_I$ is complete with respect to the Hilbert module norm and $X_I$ is a Hilbert $I$-module.\\
	Now to see $X_I\cong X\otimes_{\iota}I$ we compute 
	\begin{align*}
	\IP{xi,yj}=i^*\IP{x,y}_Xj=i^*\iota(\IP{x,y})j=\IP{i,\iota(\IP{x,y}_X)j}=\IP{x\otimes i,y\otimes j},\\
	\end{align*}
	so by Lemma~\ref{lemma: Well defined unitaries on tensor product} there is a well defined inner-product preserving operator $U:X\otimes_{\iota} I\to X\cdot I$ such that $U(x\otimes i)=x\cdot i$. Then $U$ is obviously surjective, hence $U$ is unitary and $X_I\cong X\otimes_{\iota}I$ as Hilbert modules. If $\phi:A\to\LL(X)$ is a left action of $A$ on $X$ then
	\begin{equation*}
	U(\phi(a)\otimes 1(x\otimes i))=U(\phi(a)x\otimes i)=\phi(a)xi=\phi(a)U(x\otimes i),
	\end{equation*}
	so $U$ intertwines the left actions and so $X\otimes_{\iota}I\cong X_I$ as $A$--$I$ correspondences.\\
	We lastly show that $I\otimes_{\phi}X\cong \overline{\phi(I)X}$. Since adjointables are $A$-linear we see that $\phi(i)(xa)~=~(\phi(i)x)a$ so the right action of $A$ preserves $\overline{\phi(I)X}$ and $\overline{\phi(I)X}$ is a Hilbert $I$-module. Now we compute
	\begin{align*}
	\IP{\phi(i)x,\phi(j)y}=\IP{x,\phi(i^*j)y}=\IP{i\otimes x,j\otimes y},\\
	\end{align*}
	so there is a well defined inner-product preserving operator $U:X\otimes_{\iota} I\to X\cdot I$ such that $U(x\otimes i)=x\cdot i$. By the same argument as Example~\ref{left and right}, since $\phi$ is non-degenerate, $V$ is surjective and hence unitary and we conclude that $I\otimes_{\phi}X\cong \overline{\phi(I)X}$.		
	Now lastly we check that if $\iota:A\to\LL(I)$ is the inclusion of $a$ as left multiplication operators, then
	\begin{equation*}
	V\big((\iota(a)\otimes 1)(i\otimes x)\big)=V(ai\otimes x)=\phi(a)\phi(i)x=\phi(a)V(i\otimes x)
	\end{equation*}
	so $\overline{\phi(I)X}$ and $I\otimes_{\phi}X$ are isomorphic as $A$--$A$ correspondences.
\end{proof}
\begin{lemma}\label{T}
	Let $X_A$ be a right $A$-module and $Y$ an $A$--$B$ correspondence with $\phi:A\to\LL(Y)$ the left action. For each $T\in\LL(X)$ there exists $(T\otimes 1)\in\LL(X\otimes_AY)$ such that $(T\otimes 1)(x\otimes y)=Tx\otimes y$ for all $x\in X$ and $y\in Y$. The map $T\mapsto T\otimes 1$ is a homomorphism and it is injective if $\phi$ is injective.
\end{lemma}
\begin{proof}
	First we wish to show that given $T\in\LL(X)$, the formula $(T\otimes 1)(x\otimes y)=Tx\otimes y$ is well defined on $X\otimes_AY$. Since for $x\in X$, $y\in Y$ and $a\in A$ we have
	\begin{equation*}
		(T\otimes 1)(xa\otimes y-x\otimes\phi(a)y)=T(x)a\otimes y-T(x)\otimes\phi(a)y,  
	\end{equation*}
	we see that $(T\otimes 1)$ is well defined on $(X\odot Y)/N$, where $N$ is subspace
	\begin{equation*}
		N=\Span\{xa\otimes y-x\otimes\phi(a)y:x\in X,y\in Y,a\in A\}.
	\end{equation*}
	For any $x\in X$ \cite[Corollary~1.25]{Crossedandunitizations} implies
	\begin{equation*}
		\normof{T}^2\IP{x,x}-\IP{Tx,Tx}\geq 0,
	\end{equation*}
	so since homomorphisms preserve positivity, employing~\ref{Panang} gives for all $y\in Y$
	\begin{equation*}
		\normof{T}\IP{y,\phi(\IP{x,x})y}-\IP{y,\phi(\IP{Tx,Tx})y}=\IP{y,\phi(\normof{T}^2\IP{x,x}-\IP{Tx,Tx})y}\geq 0.
	\end{equation*}
	Hence we obtain
	\begin{align*}
		\normof{T\otimes 1}^2&=\sup_{\normof{x\otimes y}\leq 1}\normof{Tx\otimes y}^2=\sup_{\normof{x\otimes y}\leq 1}\normof{\IP{y,\phi(\IP{Tx,Tx})y}}\\
		&\leq \sup_{\normof{x\otimes y}\leq 1}\normof{T}^2\normof{\IP{y,\phi(\IP{x,x})y}}=\normof{T}^2
	\end{align*}
	and so $T\otimes 1$ is uniformly continuous. Theorem~\ref{Sang} states there is an unique continuous extension of $T\otimes 1$ to $X\otimes_A Y$ such that if $S_n$ is a Cauchy sequence in $(X\dot Y)/N$ then
	\begin{equation*}
		(T\otimes 1)(\lim_{n\to\infty}S_n)=\lim_{n\to\infty}(T\otimes 1)(S_n).
	\end{equation*}
	That $T\mapsto T\otimes 1$ is a homomorphism is a straightforward computation, following \cite[Lemma~2.47]{Crossedandunitizations} we show that if $\phi$ is injective that then so is $T\mapsto T\otimes 1$. We do this by showing that the kernel of this map is zero. Suppose that $T\in\LL(X)$ is non-zero, that is, we may pick $x\in X$ such that $Tx\neq 0$. Then for any $y\in Y$
	\begin{align*}
	\normof{(T\otimes 1)(x\otimes y)}&=\normof{\IP{y,\phi(\IP{Tx,Tx})y}}^{1/2}\\
	&=\normof{\IP{\phi(\IP{Tx,Tx})^{1/2}y,\phi(\IP{Tx,Tx})^{1/2}y}}^{1/2}\\
	&=\normof{\phi(\IP{Tx,Tx}^{1/2})y}.\\
	\end{align*}
	Since $\phi$ is injective, there exists some $y\in Y$ such that $\phi(\IP{Tx,Tx}^{1/2})y\neq 0$, and so $(T\otimes 1)(x\otimes y)\neq 0$. Hence $T\otimes 1\neq 0$ so $T\mapsto T\otimes 1$ has zero kernel and is injective.
\end{proof}
\begin{lemma}\label{compact otimes 1}
	Let $X_A$ be a right $A$-module and $Y$ an $A$--$B$ correspondence with $\phi:A\to\LL(F)$ the left action. Suppose that $\phi$ is injective, and let $I=\phi^{-1}(\KK(Y))$. Then $I$ is an ideal in $A$, so that $X_I$ is a right $A$-module as in Lemma~\ref{ONE}, and
	\begin{equation*}
	\KK(X_I)=\{k\in\KK(X):k\otimes 1\in\KK(X\otimes_{\phi}Y)\}.
	\end{equation*}
\end{lemma}
\begin{proof}
	Since $\KK(X)$ is an ideal in $\LL(X)$ we may form the quotient $\LL(X)/\KK(X)$. Let $q$ be the quotient homomorphism. Then $I$ is precisely the kernel of $q\circ\phi$, hence $I$ is an ideal in $A$.
	Now to prove the second claim we expand of the proof of \cite[Lemma~3.6]{FreeProbTheory}. Consider the map $T_x:Y\to X\otimes_{\phi} Y$ defined by $T_x(y)=x\otimes y$. The calculation 
	\begin{align*}
		\IP{T_xy,u\otimes v}&=\IP{x\otimes y,u\otimes v}=\IP{y,\phi(\IP{x,u})v}\\
	\end{align*}
	shows that $T_x$ is adjointable	with adjoint $T^*_x(y\otimes z)=\phi(\IP{x,y})z$. Direct computation shows that
	\begin{align*}
	T_xT^*_y(u\otimes v)&=x\otimes \phi(\IP{y,u})v\\
	&=x\cdot\IP{y,u}\otimes v\\
	&=\Theta_{x,y}(u)\otimes v.\\
	\end{align*}
	Hence
	\begin{equation}
	T_xT^*_y=\Theta_{x,y}\otimes 1.\label{ZZZ}
	\end{equation}
	We also see for any $v\in Y$ that
	\begin{equation*}
		T_x^*T_y(v)=T_xy\otimes v=\phi(\IP{x,y})v, 
	\end{equation*}
	so we have
	\begin{equation*}
	T_x^*T_y=\phi(\IP{x,y}).
	\end{equation*}
	Since $\phi$ is an injective homomorphism, $A\cong \phi(A)$. Let $\widetilde{X}$ be the right $\phi(A)$ module
	\begin{equation*}
	\widetilde{X}\coloneqq\{T_x:x\in X\}\subseteq \LL(X\otimes_{\phi}Y)
	\end{equation*}
	with inner-product
	\begin{equation*}
	\IP{T_x,T_y}=T^*_xT_y=\phi(\IP{x,y})
	\end{equation*}
	and right action
	\begin{equation}\label{A}
	T_x\cdot\phi(a)=T_x\phi(a)=T_{x\cdot a}.
	\end{equation}
	The linear map $\rho:X\to\widetilde{X}$, $\rho(x)=T_x$ is clearly surjective and preserves inner-products since
	\begin{equation*}
	\IP{\rho(x),\rho(y)}=T_x^*T_y=\phi(\IP{x,y}),
	\end{equation*}
	so given Equation \eqref{A}, the maps $\rho$ and $\phi$ satisfy the conditions of Definition~\ref{Heidi Klum} and $X\cong\widetilde{X}$. This induces an isomorphism of $\KK(X)$ onto $\KK(\widetilde{X})$ that takes $\Theta_{x,y}$ to $\rho\Theta_{x,y}\rho^{-1}=\Theta_{T_x,T_y}$. By Lemma~\ref{ONE} the right $A$-module $X_I$ is equal to $X\cdot I=\{x\cdot i:x\in X,i\in I\}$. Thus for $x\in X_I$ we may write $x=yi$ for some $y\in X$ and $i\in I$, so
	\begin{align*}
	\rho(x)&=\rho(y\cdot i)=T_{y\cdot i}=T_y\phi(i)\in\widetilde{X}\cdot\phi(I)=\widetilde{X}_{\phi(I)}.\\
	\end{align*}
	Since every element of the form $T_x\phi(i)$ for $i\in\phi(I)$ arises this way, $\rho$ restricts to an isomorphism of $X_I$ onto $\widetilde{X}_{\phi(I)}$. Similarly, $\rho$ induces an isomorphism of $\KK(X_I)$ onto $\KK(\widetilde{X}_{\phi(I)})$ that takes $\Theta_{xi,yj}$ to $\Theta_{T_{xi},T_{yj}}$. Next we would like to show that $T_xT^*_y\in \KK(X\otimes_{\phi}Y)$ if and only if $T_x^*T_y\in\KK(Y)$. To this end we wish to show that if we define
	\begin{equation*}
		K\coloneqq\overline{\Span}\{T_xkT_y^*:k\in\KK(Y),x,y\in X\},
	\end{equation*}
	then $\KK(X\otimes_{\phi}Y)=K$.
	It is clear that $K$ is an ideal in $\KK(X\otimes_{\phi}Y)$, so we need only show that $K$ contains $\KK(X\otimes_{\phi}Y)$. But this is clear since $\Theta_{x\otimes y,u\otimes v}=T_x\Theta_{y,v}T^*_u$ for all $x,u\in X$ and $y,v\in Y$. Hence $K$ contains the generating elements of $\KK(X\otimes_{\phi}Y)$, so we must have $K=\KK(X\otimes_{\phi}Y)$. Now suppose that $T^*_xT_x\in\KK(Y)$. Then since $T^*_xT_x=\phi(\IP{x,x})$ is compact, $\IP{x,x}\in \phi(I)$, and so $T_x\in\widetilde{X}_{\phi(I)}$. This means that $T_x=T_y\phi(i)$ for some $y\in X$, $i\in I$, hence
	\begin{equation*}
	T_xT^*_x=T_y\phi(ii^*)T^*_y
	\end{equation*}
	is an element of $K=\KK(X\otimes_{\phi}Y)$. Now suppose instead that $T_xT^*_x\in \KK(X\otimes_{\phi}Y)$. Then by definition $T^*_x\Big(T_xT^*_x\Big)T_x=\Big(T^*_xT_x\Big)^2$ is an element of $K$ and hence $\KK(Y)$. Now $T^*_xT_x$ is a positive element of $\LL(Y)$, so $\Big(T^*_xT_x\Big)^2$ is too, and we may apply functional calculus to obtain
	\begin{equation*}
	\Big(\big(T^*_xT_x\big)^2\Big)^{1/2}=T^*_xT_x.
	\end{equation*}
	Since all functions generated by non-unital functional calculus arise from products and sums of $\big(T^*_xT_x\big)^2$ (and its adjoint but $\big(T^*_xT_x\big)^2$ is self adjoint), and since $\big(T^*_xT_x\big)^2$ is in the ideal $\KK(Y)$, all elements obtained using non-unital functional calculus must also be in $\KK(Y)$, so we deduce that $T^*_xT_x\in\KK(Y)$. Now polarisation gives us
	\begin{equation*}
	T_xT_y^*\in \KK(X\otimes_{\phi}Y)\iff T_x^*T_y\in \KK(Y)
	\end{equation*}
	as we wanted. Now suppose that $\Theta_{x,y}\in X_I$. Then $\Theta_{T_x,T_y}=T^*_xT_y\in\KK(\widetilde{X}_{\phi(I)})$ and so by definition $T_x^*T_x$ and $T_y^*T_y$ are both in $\KK(Y)$. Polarisation shows that this is exactly equivalent to $T_x^*T_y$ being an element of $\KK(Y)$, which in turn is equivalent to $T_xT_y^*$ being an element of $\KK(X\otimes_{\phi}Y)$ as we have shown above. Since each $T_xT_y^*=\Theta_{x,y}\otimes 1$ by Equation \eqref{ZZZ}, this is equivalent to $\Theta_{x,y}\otimes 1$ being in $\KK(X\otimes_{\phi}Y)$. Since each of these statements are if and only ifs, we see that
	\begin{equation*}
	\Theta_{x,y}\in\KK(X_I)\iff \Theta_{x,y}\otimes 1\in\KK(X\otimes_{\phi}Y).
	\end{equation*}
	Bilinearity of the tensor product shows that $K=\sum_n\Theta_{x_n,y_n}$ is a finite rank operator, then
	\begin{equation*}
	K\in\KK(X_I)\iff K\otimes 1\in\KK(X\otimes_{\phi}Y).
	\end{equation*}
	So now finally to pass to a general compact operator: let $k_n$ be a sequence of finite rank operators converging to $k$. By Lemma~\ref{T}, $k_n\to k$ if and only if $k_n\otimes 1\to k\otimes 1$. Thus, $k_n$ is a convergent sequence in $\KK(X_I)$ if and only if $k_n\otimes 1$ is a convergent sequence in $\KK(X\otimes_{\phi}Y)$, hence we deduce
	\begin{equation*}
	k\in\KK(X_I)\iff k\otimes 1\in\KK(X\otimes_{\phi}Y)
	\end{equation*}
	as required.
\end{proof}

%% file: chapter3.tex
\chapter{Tensor products, gradings and unitizations}
\label{C^*-algebras}

\section{Tensor products and norms}

Given two Hilbert spaces $\HH$ and $\KK$ we may consider their vector space tensor product $\HH\otimes \KK$. There is a natural choice of inner-product on $\HH\otimes\KK$ determined by
\begin{equation}\label{Vaughn}
\IP{h_1\otimes k_1,h_2\otimes k_2}_{\HH\otimes\KK}=\IP{h_1,h_2}_{\HH}\IP{k_1,k_2}_{\KK}
\end{equation}
which makes $\HH\otimes \KK$ into an inner-product space. The vector space tensor product is not guaranteed to be complete with respect to the inner-product norm that comes from this definition. When we take the Hilbert space tensor product of $\HH$ and $\KK$ we take the vector space tensor product equipped with the inner-product in equation \eqref{Vaughn} and we complete it with respect to the inner-product norm.

When taking the tensor product of two $C^*$-algebras there are more considerations concerning norms and completion. Unlike the case for Hilbert spaces, there are multiple natural norms with which may equip the algebraic tensor product that give non-isomorphic $C^*$-completions. 
Given two $C^*$-algebras $A$ and $B$ we may form the `algebraic' tensor product $A\otimes B$, which is the vector space inner-product equipped with the multiplication and involution determined by
\begin{equation*}
(a\otimes b)\cdot (c\otimes d)=ac\otimes bd\qquad\text{and}\qquad (a\otimes b)^*=a^*\otimes b^*.
\end{equation*}
For clarity, from here on out we will denote by $A\odot B$ the algebraic tensor product, and we reserve the notation $A\otimes _{\normof{\cdot}}B$ for the completion of $A\odot B$ with respect to a particular norm $\normof{\cdot}$. There are two main norms that we consider on the tensor product: the minimal and maximal norms. Let $A$ and $B$ be faithfully represented by $\pi:A\to\BB(\HH)$ and $\rho:B\to\BB(\KK)$ on Hilbert spaces $\HH$ and $\KK$. We may consider the algebraic tensor product $\BB(\HH)\odot\BB(\KK)$ of the $C^*$-algebras $\BB(\HH)$ and $\BB(\KK)$ as being included into the $C^*$-algebra $B(\HH\otimes\KK)$ of bounded operators on the Hilbert space tensor product of $\HH$ with $\KK$. If $T\in\BB(\HH)$ and $S\in\BB(\KK)$ then there is a unique operator $T\otimes S\in \BB(\HH\otimes \KK)$ such that
\begin{equation*}
	T\otimes S(h\otimes k)=Th\otimes Sk
\end{equation*}
for all $h\in\HH$ and $k\in\KK$. The map $T\odot S\mapsto T\otimes S$ is an inclusion of $\BB(\HH)\odot\BB(\KK)$ in $\BB(\HH\otimes \KK)$.
We can then define the \textit{minimal tensor product norm} by
\begin{equation*}
\normof{a\otimes b}_{\min}\coloneqq \normof{\pi(a)\otimes\rho(b)}_{\BB(\HH\otimes\KK)}.
\end{equation*}
One can show that this norm is independent of choice of representations.
The second natural norm on the tensor product is called the \textit{maximal tensor product norm} and is defined by
\begin{equation*}
\normof{a\otimes b}_{\max}\coloneqq \sup\{\normof{a\otimes b}_{N}:\normof{\cdot}_{N} \text{ is a $C^*$-semi-norm on }A\odot B\}.
\end{equation*}
By $C^*$-semi-norm we mean a semi-norm such that $\normof{a^*a}=\normof{a}^2$ and $\normof{ab}\leq\normof{a}\normof{b}$.
As the names of these norms suggest, if $\normof{\cdot}$ is another norm on $A\odot B$ then we must have
\begin{equation*}
\normof{a}_{\text{min}}\leq\normof{a}\leq\normof{a}_{\text{max}}
\end{equation*}
for all $a\in A\odot B$.
Each of these norms has its advantages and disadvantages, and in general one must be careful about which is being used. Luckily, there is a large class of $C^*$-algebras $A$ for which there is a unique norm on $A\odot B$ for every $C^*$-algebra $B$ under which the completion is a $C^*$-algebra. This definition is taken from \cite[II.9.4.1]{Encyclopedia}
\begin{definition}
	Let $A$ be a $C^*$-algebra. We say that $A$ is nuclear if $A\otimes_{\text{min}}B\cong A\otimes_{\text{max}}B$ for all $C^*$-algebras $B$.
\end{definition}
Nuclear $C^*$-algebras include all finite dimensional $C^*$-algebras, all commutative $C^*$-algebras and lots more. There is another characterisation of nuclear $C^*$-algebras as algebras which are `approximately finite dimensional'. The following theorem can be found in \cite[Theorem IV 3.1.5]{Encyclopedia}. Recall the definition of a completely positive map from Definition~\ref{Sprang}.
\begin{theorem}\label{Whut}
	If $A$ is a $C^*$-algebra, then $A$ is nuclear if and only if for every finite subset $F\subset A$ and every $\ep>0$ there exist a natural number $n$ and completely positive maps $\phi:A\to M_n(\C)$ and $\psi:M_n(\C)\to A$ such that $\normof{\phi},\normof{\psi}\leq 1$ and
	\begin{equation*}
	\normof{a-\psi\circ\phi(a)}<\ep
	\end{equation*}
	for all $a\in F$.
\end{theorem}
It is clear from Theorem~\ref{Whut} why all finite dimensional $C^*$-algebras are nuclear.

\section{Graded algebras}

Let $A$ be a $C^*$-algebra. A grading of $A$ is an automorphism $\alpha$ such that $\alpha^2=1$. Such a map is sometimes called a `$\Z_2$ grading', and we will refer to $\alpha$ as the grading automorphism. A graded $C^*$-algebra $A$ admits a direct sum decomposition $A=A_0\oplus A_1$ where
\begin{equation*}
A_0=\{a\in A:\alpha(a)=a\}\qquad\text{and}\qquad A_1=\{a\in A:\alpha(a)=-a\}.
\end{equation*}
This is because every $a\in A$ can be written as $a=a_0+a_1$ where $a_0=\frac{a+\alpha(a)}{2}$ is in $A_0$ and $a_1=\frac{a-\alpha(a)}{2}$ is in $A_1$. If $a\in A_i$ then we say that $a$ is \textit{homogeneous of degree i} and write $\partial a=i$. We say $a$ is even if $\partial a=0$ and we say $a$ is odd if $\partial a=1$. If $a$ and $b$ are homogeneous elements of degrees $\partial a=i$ and $\partial b=j$ then
\begin{equation*}
\alpha(ab)=\alpha(a)\alpha(b)=(-1)^{i+j}ab
\end{equation*}
and so $A_i\cdot A_j\subseteq A_{i+j}$ where our subscripts are mod 2. If $\alpha$ is the identity automorphism then $A_0=A$ and $A_1=\{0\}$. We call this grading the trivial grading. An algebra may admit multiple gradings, so we write the pair $(A,\alpha)$ for a $C^*$ algebra with a particular grading $\alpha$. We call a graded $C^*$-algebra \textit{inner graded} if there exists a self adjoint unitary $U\in\MM(A)$ such that $\alpha(a)=UaU$ for all $a\in A$ where $\MM(A)$ is the multiplier algebra of $A$ (see Section~\ref{Multiplier}). 
Note that if we let $\alpha_1=\alpha$ and $\alpha_0=$id then $\alpha_n$ is an action of $\Z_2$ in the sense of definition~\ref{denizen}, so every graded $C^*$-algebra carries a canonical action of $\Z_2$.
\begin{example}
	The trivial grading is the only grading on $\C$. This is because for any grading $\alpha$ we have $\alpha(1)=1$ and for any $z\in\C$
	\begin{equation*}
	\alpha(z)=z\alpha(1)=z.
	\end{equation*}
\end{example}
\begin{example}
The operator $\alpha(f)(x)=f(-x)$ on $C_0(\R)$ is a grading. This example motivates terminology: if $f\in C_0(\R)_1$ then $f(-x)=\alpha(f)(x)=-f(x)$ so $f$ is an odd function, and if $f\in C_0(\R)_0$ then $f(-x)=\alpha(f)(x)=f(x)$ so $f$ is an even function. The odd and even functions are then the odd and even subspaces for this grading.
\end{example}\label{Eg: Matrix grading}
\begin{example}
	The matrix algebra $M_2(\C)$ admits a grading
	\begin{equation*}
	\begin{pmatrix}
	a&b\\
	c&d\\
	\end{pmatrix}\mapsto\begin{pmatrix}
	a&-b\\
	-c&d\\
	\end{pmatrix}.
	\end{equation*}
	The even and odd subspaces are then diagonal and off-diagonal matrices. This grading is inner since it is implemented by 
	\begin{equation*}
	\begin{pmatrix}
	a&-b\\
	-c&d\\
	\end{pmatrix}=
	\begin{pmatrix}
	1&0\\
	0&-1\\
	\end{pmatrix}
	\begin{pmatrix}
	a&b\\
	c&d\\
	\end{pmatrix}
	\begin{pmatrix}
	1&0\\
	0&-1\\
	\end{pmatrix}.
	\end{equation*}
\end{example}

\vspace{0.5 cm}
A graded homomorphism $\pi:(A,\alpha)\to (B,\beta)$ is a $C^*$-homomorphism such that
\begin{equation*}
\pi\circ \alpha=\beta\circ \pi.
\end{equation*}
We say that $(A,\alpha)$ and $(B,\beta)$ are graded-isomorphic if there exists a bijective graded homomorphism $\pi:A\to B$. 
Given a graded $C^*$-algebra $(A,\alpha)$ and homogeneous elements $a,b$ of degree $i$ and $j$ we define the graded commutator by
\begin{equation*}
[a,b]^{\gr}=ab-(-1)^{\partial a\partial b}ba.
\end{equation*}
For $a,b\in A$ we write $a=a_0+a_1$ and $b=b_0+b_1$ where $a_0,b_0\in A_0$ and $a_1,b_1\in A_1$ and define
\begin{equation*}
	[a,b]^{\gr}=\sum_{i,j=1}^2[a_i,b_j]^{\gr}.
\end{equation*}
If $A$ is trivially graded then this is the usual commutator $[a,b]=ab-ba$. It is useful to note that if $a$ is odd then writing $b=b_0+b_1$ where $b_0\in A_0$ and $b_1\in A_1$ then
\begin{equation}\label{Neat}
[a,b]^{\gr}=[a,b_0]^{\gr}+[a,b_1]^{\gr}=ab_0-b_0a+ab_1+b_1a=ab-\alpha(b)a.
\end{equation}

Given two graded $C^*$-algebras $(A,\alpha)$ and $(B,\beta)$ we may form a graded $C^*$-algebra $A\Otimes B$. We equip the algebraic tensor product $A\odot B$ with the multiplication and adjoint given on homogeneous elements by
\begin{equation}\label{Def}
a\Otimes b\cdot c\Otimes d=(-1)^{\partial a\partial b}ac\Otimes bd\qquad\text{and}\qquad (a\Otimes b)^*=(-1)^{\partial a\partial b}a^*\Otimes b^*.
\end{equation}
The \textit{graded tensor product} $A\Otimes  B$ is then obtained by completing with respect to the minimal tensor product norm. The grading $\alpha\otimes\beta$ on $A\Otimes  B$ is then determined by
\begin{equation*}
\alpha\otimes\beta(a\otimes b)=\alpha(a)\otimes\beta(b).
\end{equation*}
In terms of the graded components of $A$ and $B$, we have
\begin{align*}
(A\Otimes B)_0&=(A_0\Otimes  B_0)\oplus (A_1\Otimes  B_1),\;\;\text{ and }\\
(A\Otimes B)_1&=(A_1\Otimes B_0)\oplus(A_0\Otimes B_1).\\
\end{align*}
\begin{example}
	Let $B$ be a trivially graded $C^*$-algebra. The Clifford algebras $\Cliff_n$ come with a natural grading $\alpha_{\Cliff_n}$ (see Section~\ref{Cliff}), and we may form the graded tensor product $B\Otimes \Cliff_n$. Since all elements in $B$ are of degree zero, the minus signs appearing in~\eqref{Def} disappear and $B\Otimes \Cliff_n\cong B\otimes\Cliff_n$ as ungraded algebras. Note that since $\Cliff_n$ is finite dimensional, there is no ambiguity about with tensor norm we are using. Even though $B$ is trivially graded, the graded tensor product $B\Otimes \Cliff_n$ is not. We have
	\begin{equation*}
	\alpha_{B\Otimes \Cliff_n}(b\otimes e)=b\otimes\alpha_{\Cliff_n}(e).
	\end{equation*}
	If $B$ is not trivially graded then the structure becomes slightly more convoluted. In Example~\ref{Good} we look at the structure of $B\Otimes \Cliff_1$ when $B$ is not trivially graded.
\end{example}
The following Lemma is from \cite[Proposition~14.5.1]{Blackadar}.
\begin{lemma}\label{lemma: graded tensor isom ungraded}
	Suppose that $(A<\alpha_A)$ and $(B,\alpha_B)$ are graded $C^*$-algebras and that the grading on $B$ is inner. That is, there is some unitary $U\in\MM(B)$ such that $\alpha_B(b)=UbU^*$ for all $b\in B$. Then the graded tensor product $A\Otimes B$ and the ungraded tensor product $A\otimes B$ are isomorphic.
\end{lemma}
\begin{remark}\label{remark: graded and ungraded tensor product}
	In Example~\ref{Eg: Matrix grading} we saw that the grading on $M_2(\C)$ determined by 
	\begin{equation*}
		\alpha_{M_2}\Bigg(\begin{pmatrix}
			a&b\\
			c&d\\
		\end{pmatrix}\Bigg)=
		\begin{pmatrix}
			a&-b\\
			-c&d\\
		\end{pmatrix}
	\end{equation*}
	is inner. Thus by Lemma \ref{lemma: graded tensor isom ungraded}, the graded tensor product of any $C^*$-algebra $A$ with $M_2(\C)$ is isomorphic to the ungraded tensor product. We will need to make use of this fact in Chapter~\ref{KK-theory} when proving periodicity of $KK$-theory.
\end{remark}
If $X_B$ is a Hilbert $B$ module $B$ is graded by $\alpha_B$, a grading on $X$ is a linear map $\alpha_X:X\to X$ such that for all $x\in X$ and $b\in B$ we have the following: 
\begin{align*}
	\alpha_X^2(x)&=x\\
	\alpha_X(xb)&=\alpha_X(x)\alpha_B(b)\\
	\alpha_B(\IP{x,y})&=\IP{\alpha_X(x),\alpha_X(y)}.\\
\end{align*}
Define the subspaces
\begin{align*}
X_0&=\{x\in X:\alpha_X(x)=x\}=\{\alpha_X(x)+x:x\in X\}\\
X_1&=\{x\in X:\alpha_X(x)=-x\}=\{\alpha_X(x)-x:x\in X\}.\\
\end{align*}
As vector spaces $X$ decomposes as $X=X_0\oplus X_1$, though $X_0$ and $X_1$ are not sub-modules unless $\alpha_B$ is trivial (this is proved in Lemma~\ref{lemma: trivially graded A gives direct sum}).\\
Let $(A,\alpha_A)$ and $(B,\alpha_B)$ be graded $C^*$-algebras and let $X$ be an $A$--$B$ correspondence with left action $\phi$. We say that $X$ is graded by $\alpha_X$ if $\alpha_X$ defines a Hilbert $B$ module grading and $\alpha_X$ has the additional properties that for all $x\in X$ and $a\in A$ we have $\alpha_X(\phi(a)x)=\phi(\alpha_A(a))\alpha_X(x)$.

We can re-express the conditions on $\alpha_X$ as $X_iB_j\subseteq X_{i+j}$, $\phi(A_i)X_j\subseteq X_{i+j}$ and $\IP{X_i,X_j}\subseteq B_{i+j}$. Since $\alpha_A$ is an automorphism of $A$ we have $\normof{\alpha_A(a)}=\normof{a}$ for all $a\in A$. Hence we have
\begin{align*}
	\normof{\alpha_X}=\sup_{\normof{x}\leq 1}\normof{\alpha_X(x)}&=\sup_{\normof{x}\leq 1}\normof{\IP{\alpha_X(x),\alpha_X(x)}}^{1/2}=\sup_{\normof{x}\leq 1}\normof{\alpha_A(\IP{x,x})}^{1/2}\\
	&=\sup_{\normof{x}\leq 1}\normof{\IP{x,x}}^{1/2}=\sup_{\normof{x}\leq 1}\normof{x}=1\\
\end{align*}

A grading $\alpha_X$ on a correspondence (or Hilbert module) induces a grading $\widetilde{\alpha}_X$ on $\LL(X)$ given by
\begin{equation*}
\widetilde{\alpha}_X(T)=\alpha_X\circ T\circ\alpha_X.
\end{equation*}
In general $\alpha_X$ is not an adjointable operator on $X$ unless $B$ is trivially graded, since $\alpha_X(xb)=\alpha_X(x)\alpha_B(b)$ shows that $\alpha_X$ is not $B$-linear. So it is not obvious a priori that $\alpha_X T\alpha_X$ is in $\LL(X)$. However the computation
\begin{equation*}
\IP{\alpha_X(T\alpha_X(x)),y}=\alpha_B(\IP{T\alpha_X(x),\alpha_X(y)})=\alpha_B(\IP{\alpha_X(x),T^*\alpha_X(y)})=\IP{x,\alpha_XT^*\alpha_X(y)}
\end{equation*}
shows that $\alpha_X\circ T^*\circ\alpha_X$ is an adjoint for $\alpha_X\circ T\circ \alpha_X$. If $\phi:A\to\LL(X)$ is a left action of $A$ then
\begin{equation*}
\widetilde{\alpha}_X(\phi(a))x=\alpha_X(\phi(a)\alpha_X(\alpha(x)))=\phi(\alpha_A(a))x,
\end{equation*}
so $\widetilde{\alpha}_X(\phi(a))=\phi(\alpha_A(a))$. Thus we may interpret the property $\alpha_X(\phi(a)x)=\phi(\alpha_A(a))\alpha_X(x)$ as requiring that $\phi$ be a graded homomorphism for the gradings $\alpha_A$ on $A$ and $\widetilde{\alpha}_X$ on $\LL(X)$.

\begin{example}\label{Graded compact isomorphism}
	Recall in Example~\ref{Fang} we saw that $\KK(A_A)\cong A$ via the map determined by $\theta_{a,b}\mapsto a^*b$. If $A$ is graded by $\alpha_A$ then $\alpha_A$ is also a grading for $A_A$ because for all $a,b\in A$ we have
	\begin{align*}
	\alpha_A(a\cdot b)&=\alpha_A(a)\alpha_A(b)\quad\text{ and }\\
	\alpha_A(\IP{a,b})&=\alpha_A(a^*b)=\alpha(a)^*\alpha(b)=\IP{\alpha(a),\alpha(b)}.
	\end{align*}
	This induces the grading on $\widetilde{\alpha_A}$ on $\LL(A_A)$. We see that
	\begin{align*}
		\phi\widetilde{\alpha_A}((\Theta_{a,b}))&=\phi(\alpha_A\Theta_{a,b}\alpha_A)=\phi(\Theta_{\alpha_A(a),\alpha_A(b)})\\
		&=\alpha(a)^*\alpha(b)=\alpha_A(a^*b)=\alpha_A(\phi(\Theta_{a,b})),\\
	\end{align*}
	so $\phi$ is in fact a graded isomorphism between $\KK(A_A)$ and $A$.
\end{example}
\begin{example}\label{eg: direct sum grading}
	Given a graded $C^*$-algebra $(A,\alpha_A)$ and a finite amount of graded Hilbert $A$-modules $(X_i,\alpha_{X_i})$ we may define a grading $\alpha_{\oplus X_i}$ on the direct sum $\bigoplus_{i=1}^n X_i$ such that if $x_i\in X_i$ then 
	\begin{equation*}
	\alpha_{\oplus X_i}\Big(\sum_{i=1}^n x_i\Big)=\sum_{i=1}^n\alpha_{X_i}(x_i).
	\end{equation*}
	Suppose we have a countably infinite amount of graded Hilbert modules $(X_i,\alpha_{X_i})$. Then since $\normof{\alpha_{X_i}}=1$ for each $i$, if we have a sequence $x_i\in X_i$ such that $\sum_{i=1}^{\infty} \IP{x_i,x_i}$ converges in $A$, then
	\begin{equation*}
		\sum_{i=1}^{\infty}\IP{\alpha_{X_i}(x_i),\alpha_{X_i}(x_i)}\leq \sum_{i=1}^{\infty}\IP{x_i,x_i}.
	\end{equation*}
	Hence there is a well defined grading $\alpha_{\oplus X_i}$ on $\bigoplus_{i=1}^{\infty}X_i$ such that
	\begin{equation*}
		\alpha_{\oplus X_i}\Big(\sum_{i=1}^{\infty} x_i\Big)=\sum_{i=1}^{\infty}\alpha_{X_i}(x_i).
	\end{equation*}
	In particular, there is a grading $\alpha_A^{\infty}$ on $\HH_A$ (see Definition~\ref{def: standard moduel}) such that for $(a)\in\HH_A$, we have
	\begin{equation*}
		\alpha_A^{\infty}(a)_i=\alpha_A(a_i).
	\end{equation*}
	If $X_i^0$ and $X_i^1$ denote the even and odd subspaces of $X_i$ under $\alpha_X^i$ then the even and odd subspaces of $\oplus_i X_i$ are given by
	\begin{equation*}
		\Big(\bigoplus_i X_i\Big)^0=\bigoplus_iX_i^0\qquad\text{ and }\qquad\Big(\bigoplus_i X_i\Big)^1=\bigoplus_iX_i^1.
	\end{equation*}
\end{example}

If $_AX_B$ is a graded $C^*$-correspondence with $\alpha$ the grading operator, and if $X$ decomposes as $X_1\oplus X_2$, then we say that this direct sum is graded if $\alpha(X_i)\subseteq X_i$, so that each $X_i$ is then a graded correspondence with grading operator $\alpha\big|_{X_i}$.
\begin{example}
	If $(A,\alpha)$ and $(B,\beta)$ are graded $C^*$-algebras and $\phi:B\to A$ is a graded homomorphism then the $B$-$A$ correspondence $_BA_A$ of Example~\ref{Marnie} has grading $\alpha$ since
	\begin{equation*}
	\alpha(\phi(b)xa)=\phi(\beta(b))\alpha(x)\alpha(a)
	\end{equation*}
	for all $b\in B$ and $a,x\in A$.
\end{example}
If $\alpha_X$ and $\alpha_Y$ are two gradings on $A$--$B$ and $B$--$C$ correspondences $_{\phi}X_B$ and $_{\psi}Y_C$ then there is a well defined operator $\alpha_X\otimes\alpha_Y$ on $X\otimes_BY$ such that
$\alpha_X\otimes\alpha_Y(x\otimes y)=\alpha_X(x)\otimes\alpha_Y(y)$. This is because for all $x\in X$, $y\in Y$ and $b\in B$ we have
\begin{align*}
\alpha_X\otimes\alpha_Y(x\otimes\psi(b)y-xb\otimes y)&=\alpha_X(x)\otimes\alpha_Y(\psi(b)y)-\alpha_X(xb)\otimes\alpha_Y(y)\\
&=\alpha_X(x)\otimes\psi(\alpha_B(b))\alpha_Y(y)-\alpha_X(x)\alpha_B(b)\otimes\alpha_Y(y),\\
\end{align*}
so $\alpha_X\otimes\alpha_Y$ preserves the subspace
\begin{equation*}
\Span\{x\otimes\phi(b)y-bx\otimes y:x\in X,y\in Y,b\in B\}.
\end{equation*}
\section{Unitizations}\label{Multiplier}
In studying $C^*$-algebras the first thing one does is study unital $C^*$-algebras. The tools developed for unital $C^*$-algebras like the spectrum and functional calculus can still be used for non-unital algebras, but one needs to study how to embed non-unital algebras into unital ones where these tools make sense.

The following definitions and results come from \cite[Section 2.3]{Raeburn Williams}.
\begin{definition}
	Let $A$ be a $C^*$-algebra and $I$ be an ideal in $A$. We say that $I$ is an essential ideal if for every ideal $J\subset A$, $I\cap J\neq \emptyset$.
\end{definition}
\begin{lemma}\label{I cap J}
	Let $A$ be a $C^*$-algebra and let $I$ and $J$ be ideals in $A$. Let $I\cdot J$ be the closed span of elements of the form $ij$ with $i\in I$ and $j\in J$. Then $I\cap J=I\cdot J$.
\end{lemma}
\begin{proof}
	Since $I$ and $J$ are ideals, linear combinations of elements of the form $ij$ with $i\in I$ and $j\in J$ are elements of $I\cap J$. Since $I$ and $J$ are closed so is $I\cap J$, so taking the closure of said linear combinations gives $I\cdot J\subseteq I\cap J$. For reverse containment let $e_{\lambda}$ be an approximate identity for $J$. Then since $I\cap J\subseteq J$, $e_{\lambda}$ is also an approximate identity for $I\cap J$, and so every $x\in I\cap J$ can be written
	\begin{equation*}
		x=\lim_{\lambda\in\Lambda}xe_{\lambda}.
	\end{equation*}
	The net $xe_{\lambda}$ consists of elements in $I\cdot J$, so since $I\cdot J$ is closed, the limit $x$ is an element of $I\cdot J$. We conclude that $I\cap J\subseteq I\cdot J$ and so $I\cap J= I\cdot J$.
\end{proof}
\begin{lemma}\label{Intern}
	An ideal $I$ is essential if and only if $aI=\{0\}$ implies $a=0$ for all $a\in A$.
\end{lemma}
\begin{proof}
	This is directly taken from \cite[Lemma 2.36]{RaeburnWilliams}. For $a\in A$ let $J_a$ be the ideal $\overline{\Span}AaA$ generated by $a$. Since $A\cdot I=I$ we have $J_a\cdot I=\{0\}$ if and only if $a\cdot I=\{0\}$. Hence if $I$ is an essential ideal and $aI=\{0\}$ then by Lemma~\ref{I cap J} we have $J_a\cap I=\{0\}$ and so $J_a=\{0\}$ whence $a=0$. Conversely, suppose that $aI=\{0\}$ implies $a=0$. If $J$ is a non-zero ideal and $a\in J$ is not zero then
	\begin{equation*}
	a\neq 0\implies aI\neq\{0\}\implies J_a\cap I\implies J\cap I\neq 0.
	\end{equation*}
\end{proof}
The two characterisations of essential ideals above will both be necessary in what follows.
\begin{definition}
	Let $A$ be a $C^*$-algebra. A unitization of $A$ is a unital $C^*$-algebra $B$ and an embedding $\iota:A\hookrightarrow B$ such that $\iota(A)$ is an essential ideal of $B$.
\end{definition}
\begin{example}\label{Mang}
	Recall in Example~\ref{Fang} we deduced that $A\cong \KK(A_A)$ where the left action $L_a:A\to \KK(A_A)\subset \LL(A_A)$ of $A$ on $A_A$ is defined by $L_a(b)=ab$ (we have changed notation to $L_a$ here to match convention). Since the identity operator is adjointable (it is self-adjoint), $\LL(A_A)$ is unital. If $T\in\LL(A_A)$ satisfies $TK=0$ for all $K\in \KK(A_A)$ then $T\Theta_{a,b}c=Tab^*c=0$ for all $a,b,c\in A$. Since products of the form $ab^*c$ are dense in $A$, we deduce that $T=0$. Thus $T\KK(A_A)=\{0\}$ implies $T=0$ and $\KK(A_A)$ is an essential ideal in $\LL(A_A)$. So $L_a:A\to\LL(A_A)$ is a unitization of $A$.
\end{example}
\begin{example}\label{minimal unitization}
	Suppose that $A$ is a non-unital $C^*$-algebra. Consider the vector space direct sum $A\oplus \C$. The operations $(a,\lambda)(b,\mu)\coloneqq (ab+\mu a+\lambda b,\lambda\mu)$ and $(a,\lambda)^*\coloneqq(a^*,\overline{\lambda})$ define a $*$-algebra structure on $A\oplus \C$. Since for all $(a,\lambda)\in A\oplus\C$ we have
	\begin{equation*}
		(a,\lambda)(0,1)=(a,\lambda)=(0,1)(a,\lambda),
	\end{equation*}
	the $*$-algebra $A\oplus\C$ is unital. In \cite[Lemma~1.3]{Crossedandunitizations} it is shown that the formula $\normof{(a,\lambda)}=\sup_{\normof{b}\leq 1}\normof{ab+\lambda b}$ defines a $C^*$-norm on $A\oplus\C$ under which it is a $C^*$-algebra. We claim that the map $\iota:A\to A\oplus\C$ defined by $\iota(a)=(a,0)$ is a unitization of $A$. The map $\iota$ is clearly a homomorphism, and we have
	\begin{equation*}
		\normof{\iota(a)}=\sup_{\normof{b}\leq 1}\normof{ab}\leq\sup_{\normof{b}\leq 1}\normof{a}\normof{b}=\normof{a}.
	\end{equation*}
	Letting $b=a^*/\normof{a^*}$ gives the reverse inequality and we deduce that $\normof{\iota(a)}=\normof{a}$. Hence $\iota$ is injective. The image of $\iota$ is clearly in ideal in $A\oplus\C$, so we need only show that it is essential. Suppose that for some $(b,\lambda)$ we have $(ba+\lambda a,0)=(b,\lambda)(a,0)=(0,0)$ for all $a\in A$. Then if $\lambda\neq 0$ we have $-ba/\lambda=a$ for all $a\in A$, and so $-b/\lambda$ is a unit for $A$. Since by assumption $A$ is non-unital by hypothesis we have a contradiction and deduce that $\lambda=0$. Then we have $ba=0$ for all $a\in A$ and so $b=0$. Hence $\iota(A)$ is an essential ideal in $A\oplus\C$ and so $A\oplus\C$ is a unitization of $A$.
\end{example}
If $A$ is a unital $C^*$-algebra then $A$ together with the identity map is a unitization. This is the only unitization of $A$. If $A$ could be embedded a an essential ideal in a larger unital algebra $B$ then for $b\in B\bs A$, $b1_A$ is an element of $A$, and so $b-b1_A\neq 0$. However $(b-b1_A)a=0$ for all $a\in A$, and $b-b1_a\neq 0$, so Lemma~\ref{Intern} shows that $A$ is not essential. If we did not require that $A$ be an essential ideal, then we could keep embedding $A$ in larger and larger $C^*$-algebras to continue creating larger unitizations. If we require that $\iota(A)$ be an essential ideal, then there exists a `maximal' unitization, which turns out to be unique up to isomorphism.
\begin{definition}\label{licence}
	We say that a unitization $\iota:A\hookrightarrow B$ is maximal if for every embedding $j:A\hookrightarrow C$ into a $C^*$-algebra $C$ as an essential ideal, there exists a homomorphism $\phi:C\to B$ such that the diagram
	\begin{equation}\label{commutative diagram}
	\begin{tikzcd}[row sep=tiny]
	& B \\
	A \arrow[ru, "\iota", hook]\arrow[rd, swap, "j", hook]
	& \\
	&C\arrow{uu}[swap]{\phi}\\ 
	\end{tikzcd}
	\end{equation}
	commutes.
\end{definition}
\begin{remark}
	One may it unusual not to require that the map $\phi$ in Definition~\ref{licence} be injective, but we will see in Corollary~\ref{corollary} that injectivity of $\phi$ is in fact automatic.
\end{remark}
\begin{proposition}\label{props}
	Suppose that $X_A$ is a right-Hilbert module, that $C$ is a $C^*$-algebra with $B$ an ideal in $C$ and that $\phi:B\to \LL(X)$ is non-degenerate in the sense of Remark~\ref{Sam}. Then there exists a unique homomorphism $\widetilde{\phi}:C\to\LL(X)$ such that $\widetilde{\phi}|_{B}=\phi$. We call $\widetilde{\phi}$ the extension of $\phi$ to $C$. If $B$ is essential in $C$ and $\phi$ is injective then $\widetilde{\phi}$ is also injective. 
\end{proposition}
\begin{proof}
	Since $\phi$ is non-degenerate, for $x\in X$ we may find $b_n\in B$ and $x_n\in X$ such that $x=\sum_n\phi(b_n)x_n$. For $c\in C$ we aim to define
	\begin{equation*}
		\widetilde{\phi}(c)x=\sum_n\phi(cb_n)x_n.
	\end{equation*}
	Fix an approximate identity $e_{\lambda}$ for $B$ with $0\leq e_{\lambda}\leq 1$. For finite sums, if $\sum_n\phi(a_n)y_n=\sum_n\phi(b_n)x_n$ then we have
	\begin{align*}
	\Bignormof{\sum_n\phi(ca_n)y_n-\phi(cb_n)x_n}&=\lim_{\lambda\to\infty}\Bignormof{\sum_n\phi(ce_{\lambda}a_n)y_n-\phi(ce_{\lambda}b_n)x_n}\\
	&=\lim_{\lambda\to\infty}\Bignormof{\phi(ce_{\lambda})\sum_n\left[\phi(a_n)y_n-\phi(b_n)x_n\right]}\\
	&=0.\\
	\end{align*}
	Thus there exists a linear map $M_c:\Span\phi(B)X\to X$ such that $M_c(\phi(b)x)=\phi(cb)x$ for all $b\in B$ and $x\in X$. Since $\normof{\phi(ce_{\lambda})}\leq\normof{c}$ for all $\lambda$, for any finite linear combination $\sum_n\phi(b_n)x_n\in\Span\phi(B)X$,
	\begin{align*}
	\Bignormof{M_c\sum_n\phi(b_n)x_n}&=\lim_{\lambda\in\Lambda}\Bignormof{\sum_n\phi(ce_{\lambda}b_n)x_n}\\
	&=\lim_{\lambda\in\Lambda}\Bignormof{\phi(ce_{\lambda})\sum_n\phi(b_n)x_n}\\
	&=\leq \normof{c}\Bignormof{\sum_n\phi(b_n)x_n}.\\
	\end{align*}
	Thus $M_c$ is bounded and hence uniformly continuous. Theorem~\ref{Sang} implies that there is a uniformly continuous linear map $\widetilde{\phi}(c):\overline{\Span}\phi(B)X\to X$ such that $\widetilde{\phi}(c)$ restricted to $\Span\phi(B)X$ is equal to $M_c$. Non-degeneracy of $\phi$ gives $\overline{\Span}\phi(B)X=X$, so $\widetilde{\phi}(c)$ is a well defined operator on $X$. Since $c\in C$ was arbitrary, we have a map $\widetilde{\phi}$ that takes an element $c\in C$ to an operator $\widetilde{\phi}(c)$ on $X$. Let $\sum_n\phi(b_n)x_n$ and $\sum_n\phi(a_n)y_n$ be in $\Span\phi(B)X$. We compute
	\begin{align*}
		\sum_{n,m=1}\IP{\widetilde{\phi}(c)\phi(b_n)x_n,\phi(a_m)y_m}&=\sum_{n,m=1}\IP{\phi(cb_n)x_n,\phi(a_m)y_m}=\sum_{n,m=1}\IP{x_n,\phi(b_n^*c^*a_m)y_m}\\
		&=\sum_{n,m=1}\IP{\phi(b_n)x_n,\phi(c^*a_m)y_m}=\sum_{n,m=1}\IP{\phi(b_n)x_n,\widetilde{\phi}(c^*)\phi(a_m)y_m}.\\
	\end{align*}
	Since $\widetilde{\phi}$ is determined by its values on $\Span\phi(B)X$, we deduce that $\widetilde{\phi}(c^*)$ is an adjoint for $\widetilde{\phi}(c)$. We deduce that each $\widetilde{\phi}(c)$ is an element of $\LL(X)$. For all $c,d\in C$ and $\sum_n \phi(b_n)x_n\in\Span\phi(B)X$ we have
	\begin{equation*}
		\widetilde{\phi}(c)\widetilde{\phi}(d)\sum_n\phi(b_n)x_n=\sum_n\phi(cdb_n)x_n=\widetilde{\phi}(cd)\sum_n\phi(b_n)x_n.
	\end{equation*}
	Since both $\widetilde{\phi}(cd)$ and $\widetilde{\phi}(c)\widetilde{\phi}(d)$ are continuous and agree on $\Span\phi(B)X$ we deduce that $\widetilde{\phi}(cd)=\widetilde{\phi}(c)\widetilde{\phi}(d)$. Thus, since $\widetilde{\phi}(c)^*=\widetilde{\phi}(c^*)$, we deduce that $\widetilde{\phi}$ is a homomorphism from $C$ into $\LL(X)$.\\	
	If $b\in B$ then for all $\sum_n\phi(b_n)x_n\in\Span\phi(B)X$,
	\begin{equation*}
	\widetilde{\phi}(b)\sum_n\phi(b_n)x_n=\sum_n\phi(bb_n)x_n=\phi(b)\sum_n\phi(b_n)x_n
	\end{equation*}
	so $\widetilde{\phi}|_{B}=\phi$. If $\overline{\phi}$ is another continuous extension of $\phi$ to $C$ such that $\widetilde{\phi}|_B=\overline{\phi}|_B=\phi$, then
	\begin{equation*}
	\left(\widetilde{\phi}(c)-\overline{\phi}(c)\right)\sum_n\phi(b_n)x_n=\left(\phi(cb_n)-\phi(cb_n)\right)x_n=0.
	\end{equation*}
	Since $\Span\phi(B)X$ is dense in $X$ and $\widetilde{\phi}$ and $\overline{\phi}$ are continuous, we deduce that $\widetilde{\phi}=\overline{\phi}$. If $\phi$ is injective and $B$ is essential then $\Ker(\widetilde{\phi})\cap B=\Ker(\phi)\cap B=\{0\}$. Since $B$ is essential, we deduce that $\Ker(\widetilde{\phi})=\{0\}$ whence $\widetilde{\phi}$ is injective.
\end{proof}
\begin{corollary}\label{corollary}
	If $i:A\hookrightarrow B$ is a maximal unitization of $A$ and $j:A\hookrightarrow C$ is an embedding of $A$ in $C$ as an essential ideal then there is only one homomorphism $\phi:C\to B$ such that \eqref{commutative diagram} commutes, and it is injective.
\end{corollary}
\begin{proof}
	If $c\in \Ker(\phi)\cap j(A)$ then $0=\phi(c)=\phi(j(a))=i(a)$ for some $a\in A$. Since $i$ is injective we have $a=0$, so since $j$ is injective $c=j(a)=0$. Thus $\Ker(\phi)\cap j(A)=\{0\}$ whence $\Ker(\phi)=\{0\}$ since $j(A)$ is essential. So $\phi$ is injective. If $\psi:C\to B$ is another homomorphism such that \eqref{commutative diagram} commutes, then for $c\in C$
	\begin{align*}
	\left(\phi(c)-\psi(c)\right)i(a)&=\phi(c)\phi(j(a))-\psi(c)\psi(j(a))=\phi(cj(a))-\psi(cj(a))=0\\
	\end{align*}
	because $cj(a)\in j(A)$ and $\phi\big|_{j(A)}=\psi\big|_{j(A)}$. Since $i(A)$ is essential in $B$, Lemma~\ref{Intern} shows that $\phi(c)-\psi(c)=0$. So $\phi\cong\psi$.
\end{proof}
\begin{remark}\label{WHAT}
	Corollary~\ref{corollary} tells us that if $i:A\to B$ is a maximal unitization then every other unitization of $A$ must be contained in $B$. In particular the maximal unitization of $A$ is a maximal element in the set of all unitizations of $A$.
\end{remark}
\begin{proposition}
	Let $A$ be a $C^*$-algebra. There exists a unique maximal unitization $(\MM(A),i_A)$ of $A$, moreover $(\MM(A),i_A)\cong(\LL(A_A),L_a)$ from Example~\ref{Mang}. We call $\MM(A)$ the multiplier algebra of $A$.
\end{proposition}
\begin{proof}
	If $i:A\to B$ and $j:A\to C$ are two maximal unitizations of $A$ then there exist homomorphisms $\phi:B\to C$ and $\psi:C\to B$ such that $i(a)=\psi(j(a))=\psi(\phi(i(a)))$ for all $a\in A$. For $b\in B$, since $bi(a)\in i(A)$, we have
	\begin{equation*}
	(\psi(\phi(b))-b)i(a)=\psi(\phi(bi(a)))-bi(a)=0,
	\end{equation*}
	whence $(\psi(\phi(b))-b)i(A)=\{0\}$. Since $i(A)$ is essential we deduce that $\psi(\phi(b))=b$ whence $\psi\circ\phi$ is the identity on $B$. Repeating this argument on the essential idea $j(A)$ shows $\phi\circ\psi$ is the identity on $C$ whence $B\cong C$. It remains to show that the unitization $\LL(A_A)$ is maximal. Let $j:A\hookrightarrow C$ be a unitization of $C$ and let $\phi:A\to \LL(A_A)$ be the embedding of $A$ as left multiplication operators. We saw in Example~\ref{Fish} that $\phi$ is non-degenerate, so Proposition~\ref{props} applied to $X=A_A$ gives us an extension $\widetilde{\phi}:C\to \LL(A)$ such that $\widetilde{\phi}\circ j=\widetilde{\phi}|_A=\phi$.
\end{proof}
It is worth quickly restating Proposition~\ref{props} in terms of Multiplier algebras.
\begin{proposition}\label{REF}
	If $A$ and $B$ are $C^*$-algebras and $A\to\LL(B_B)\cong\MM(B)$ is a non-degenerate homomorphism, then there exists a unique extension $\widetilde{\phi}:\MM(A)\to\MM(B)$ that agrees with $\phi$ on $A$. If $\phi$ is injective then so is $\widetilde{\phi}$.
\end{proposition}
Now that we know about unitizations we can talk about the spectrum of non-unital algebras.
\begin{definition}\label{L}
	Let $A$ be a non-unital $C^*$-algebra. We define the spectrum $\sigma(a)$ of an element $a\in A$ to be the spectrum $\sigma_{A^+}$ of $a$ in the minimal unitization $A^+$ of $A$ in the sense of Example~\ref{minimal unitization}.
\end{definition}
\begin{remark}
	Since every unitization of $A$ is (unitally) contained in the Multiplier algebra $\MM(A)$ (recall Remark~\ref{WHAT}), the Spectral Permanence Theorem says that the spectrum of $a$ inside every unitization is the same as the spectrum of $a$ inside $\MM(A)$. So Definition~\ref{L} is independent of choice of unitization. It can be useful to compute the spectrum of an element in the minimal unitization since all elements of $A^+$ are quite tractable.
\end{remark}
By considering non-unital algebras $A$ embedded inside their minimal unitizations where we have a well defined spectrum, the usual continuous functional calculus holds. That is, for $a\in A$ there is a $*$-isomorphism between $C(\sigma(a)_{A^+})$ and $C^*(a,1)$. The algebra $C^*(a,1)$ is not however contained in the original algebra $A$, so we may ask if there is a restriction of this isomorphism onto the $C^*$-algebra $C^*(a)$. Since $A$ is an ideal in $A^+$, every element of $A$ is not invertible in $A^+$ and so contains 0 in the spectrum. The pre-image of $C^*(a)$ under the functional calculus consists of continuous functions of the spectrum spanned by polynomials without constant term. Every function of this form vanishes at zero, so we determine for $a\in A$ that functional calculus restricts to an isomorphism of $\{f\in C(\sigma(a)):f(0)=0\}$ onto $C^*(a)$. In particular the function $x^{\alpha}$ vanishes at zero for all $\alpha>0$ and so for self-adjoint $a\in A$, $a^{\alpha}$ is well defined and is an element of $A$ rather than an element of $A^+$.

%% file: chapter4.tex
\chapter{Special $C^*$-algebras}
\label{Special Funky C^*-algebras}

\section{Clifford algebras}\label{Cliff}
Essential in the study of $K$-theory are Clifford algebras. In this section we carefully construct and analyse these fascinating algebras.
\begin{definition}
	Let $V$ be a finite dimensional real vector space with an inner-product $\IP{\cdot,\cdot}$.
	Let $T(V)\subset \oplus_{n=0}^{\infty}V^{\otimes n}$ be the span of all finite sums of elements $v_1\otimes\dots\otimes v_n$ with multiplication given by $a\cdot b=a\otimes b$ and declaring that this multiplication is distributive. We make the convention that $V^0=\R$ and that multiplication by elements in $V^0$ is scalar multiplication. We may define an inner-product on tensor powers by $\IP{a\otimes b,c\otimes d}\coloneqq \IP{a,c}\IP{b,d}$ which then lifts to an inner-product on $T(V)$ by declaring $\IP{a,b}=0$ when $a,b$ are in differing tensor powers. Let $J(V)\subset T(V)$ be the ideal generated by elements in of the form $v\cdot v+\IP{v,v}1$. Then the Clifford algebra $\Cl(V)$ defined as the quotient space $T(V)/J(V)$.
\end{definition}
This definition leaves us with a vector space having a natural multiplication satisfying the fundamental relation $v\cdot v=-\IP{v,v}1$. By applying this to $\IP{v+w,v+w}$ we learn that $v\cdot w+w\cdot v=-2\IP{v,w}$. We have a natural inclusion $V\hookrightarrow \text{Cl}(V)$ by including $v\in V$ into $T(V)$ and composing with the quotient map. Thus we may make sense of the expression $v\cdot w$ in the Clifford algebra for elements $v,w\in V$. A basis of $\Cl(V)$ is then given by
\begin{equation*}
\{v_{i_1}\cdot v_{i_2}\cdot ... \cdot v_{i_k}:1\leq i_1<\dots<i_k\leq \dim(V)\}\cup \{1\}
\end{equation*}
where $\{v_i\}$ is a basis for $V$ and $1\in\R$. We usually abbreviate  $v_1\cdot v_2$ to just $v_1v_2$.\\
We will be interested in the algebras $\Cliff_n$. The algebra Cliff$(V)$ is a real algebra for $V$ a real vector space, so one may think that to form complex Clifford algebras we take $V$ to be a complex vector space. It turns out however that there is a more natural definition for complex Clifford algebras as
\begin{equation*}
\Cliff(V)=\text{Cliff}(V)\otimes_{\R}\C.
\end{equation*}
We write $\otimes_{\R}$ to denote that we are considering Cliff$(V)$ and $\C$ real vector spaces and taking the tensor product with coefficients in $\R$.
Note that since Cliff$(V)$ is a real vector space, we are not able to make sense of the tensor product with coefficients in $\C$. We are however able to make sense of $\Cliff(V)$ as a complex vector space with this definition by defining for $y,z\in \C$ and $v\in V$
\begin{equation*}
z(v\otimes y)=v\otimes zy.
\end{equation*}
Often we will be concerned with the case when $V=\R^n$ with the usual dot product. We abbreviate Cliff$(\R^n)$ and $\Cliff(\R^n)$ to just Cliff$_n$ and $\Cliff_n$. Some interesting properties of Cliff$_n$ and $\Cliff_n$ are their `periodicity'. For the real case we have that 
\begin{equation*}
\text{Cliff}_{n+8}\cong \text{Cliff}_n\otimes\text{Cliff}_8\cong \text{Cliff}_n\otimes M_{16}(\R) 
\end{equation*}
and for complex algebras
\begin{equation}\label{Cliff tensor 2 period}
\Cliff_{n+2}\cong \Cliff_n\otimes \Cliff_2\cong \Cliff_n\otimes M_2(\C).
\end{equation}
A proof for these facts can be obtained from \cite[Theorem 8.11]{AlgTop}. These facts reflect `Bott periodicity' in K-theory. 

Clifford algebras are also naturally graded: The map $-1\in $End$(V)$ preserves the ideal $J$ because
\begin{equation*}
(-v)\otimes (-v)+\IP{-v,-v}=v\otimes v+\IP{v,v}
\end{equation*}
and so we may define a map $\widetilde{-1}$ on Cliff$(V)$ by
\begin{equation*}
\widetilde{-1}(e_{i_1}\dots e_{i_k})=(-1e_{i_1})\dots(-1e_{i_k}),\quad \widetilde{-1}(1)=1.
\end{equation*}
and extending linearly. Clearly this map is a self inverse automorphism and so defines a grading on Cliff$(V)$ where the even components are products and sums of even numbers of basis vectors of $V$ and the odd component being products and sums of odd numbers of basis vectors of $V$. One can check that this map is independent of basis chosen. We may extend $\widetilde{-1}$ to a grading of $\Cliff(V)$ by giving $\C$ the identity grading. That is,
\begin{equation*}
\widetilde{-1}(v\otimes z)=\widetilde{v}\otimes z.
\end{equation*}

We can define an adjoint on $\Cliff_n$ making it into a $*$-algebra. We define for $\{e_i\}$ an orthonormal basis of $\R^n$
\begin{equation*}
(e_{i_1}\dots e_{i_k}\otimes z)^*\coloneqq (-1)^{k}e_{i_k}\dots e_{i_1}\otimes \overline{z},\quad (1\otimes z)^*=1\otimes \overline{z}
\end{equation*}
and extend this adjoint by linearity. Clearly this definition gives a self-inverse conjugate linear mapping so we do indeed have an adjoint. It can be shown that there is a unique norm on $\Cliff_n$ under which it is a $C^*$-algebra. Defining this norm in terms of our above constructions is surprisingly difficult. We will show that $\Cliff_n$ is isomorphic to a direct sum of matrix $C^*$-algebras, and if required we may define the norm on $\Cliff_n$ as the pull-back metric under this isomorphism.

\subsection{Real Clifford algebras}

Okay, lets jump in and get a solid handle on what these Clifford algebras really are.
\begin{example}
	The first real Clifford algebra Cliff$_1$ is just $\C$. An orthonormal basis of $\R$ is $\{e\}$ where $e=1$. We are calling this basis vector $e$ to differentiate it from the scalar 1. The vectors $1$ and $e$ then are generators for Cliff$_1$ and satisfy the relations $1^2=1$ and $e^2=-\IP{e,e}=-1$. We then have an algebra isomorphism between Cliff$_1$ and $\C$ by mapping $1\in$Cliff$_1$ to $1\in\C$ and $e$ to $i\in\C$. 
\end{example}
\begin{example}
	Taking $\{e_1,e_2\}$ to be an orthonormal basis of $\R^2$ we obtain the generators
	\begin{equation*}
	1\qquad e_1\qquad e_2\qquad e_1e_2
	\end{equation*}
	of Cliff$_2$. They satisfy the relations
	\begin{align*}
	e_1^2&=-\IP{e_1,e_1}=-1\\
	e_2^2&=-\IP{e_2,e_2}=-1\\
	e_1e_2+e_2e_1&=-2\IP{e_1,e_2}=0\implies e_1e_2=-e_2e_1\\
	(e_1e_2)^2&=e_1e_2e_1e_2=-e_1^2e_2^2=-1\\
	(e_1e_2)e_2&=-e_1^2e_2=e_2\\
	(e_1e_2)e_1&=-e_1\\
	\end{align*}
	Which are exactly the generating relations of the quarternions. Thus by mapping
	\begin{align*}
	\text{Cliff}_2\ni 1&\mapsto 1\in\h\\
	e_1&\mapsto i\\
	e_2&\mapsto j\\
	e_1e_2&\mapsto k\\
	\end{align*}
	we obtain an algebra isomorphism of Cliff$_2$ into $\h$. 
\end{example}

\subsection{Complex Clifford algebras}

\begin{example}
	The complex Clifford algebras $\Cliff_1$ is isomorphic to $\C\oplus\C$ via the map 
	\begin{equation*}
	\phi(a\otimes z+be\otimes w)=(za+iwb,za-iwb).
	\end{equation*} 
	To see that $\phi$ is injective, suppose that 
	\begin{equation*}
	\phi(a\otimes z + be\otimes w)=\phi(c\otimes u+de\otimes v).
	\end{equation*} 
	Then we obtain the equations
	\begin{align*}
	za+iwb&=uc+ivd\\
	za-iwb&=uc-ivd\\
	\end{align*}
	which gives
	\begin{align*}
	za&=uc\\
	wb&=vd\\
	\end{align*}
	This then gives
	\begin{align*}
	a\otimes z + be\otimes w&=1\otimes az+e\otimes bw\\
	&=1\otimes uc+e\otimes vd\\
	&=c\otimes u+de\otimes v\\
	\end{align*}
	whence $\phi$ is injective. Since $\phi$ is clearly linear and multiplicative, rank nullity says, since we are in finite dimensions, that $\phi$ is surjective and we have an algebra isomorphism. In fact, $\phi$ is $*$-preserving since
	\begin{equation*}
	\phi(((a+be)\otimes z)^*)=\phi((a-be)\otimes z)=(z(a-ib),z(a+ib))=(\phi((a+be)\otimes z))^*.
	\end{equation*}
	Hence we have a $*$-isomorphism. The inverse $\phi^{-1}$ is then given by $\phi^{-1}(z,w)=1\otimes\frac{z+w}{2} + e\otimes\frac{z-w}{2}$. This then induces a grading $\beta$ on $\C\oplus \C$ by
	\begin{align*}
	\beta(z,w)&=\phi\circ\widetilde{-1}\circ\phi^{-1}(z,w)\\
	&=\phi\circ\widetilde{-1}\left(1\otimes\frac{z+w}{2} + e\otimes\frac{z-w}{2}\right)\\
	&=\phi\left(1\otimes\frac{z+w}{2} - e\otimes\frac{z-w}{2}\right)\\
	&=(w,z).\\
	\end{align*}
	
\end{example}
\begin{example}\label{Eg: Cliff 2}
	The complex Clifford algebra $\Cliff_2$ is isomorphic to $M_2(\C)$. We implement this isomorphism explicitly using the `Pauli Matrices'. Let
	\begin{align*}
	M_1&=\begin{pmatrix}
	1&0\\
	0&1\\
	\end{pmatrix}
	&&M_i=\begin{pmatrix}
	0&1\\
	-1&0\\
	\end{pmatrix}\\
	M_j&=\begin{pmatrix}
	0&i\\
	i&0\\
	\end{pmatrix}
	&&M_k=\begin{pmatrix}
	i&0\\
	0&-i\\
	\end{pmatrix}.\\
	\end{align*}
	The assignment of $1,e_1,e_2$ and $e_1e_2$ to $M_1,M_i,M_j$ and $M_k$ then is an isomorphism of Cliff$_2$ into the sub-algebra of matrices of the form $\begin{pmatrix}
	\alpha&\beta\\
	-\overline{\beta}&\overline{\alpha}
	\end{pmatrix}$. 
	When we tensor by $\C$ to form the complex Clifford algebra we lose this neat sub-algebra structure and end up with the entirety of $M_2(\C)$.
	We define the algebra homomorphism $\phi:\Cliff_2\to M_2(\C)$ by
	\begin{equation*}
	\phi(1\otimes z_1+e_1\otimes z_2+e_2\otimes z_3+e_1e_2\otimes z_4)=z_1M_1+z_2M_i+z_3M_j+z_4M_k.
	\end{equation*}
	To see surjectivity, fix
	\begin{equation*}
	A=\begin{pmatrix}
	a&b\\
	c&d\\
	\end{pmatrix}
	\end{equation*}
	in $M_2(\C)$. Then
	\begin{align*}
	\phi(1\otimes\frac{a+d}{2}+e_1e_2\otimes i\frac{d-a}{2}+&e_1\otimes\frac{b-c}{2}+e_2\otimes-i\frac{(b+c)}{2})\\
	&=\frac{1}{2}\begin{pmatrix}
	a+d+a-d&b-c+b+c\\
	c-b+b+c&a+d-a+d\\
	\end{pmatrix}\\
	&=\begin{pmatrix}
	a&b\\
	c&d\\
	\end{pmatrix}.
	\end{align*}
	Hence $\phi$ is surjective. Since we're in finite dimensions, rank nullity then says that $\phi$ is injective whence we have an algebra isomorphism. To see that $\phi$ is also $*$-preserving we compute
	\begin{align*}
	\phi((1\otimes z_1+e_1\otimes z_2+e_2\otimes z_3+e_1e_2\otimes z_4)^*)&=\phi(1\otimes \overline{z_1}-e_1\otimes \overline{z_2}-e_2\otimes \overline{z_3}-e_1e_2\otimes \overline{z_4})\\
	&=\overline{z_1}M_1-\overline{z_2}M_k-\overline{z_3}M_i-\overline{z_4}M_j\\
	&=(z_1M_1+z_2M_k+z_3M_i+z_4M_j)^*\\
	&=\phi(1\otimes z_1+e_1\otimes z_2+e_2\otimes z_3+e_1e_2\otimes z_4)^*\\
	\end{align*}
	so we in fact have a $*$-algebra isomorphism. If $A$ is as above, then the induced grading $\beta$ on $M_2(\C)$ is given by
	\begin{align*}
	\beta(A)&=\phi\widetilde{-1}\phi^{-1}(A)\\
	&=\phi\widetilde{-1}\left(1\otimes\frac{a+d}{2}+e_1e_2\otimes i\frac{d-a}{2}+e_1\otimes\frac{b-c}{2}+e_2\otimes-i\frac{(b+c)}{2}\right)\\
	&=\phi\left(1\otimes\frac{a+d}{2}+e_1e_2\otimes i\frac{d-a}{2}-e_1\otimes\frac{b-c}{2}-e_2\otimes-i\frac{(b+c)}{2}\right)\\
	&=\begin{pmatrix}
	a&-b\\
	-c&d\\
	\end{pmatrix}\\
	\end{align*}
	which is the usual grading of on and off diagonal matrices on $M_2(\C)$. 
\end{example}

\section{The Cuntz-Krieger algebra $C^*(E)$}
In this section we give an outline of the construction of graph algebras following the content of \cite{GraphAlgebras}. Note that our conventions for range and source are chosen to be the opposite to that of \cite{GraphAlgebras}. Given a graph $E$, a Cuntz-Krieger $E$ family is a family of mutually orthogonal projections and partial isometries in a $C^*$-algebra obeying multiplication relations that reflect the structure of $E$. A Cuntz-Krieger family can be thought of as a representation of $E$ in a $C^*$-algebra. The Cuntz-Krieger algebra $C^*(E)$ (often referred to simply as the graph algebra $C^*(E)$) is a $C^*$-algebra generated by a universal Cuntz-Krieger $E$-family. We will be interested in giving graph algebras a grading and computing the graded $K$-theory in Chapter~\ref{Graph algebras}. The (ungraded) $K$-theory of graph algebras is entirely known and quite intuitive: the $K$-groups of $C^*(E)$ are generated elements satisfying analogous relations to the Cuntz-Krieger relations of the Cuntz-Krieger family which generates $C^*(E)$ (see Chapter~\ref{ch: closing remarks}). We also state some results about the $C^*$-correspondence of Example~\ref{Marns} and how it relates to $C^*(E)$.

\begin{definition}
	Let $E$ be a directed graph. A Cuntz-Krieger $E$-family $\{S,P\}$ is consists of a set $P=\{P_v:v\in E^0\}$ of mutually orthogonal projections in a $C^*$-algebra $A$ and a set $S=\{S_e:e\in E^1\}$ of partial isometries in $A$ such that
	\begin{itemize}
		\item[(CK1)]\label{CK1}
			$S^*_eS_e=P_{r(e)}$ for all $e\in E^1$, and
		\item[(CK2)]\label{CK2}
			$P_v=\sum_{s(e)=v}S_eS^*_e$ whenever $0<|s^{-1}(v)|<\infty$.
	\end{itemize}
	We write $C^*(S,P)$ for the $C^*$-algebra generated in $A$ by the projections and isometries $\{S,P\}$. 
\end{definition}
The following Theorem states that there is a universal Cuntz-Krieger $E$-family. The statement is that of \cite[Proposition~1.21]{GraphAlgebras} and the proof uses the existence of an action of $\T$ on the $C^*$-algebra generated by the universal Cuntz-Krieger family, the details of which can be found in \cite[Chapter~2]{GraphAlgebras}.
\begin{theorem}
	For any graph $E$ there is a $C^*$-algebra $C^*(E)$ generated by a Cuntz-Krieger family $\{s,p\}$ such that for every Cuntz-Krieger $E$-family $\{T,Q\}$ there is a homomorphism $\pi_{T,Q}:C^*(E)\to C^*(T,Q)$ satisfying $\pi_{T,Q}(s_v)=T_v$ for every $v\in E^0$ and $\pi_{T,Q}(S_e)=Q_e$ for every $e\in E^1$. We call $C^*(E)$ the \emph{Cuntz-Krieger algebra} of $E$, or often simply the graph algebra of $E$.
\end{theorem}

In \cite[Proposition~12]{CuntzKreigerIsCuntzPimsner} it is shown that if $E$ has no sinks then $C^*(E)$ is isomorphic to the $C^*$-algebra $\OO_{X}$ where $X$ is the graph correspondence of Example~\ref{Marns} and $\OO_X$ is the Cuntz-Pimsner algebra of $X$ (see Definition~\ref{def: Cuntz-Pimsner}). In Chapter~\ref{The Exact Sequence!} we construct a method of computing the graded $K$-theory of $\OO_X$. In particular when $X$ is the graph correspondence of $E$, this will enable us to compute the graded $K$-theory of $C^*(E)$. We devote Chapter~\ref{Graph algebras} to making this computation.

\begin{definition}\label{def: graph lingo}
	Let $E$ be a directed graph. We say that a vertex $v\in V$ is a \emph{source} if $|r^{-1}(v)|=0$, and we say that $v$ is a \emph{sink} if $|s^{-1}(v)|=0$. We say that $v$ is an \emph{infinite emitter} if $|s^{-1}(v)|=\infty$ and we say that $E$ is \emph{row-finite} if $E^0$ contains no infinite emitters.
\end{definition}

It can be shown that algebraic properties of the graph correspondence $X$ correspond directly to the properties defined in Definition~\ref{def: graph lingo} of the graph $E$. The following table of such properties is taken from \cite[Chapter~3]{AddingTails}:
\begin{table}[h!]
\begin{center}
	\begin{tabular}{c|ccccc}\label{Table}
	\textbf{Property of $X$}	& \textbf{Property of $E$}\\  
	\hline
		$\phi(\delta_v)\in\KK(X)$& $v$ emits a finite number of edges\\  
		$\phi(C_0(E^0))\subseteq\KK(X)$& $E$ is row-finite\\  
		$\phi$ is injective & $E$ has no sinks\\
		$X$ is full & $E$ has no sources\\
		$\phi$ is non-degenerate & True of all graphs\\  
	\end{tabular}
\end{center}
\end{table}

The following Lemma is a proof of the first two entries of this table.
\begin{lemma}\label{Graph compact left action}
	Let $(E^0,E^1,r,s)$ be a directed graph with discrete topology, and let $X$ be the $C_0(E^0)$--$C_0(E^0)$ correspondence of Example~\ref{Marns}. For $f\in C_0(E^0)$, the operator $\phi(f)\in\LL(X)$ is compact if and only if $f(v)=0$ whenever $|s^{-1}(v)|=\infty$. That is, if $F$ is the set of all vertices that emit finitely many edges, then
	\begin{equation*}
	\phi^{-1}(\KK(X))=C_0(F).
	\end{equation*}
\end{lemma}
\begin{proof}
	This wonderful proof is taken from \cite[Proposition~4.4]{FowlerRaeburn}.\\
	The set $I=\phi^{-1}(\KK(X))\subseteq C_0(E^0)$ is an ideal because $\KK(X)$ is an ideal and if $q:\LL(X)\to \LL(X)/\KK(X)$ is the quotient homomorphism then $I$ is the kernel of the homomorphism $q\circ \phi$. We then know that $I$ is of the form
	\begin{equation*}
	I=\{f\in C_0(E^0):f(v)=0 \  \forall \ v\notin F\}=\overline{\Span}\{\delta_v:v\in F\}
	\end{equation*}
	for some subset $F\subset E^0$. So it suffices to show that $\delta_v\in I$ if and only if $v$ emits finitely many edges. If $v$ does emit finitely many edges then
	\begin{equation*}
	\phi(\delta_v)=\sum_{f\in E^1v}\Theta_{\delta_f,\delta_f}
	\end{equation*}
	so $\delta_v\in I$. Now suppose then that $v$ emits infinitely many edges. Since $\Span \{\delta_f\}$ is dense in $X$ and $(x,y)\mapsto \Theta_{x,y}$ is continuous, we may approximate any compact operator on $X$ by a finite linear combination of the form $K=\sum_{e,f\in F}c_{e,f}\Theta_{\delta_e,\delta_f}$. Since $v$ emits infinitely many edges, there exists an edge $g\notin F$ such that $r(g)=v$. So then \begin{equation*}
	\sum_{e,f\in F}c_{e,f}\Theta_{\delta_e,\delta_f}(\delta_g)=\delta_e\IP{\delta_f,\delta_g}=0
	\end{equation*}
	for all $e,f\in F$. Thus
	\begin{align*}
	\normof{\phi(\delta_v)-K}&=\sup_{\normof{x}\leq 1}\normof{(\phi(\delta_v)-K)x}\\
	&\geq \normof{\phi(\delta_v)(\delta_g)-K(\delta_g)}\\
	&=\normof{\delta_g-0}=1,
	\end{align*}
	and hence $\phi(\delta_v)$ is not compact.
\end{proof}

\section{The discrete crossed product $A\times_{\alpha}G$}
In this section we provide a thorough construction of the crossed product of a $C^*$-algebra by a discrete group $G$. We will be interested in crossed products by the group $\Z$, so we restrict our attention to discrete groups for which the construction simplifies considerably. This section follows the construction in \cite[Chapter 2]{Crossed and unitizations}. We start with some definitions. Note that the following definitions can be made for non-discrete groups, with some added topological constraints. 
\begin{definition}\label{denizen}
	Let $G$ be a discrete group and $A$ be a $C^*$-algebra. An action of $G$ on $A$ is a group homomorphism
	\begin{equation*}
	\alpha: G\to \text{Aut}(A),\qquad g\mapsto \alpha_g.
	\end{equation*}
	A $C^*$-dynamical system $(A,G,\alpha)$ is a $C^*$-algebra $A$, a discrete group $G$ and an action $\alpha$ of $G$ on $A$.
\end{definition}
\begin{example}\label{mankey}
	If $\alpha$ is an automorphism of a $C^*$-algebra $A$ then we may define an action of $\Z$ on $A$ by $\alpha_n(a)=\alpha^n(a)$ with the convention that $\alpha^0$ is the identity map. Clearly we have $\alpha^n(\alpha^m(a))=\alpha^{n+m}(a)$ so $n\mapsto \alpha^n$ is a homomorphism. Conversely, if $\alpha_n$ is an action of $\Z$ on $A$ then $\alpha_1$ is an automorphism of $A$ which generates the action, so every action of $\Z$ arises this way.
\end{example}
\begin{definition}
	Let $(A,G,\alpha)$ be a $C^*$-dynamical system and let $B$ be a $C^*$-algebra. A covariant homomorphism from $(A,G,\alpha)$ to $B$ is a pair of maps $(\pi,u)$ where $\pi:A\to B$ is a homomorphism of $C^*$-algebras, and $u:G\to \UU\MM(A)$ is a group homomorphism such that
	\begin{equation*}
	u_s\pi(a)u_s^*=\pi(\alpha_s(a))
	\end{equation*}
	for all $s\in G$ and $a\in A$. We say that $(\pi,u)$ is non-degenerate if $\pi$ is non-degenerate when viewed as a map into $\LL(B_B)$.
\end{definition}

Now we begin constructing the crossed product $A\times_{\alpha} G$. Explicitly, $A\times_{\alpha}G$ is the completion of $C_c(G,A)$ with multiplication given by convolution under a norm that we construct out of the collection of all covariant homomorphisms.
The crossed product is by definition the $C^*$-algebra generated by a distinguished universal covariant homomorphism of $(A,G,\alpha)$. We will make this precise below. We first need to show that $C_c(G,A)$ is a $*$-algebra.
\begin{proposition}\label{Bub}
	Let $(A,G,\alpha)$ be a $C^*$-dynamical system and suppose that $G$ is discrete. Then the vector space $C_c(G,A)$ of finitely supported functions from $G$ into $A$ is a complex $*$-algebra under the operations
	\begin{equation*}
	(f*g)(s)=\sum_{t\in G}f(t)\alpha_t(g(t^{-1}s))\qquad\text{ and }\qquad f^*(s)=\alpha_s(f(s^{-1}))^*.
	\end{equation*}
	If $(\pi,u)$ is a covariant homomorphism from $(A,G,\alpha)$ to $B$ then the map
	\begin{equation*}
	i_{\pi,u}:C_c(G,A)\to \MM(B)
	\end{equation*}
	given by $i_{\pi,u}(f)=\sum_{s\in G}\pi(f(s))u_s$ is a homomorphism.
\end{proposition}
\begin{proof}
	This proof is due to \cite[Proposition 2.22]{Crossed and unitizations}.
	We begin by checking that convolution is associative and distributive. For $f,g,h\in C_c(G,A)$ we compute
	\begin{align*}
	(f*g)*h(s)&=\sum_{t\in G}(f*g)(t)\alpha_t h((t^{-1}s))\\
	&=\sum_{t,k\in G}f(k)\alpha_k(g(k^{-1}t))\alpha_t(h(t^{-1}s))\\
	&=\sum_{k,r\in G}f(k)\alpha_{k}(g(r)\alpha_r(h(r^{-1}k^{-1}s)))\\
	&=\sum_{r\in G}f(k)\alpha_k(g*h(k^{-1}s))\\
	&=f*(g*h)(s)\\
	\end{align*}
	and so convolution is associative. One sees that
	\begin{align*}
	f*(g+h)(s)&=\sum_{t\in G}f(t)\alpha_t(g(t^{-1}s)+h(t^{-1}s))\\
	&=\sum_{t\in G}f(t)\alpha_t(g(t^{-1}s))+\sum_{t\in G}f(t)\alpha_t(h(t^{-1}s))\\
	&=f*g(s)+f*h(s)\\
	\end{align*}
	and similarly $(f+g)*h=f*h+g*h$, so convolution is also distributive. We have for $\lambda\in\C$ the relations
	\begin{align*}
	(\lambda f)^*(s)&=\alpha_s((\lambda f(s^{-1}))^*)=\overline{\lambda}\alpha_s(f(s^{-1})^*)=\overline{\lambda}f^*(s),\\
	(f^*)^*	(s)&=\alpha_s(f^*(s^{-1})^*)=\alpha_s(\alpha_{s^{-1}}(f(s)^*)^*)=f(s), \text{ and }\\
	(f+g)^*(s)&=\alpha_s((f(s^{-1})+g(s^{-1}))^*)=\alpha_s(f(s^{-1})^*)+\alpha_s(g(s^{-1})^*)=f^*(s)+g^*(s).\\
	\end{align*}
	Moreover we see that
	\begin{align*}
	(f*g)^*(s)&=\alpha_s((f*g)(s^{-1})^*)\\
	&=\left(\alpha_s(\sum_{t\in G}f(t)^*\alpha_t(g(t^{-1}s^{-1})^*)\right)\\
	&=\sum_{k\in G}\alpha_s(f(s^{-1}k)^*)\alpha_k(g(k^{-1})^*)\\
	&=\sum_{k\in G}g^*(k)\alpha_k(f^*(k^{-1}s))\\
	&=g^**f^*(s).\\
	\end{align*}
	So $C_c(G,A)$ is a $*$-algebra.
	Note that we have three `$*$' symbols in play in the above computation: the adjoint of the image of $f$ in $A$, the convolution product of $f$ and $g$ and the adjoint operation on $C_c(G,A)$. 
	Now we show that $i_{\pi,u}$ is a homomorphism. Clearly $i_{\pi,u}$ is linear. For $f,g\in C_c(G,A)$ we calculate:
	\begin{align*}
	i_{\pi,u}(f*g)&=\sum_{s\in G}\pi((f*g)(s))u_s\\
	&=\sum_{s,t\in G}\pi(f(t)\alpha_t(g(t^{-1}s)))u_s\\
	&=\sum_{t,k\in G}\pi(f(t))\pi(\alpha_t(g(k)))u_tu_k\\
	&=\sum_{t,k\in G}\pi(f(t))u_t\pi(g(k))u_k\\
	&=i_{\pi,u}(f)*i_{\pi,u}(g)\\
	\end{align*}
	where we have used that $u_s\pi(f(s))u_s^*=\pi(\alpha_s(f(s))$ to move $u_t$ past $\pi$ in the second last line. Lastly, $i_{\pi,u}$ is $*$-preserving because
	\begin{align*}
	i_{\pi,u}(f^*)&=\sum_{t\in G}\pi(f^*(t))u_t\\
	&=\sum_{t\in G}\pi(\alpha_t(f(t^{-1})^*))u_t\\
	&=\sum_{t\in G}u_t\pi(f(t^{-1}))^*\\
	&=\sum_{t\in G}(\pi(f(t^{-1}))u_{t^{-1}})^*\\
	&=\left(\sum_{s\in G}\pi(f(s))u_s\right)^*\\
	&=i_{\pi,u}(f)^*.\\
	\end{align*}
\end{proof}
\begin{remark}
	We will later extend $i_{\pi,u}$ to a map $\pi\times u:A\times_{\alpha} G\to\MM(B)$. We denote by $i_{\pi,u}$ map defined on $C_c(G,A)$ so that we do not overload our notation between $i_{\pi,u}$ and its extension $\pi\times u$.
\end{remark}
Now that we know that $C_c(G,A)$ is a $*$-algebra, the next thing we will need to know in order to build the crossed product is that `nice' semi-norms descend to proper $C^*$-norms when we quotient by elements of length zero.

\begin{lemma}\label{lemma: C star completion}
	If $A$ is a $*$-algebra and $\rho$ is a semi-norm on $A$ such that $\rho(a^*a)=\rho(a)^2$ and $\rho(ab)\leq\rho(a)\rho(b)$ then
	\begin{equation*}
	I=\{a\in A:\rho(a)=0\}
	\end{equation*}
	is an ideal and the closure of the quotient space $A/I$ is a $C^*$-algebra. We call $\overline{A/I}$ the $C^*$-completion of $A$ with respect to the induced norm $\rho$.
\end{lemma}
\begin{proof}
	The following is an expansion of \cite[20.2]{Marco}.
	First we show that $I$ is an ideal. If $a\in I$ then
	\begin{equation*}
	\rho(ab)\leq \rho(a)\rho(b)=0
	\end{equation*}
	and so $\rho(ab)=0$. Thus $ab\in I$ making $I$ a right ideal. Repeating the argument for $ba$ shows that $I$ is in fact a two sided ideal. Since $\rho(a^*a)=\rho(a)^2$ we have
	\begin{equation*}
	\rho(a)^2=\rho(a^*a)\leq \rho(a^*)\rho(a)
	\end{equation*}
	and so $\rho(a)\leq \rho(a^*)$. Applying this now to $a^*$ gives $\rho(a^*)\leq \rho(a)$ and so $\rho(a^*)=\rho(a)$. So if $a\in I$, then
	\begin{equation*}
	\rho(a^*)=\rho(a)=0
	\end{equation*}
	and so $I$ is $*$-closed. The ideal $I$ is closed because if $a=\lim_{n\to\infty}a_n$ is a limit of elements of $I$, then by continuity of semi-norms
	\begin{equation*}
	\rho(a)=\rho(\lim_{n\to\infty}a_n)=\lim_{n\to\infty}\rho(a_n)=0.
	\end{equation*}
	So $I$ is a two sided closed $*$-ideal and the quotient space $A/I$ is then a $*$-algebra. We define a norm on $A/I$ by $\normof{a+I}=\rho(a)$. To see that this is a well defined norm, suppose that $a=b+i$ for some $i\in I$. Then
	\begin{equation*}
	\rho(a)=\rho(b+i)\leq \rho(b)+\rho(i)=\rho(b).
	\end{equation*}
	Similarly
	\begin{equation*}
	\rho(b)=\rho(a-i)\leq \rho(a)+\rho(i)=\rho(a)
	\end{equation*}
	and so $\rho(a)=\rho(b)$. Hence $\normof{a+I}=\rho(a)$ is well defined. If $\normof{a+I}=0$ then $\rho(a)=0$ and so $a\in I$ which means that $a$ represents zero in $A/I$.
	The relations
	\begin{align*}
	\normof{a+b+I}&=\rho(a+b)\leq\rho(a)+\rho(b)=\normof{a+I}+\normof{b+I}\\
	\normof{\lambda a+I}&=\rho(\lambda a)=|\lambda|\rho(a)=|\lambda|\normof{a+I}\\
	\normof{ab+I}&=\rho(ab)\leq\rho(a)\rho(b)=\normof{a+I}\normof{b+I}\\
	\normof{a^*a+I}&=\rho(a^*a)=\rho(a)^2=\normof{a+I}^2\\
	\end{align*}
	then show that $\normof{\cdot}$ is a $C^*$-norm. Hence $\overline{A/I}$ is a $C^*$-algebra.
\end{proof}

Finally we are ready to construct the crossed product $A\times_{\alpha}G$ of $A$ by $G$.

\begin{theorem}\label{Profit}
	Let $G$ be a discrete group, and let $(A,G,\alpha)$ be a $C^*$-dynamical system. There exists a $C^*$-algebra $A\times_{\alpha}G$ called the crossed product of $A$ by $G$ and a non-degenerate covariant homomorphism $(i_A,i_G)$ from $(A,G,\alpha)$ to $\MM(A\times_{\alpha}G)$ such that
	\begin{enumerate}
		\item 
		$A\times_{\alpha}G=\overline{\Span}\{i_A(a)i_G(g):a\in A,g\in G\}$, and
		\item 
		if $(\pi,u)$ is a non-degenerate covariant homomorphism from $(A,G,\alpha)$ to $B$, then there is a homomorphism $(\pi\times u):A\times_{\alpha}G\to B$ whose extension to $\MM(A\times_{\alpha}G)$ in the sense of Lemma~\ref{REF} satisfies $(\pi\times u)\circ i_A=\pi$ and $(\pi\times u)\circ i_G=u$.
	\end{enumerate}
	The pair $(A\times_{\alpha}G,(i_A,i_G))$ is unique up to isomorphism.
\end{theorem}
\begin{proof}
	This proof is an expansion of that of \cite[Theorem 2.21]{Crossedandunitizations}. Let $(\pi,u)$ be a covariant homomorphism of $(A,G,\alpha)$ into a $C^*$-algebra $B$ and let $i_{\pi,u}:C_c(G,A)\to \MM(B)$ denote the homomorphism $\sum_{s\in G}\pi(f(s))u_s$ defined in Proposition~\ref{Bub}. We wish to define $A\times_{\alpha}G$ as the $C^*$-completion of $C_c(G,A)$ under the semi-norm
	\begin{equation}\label{4.5}
	\rho(f)=\sup\{\normof{i_{\pi,u}(f)}:(\pi,u)\;\text{is a covariant homomorphism}\}.
	\end{equation}
	First we need to show that $\rho$ defines a semi-norm on $C_c(G,A)$. Firstly we note that for any covariant homomorphism
	\begin{equation}\label{master}
	\normof{i_{\pi,u}(f)}=\Bignormof{\sum_{s\in G}\pi(f(s))u_s}\leq \sum_{s\in G}\normof{\pi(f(s))}\normof{u_s}\leq \sum_{s\in G}\normof{f(s)}
	\end{equation}
	where we have used that unitary elements have norm one and that $\pi$ being a homomorphism of $C^*$-algebras is norm decreasing. Hence the supremum of Equation \eqref{4.5} is finite. For all $\lambda\in\C$ and $f,g\in C_c(G,A)$, linearity of $\pi$ ensures we have $\rho(\lambda f)=|\lambda|\rho(f)$, and we calculate 
	\begin{align*}
		\rho(f+g)&=\sup\{\normof{i_{\pi,u}(f+g)}\}=\sup\{\normof{i_{\pi,u}(f)+i_{\pi,u}(g)}\}\leq \sup\{\normof{i_{\pi,u}(f)}+\normof{i_{\pi,u}(g)}\}\\
		&=\sup\{\normof{i_{\pi,u}(f)}\}+\sup\{\normof{i_{\pi,u}(g)}\}=\rho(f)+\rho(g).\\
	\end{align*}
	Hence $\rho$ is a semi-norm. Our semi-norm $\rho$ is sub-multiplicative because
	\begin{align*}
	\rho(f*g)&=\sup\{\normof{i_{\pi,u}(f*g)}\}=\sup\{\normof{i_{\pi,u}(f)i_{\pi,u}(g)}\}\\
	&\leq \sup\{\normof{i_{\pi,u}(f)}\normof{i_{\pi,u}(g)}\}\leq\sup\{\normof{i_{\pi,u}(f)}\}\sup\{\normof{i_{\pi,u}(g)}\}=\rho(f)\rho(g)
	\end{align*}
	and we have the $C^*$ identity since
	\begin{equation*}
	\rho(f^*f)=\sup\{\normof{i_{\pi,u}(f^*f)}\}=\sup\{\normof{i_{\pi,u}(f)^*i_{\pi,u}(f)}\}=\sup\{\normof{i_{\pi,u}(f)}\}^2=\rho(f)^2.
	\end{equation*}
	As in Lemma~\ref{lemma: C star completion}, let
	\begin{equation*}
		I=\{a\in A:\rho(a)=0\}
	\end{equation*}
	be the kernel of $\rho$. Using Lemma~\ref{lemma: C star completion}, we may form the $C^*$-completion of $C_c(G,A)$ which we call the crossed product $A\times_{\alpha}G$. Let $q:C_c(G,A)\to C_c(G,A)/I$ be the quotient homomorphism $q(f)=f+I$. We now define the maps $i_A$ and $i_G$. For $a\in A$, let $a1_s$ be the function $a1_s(t)=\delta_{s,t}a$. We define $i_A(a)=q(a1_e)$. Since $q$ is a homomorphism, $i_A$ will be too. We will need to work a little harder to define $i_G$. We start by defining the `left translation map' $\lt_s:C_c(G,A)\to C_c(G,A)$ by $\lt_s(f)(t)=\alpha_s(f(s^{-1}t))$. We wish to show that $\lt$ extends to a well defined operator on the crossed product. We see that
	\begin{align*}
	i_{\pi,u}(\lt_s(f))&=\sum_{t\in G}\pi(\alpha_s(f(s^{-1}t)))u_t\\
	&=\sum_{k\in G}\pi(\alpha_s(f(k)))u_su_k\\
	&=u_s\sum_{k\in G}\pi(f(k))u_k\\
	&=u_si_{\pi,u}(f).\\
	\end{align*}
	If $i_{\pi,u}(f)=0$ for all covariant homomorphisms $(\pi,u)$ then this shows that $i_{\pi,u}(\lt_s(f))=u_si_{\pi,u}(f)=0$. So $\lt_s$ induces a well defined map $\lt_s(q(f))=q(\lt_s(f))$ on the quotient. This map is norm decreasing since
	\begin{equation*}
	\normof{i_{\pi,u}(\lt_s(f))}=\normof{u_si_{\pi,u}(f)}\leq \normof{i_{\pi,u}(f)}.
	\end{equation*}
	So by definition, for each $f\in C_c(G,A)$,
	\begin{equation*}
	\normof{\lt_s(q(f))}=\normof{q(\lt_s(f))}=\normof{\lt_s(f)}\leq \normof{f}=\normof{q(f)}.
	\end{equation*}
	We deduce that $\lt_s$ is uniformly continuous on the quotient whence it extends to the $C^*$-completion $A\times_{\alpha}G$ of $C_c(G,A)$. We are then ready to define 
	\begin{equation*}
	i_G(s):A\times_{\alpha}G\to A\times_{\alpha}G
	\end{equation*}
	by 
	\begin{equation*}
	i_G(s)(q(f))=q(\lt_s(f)).
	\end{equation*}
	We now want to show that $i_G(s)$ is unitary in $\MM(A\times_{\alpha}G)$; that is, that each $i_G(s)$ is a unitary operator on $A\times_{\alpha}G$ viewed as a right Hilbert module with $A\times_{\alpha}G$-valued inner-product $\IP{f,g}=f^**g$. We calculate for $f,g\in C_c(G,A)$
	\begin{align*}
	\lt_s(f)^**g(t)&=\sum_{k\in G}(\lt_s(f)^*(k))\alpha_k(g(k^{-1}t))\\
	&=\sum_{k\in G}\alpha_k(\lt_s(f)(k^{-1})^*)\alpha_k(g(k^{-1}t))\\
	&=\sum_{k\in G}\alpha_{k}(\alpha_s(f(s^{-1}k^{-1})^*))\alpha_k(g(k^{-1}t))\\
	&=\sum_{l\in G}\alpha_{l}f(l^{-1})^*\alpha_{l}(\alpha_{s^{-1}}(g(sl^{-1}t)))\\
	&=\sum_{l\in G}f^*(l)\alpha_{l}(\lt_{s^{-1}}(g)(l^{-1}t))\\
	&=f^**\lt_{s^{-1}}(g)(t).\\
	\end{align*}
	Since $q$ is a homomorphism it follows that for all $f,g\in C_c(G,A)$
	\begin{align*}
	\IP{i_G(s)(q(f)),q(g)}&=\IP{q(f),i_G(s^{-1})(q(g))},\\
	\end{align*}
	and so $i_G(s)$ is a unitary operator on the quotient space. By continuity $i_G(s)$ extends to the completion $A\times_{\alpha}G$ as a unitary operator.\\
	Next we check that $i_A$ is non-degenerate. Fix $f\in C_c(G,A)$, fix $\ep>0$ and let $\{e_{\lambda}\}$ be an approximate identity for $A$. Since $f\in C_c(G,A)$, we know that $f(s)\neq 0$ for only finitely many $s$. Let $N=|\{s\in G: f(s)\neq 0\}|$. Then there exists $\lambda$ large enough so that $\normof{e_{\lambda}f(s)-f(s)}<\ep/N$ for all $s\in G$. By Equation \eqref{master} we have
	\begin{equation*}
	\normof{i_A(e_{\lambda})f-f}=\normof{e_{\lambda}1_e*f-f}\leq \sum_{s\in G}\normof{e_{\lambda}f(s)-f(s)}<\ep.
	\end{equation*}
	Thus $i_A(e_{\lambda})f\to f$ in $C_c(G,A)$. Since $q$ is continuous, $i_A(e_{\lambda})q(f)\to q(f)$ in the quotient. Since $C_c(G,A)$ is dense in $A\times_{\alpha}G$, we have $\overline{i_A(A)A\times_{\alpha}G}=A\times_{\alpha}G$ and $i_A$ is non-degenerate.\\
	We must show that $(i_A,i_G)$ is a covariant representation of $(A,G,\alpha)$ on $A\times_{\alpha}G$. That is, we need to check that
	$$i_G(s)i_A(a)i_G(s^{-1})=i_A(\alpha_s(a)).$$
	For $f\in C_c(G,A)$ we have
	\begin{align*}
	\lt_g(a1_e*\lt_{g^{-1}}(f))(s)&=\alpha_g\left(a1_e*\lt_{g^{-1}}(f)(g^{-1}s)\right)\\
	&=\alpha_g\Big(\sum_{t\in G}a1_e(t)\alpha_t(\lt_{g^{-1}}(f)(t^{-1}g^{-1}s))\Big)\\
	&=\alpha_g\left(a\lt_{g^{-1}}(f)(g^{-1}s)\right)=\alpha_g\left(a\alpha_{g^{-1}}(f(s))\right)\\
	&=\alpha_g(a)f(s)=(\alpha_g(a1_e)*f )(s)\\
	\end{align*}
	so that $\lt_g(a1_e*\lt_{g^{-1}})(f)=\alpha_g(a)1_e*f$. It then follows that
	\begin{equation*}
		i_G(g)i_A(a)i_G(g^{-1})(q(f))=i_A(\alpha_g(a))(q(f))
	\end{equation*}
	and so $(i_A,i_G)$ is covariant on $q(C_c(G,A))$. Since this set is dense in $A\times_{\alpha}G$, it follows that $(i_A,i_G)$ is a covariant homomorphism.\\
	
	Now we wish to show that property (1) holds. We note that
	\begin{align*}
	(a1_{hg}*f)(s)&=\sum_{t\in G}a1_{hg}(t)\alpha_{t}(f(t^{-1}s))\\
	&=a\alpha_{hg}f(g^{-1}h^{-1}s)=a1_h(h)\alpha_h(\lt_g(f(h^{-1}s)))\\
	&=\sum_{t\in G}a1_h(t)\alpha_t(\lt_g(t^{-1}s))=(a1_h*\lt_g(f))(s),
	\end{align*}
	so
	\begin{equation}
	a1_{hg}*f=a1_h*\lt_g(f).\label{Yanaginagi}
	\end{equation}
	Fix $a\in A$ and $g\in G$. Since $A\times_{\alpha}G$ is an ideal in $\MM(A\times_{\alpha}G)$ and since $i_A(a)\in A\times_{\alpha}G$ we have $i_A(a)i_G(g)\in A\times_{\alpha}G$. In particular, for $f\in C_c(G,A)$,
	\begin{align*}
	i_A(a)*i_G(g)q(f)&=q(a1_e)*q(\lt_g(f))=q(a1_e*\lt_g(f))\qquad\text{by }\eqref{Yanaginagi}\\
	&=q(a1_g*f)=q(a1_g)*q(f).\\
	\end{align*}
	Thus $i_A(a)i_G(g)=q(a1_{g})$ as left multiplication operators. So $\overline{\Span}\{i_A(a)i_G(g):a\in A,g\in G\}\subseteq A\times_{\alpha}G$. Choosing an approximate identity $e_{\lambda}$ in $A$ and letting $a=e_{\lambda}$ and $g=e$ then gives reverse containment and we have shown (1). 
	
	Now we wish to show that $(A\times_{\alpha}G,i_A,i_G)$ has the universal property (2), that is. Fix a covariant homomorphism of $(A,G,\alpha)$ into $B$. For $f\in C_c(G,A)$ we aim to define $\pi\times u(q(f))=i_{\pi,u}(f)$ on the quotient space. If $f=g+h$ with $\normof{h}=0$ then $i_{\pi,u}(h)=0$ and $i_{\pi,u}(f)=i_{\pi,u}(g)$. So there is a linear map $\pi\times u:C_c(G,A)/I\to B$ such that $\pi\times u(f+I)=i_{\pi,u}(f)$ for all $f$. Since $i_{\pi,u}$ is a homomorphism, so is $\pi\times u$. By definition of the norm on $C_c(G,A)$, for each $f\in C_c(G,A)$ we have
	\begin{equation*}
	\normof{\pi\times u(q(f))}=\normof{i_{\pi,u}(f)}\leq \normof{f}=\normof{q(f)}.
	\end{equation*}
	Hence $\pi\times u$ is norm decreasing, and therefore it extends to a map $\pi\times u:A\times_{\alpha}G\to B$. 
	Now we wish to check the universal properties of $\pi\times u$.		
	Recall from Proposition~\ref{props} that if $\phi:A\to \MM(B)\cong \LL(B_B)$ is a non-degenerate homomorphism then its extension to $\MM(A)$ is given by $\phi(m)(b)=\sum_n \phi(ma_n)(b_n)$ where $b=\sum_n \phi(a_n)(b_n)$. Since $\pi$ is non-degenerate so is $\pi\times u$, so fix $b\in B$ and let $b=\sum_n \pi\times u(q(f_n))(b_n)$ for some $f_n\in C_c(G,A)$ and $b_n\in B$. We compute
	\begin{align*}
	(\pi\times u)\circ i_G(s)(b)&=\sum_n \pi\times u(i_G(s)q(f_n))(b_n)\\
	&=\sum_n\sum_{t\in G}\pi(\lt_s(f_n)(t))u_t(b_n)=\sum_n\sum_{k\in G}\pi(\alpha_s(f_n(k)))u_su_k(b_n)\\
	&=\sum_n\sum_{k\in G}u_s\pi(f_n(k))u_k(b_n)=u_s\sum_{n}\pi\times u(f_n)(b_n)=u_s(b)\\
	\end{align*}
	and so $\pi\times u\circ i_G=u$. Since $i_A(A)\subseteq A\times_{\alpha}G$ we do not need to worry about extending $\pi\times u$ to $\MM(A\times_{\alpha}G)$ to compute the composition $(\pi\times u)\circ i_A$:
	\begin{align*}
	(\pi\times u)\circ i_A(a)&=\pi\times u(q(a1_e))=(\pi\times u) (a1_e)\\
	&=\sum_{s\in G}\pi(a1_e(s))u_s=\pi(a)u_e=\pi(a)\\
	\end{align*}
	and so $\pi\times u\circ i_A=\pi$ as required.\\
	Now all that is left is to check that the crossed product is unique. 
	To show uniqueness, suppose $B$ is another $C^*$-algebra with a covariant homomorphism $(\pi,u)$ such that (1) and (2) hold. Then property (2) gives us two homomorphisms
	\begin{equation*}
	\pi\times u:A\times_{\alpha}G\to B\qquad \text{ and }\qquad i_A\times i_G:B\to A\times_{\alpha}G
	\end{equation*}
	such that
	\begin{equation*}
	\pi\times u\circ i_A\times i_G(\pi(a)u_g)=\pi\times u (i_A(a)i_G(g)=\pi(a)u_g.
	\end{equation*}
	Property (1) tells us that these elements $\pi(a)u_g$ generate $B$, so $(\pi\times u)\circ i_A\times i_G$ is the identity on $B$. Similarly, property (1) for $A\times_{\alpha}G$ tells us that $(i_A\times i_G)\circ(\pi\times u)$ is the identity on $A\times_{\alpha}G$. Hence $i_A\times i_G$ is an isomorphism with inverse $\pi\times u$ and so $A\times_{\alpha}G\cong B$.
\end{proof}
In the proof above we quotiented out by functions $f\in C_c(G,A)$ such that $i_{\pi,u}(f)=0$ for all covariant homomorphisms $(\pi,u)$. Our next Lemma shows that for each $f\in C_c(G,A)$ there is always a covariant homomorphism $(\pi,u)$ for which $i_{\pi,u}(f)\neq 0$, so that $C_c(G,A)$ embeds in $A\times_{\alpha}G$. We can show slightly more than this: there is a special covariant representation $(\pi,u)$ called the Reduced representation of $A$ in $\LL(\oplus_{g\in G}A_A)$, for which $i_{\pi,u}(f)=0$ if and only if $f=0$.
\begin{lemma}
	Let $(A,G,\alpha)$ be a $C^*$-dynamical system, and let $i_G:G\to \UU(A\times_{\alpha}G)$ and $i_A:A\to A\times_{\alpha}G$ be the universal representation of $(A,G,\alpha)$ in $A\times_{\alpha}G$. Then there exists a covariant homomorphism $(\pi,u)$ of $A$ in $\LL(\oplus_{g\in G}A_A)$ such that the homomorphism $i_{\pi,u}:C_c(G,A)\to\LL(\bigoplus_{g\in G}A_A)$ is injective. In particular, the homomorphism $i_{i_A,i_G}:C_c(G,A)\to A\times_{\alpha}G$ obtained from the universal property of $A\times_{\alpha}G$ is injective, and consequently $i_A$ is injective. 
\end{lemma}
\begin{proof}
	As discussed in Appendix~\ref{Direct sums}, the direct sum $\bigoplus_{g\in G}A$ is the set of sequences $(a)_{g\in G}$ indexed by $G$ such that $\sum_{g\in G}a_g^*a_g$ converges in norm. For each $a\in A$ we define an operator $\pi(a)$ on $\bigoplus_{g\in G}A_A$ by $(\pi(a)(b))_g=\alpha_{g^{-1}}(a)b_g$. The operator $\pi(a)$ is adjointable with adjoint $\pi(a)^*=\pi(a^*)$, and clearly we have both $\pi(ab)=\pi(a)\pi(b)$ and $\pi(a+b)=\pi(a)+\pi(b)$, so $\pi:A\to\LL(\oplus_{g\in G}A_A)$ is a homomorphism. For each non-zero $a\in A$, we have $aa^*\neq 0$ and if $x$ is any sequence in $\bigoplus_{g\in G}A_A$ with $(x)_e=a^*$, then $\pi(a)x\neq 0$. Thus the kernel of $\pi$ contains only zero and we deduce that $\pi$ is injective. For $g\in G$ we define an operator $U_g$ on $\bigoplus_{h\in G}A_A$ by $(U_g(a))_h=a_{g^{-1}h}$. We compute
	\begin{align*}
	\IP{U_g(a),b}&=\sum_{h\in G}(U_g(a))_hb_h=\sum_{h\in G}a_{g^{-1}h}^*b_h\\
	&=\sum_{k\in G}a_k^*b_{gk}=\IP{a,U_{g^{-1}}(b)},\\
	\end{align*}
	so $U_g$ is adjointable with $U^*_g=U_{g^{-1}}$. Since
	\begin{equation*}
	(U_gU_{g^{-1}}(a))_h=(U_{g^{-1}}(a))_{g^{-1}h}=a_{gg^{-1}h}=a_h
	\end{equation*}
	we see that $U_gU_{g^{-1}}(a)=a$. Similarly $U_{g^{-1}}U_g(a)=a$ and so $U_g$ is invertible with $U_{g^{-1}}=U_g^{-1}$, so $U$ is unitary. Since we also have $U_{gh}=U_gU_h$, we see that $u(g)=U_g$ is a group homomorphism. Now we check that $(\pi,u)$ is a covariant homomorphism. We compute
	\begin{align*}
	(U_g\pi(a)U_g^*(b))_h&=(\pi(a)U_g^*(b))_{g^{-1}h}=\alpha_{h^{-1}g}(U_g^*(b))_{g^{-1}h}\\
	&=\alpha_{h^{-1}}(\alpha_g(b_h))=\pi(\alpha_g(b))_h\\
	\end{align*}
	and so we see that $(\pi,u)$ is covariant. Now suppose that $f\in C_c(G,A)$ satisfies $i_{\pi,u}(f)=0$. Then for each $g\in G$ and all $a\in \bigoplus_{h\in G}A_A$ we have
	\begin{align*}
	\big(\sum_{h\in G}\pi(f(h))u_h(a)\big)_g=\sum_{h\in G}\alpha_{g^{-1}}(f(h))a_{h^{-1}g}=0.\\
	\end{align*}
	Suppose $f\neq 0$, say $f(k)\neq 0$. Let $b=f(k)^*\in A$ so that $f(k)b\neq 0$. Let $a$ be the sequence defined by $a_{k^{-1}g}=\alpha_{g^{-1}}(b)$ and $a_h=0$ for $h\in G\bs\{k^{-1}g\}$. Then
	\begin{equation*}
	\sum_{h\in G}\alpha_{g^{-1}}(f(h))a_{h^{-1}g}=\alpha_{g^{-1}}(f(h)b)=0.
	\end{equation*}
	Since $\alpha_{g^{-1}}$ is an automorphism we deduce that $f(h)b=0$ contracting that $f(h)\neq 0$. Thus $f$ must be the zero function whence $i_{\pi,u}(f)=0$ if and only if $f=0$. We conclude that the quotient mapping $q$ in the proof of Theorem~\ref{Profit} is the identity mapping and so $C_c(G,A)$ embeds into $A\times_{\alpha}G$.
\end{proof}
\begin{remark}
	The universal property of $A\times_{\alpha}G$ gives a homomorphism $\pi\times u:A\times_{\alpha} G\to \LL(\oplus_{g\in G}A_A)$ where $(\pi,u)$ is the reduced representation of $A$ in $\LL(\oplus_{g\in G}A_A)$. This makes $\bigoplus_{g\in G}A_A$ into an $A\times_{\alpha}G$-$A$ correspondence.
\end{remark}
\begin{example}
	We will be interested in crossed products by $\Z$. Example~\ref{mankey} showed that any automorphism $\alpha$ of a $C^*$-algebra $A$ defines an action of $\Z$ on $A$ by sending $n\in\Z$ to the $n$th power $\alpha^n$. Given an automorphism $\alpha$ we write $A\times_{\alpha}\Z$ for the crossed product defined by this action. In fact, given any action $\alpha_n$ of $\Z$ on $A$, the automorphism $\alpha_1$ generates the entire action, so we see that every action of $\Z$ arises this way. For a more concrete example, we look at the \textit{Rotation Algebra} $A_{\theta}$. Let $\theta\in[0,1)$ be a fixed and define an action $\alpha_n^{\theta}$ of $\Z$ on $C(\T)$ by
	\begin{equation*}
	\alpha_n^{\theta}(f)(z)=f(e^{2\pi in\theta}z),
	\end{equation*}
	so that $\alpha_n$ is generated by the automorphism
	\begin{equation*}
	\alpha^{\theta}_1(f)(z)=f(e^{2\pi i n\theta}z).
	\end{equation*}
	We call the crossed product $C(\T)\times_{\alpha^{\theta}}\Z$ the rotation algebra and will often denote it by $A_{\theta}$. It is an exercise in \cite[Theorem~2.31]{Crossedandunitizations} that $A_{\theta}$ is a simple $C^*$-algebra if and only if $\theta$ is irrational.
\end{example}	

\begin{example}
	Let $(A,G,\alpha)$ be a $C^*$-dynamical system and consider the crossed product $A\times_{\alpha}G$. Recall that the dual group of $G$ is the group of homomorphisms
	\begin{equation*}
		\widehat{G}=\{\chi:G\to\T, \chi\text{ is a homomorphism}\}
	\end{equation*}
	with the group operation being multiplication. Even if $G$ is discreet, $\widehat{G}$ may not be. For example we have $\widehat{\Z}=\T$, and $\T$ is not discrete. When we speak of a non-discrete group action, we must make some topological considerations. We will avoid discussion of this since we will only be interested in the case when $G=\Z_2$ so that the dual $\widehat{G}=\Z_2$ is discrete. The crossed product has a natural action $\widehat{\alpha}$ of $\widehat{G}$ such that for any $\chi\in\widehat{G}$, $a\in A$ and $g\in G$ we have
	\begin{equation*}
		\widehat{\alpha}_{\chi}(i_A(a)i_G(g))=i_A(a)\chi(g)i_G(g).
	\end{equation*}
	To see this, we observe that for fixed $\chi\in\widehat{G}$ the pair $(i_A,\chi i_G)$ is a covariant representation of $(A,G,\alpha)$ in $A\times_{\alpha}G$, and so the universal property of $A\times_{\alpha}G$ gives a homomorphism $(i_A\times \chi i_G):A\times_{\alpha}G\to A\times_{\alpha}G$ such that
	\begin{equation*}
		(i_A\times\chi i_G)(i_A(a)i_G(g))=i_(a)\chi(g)i_G(g).
	\end{equation*}
	The homomorphism $i_A\times\chi i_G$ is clearly a bijection, so $\chi\mapsto i_A\times \chi i_G$ is the desired action, which we call $\widehat{\alpha}$. In particular, when $G=\Z_2$ so that $\widehat{G}=\Z_2$, the action $\widehat{\alpha}$ is a self-inverse automorphism of $A\times_{\alpha}G$, and so $\widehat{\alpha}$ is a grading on $A\times_{\alpha}\Z_2$. We call this grading the \emph{dual grading}.
\end{example}
\begin{example}\label{Good}
	Let $(B,\beta)$ be a graded $C^*$-algebra. Then $(B,\Z_2,\beta)$ is a $C^*$-dynamical system and we may form the crossed product $B\times_{\beta}\Z_2$. Let us write $u_0$ and $u_1$ for the unitaries $i_{\Z_2}(0),i_{\Z_2}(1)\in\UU(A\times_{\beta}\Z_2)$. Then $B\times_{\beta}\Z_2$ is the closure of elements of the form $b_0u_0+b_1u_1$ such that $b_0,b_1\in B$. Let $\beta\times\alpha$ be the grading on $\Cliff_1$ and let $\hat{\beta}\circ (\beta\times 1)$ be the dual grading on $B\times_{\beta}\Z_2$. Consider the graded $C^*$-algebras $B\Otimes \Cliff_1$ and $B\times_{\beta}\Z_2$ with gradings $\beta\times\alpha$ and $\widehat{\beta}\circ(\beta\times 1)$ respectively. That is, for $b\otimes (x,y)\in B\Otimes \Cliff_1$ and $b_0u_0+b_1u_1\in B\times_{\beta}\Z_2$, the gradings are given by
	\begin{align*}
	(\beta\times\alpha)(b\otimes (x,y))&=\beta(b)\otimes (y,x)\\ \hat{\beta}\circ(\beta\times 1)(b_0u_0+b_1u_1)&=\beta(b_0)u_0-\beta(b_1)u_1.\\
	\end{align*}
	We show that $(B\Otimes \Cliff_1)$ and $B\times_{\beta}\Z_2$ are graded isomorphic. Let $u$ denote the self-adjoint unitary generator of $\Cliff_1$. If we are viewing $\Cliff_1$ as $\C^2$ then $u$ is the element $(1,-1)$. If we are viewing $\Cliff_1$ as Cliff$_1\otimes \C$ then $u$ is the element $e\otimes i$. Let $1$ denote the element $u^2=(1,1)$ of $\Cliff_1$ (or $1\otimes 1$ from the Cliff$_1\otimes \C$ perspective). We may write every element of $\Cliff_1$ in the form $x\cdot 1+y\cdot u$ for some $x,y\in\C$. That is, $\{1,u\}$ is a basis of $\Cliff_1$ and so every element of $B\Otimes \Cliff_1$ is of the form
	\begin{equation*}
	b\otimes 1+c\otimes u
	\end{equation*}
	for some $b,c\in B$. We define a linear map $\phi:B\Otimes \Cliff_1\to B\times_{\beta}\Z_2$ by
	\begin{equation*}
	\phi(b\otimes 1+c\otimes u)=bu_0+cu_1.
	\end{equation*}
	We show that $\phi$ is multiplicative, $*$-preserving, bijective and intertwines the gradings. We compute
	\begin{align*}
	\phi((b\otimes 1+c\otimes u)(d\otimes 1+e\otimes u))&=\phi((bd+c\beta(e))\otimes 1+(be+c\beta(d))\otimes u)\\
	&=(bd+c\beta(e))u_0+(be+c\beta(d))u_1\\
	&=(bu_0+cu_1)*(du_0+eu_1)\\
	&=\phi(b\otimes 1+c\otimes u)\phi(d\otimes 1+e\otimes u)\\
	\end{align*}
	showing $\phi$ is multiplicative. Next we compute
	\begin{align*}
	\phi((b\otimes 1+c\otimes u)^*)&=\phi(b^*\otimes 1-\beta(c^*)\otimes -u)\\
	&=\phi(b^*\otimes 1+\beta(c^*)\otimes u)\\
	&=b^*u_0+\beta(c)^*u_1\\
	&\phi(b\otimes 1+b\otimes u)^*\\
	\end{align*}
	showing $\phi$ is $*$-preserving. For $bu_0+cu_1\in B\times_{\beta}\Z_2$ we have $\phi(b\otimes 1+c\otimes u)=bu_0+cu_1$, so $\phi$ is surjective. If
	\begin{align*}
	\phi(b\otimes 1+c\otimes u)&=0, \text{then}\\
	bu_0-cu_1&=0, \text{ forcing}\\
	b&=c=0.\\
	\end{align*}
	So $\phi$ is injective. Lastly we show that $\phi$ intertwines gradings. We have
	\begin{align*}
	\phi((\beta\times\alpha)(b\otimes 1+c\otimes c))&=\phi(\beta(b)\otimes 1-\beta(c)\otimes u)\\
	&=\beta(b)u_0-\beta(c)u_1\\
	&=\hat{\beta}\circ(\beta\otimes 1)(bu_0+cu_1)\\
	&=\hat{\beta}\circ(\beta\otimes 1)(\phi(b\otimes 1+c\otimes u)).\\
	\end{align*}
	Thus $(B\Otimes \Cliff_1,\beta\times\alpha)$ and $(B\times_{\beta}\Z_2,\hat{\beta}\circ(\beta\times 1))$ are graded isomorphic.
\end{example}

\section{The Toeplitz algebra $\TT_X$}
Given an $A$-$A$ correspondence $X$, the Toeplitz algebra $\TT_X$ is a $C^*$-algebra which is in some sense the `universal' $C^*$-algebra onto which $X$ can be represented. Whether the left action $\phi:A\to\LL(X)$ is injective will make a significant difference to the standard definition of $\OO_X$ in Section~\ref{OO} but not in the definitions of $\TT_X$. We will follow Pimsner's original definition in \cite{FreeProbTheory} of $\TT_X$ as operators on the Fock space (see Example~\ref{Fock this}), though one can equivalently define $\TT_X$ according to its universal properties. 
All definitions and examples come in this section from \cite{FreeProbTheory}.

\begin{definition}
	Let $A$ and $B$ be $C^*$-algebras and $X$ be an $A$-$A$ correspondence. A representation of $X$ in $B$ is a pair of maps $(\pi,t)$ where $\pi:A\to B$ is a homomorphism and $t:X\to B$ is a linear map such that
	\begin{align*}
	t(\xi)^*t(\eta)&=\pi(\IP{\xi,\eta}_A)\\
	t(\phi(a)\xi)&=\pi(a)t(\xi)\\
	\end{align*}
	for all $\xi,\eta\in X$ and $a\in A$.
\end{definition}

If $(\pi,t)$ is a representation of $X$ in $B$, then
\begin{equation*}
t(\xi a)=t(\xi)\pi(a)\qquad \text{ for all $\xi\in X$ and $a\in A$}.
\end{equation*}
This is because
\begin{align*}
\normof{t(\xi a)-t(\xi)\pi(a)}^2&=\normof{\left(t(\xi a)-t(\xi)\pi(a)\right)^*\left(t(\xi a)-t(\xi)\pi(a)\right)}\\
&=\normof{t(\xi a)^*t(\xi a)-t(\xi a)^*t(\xi)\pi(a)-\pi(a^*)t(\xi)^*t(\xi a)+\pi(a^*)t(\xi)^*t(\xi)\pi(a)}\\
&=\normof{\pi(\IP{\xi a,\xi a})-\pi(\IP{\xi a,\xi}a)-\pi(a^*\IP{\xi,\xi a})+\pi(a^*\IP{\xi,\xi}a)}\\
&=\normof{2\IP{\xi a,\xi a}-2\IP{\xi a,\xi a}}=0.\\
\end{align*}
\begin{remark}
	Given a representation $(\pi,t)$ of $X$ in $B$, if we consider $B$ as an $B$-$B$ correspondence with right action $b\cdot c=bc$, inner-product $\IP{a,b}=a^*b$ and left action $a\cdot b=ab$ then the condition
	$t^*(x)t(y)=\pi(\IP{x,y})$ says that $\pi$ and $t$ intertwine inner-products in the sense that
	\begin{equation*}
	\IP{t(x),t(y)}_B=\pi(\IP{x,y}_X).
	\end{equation*}
	Similarly, the condition $\pi(a)t(x)=t(\phi(a)x)$ and the result that $t(x)\pi(a)=t(xa)$ says that $t$ intertwines the left and right actions of $A$ and $B$ on $X$ and $B$ whenever $b\in B$ is in the range of $\pi$. So $(\pi,t)$ is a morphism of correspondences.
\end{remark}
\begin{example}\label{Fock this}
	This example will be very important. Given an $A$-$A$ correspondence $_{\phi}X$, we define $X^{\otimes n}$ to be the $n$-fold tensor product
	\begin{equation*}
		 X^{\otimes n}\coloneqq \overbrace{X\otimes_A X\otimes_A\dots\otimes_A X}^{n\text{ times}} 
	\end{equation*}
	and we define $X^{\otimes 0}=A$. The Fock space $\FF_X$ is the infinite direct sum $\FF_X=\bigoplus_{n=0}^{\infty}X^{\otimes n}$, which is a Hilbert-$A$ module as discussed in Appendix~\ref{Direct sums}. We may extend the left action $\phi$ on $X$ to a left action $\phi^{\infty}$ of $A$ on $\FF_X$ as follows. Consider for $a\in A$, the operators $\phi_n$ on $X^{\otimes n}$ determined by the formulas
	\begin{align*}
		\phi^{0}(a)x&=ax\quad x\in A\\ \phi^{n}(a)\big(x_1\otimes\dots\otimes x_n\big)&=(\phi(a)x_1)\otimes\dots\otimes x_n\quad x_1\otimes\dots\otimes x_n\in X^{\otimes n}.
	\end{align*}
	Each $\phi^n(a)$ is an adjointable operator in $\LL(X^{\otimes n})$ with adjoint $\phi^n(a^*)$, so by Lemma~\ref{direct sum of operators} there is a well defined operator $\phi^{\infty}(a)\coloneqq \oplus \phi^n(a)$ in $\LL(\FF_X)$ such that for $x\in X^{\otimes n}$, we have $\phi^{\infty}(a)x=\phi^n(a)x$.\\
	For $\xi\in X$, we aim to define an operator $T_{\xi}$ by the following formulas
	\begin{equation*}
		T_{\xi}(x_1\otimes\dots\otimes x_n)=\xi\otimes x_1\otimes\cdots\otimes x_n,\quad T_{\xi}(a)=\xi\cdot a.
	\end{equation*}
	For $x_1\otimes\dots \otimes x_n\in X^{\otimes n}$ we see that
	\begin{align*}
		\IP{T_{\xi} x_1\otimes\dots\otimes x_n,T_{\xi} x_1\otimes\dots\otimes x_n}&=\IP{\xi\otimes x_1\otimes\dots\otimes x_n,\xi\otimes x_1\otimes\dots\otimes x_n}\\
		&=\IP{x_1\otimes\dots\otimes x_n,\phi^{\infty}(\IP{\xi,\xi}) x_1\otimes\dots\otimes x_n}\\
		&=\IP{\phi^{\infty}(\IP{\xi,\xi})^{1/2}x_1\otimes\dots\otimes x_n,\phi^{\infty}(\IP{\xi,\xi})^{1/2} x_1\otimes\dots\otimes x_n}\\
		&\leq\normof{\phi^{\infty}(\IP{\xi,\xi})}^{1/2}\IP{x_1\otimes\dots\otimes x_n,x_1\otimes\dots\otimes x_n}\\
	\end{align*}
	where we have used \cite[Corollary~1.25]{Crossedandunitizations} in the last line. Let  $x=(x_1^n\otimes\dots\otimes x_n^n)^{\infty}_{n=0}$ and $y=(y_1^n\otimes\dots\otimes y_n^n)^{\infty}_{n=0}$ be spanning elements of $\FF_X$, so that 
	\begin{equation}\label{an equation}
		\IP{x_0,x_0}+\sum_{n=1}^{\infty}\IP{x_1^n\otimes\dots\otimes x_n^n,x_1^n\otimes\dots\otimes x_n^n}
	\end{equation}
	converges in $A$, and similarly for $y$. Then for any finite sum we have
	\begin{align*}
		&\IP{\xi\cdot x_0,\xi\cdot x_0}+\sum_{n=1}^{N}\IP{\xi\otimes x_1^n\otimes\dots\otimes x_n^n,\xi\otimes x_1^n\otimes\dots\otimes x_n^n}\\
		&\leq \normof{\phi^{\infty}(\IP{\xi,\xi})}^{1/2}\left(\IP{x_0,x_0}+\sum_{n=1}^{N}\IP{x_1^n\otimes\dots\otimes x_n^n,x_1^n\otimes\dots\otimes x_n^n}\right),\\
	\end{align*}
	which converges by Equation \eqref{an equation}. We deduce that $T_{\xi}$ preserves $\FF_X$. Now we wish to show that $T_{\xi}$ is adjointable. Let $x$ and $y$ be as above.
	\begin{align*}
	\IP{T_{\xi}x,y}&=\sum_{n=1}^{\infty}\IP{\xi\otimes x^n_1\otimes\dots\otimes x^n_n,y^n_1\otimes\dots\otimes y^{n+1}_{n+1}}+\IP{\xi\cdot x_0,y^1}_X\\
	&=\sum_{n=1}^{\infty}\IP{x^n_1\otimes\dots\otimes x^n_n,\phi^{\infty}(\IP{\xi,y^n_1})y^n_2\otimes\dots\otimes y^{n+1}_{n+1}}+\IP{x_0,\IP{\xi,y^1}}_A.\\
	\end{align*}
	Therefore $T_{\xi}$ is adjointable with adjoint determined by
	\begin{equation*}
	T_{\xi}^*y=\begin{cases}
	\phi^{\infty}(\IP{\xi,y_1})y_2\otimes\dots\otimes y_n&y\in X^{\otimes n},n> 1\\
	\IP{\xi,y_1}&y\in X\\
	0&y\in A.\\
	\end{cases}
	\end{equation*}
	We deduce that for every $\xi\in X$, the operator $T_{\xi}$ is a member in $\LL(\FF_X)$.
	The operators $T_{\xi}$ and $T_{\xi}^*$ are known as the \textit{creation} and \textit{annihilation} operators associated to $\xi$.
	We see that
	\begin{align*}
	T_{\xi}^*T_{\eta}x_1\otimes\dots\otimes x_n&=\phi^{\infty}(\IP{\xi,\eta})x_1\otimes\dots\otimes x_n\\
	\end{align*}
	so $T_{\xi}^*T_{\eta}=\phi^{\infty}(\IP{\xi,\eta})$.
	One checks that
	\begin{align*}
	\phi^{\infty}(a)T_{\xi}x_1\otimes\dots x_n&=(\phi(a)\xi)\otimes x_1\otimes\dots\otimes x_n\\
	&=T_{\phi(a)\xi}x_1\otimes\dots\otimes x_n\\
	\end{align*}
	and
	\begin{align*}
	T_{\xi}\phi^{\infty}(a)x_1\otimes\dots\otimes x_n&=\xi\otimes\phi(a)x_1\otimes\dots\otimes x_n\\
	&=\xi\cdot a\otimes x_1\otimes\dots\otimes x_n\\
	&=T_{\xi a}x_1\otimes\dots\otimes x_n\\
	\end{align*}
	where we have used that $xa\otimes y=x\otimes \phi(a)y$ in the balanced tensor product. If we define $T:X\to\LL(\FF_X)$ by $T(\xi)=T_{\xi}$ then these properties show that $(\phi^{\infty},T)$ is a representation of $X$ on $\LL(\FF_X)$. 
\end{example}
\begin{notation}
	To match notation with \cite{Crossedandunitizations} we will now refer to the representation $(\phi^{\infty},T)$ as $(i_A,i_X)$ instead. If $(\pi,t)$ is a representation of $X$ in $B$ then we write $C^*(\pi,t)$ for the $C^*$-algebra generated by the image of $\pi$ and $t$ in $B$.
\end{notation}
\begin{definition}
	Let $A$ be a $C^*$-algebra and $X$ an $A$--$A$ correspondence. Let $\FF_X$ and $(i_A,i_X)$ be the Fock space and representation of $X$ in $\LL(\FF_X)$ of Example~\ref{Fock this}. We define the Toeplitz algebra $\TT_X$ associated to $X$ to be the $C^*$-algebra $C^*(i_A,i_X)$ in $\LL(\FF_X)$.
\end{definition}

It is shown in \cite[Proposition~3.3]{FreeProbTheory} that $\TT_X$ can be characterised by the following universal property.
\begin{theorem}\label{theorem: universal property of Toeplitz}
	Let $X$ be an $A$-$A$ correspondence. Let $\TT_X$ be the Toeplitz algebra and let $(i_A,i_X)$ be the Fock space representation. Then $\TT_X$ has the property that for any representation $(\pi,t)$ of $X$ in a $C^*$-algebra $B$, there exists a homomorphism $\pi\times t:\TT_X\to B$ such that $(\pi\times t)\circ i_X=t$ and $(\pi\times t)\circ i_A=\pi$.
\end{theorem}
\begin{remark}
	If $X$ is graded by $\alpha_X$, then as in Example~\ref{eg: direct sum grading} $\alpha_X$ induces a grading $\alpha^{\infty}_X$ on $\FF_X$ such that for $x_1\otimes\dots\otimes x_n\in X^{\otimes n}$ we have
	\begin{equation*}
		\alpha^{\infty}_X(x_1\otimes\dots\otimes x_n)=\alpha_X(x_1)\otimes\dots\otimes\alpha_X(x_n),
	\end{equation*}
	and for $a\in A$ we have $\alpha^{\infty}_X(a)=\alpha_A(a)$. The Toeplitz algebra then carries the induced grading $\widetilde{\alpha^{\infty}_X}$ on $\LL(\FF_X)$. We compute for $x\in X$ and $x_1\otimes\dots\otimes x_n\in X^{\otimes n}$
	\begin{align*}
		\widetilde{\alpha^{\infty}_X}(T_x)(x_1\otimes\dots\otimes x_n)&=\alpha^{\infty}_X\big(x\otimes \alpha_X(x_1)\otimes\dots\otimes\alpha_X(x_n)\big)\\
		&=\alpha_X(x)\otimes x_1\otimes\dots\otimes x_n\\
		&=T_{\alpha_X(x)}x_1\otimes\dots\otimes x_n.\\
	\end{align*}
	Hence $\widetilde{\alpha^{\infty}_X}(T_x)=T_{\alpha_X(x)}$.
\end{remark}

\begin{example}
	Let $A$ be a $C^*$-algebra and consider the $A$--$A$ correspondence $_AA_A$. Then Example~\ref{Gen hom balance} shows that $A\otimes_AA\cong A$, so we have $A^{\otimes n}\cong A$. The Fock space $\FF_A$ then is simply
	\begin{equation*}
		\FF_A=\bigoplus_{n=0}^{\infty}A_A=\HH_A.
	\end{equation*}
	For $(b)\in\HH_A$, under the isomorphism $A^{\otimes n}\cong A$, the maps $T_a$ are given by
	\begin{equation*}
		(T_a(b))_n=ab_{n-1}.
	\end{equation*}
\end{example}

It turns out that $\TT_X$ contains a copy of the compact operators on $\FF_X$, which we will use of later.
\begin{lemma}\label{comapcts inclusion in Toeplitz}
	Let $X$ be a countably generated $A$-$A$ correspondence, suppose that the left action $\varphi:A\to\LL(X)$ is injective and let $I$ be the ideal $\varphi^{-1}(\KK(X))$. Then there is an embedding $j:\KK(\FF_{X,I})\to \TT_X$, where $\FF_{X,I}$ denotes the Hilbert $I$ module of Lemma~\ref{ONE}.
\end{lemma}
\begin{proof}
	For $x=x_1\otimes\dots\otimes x_n\in X^{\otimes n}$ we define
	\begin{equation*}
		T_x=T_{x_1}\dots T_{x_n}\qquad\text{ and }\qquad T_x^*=T_{x_n}\dots T_{x_1}^*,
	\end{equation*}
	and for $a\in A$ we define $T_a=\varphi^{\infty}(a)$ and $T_a^*=\varphi^{\infty}(a^*)$. Let $e_k$ be an approximate identity for $A$ and let $\{x_i\}$ be a frame for $X$ (such a frame exists by Lemma~\ref{frames exist}).
First we prove that 
\begin{equation*}
	T_a \Big(\varphi^\infty(e_n) - \sum^n_{i=1} T_{x_i} T^*_{x_i}\Big) T^*_b \to \Theta_{a, b},
\end{equation*}
for $a,b\in I$, where we have $T_b^*=\varphi^{\infty}(b^*)$ and $T_a=\varphi^{\infty}(a)$. Fix $\ep>0$.  If $a=0$ then the result follows immediately, so suppose $a\neq 0$. We wish to show that for large $k$
\begin{equation*}
	\Bignormof{T_a \Big(\varphi^\infty(e_k) - \sum^k_{i=1} T_{x_i} T^*_{x_i}\Big) T^*_b - \Theta_{a, b}}<\ep.
\end{equation*}
Employing Lemma~\ref{direct sum of operators}, it suffices to show that
\begin{equation*}
	\Biggnormof{\Bigg(T_a \Big(\varphi^\infty(e_k) - \sum^k_{i=1} T_{x_i} T^*_{x_i}\Big) T^*_b - \Theta_{a, b}\Bigg)\Bigg|_{X^{\otimes n}}}<\ep
\end{equation*}
for all $n\geq 0$, with $X^{\otimes 0}=A$. For any $c\in A$ we have $T_{x_i}^*c=0$, hence we compute
\begin{equation*}
	T_a \Big(\varphi^\infty(e_k) - \sum^k_{i=1} T_{x_i} T^*_{x_i}\Big) T^*_bc - \Theta_{a, b}c=ae_kb^*c-ab^*c.
\end{equation*}
Thus on the subspace $A$ we have
\begin{equation*}
	T_a \Big(\varphi^\infty(e_k) - \sum^k_{i=1} T_{x_i} T^*_{x_i}\Big) T^*_b - \Theta_{a, b}=ae_kb^*-ab^*.
\end{equation*}
Since $e_k$ is an approximate unit for $I$, there exists $K_0$ such that for $k\geq K_0$ we have
\begin{equation*}
	\normof{ae_kb^*-ab^*}<\ep.
\end{equation*}
Hence for $k\geq K_0$ we have
\begin{equation*}
	\Bignormof{\Big(T_a \Big(\varphi^\infty(e_k) - \sum^k_{i=1} T_{x_i} T^*_{x_i}\Big) T^*_b - \Theta_{a, b}\Big)\Big|_{A}}=\normof{ae_kb^*-ab^*}<\ep.
\end{equation*}
Now the operator $\Theta_{a,b}$ when restricted to $X^{\otimes n}$ for $n\geq 1$ is zero. Hence we have
\begin{equation*}
	\Bignormof{\Big(T_a \Big(\varphi^\infty(e_k) - \sum^k_{i=1} T_{x_i} T^*_{x_i}\Big) T^*_b - \Theta_{a, b}\Big)\Big|_{X^{\otimes n}}}=\Bignormof{T_a \Big(\varphi^\infty(e_k) - \sum^k_{i=1} T_{x_i} T^*_{x_i}\Big) T^*_b \Big|_{X^{\otimes n}}}.
\end{equation*}

We compute for simple tensors $x_1\otimes\dots \otimes x_n\in X^{\otimes n}$ with $n\geq 1$
\begin{align*}
	T_a \Big(&\varphi^\infty(e_k) - \sum^k_{i=1} T_{x_i} T^*_{x_i}\Big) T^*_b(x_1\otimes\dots\otimes x_n)\\
	&=(\varphi(ae_kb^*)x_1)\otimes x_2\otimes\dots\otimes x_n-\sum_{i=1}^k(\varphi(a)x_i)\otimes \varphi(\IP{x_i,\varphi(b^*)x_1})x_2\otimes\dots\otimes x_n\\
	&=\varphi(a)\Big(\varphi(e_kb^*)x_1-\sum_{i=1}^kx_i\IP{x_i,\varphi(b^*)x_1}\Big)\otimes x_2\otimes\dots\otimes x_n\\
	&=\varphi(a)\Big(\big(\varphi(e_kb^*)-\sum_{i=1}^k\Theta_{x_i,x_i}\varphi(b^*)\big)x_1\Big)\otimes x_2\otimes\dots\otimes x_n.\\
\end{align*}
Thus on $X^{\otimes n}$, we see that $T_a \Big(\varphi^\infty(e_k) - \sum^k_{i=1} T_{x_i} T^*_{x_i}\Big) T^*_b$ is of the form
\begin{equation*}
	 \varphi(a)\Big(\varphi(e_kb^*)-\sum_{i=1}^k\Theta_{x_i,x_i}\varphi(b^*)\Big)\Otimes 1. 
\end{equation*}
Since the left action $\varphi$ is injective, by Lemma~\ref{T}, in order to compute
\begin{equation*}
	\Bignormof{T_a \Big(\varphi^\infty(e_k) - \sum^k_{i=1} T_{x_i} T^*_{x_i}\Big) T^*_b\Big|_{X^{\otimes n},n\geq 1}},
\end{equation*}
it suffices to calculate
\begin{equation*}
	\Bignormof{T_a \Big(\varphi^\infty(e_k) - \sum^k_{i=1} T_{x_i} T^*_{x_i}\Big) T^*_b\Big|_{X}}=\bignormof{\varphi(a)\big(\varphi(e_kb^*)-\sum_{i=1}^k\Theta_{x_i,x_i}\varphi(b^*)\big)}.
\end{equation*}
Since $e_k$ is an approximate unit for $A$ and since $a\neq 0$, there is an $N_0$ such that for $k\geq N_0$ we have
\begin{equation*}
	\normof{e_kb^*-b^*}<\ep/2\normof{a}.
\end{equation*}
Since $x_i$ is a frame for $X$ and $b\in I$, we have $\varphi(b^*)\in\KK(X)$ and so 
\begin{equation*}
	\lim_{k\to\infty}\sum_{i=1}^k\Theta_{x_i,x_i}\varphi(b^*)\to \varphi(b^*).
\end{equation*}
That is, there is some $N_1$ such that for $k\geq N_1$ we have
\begin{equation*}
	\normof{\varphi(b^*)-\sum_{i=1}^k\Theta_{x_i,x_i}\varphi(b^*)}<\ep/2\normof{a}.
\end{equation*}
Hence for $k\geq N_0,N_1$ we have
\begin{align*}
	\normof{\varphi(a)\big(\varphi(e_kb^*)-\sum_{i=1}^k\Theta_{x_i,x_i}\varphi(b^*)\big)}&\leq\normof{\varphi(a)}\normof{\varphi(e_kb^*-b^*)}+\normof{\varphi(a)}\normof{\varphi(b^*)-\sum_{i=1}^k\Theta_{x_i,x_i}\varphi(b^*)}\\
	&< \normof{a}\normof{e_kb^*-b^*}+\ep/2<\ep.\\
\end{align*}
Thus we conclude that for $k\geq K_0,N_0,N_1$, we have
\begin{align*}
	\Bignormof{T_a&\Big(\varphi^{\infty}(e_k) - \sum^k_{i=1} T_{x_i} T^*_{x_i}\Big) T^*_b-\Theta_{a,b}}\\
	&=\sup_{n=0,1,\dots}\Bignormof{\Big(T_a \Big(\varphi^{\infty}(e_k) - \sum^k_{i=1} T_{x_i} T^*_{x_i}\Big) T^*_b-\Theta_{a,b}\Big)\Big|_{X^{\otimes n}}}<\ep\\
\end{align*}
and so
\begin{equation*}
	T_a \Big(\varphi^{\infty}(e_k) - \sum^n_{i=k} T_{x_i} T^*_{x_i}\Big) T^*_b\to \Theta_{a,b}
\end{equation*}
as claimed. Now we wish to show that for any $x,y\in \FF_{X,I}$, we have
\begin{equation*}
	T_x \Big(\varphi^{\infty}(e_k) - \sum^n_{i=k} T_{x_i} T^*_{x_i}\Big) T^*_y\to \Theta_{x,y}.
\end{equation*}
Consider the operator $T_x\Theta_{a,b}T_y^*$ for $a,b\in A$ and $y\in X^{\otimes n}$. This operator is non-zero only on $X^{\otimes n}$ and if $z\in X^{\otimes n}$ we compute
\begin{equation*}
	T_x\Theta_{a,b}T_y^*(z)=x\dot (ab^*\IP{y,z})=xa\cdot \IP{yb,z}=\Theta_{xa,yb}(z).
\end{equation*}
For $z\in (X^{\otimes n})^{\perp}$ we have
\begin{equation*}
	T_x\Theta_{a,b}T_y^*(z)=\Theta_{xa,yb}(z)=0.
\end{equation*}
Hence we have $T_x\Theta_{a,b}T_y^*=\Theta_{xa,yb}$. Linearity and continuity of $y\mapsto \Theta_{x,y}$ then tells us that for any $x,y\in\FF_X$, we have $T_x\Theta_{a,b}T_y^*=\Theta_{xa,yb}$. Now fix $x,y\in\FF_{X,I}$. Lemma~\ref{ONE} implies that there are $u,v\in\FF_X$ and and $a,b\in I$ such that $x=ua$ and $y=vb$. We then compute
\begin{align*}
	T_x \Big(\varphi^\infty(e_k) - \sum^k_{i=1} T_{x_i} T^*_{x_i}\Big) T^*_y&=  T_{u} \Big(T_a \Big(\varphi^\infty(e_k) - \sum^k_{i=1} T_{x_i} 	T^*_{x_i}\Big) T^*_b\Big)T^*_{v},\\
\end{align*}
which converges to
\begin{equation*}
	T_{u} \Theta_{a,b} T^*_{v}= \Theta_{ua,vb}=\Theta_{x,y}.
\end{equation*}
Hence $\TT_X$ is a $C^*$-algebra containing the finite rank operators on $\FF_{X,I}$, and so we deduce that $\KK(\FF_{X,I})\subseteq \TT_X$.
\end{proof}

\section{The Cuntz-Pimsner algebra $\OO_X$}\label{OO}
 In this section we will look at two definitions of $\OO_X$. When the left action $\phi$ is injective, these definitions coincide, but when $\phi$ is not injective they no longer agree. For our main theorem in Chapter~\ref{The Exact Sequence!} we will require $\phi$ to be injective and so will not need to differentiate.
 \subsection{Pimsner's definition}
 To define $\OO_X$ according to Pimsner we will need a technical lemma which we will not prove.
\begin{lemma}(\cite[Proposition 2.43]{Crossedandunitizations})
	Let $X$ be an $A$-$A$ correspondence and $(\pi, t)$ a representation of $X$ in $B$. There exists a homomorphism $\pi^{(1)}:\KK(X)\to B$ such that $\pi^{(1)}(\theta_{x,y})=t(x)t(y)^*$ for all $x,y\in X$.
\end{lemma}

\begin{definition}
	Let $X$ be an $A$-$A$ correspondence with $\phi:A\to\LL(X)$ the left action and $I=\phi^{-1}(\KK(X))$ as in Lemma~\ref{compact otimes 1}. Let $(\pi,t)$ be a representation of $X$ in $B$. We say that $(\pi,t)$ is \textit{Cuntz-Pimsner covariant} if for all $a\in I$ we have
	\begin{equation*}
	\pi(a)=\pi^{(1)}(\phi(a)).
	\end{equation*}
\end{definition}
The following definition of the Cuntz-Pimsner algebra $\OO_X$ is slightly different than Pimsner's original definition in \cite{FreeProbTheory}. Defining $\OO_X$ in the following way will save us some work later on.
\begin{definition}\label{def: Cuntz-Pimsner}
	Let $X$ be an $A$-$A$ correspondence with $\phi:A\to\LL(X)$ the left action and $I=\phi^{-1}(\KK(X))$ as in Lemma~\ref{compact otimes 1}. Let $(i_A,i_X)$ be the Fock space representation, let $J$ be the ideal in $\TT_X$ generated by elements of the form $i_A(a)-i_A^{(1)}(\phi(a))$ for $a\in I$, and let $q:\LL(\FF_X)\to \LL(\FF_X)/J$ be the quotient map. The Cuntz-Pimsner algebra $\OO_X$ associated to $X$ is the $C^*$-algebra $C^*(i_A\circ q,i_X\circ q)$ in $\LL(\FF_X)/J$.
\end{definition}
Now we have defined $\OO_X$ we would like to show that it has some nice properties. First we will need a technical lemma.
\begin{lemma}\label{Lem:Technical}
	Let $X$ be an $A$--$A$ correspondence. For each $T\in\KK(\FF_X)$ we have 
	\begin{equation*}
		\normof{T\big|_{X^{\otimes j}}}\to 0\text{ as }j\to \infty.
	\end{equation*}
\end{lemma}
\begin{proof}
	This proof is taken from \cite[Lemma~2.48]{Crossedandunitizations}. Fix $T\in\KK(\FF_X)$ and fix $\ep>0$. Choose $\xi_i,\eta_i\in\FF_X$ such that $\bignormof{T-\sum_{i=1}^n\Theta_{\xi_i,\eta_i}}<\ep/2$. Let $P_j:\FF_X\to X^{\otimes j}$ be the projection onto $X^{\otimes j}$. By construction of the infinite direct sum, there exists some $J$ such for $j\geq J$ we have $\normof{P_j\eta_i}<\ep/(2n\max_i\normof{\xi_i})$ for all $0\leq i\leq n$. We then see that
	\begin{equation*}
		\Bignormof{\Big(\sum_{i=1}^n\Theta_{\xi_i,\eta_i}\Big)\big|_{X^{\otimes j}}}=\Bignormof{\sum_{i=1}^n\Theta_{\xi_i,P_j\eta_i}}\leq \sum_{i=1}^n\sup_{\normof{z}\leq 1}\normof{\xi_i\IP{P_j\eta_i,z}}\leq\sum_{i=1}^n\normof{\xi_i}\normof{P_j\eta_i}<\ep/2.
	\end{equation*}
	So for $j\geq J$, we have
	\begin{align*}
		\normof{T\big|_{X^{\otimes j}}}&\leq\normof{T\big|_{X^{\otimes j}}-\Big(\sum_{i=1}^n\Theta_{\xi_i,\eta_i}\Big)\big|_{X^{\otimes j}}}+\normof{\Big(\sum_{i=1}^n\Theta_{\xi_i,\eta_i}\Big)\big|_{X^{\otimes j}}}\\
		&< \normof{T-\sum_{i=1}^n\Theta_{\xi_i,\eta_i}}+\ep/2<\ep.\\
	\end{align*}
\end{proof}

\begin{theorem}\label{Cuntz-Pimsner exact sequence}
	With notation as above, let $\FF_{X,I}$ denote the $\FF_X$ as a right $I$ module as in Lemma~\ref{ONE}. Suppose that $\phi$ is injective, let $j:\KK(\FF_{X,I})\to\TT_X$ denote the inclusion of Lemma~\ref{comapcts inclusion in Toeplitz}, and let $q:\LL(\FF_X)\to\LL(F_X)/\KK(\FF_{X,I})$ be the quotient map. Then $J$ is precisely $\KK(\FF_{X,I})$ and so we have a short exact sequence 
	\begin{equation*}
		0\to \KK(\FF_{X,I})\xrightarrow{j}\TT_X\xrightarrow{q}\OO_X\to 0.
	\end{equation*}
	The maps $(j_A,j_X)=(q\circ i_A,q\circ i_X)$ are a Cuntz-Pimsner covariant representation of $X$ in $\OO_X$, and $j_A$ is injective.
\end{theorem}
\begin{proof}
	We wish to show that $J=\KK(\FF_{X,I})$. In Lemma~\ref{comapcts inclusion in Toeplitz} we showed that every element $K\in \KK(\FF_{X,I})$ is a limit of elements in $J$, so we need only show that $J\subseteq \KK(\FF_{X,I})$. Fix a generating element $i_A(a)-i_A^{(1)}(\phi(a))$ with $a\in I$. That is, we wish to write $i_A(a)-i_A^{(1)}(\phi(a))$ as a limit of linear combinations of rank one operators $\Theta_{x,y}$ with $x,y\in\FF_{X,I}$. Since $\phi(a)\in\KK(X)$, we may write
	\begin{equation*}
		\phi(a)=\lim_{n\to\infty}\sum_{i=1}^n\Theta_{x_i^n,y_i^n}
	\end{equation*}
	for some $x_i^n,y_i^n\in X$. Then for $x_1\otimes \dots\otimes x_n\in A^{\otimes n}$, $n\geq 1$ we have
	\begin{align*}
		\big(i_A(a)-i_A^{(1)}(\phi(a))\big)\big(x_1\otimes\dots\otimes x_n\big)&=\phi^{\infty}(a)(x_1\otimes\dots\otimes x_n)-\lim_{n\to\infty}\sum_{i=1}^ni_A^{(1)}(\Theta_{x_i^n,y_i^n})(x_1\otimes\dots\otimes x_n)\\
		&=(\phi(a)x_1)\otimes\dots\otimes x_n-\lim_{n\to\infty}\sum_{i=1}^nT_{x_i^n}T_{y_i^n}^*(x_1\otimes\dots\otimes x_n)\\
		&=(\phi(a)x_1)\otimes\dots\otimes x_n-\lim_{n\to\infty}\sum_{i=1}^n(\Theta_{x_i^n,y_i^n}x_1)\otimes x_2\otimes\dots\otimes x_n\\
		&=(\phi(a)x_1)\otimes\dots\otimes x_n-(\phi(a)x_1)\otimes x_2\otimes\dots\otimes x_n\\
		&=0.\\
	\end{align*}
	For $b\in A$ we compute
	\begin{align*}
	i_A(a)-i_A^{(1)}(\phi(a))\big(b\big)&=ab-\lim_{n\to\infty}\sum_{i=1}^nT_{x_i^n}T_{y_i^n}^*(b)=ab.\\
	\end{align*}
	Thus, letting $e_i$ be an approximate identity for $I$, we deduce that $i_A(a)-i_A^{(1)}$ is precisely the operator $\lim_{i\to\infty}\Theta_{a,e_i}$. Since $a,e_i\in I$, we have $i_A(a)-i_A^{(1)}\in\KK(\FF_{X,I})$. Since elements of the form $i_A(a)-i_A^{(1)}$ generate $J$ and since $\KK(\FF_{X,I})$ is an ideal, we deduce that $J\subseteq \KK(\FF_{X,I})$. Thus we conclude that $J=\KK(\FF_{X,I})$. Now it is clear that the representation $(j_A,j_A)=(q\circ i_A,q\circ i_X)$ is Cuntz-Pimsner covariant, since $q$ takes elements of the form $i_A(a)-i_A^{(1)}(\phi(a))$ to the zero element $\KK(\FF_{X,I})$ of $\LL(\FF_X)/\KK(\FF_{X,I})$. Lastly we show that $j_A$ is injective. Suppose that for some $a\in A$ we have $j_A(a)=0$, that is, $i_A(a)\in J$. Since $\phi$ is injective, Lemma~\ref{T} shows that 
	\begin{equation*}
		\normof{i_A\big|_{X^{\otimes j}}(a)}=\normof{\phi^{\infty}\big|_{X^{\otimes j}}(a)}=\normof{\phi(a)}=\normof{a},
	\end{equation*}
	and so by Lemma~\ref{Lem:Technical} we deduce that $a=0$. We conclude that $j_A$ is injective as required.
\end{proof}
\begin{remark}
	If $\Theta_{x,y}\in \KK(\FF_{X,I})$ then we compute
	\begin{align*}
		\widetilde{\alpha_X^{\infty}}(\Theta_{x,y})z&=\alpha^{\infty}_X(x\IP{y,\alpha^{\infty}_X(z)})\\
		&=\alpha^{\infty}(x)\alpha_A(\IP{y,\alpha_X(z)})=\alpha^{\infty}(x)\alpha_A(\IP{\alpha_X^{\infty}(y),z})=\Theta_{\alpha_X^{\infty}(x),\alpha_X^{\infty}(y)}z.\\
	\end{align*}
	If $x,y\in \FF_{X,I}$ then the characterisation $\FF_{X,I}=\FF_X\cdot I$ shows that $\alpha_X^{\infty},\alpha_X^{\infty}\in\FF_{X,I}$. Hence $\Theta_{\alpha_X^{\infty}(x),\alpha_X^{\infty}(y)}\in\KK(\FF_{X,I})$. Linearity and continuity then imply that $\widetilde{\alpha_X^{\infty}}$ preserves the ideal $\KK(\FF_{X,I})$, and so we deduce that $\widetilde{\alpha_X^{\infty}}$ descends to a grading on $\OO_X$.
\end{remark}
As with $\TT_X$, in \cite[Theorem~3.12]{FreeProbTheory}, Pimsner proved that $\OO_X$ can be characterised by the following universal property.
\begin{theorem}
	Let $X$ be an $A$-$A$ correspondence. Let $\OO_X$ be the Cuntz-Pimsner algebra and let $(j_A,j_X)$ be the canonical representation of $X$ in $\OO_X$. Then $\OO_X$ has the property for any Cuntz-Pimsner covariant representation $(\pi,t)$ of $X$ in a $C^*$-algebra $B$, there exists a homomorphism $\pi\times t:\OO_X\to B$ such that $\pi\times t\circ j_X=t$ and $\pi\times t\circ j_A=\pi$.
\end{theorem}
These universal properties seem to be quite similar to those of the crossed product. We show that for certain classes of correspondences, the Cuntz-Pimsner algebra $\OO_X$ is in fact the crossed product by $\Z$.\\
If $\alpha: A\to B$ is a homomorphism then we have seen in Example~\ref{Marnie} how $A$ is a $A$--$B$ correspondence with left action $\alpha$. The conditions that a Cuntz-Pimsner covariant representation $(\pi,t)$ of $_{\alpha}A$ must satisfy can then be written as
\begin{align}
t(a)^*t(b)&=\pi(a^*b)\label{one}\\
\pi(a)t(b)&=t(\alpha(a)b)\label{two}\\
\pi(a)&=\psi^{(1)}(\alpha(a))=\psi^{(1)}(\Theta_{\alpha(a),1})=t(\alpha(a))t(1)^*\label{three}\\\nonumber
\end{align}
\begin{theorem}\label{thm: Crossed by Z is Cuntz Pimsner}
	Let $A$ be a unital $C^*$-algebra and $\alpha$ an automorphism of $A$. If $(j_A,j_{_\alpha A})$ is the canonical representation of $_{\alpha}A$ in $\OO_{_\alpha A}$ and $(i_A,i_{\Z})$ is the canonical representation of $(A,\Z,\alpha)$ in $A\times_{\alpha}\Z$ then there is an isomorphism $\psi:A\times_{\alpha}\Z\to \OO_{_\alpha A}$ such that $\psi(i_A(a)i_{\Z}(n))=j_A(a)(j_{_\alpha A}(1)^*)^n$.
\end{theorem}
\begin{proof}
	We follow the proof of \cite[Corollary~2.49]{Crossedandunitizations}.
	Let $(\pi,t)$ be any representation of $_\alpha A$. We aim to show that there is a group homomorphism $\Z\ni n\mapsto t(1)^n$ and that this map together with $\pi$ satisfy conditions (1) and (2) of Theorem~\ref{Profit}. To do this, we first wish to show that $\pi(1)$ is a unit for $C^*(\pi,t)$. Since $\pi$  is a homomorphism,
	\begin{equation*}
	\pi(1)\pi(a)=\pi(a)=\pi(a)\pi(1),
	\end{equation*}
	so $\pi(1)$ acts as a unit on the image of $\pi$. Since $\alpha$ is an automorphism it satisfies $\alpha(1)=1$ which gives
	\begin{equation*}
	\pi(1)t(a)=t(\alpha(1)a)=t(a)=t(a)\pi(1),
	\end{equation*}
	and so $\pi(1)$ satisfies
	\begin{equation*}
	\pi(1)b=b=b\pi(1)
	\end{equation*}
	for all $b\in C^*(\pi,t)$. Hence $\pi(1)$ it is a unit for $C^*(\pi,t)$. Now we wish to show that $t(1)$ is invertible, and in particular unitary. The identity operator is compact with $1=\theta_{1,1}$, so we have
	\begin{equation*}
	1=\pi(1)=\psi^{(1)}(\theta_{1,1})=t(1)t(1)^*.
	\end{equation*}
	We also have
	\begin{equation*}
	1=\pi(1)=t(1)^*t(1),
	\end{equation*}
	so we have shown that $t(1)$ is invertible with inverse $t(1)^*$. So $t(1)$ is unitary. Thus there is a group homomorphism $\phi:\Z\to \UU C^*(\pi,t)$ such that $\phi(n)=t(1)^{-n}=(t(1)^*)^n$. We see that
	\begin{align*}
	\phi(n)\pi(a)\phi(n)^*&=(t(1)^*)^n\pi(a)t(1)^n\\
	&=(t(1)^*)^{n-1}t(1)^*t(\alpha(a))t(1)^{n-1}\\
	&=(t(1)^*)^{n-1}\pi(\alpha(a))t(1)^{n-1}\\
	&\;\vdots\\
	&=\pi(\alpha^n(a))\\
	\end{align*}
	and so the pair $(\pi,\phi)$ is a covariant representation of the $C^*$-dynamical system $(A,\Z,\alpha)$ on $C^*(\pi,t)$ for any representation $(\pi,t)$ of $_\alpha A$. In fact, we can show that every covariant representation of $(A,\Z,\alpha)$ arises in this way. Suppose that $(f,g)$ is a covariant representation of $(A,\Z,\alpha)$. We claim that the pair $(f,t)$ where $t(a)=g(-1)f(a)$ defines a representation of $_\alpha A$ on $C^*(f,t)$. Noting that $g(1)^*=g(-1)$, we check conditions \eqref{one}, \eqref{two} and \eqref{three}
	\begin{align*}
	t^*(a)t(b)&=f(a^*)g(1)g(-1)f(b)=f(a^*b)\\
	f(a)t(b)&=f(a)g(-1)f(b)=g(-1)f(\alpha(a)b)=t(\alpha(a)b)\\
	t(\alpha(a))t(1)^*&=g(-1)f(\alpha(a))g(1)=f(a)\\
	\end{align*} 
	and we determine that $(f,t)$ is a Cuntz-Pimsner covariant representation of $_\alpha A$.\\
	
	For a representation $(\pi,t)$ of $_\alpha A$, properties \eqref{one}, \eqref{two} and \eqref{three} show that
	\begin{equation*}
	\overline{\Span}\{\pi(a)\phi(n):a\in A,n\in\Z\}=C^*(\pi,t),
	\end{equation*}
	so in particular for the representation $(j_A,j_{_\alpha A})$, $\OO_{_\alpha A}$ satisfies property (1) of Theorem~\ref{Profit}. Now suppose that $(f,g)$ is a covariant representation of $A$ on $B$. Then $(f,t)$ is a representation on $C^*(f,t)$ where $t$ is defined as above, and the universal property of $\OO_{_\alpha A}$ gives a map $f\times t:\OO_{_\alpha A}\to C^*(f,t)$ such that $f\times t\circ j_A= f$ and $f\times g\circ j_{_\alpha A}=t$. This implies that
	\begin{equation*}
	(f\times t)\circ\phi(n)=(f\times t)((j_{_\alpha A}(1)^n)^*)=(t(1)^*)^n=(g(-1)^*)^n=g(n)
	\end{equation*}
	and so $f\times t$ satisfies hypothesis (2) of Theorem~\ref{Profit}. Since $\OO_{_\alpha A}$ satisfies both (1) and (2) of Theorem~\ref{Profit}, by uniqueness of the crossed product we deduce that $\OO_{_\alpha A}\cong A\times_{\alpha}\Z$.
\end{proof}

\subsection{Katsura's definition}

In \cite{Katsura}, Katsura provided an alternative definition of the $C^*$-algebra $\OO_X$ based on its universal properties. When the left action $\phi$ on $X$ is injective, Katsura's alternative definition coincides with Pimsner's definition, however when $\phi$ is not injective, the $C^*$-algebras $\OO_X$ under both definitions are not necessarily isomorphic. Katsura showed in \cite[Proposition~3.10]{Katsura} that under this definition, when $E$ is any graph and $X$ is the graph correspondence we have $\OO_X\cong C^*(E)$ which is not true for graphs with sinks under Pimsner's definition. When the left action on $X$ is not injective, the graph $E$ has no sinks by Table~\ref{Table}. Hence, when interested in graphs with sinks it is natural to work with Katsura's definition of $\OO_X$. This short section is included to give background to the discussion of Chapter~\ref{ch: closing remarks}, we will otherwise be working with Pimsner's definition throughout this thesis.

\begin{definition}
	Let $A$ be a $C^*$-algebra, let $X$ be an $A$--$A$ correspondence with left action $\phi$, and let $J\triangleleft A$ be in ideal such that $J\subseteq \phi^{-1}(\KK(X))$. If $(i_A,i_X)$ is the Fock space representation of $X$ in $\LL(\FF_X)$ and $q:\LL(\FF_X)\to \LL(\FF_X)/\KK(\FF_{X,J})$ denotes the quotient map then we define the \emph{relative Cuntz-Pimsner algebra} $\OO(J,X)$ to be the $C^*$-algebra $C^*(i_A\circ q,i_X\circ q)$ generated in $\LL(\FF_X)/\KK(\FF_{X,J})$.
\end{definition}
\begin{remark}
	When $J$ is the ideal $\phi^{-1}(\KK(X))$, we obtain $\OO_X=\OO(J,X)$. In \cite{AddingTails} it is noted that there are Cuntz-Pimsner algebras which are quotients of relative Cuntz-Pimsner algebras, and there are relative Cuntz-Pimsner algebras which arise as quotients of Cuntz-Pimsner algebras. Thus relative Cuntz-Pimsner algebras are quite useful in the study of Cuntz-Pimsner algebras.
\end{remark}
\begin{definition}[Katsura's definition]
	Let $A$ be a $C^*$-algebra and let $X$ be an $A$--$A$ correspondence with left action $\phi$. Let $J_X$ be the ideal $J_X=\Ker(\phi)^{\perp}\cap \phi^{-1}(\KK(X))$. Then Katsura defines the \emph{Cuntz-Pimsner} algebra $\OO_X$ to be the relative Cuntz-Pimsner algebra $\OO(J_X,X)$.
\end{definition}
\begin{remark}
	The ideal $J_X=\Ker(\phi)^{\perp}\cap \phi^{-1}(\KK(X))$ is the largest ideal of $A$ on which $\phi$ is injective onto some subset of $\KK(X)$. In \cite[Chapter~2]{Katsura}, Katsura defines $\OO_X$ as the $C^*$-algebra generated by a universal representation $(\pi,t)$ satisfying $\pi(a)=\pi^{(1)}(\phi(a))$ whenever $a\in J_X$, and remarks in \cite[Section~3.2]{Katsura} that this definition coincides with the relative Cuntz-Pimsner algebra $\OO(J_X,X)$ because of \cite[Theorem~2.19]{RelativeCuntzPimsner}.
\end{remark}

%% file: chapter5.tex
\chapter{K-theory}
\label{K-theory}

K-theory for $C^*$-algebras is an important tool for classifying the structure. To each (complex) $C^*$-algebra $A$ we can associate two abelian groups $K_0(A)$ and $K_1(A)$ which encode information about projections and unitaries in $A$. Isomorphic $C^*$-algebras have isomorphic K-groups, so computing the K-theory of $C^*$-algebras gives us a way of telling when two $C^*$-algebras are different. There is in fact a large class of $C^*$-algebras for which the $K$-theory completely classifies the $C^*$-algebra: for example, any two so-called \emph{Kirchberg algebras} $A$ and $B$ satisfying $K_*(A)\cong K_*(B)$ are automatically stably isomorphic. We are interested in developing a similar tool for graded $C^*$-algebras: associating abelian groups that account for graded information so that graded isomorphic graded $C^*$-algebras have isomorphic associated groups. Before looking for such a thing, we develop some facts about regular K-theory. We will only be scratching the surface of this topic, making basic definitions and setting up some useful tools for computation. The following definitions, discussion and results are taken from \cite{BlueBook} and \cite{Mundey}.\\

\section{Equivalence relations}

For a $C^*$-algebra $A$ we would like to define $K_0(A)$ and $K_1(A)$ to be abelian groups consisting of equivalence classes of projections and unitaries. We would like our groups to be abelian, but generally multiplication in $A$ is not commutative. Addition however is commutative, but two projections only add to give another projection when they are orthogonal. The solution to this problem is to consider matrix algebras over $A$.
If $p$ and $q$ are any two projections in $A$ then we may embed them as two orthogonal projections in $M_2(A)$ via $p\mapsto \diag(p,0)$ and $q\mapsto \diag(0,q)$. Thus when we add $p$ and $q$ together in $M_2(A)$ we obtain another projection. The problem now is that we have introduced new projections in $M_2(A)$, so we must include them in larger matrix algebras. The solution then is to consider projections in matrices of all sizes at once.

\begin{definition}
	Let $A$ be a $C^*$-algebra. We define $\PP(A)$ to be the set of projections in $A$, $\UU(A)$ to be the set of unitary elements in $A$ and $M_{m,n}(A)$ the set of $m$ by $n$ matrices whose entries are elements of $A$. We write $M_n(A)$ for $M_{n,n}(A)$. We define 
	\begin{equation*}
		\PP_n(A)=\PP(M_n(A)),\qquad \PP_{\infty}(A)=\bigcup_{n=1}^{\infty}\PP_n(A),
	\end{equation*}
	as well as
	\begin{equation*}
		\UU_n(A)=\UU(M_n(A)),\qquad \UU_{\infty}(A)=\bigcup_{n=1}^{\infty}\UU_n(A),
	\end{equation*}
	viewing $\PP_n(A)$ as being pairwise disjoint and $\UU_n(A)$ as being pairwise disjoint.
\end{definition}
\begin{remark}
	The set $M_n(A)$ is itself a $C^*$ algebra under the usual matrix multiplication and operator norm. The $C^*$-algebra $M_n(A)$ is isomorphic to the tensor product $M_n(\C)\otimes A$, which is independent of tensor product norm since $M_n(\C)$ is nuclear.	
\end{remark}
Working inside $\PP_{\infty}(A)$, we may define a binary operation $\oplus$ by $p\oplus q=\diag(p,q)$; that is,
\begin{equation*}
	p\oplus q=\begin{pmatrix}
		p&0\\
		0&q\\
	\end{pmatrix}\qquad p,q\in\PP_{\infty}(A).
\end{equation*}
Under this operation $p\oplus q$ and $q\oplus p$ are different matrices. We would like this operation to be commutative, and so we introduce some equivalence relations to quotient by.
\begin{definition}
	We say that two projections $p,q\in\PP(A)$ are \emph{Murray Von-Neumann equivalent} if there exists a partial isometry $v\in A$ such that $vv^*=p$ and $v^*v=q$. Te write $p\sim q$. We say that $p$ and $q$ are \emph{unitarily equivalent} if there exists a unitary $u\in A^+$ such that $upu^*=q$, where $A^+$ is the minimal unitization of $A$ and we write $p\sim_u q$. We say that $p$ and $q$ are homotopic is there is a norm continuous path of projections $p_t$ such that $p_0=p$ and $p_1=q$, and we write $p\sim_h q$.
\end{definition}
\begin{remark}
	We learn from \cite[Proposition~2.2.7]{BlueBook} that if $p,q$ are projections in a $C^*$-algebra $A$ then $p\sim_h q$ implies $p\sim_u q$, and $p\sim_u q$ implies $p\sim q$, but in general the reverse implications do not hold. However, if $A^+$ denotes the minimal unitization of $A$, then $p\sim_uq$ if and only if $p\sim q$ and $1-p\sim 1-q$ in $A^+$ (\cite[Proposition~2.2.2]{BlueBook}), and if $\UU_0(A)$ denotes the connected component of the identity in $\UU(A)$, then $p\sim_h q$ if and only if $p=u^*qu$ for some $u\in\UU_0(A)$ (\cite[Proposition~2.2.6]{BlueBook}).
\end{remark}
We can extend our notion of Murray-Von Neumann equivalence to $\PP_{\infty}(A)$ as follows:
\begin{definition}
	Let $A$ be a $C^*$-algebra and let $p\in\PP_n(A)$, $q\in\PP_m(A)$. We say that $p$ and $q$ are Murray-Von Neumann equivalent if there exists some matrix $v\in M_{n,m}(A)$ such that $p=vv^*$ and $q=v^*v$. We write $p\sim_0 q$.
\end{definition}
Following the proof of \cite[Proposition~4.1.2]{Mundey}, for any $p\in \PP_n(A)$ and $q\in\PP_m(A)$ we have
\begin{equation*}
	\begin{pmatrix}
		p&0\\
		0&q\\
	\end{pmatrix}=
	\begin{pmatrix}
		0&q\\
		p&0\\
	\end{pmatrix}
	\begin{pmatrix}
		0&q\\
		p&0\\
	\end{pmatrix}^*
	\quad\text{and}\quad
	\begin{pmatrix}
	q&0\\
	0&p\\
	\end{pmatrix}=
	\begin{pmatrix}
	0&q\\
	p&0\\
	\end{pmatrix}^*
	\begin{pmatrix}
	0&q\\
	p&0\\
	\end{pmatrix},
\end{equation*}
so we see that $p\oplus q\sim_0 q\oplus p$. We also note that if $p\sim_0 p'$ and $q\sim_0 q'$ then there exist $v$ and $w$ such that $p=vv^*$, $p'=v^*v$, $q=ww^*$ and $q'=w^*w$, so we obtain
\begin{equation*}
	p\oplus q=\diag(v,w)\diag(v,w)^*\sim_0 \diag(v,w)^*\diag(v,w)=p'\oplus q',
\end{equation*}
so $\oplus$ gives $\PP_{\infty}(A)/\sim_0$ the well defined structure of an abelian semigroup. We would like to create $K_0(A)$ using this semigroup. For this we will need to define the Grothendieck construction for semigroups.
\begin{proposition}
	Let $S$ be an abelian semigroup. Then the relation $\sim$ on $S\times S$ $(x_1,y_1)\sim (x_2,y_2)$ if there exists some $z\in S$ such that $x_1+y_2+z=x_2+y_1+z$ is an equivalence relation, and $S\times S/\sim$ is an abelian group with identity $[(x,x)]$ for any $x\in S$ and inverses $-[(x,y)]=[(y,x)]$. We denote $S\times S/\sim $ by $\mathcal{G}(S)$ and call it the Grothendieck group of $S$. There is a semigroup homomorphism $\gamma_S:S\to\GG(S)$, $\gamma_S(x)=[(x,0)]$, and $\GG(S)$ is generated by elements of the form $\gamma_S(x)-\gamma_S(y)$ for $x,y\in S$.
\end{proposition}
A proof of these facts can be found in both \cite[3.1.2]{BlueBook} and \cite[Proposition~4.1.4]{Mundey}. Note that $\gamma_S$ is not always injective.
\begin{example}
	We have $\GG(\N,+)\cong\Z$, $\GG(\N\cup\{\infty\},+)$ is the trivial group and $\GG(\N,\times)\cong(\Q^+,\times)$. If $S$ is an abelian group then we have $\GG(S)\cong S$.
\end{example}
Finally we are ready to define $K_0(A)$ for unital $C^*$-algebras.
\begin{definition}\label{K_0}
	Let $A$ be a unital $C^*$-algebra. We define
	\begin{equation*}
		K_0(A)=\GG(\PP_{\infty}(A)/\sim_0).
	\end{equation*}
	Identifying the equivalence class $[p]_0$ under Murray-Von Neumann equivalence of a projection in $\PP_{\infty}$ with it's image under $\gamma_S$ in $K_0(A)$ we may write
	\begin{equation*}
		K_0(A)=\{[p]_0-[q]_0:p,q\in\PP_n(A),n\in\N\}.
	\end{equation*}
\end{definition}
There is of course nothing stopping us from defining $K_0(A)$ in the same way as Definition~\ref{K_0} for non-unital $C^*$-algebras, for technical reasons however we define non-unital $A$ as follows. Given a homomorphism $\phi:A\to B$ of $C^*$-algebras, if $p\in A$ is a projection then $\phi(p)$ is also a projection in $B$. It can be shown using some universal properties of the Grothendieck construction that $\phi$ induces a map $\phi_*:K_0(A)\to K_0(B)$ such that 
\begin{equation*}
	\phi_*([p]_0-[q_0])=[\phi(p)]_0-[\phi(q)]_0.
\end{equation*}
\begin{definition}
	Let $A$ be a non-unital $C^*$-algebra. Let $A^+$ be the minimal unitization of $A$ and let $\pi:A^+\to\C$ be the projection homomorphism onto $\C$. We define $K_0(A)$ to be the kernel of the induced map $\pi_*:K_0(A^+)\to K_0(\C)$.
\end{definition}

\begin{example}\label{K theory of C}
	We compute $K_0(\C)$ using the hints given in \cite[Exercise 2.9]{BlueBook}. We claim that $K_0(\C)\cong \Z$. To see this we show that equivalence classes of projections in the matrix algebra $M_{\infty}(\C)$ are determined precisely by their rank. A handy fact we will use is that the trace of a projection $P$ in $M_n(\C)$ is its rank. This is because the eigenvalues of $P$ are either zero or one, and so since projections are self-adjoint, by the spectral theorem we may diagonalise $P$ so that
	\begin{equation*}
	P=U^*DU
	\end{equation*}
	for $U$ the unitary matrix whose columns are the eigenvectors for $P$ and $D$ the diagonal matrix whose entries are the corresponding eigenvalues for $P$. That is, $D$ is a diagonal matrix with a diagonal zero for each zero-eigenspace and one for each one-eigenspace. The image of $P$ can be written as the direct sum of the non-zero eigenspaces of $P$, so the rank of $P$ is equal to the number of non-zero eigenspaces. Taking traces is invariant under unitary conjugation so
	\begin{equation*}
	\Tr(P)=\Tr(D)=\text{number of non-zero eigenspaces}=\text{Rank}(P).
	\end{equation*}
	Now to show that $K_0(\C)\cong \Z$ we will prove that two projections $P$ and $Q$ are Murray-Von-Neumann equivalent if and only if they have the same rank. First suppose that $P=V^*V$ and $Q=VV^*$. Taking the trace we see
	\begin{align*}
	\Tr(P)&=\Tr(V^*V)=\sum_{i,j=1}^nV^*_{ij}V_{ji}\\
	&=\sum_{i,j=1}^n\bar{V_{ji}}V_{ji}=\sum_{i,j=1}^nV_{ji}\bar{V_{ji}}\\
	&=\sum_{i,j=1}^nV_{ji}V^*_{ij}=\Tr(VV^*)=\Tr(Q).
	\end{align*}
	Since the trace of a projection is its rank, $P$ and $Q$ have the same rank.\\ 
	Suppose instead that $P$ and $Q$ are projections in $M_n(\C)$ with the same rank. Let
	\begin{equation*}
	B_P=\{v_1\dots,v_k,v_{k+1},\dots,v_n\}\qquad\text{ and }\qquad B_Q=\{w_1,\dots w_k,w_{k+1},\dots,w_n\}
	\end{equation*}
	be the bases of orthonormal eigenvectors for $P$ and $Q$ of $\C^n$, ordered so that $v_i$ and $w_i$ have corresponding eigenvalue 1 for $i\leq k$ and corresponding eigenvalue 0 for $i>k$. Define $\widetilde{V}$ to be the change of basis matrix from $B_P$ to $B_Q$. If $D$ is the diagonal matrix with $k$ ones and $n-k$ zeros, we define $V=D\widetilde{V}$ and claim that
	\begin{equation*}
	P=V^*V\qquad\text{ and }\qquad Q=VV^*.
	\end{equation*}
	The matrix $V$ is defined on the basis $B_P$ as
	\begin{equation*}
	Vv_i=\begin{cases}
	w_i&i\leq k\\
	0&i>k
	\end{cases}
	\end{equation*}
	and similarly $V^*$ is defined on $B_Q$ as
	\begin{equation*}
	V^*w_i=\begin{cases}
	v_i&i\leq k\\
	0&i>k\\
	\end{cases}.
	\end{equation*}
	So then writing $z\in\C^n$ as $z=\sum_{i=1}^nc_iv_i$ we compare the values of $Pz$ and $V^*Vz$ and deduce they are the same
	\begin{align*}
	Pz&=\sum_{i=1}^nc_iPv_i&&V^*Vz=\sum_{i=1}^nc_iV^*Vv_i\\
	&=\sum_{i=1}^kc_iv_i&&=\sum_{i=1}^kc_iV^*w_i\\
	& &&=\sum_{i=1}^lc_iv_i.\\
	\end{align*}
	In the same way, if we write $z$ in the basis $B_Q$ as $z=\sum_{i=1}^nc_iw_i$ then we compute
	\begin{align*}
	Qz&=\sum_{i=1}^nc_iQw_i&&VV^*z=\sum_{i=1}^nc_iVV^*w_i\\
	&=\sum_{i=1}^kc_iw_i&&=\sum_{i=1}^kc_iVv_i\\
	& &&=\sum_{i=1}^lc_iw_i\\
	\end{align*}
	whence we deduce $Q=VV^*$ and so $P$ and $Q$ are Murray-Von-Neumann equivalent. Note that $P$ and $Q$ being the same rank was essential in ensuring that we did not lose any information when forming $V$.\\
	So we can deduce from this that any rank $k$ projection in $M_n(\C)$ is equivalent to the diagonal matrix with $k$ ones and $n-k$ zeros
	\begin{equation*}
	D_k=\diag(1,\dots,1,0\dots,0).
	\end{equation*}
	What we have shown so far is that two projections in the same $M_n(\C)$ are equivalent if and only if they have the same rank. We wish to show this is true even for projections in matrix algebras of different sizes.\\
	Suppose that $P\sim Q$ and that $P\in M_m(\C)$, $Q\in M_n(\C)$ for $n\neq m$ with respective ranks $k$ and $l$. Without loss of generality, suppose $n>m$. Then $P$ is Murray-Von-Neumann equivalent to $D_k\in M_m(\C)$ and $Q$ to $D_l\in M_n(\C)$ for some $k$ and $l$, whence by transitivity $D_k\sim D_l$. We know that $D_k\sim D_k\oplus 0_{n-m}$, and $D_k\oplus 0_{n-m}$ is a rank $k$ projection in $M_n(\C)$. By transitivity we must have $D_l\sim D_k\oplus 0_{m-n}$ and so since we have two equivalent projections of the same dimension of ranks $k$ and $l$, we deduce that $l=k$, whence $P$ and $Q$ have the same rank. We conclude that two projections in $M_{\infty}(\C)$ are Murray-Von-Neumann equivalent if and only if they have the same rank. This means that the semigroup of equivalence classes is isomorphic to $\Z^{+}$ via the map $[P]\mapsto $Rank$(P)$. Consequently when we form its Grothendieck group we obtain
	\begin{equation*}
		K_0(\C)\cong \Z.
	\end{equation*}
\end{example}
We now would like to define $K_1(A)$. This requires a significantly smaller amount of work than defining $K_0(A)$. Similarly to $K_0(A)$, we wish to build $K_1(A)$ out of matrices over unitary elements of $A^+$ under an equivalence relation. 
\begin{definition}
	Let $U\in \UU_n(A^+)$ and $V\in \UU_m(A^+)$. We say that $U\sim_1 V$ if there exists some $k\geq \max{m,n}$ such that $U\oplus 1_{k-n}$ is homotopic to $V\oplus 1_{k-m}$ in $\UU_k(A^+)$. Then $\sim_1$ is an equivalence relation on $\UU_{\infty}(A^+)$ and we define
	\begin{equation*}
		K_1(A)\coloneqq \UU_{\infty}(A^+)/\sim_1.
	\end{equation*}
\end{definition}
\begin{remark}\label{Eh}
	It can be shown that $K_1(A)$ is an abelian group with addition $[u]+[v]=[u\oplus v]$ and identity element $[1]$.
	In \cite[Proposition~8.1.6]{BlueBook} it is shown that if $A$ is a unital $C^*$-algebra then $\UU_{\infty}(A^+)/\sim_1\cong \UU_{\infty}(A)/\sim_1$.
\end{remark}
\begin{example}\label{K1 of C}
	We wish to show that $K_0(\C)=0$. By Remark~\ref{Eh} we need only consider unitaries in $M_n(\C)$ rather than $M_n(\C^+)$. \cite[Lemma~2.1.3]{BlueBook} states that if $A$ is a unital $C^*$-algebra, then every unitary $u\in A$ with $\sigma(u)\neq \T$ is of the form $u=\exp(ib)$ for some self-adjoint $b\in A$. Since the spectrum of any matrix $M\in \MM_n(\C)$ is finite, we may write $U=\exp(iB)$ for some self adjoint $B\in M_n(\C)$ for all $U\in\UU(M_n(\C))=\UU_n(\C)$. Hence the map $t\mapsto \exp(itB)$ is a homotopy joining $U$ to the identity matrix. Thus, if $U\in \UU_n(\C)$ and $V\in\UU_m(\C)$ then $V\sim_1 1_m\sim_1 1_n\sim_1 U$, and so $A/\sim_1$ is the trivial group consisting of one element. Thus we have $K_1(\C)=0$.
\end{example}
The following Lemma is a restatement of \cite[Theorem~6.3.2]{BlueBook} for direct sums. Theorem~6.2.3 is stated in terms of inductive limits in \cite{BlueBook}, but since we only need it for direct sums we will restate it to save building machinery on inductive limits that we will not need.
\begin{lemma}\label{lemma: K is continuous}
	Let $\{A_n\}$ be a countable collection of $C^*$-algebras. Then we have
	\begin{equation*}
	K_i\Big(\bigoplus_{n=1}^{\infty}A_n\Big)=\oplus_{n=1}^{\infty}K_i(A_n).
	\end{equation*}
\end{lemma}

%% file: chapter6.tex
\chapter{KK-theory}
\label{KK-theory}

\begin{definition}
	Let $(A,\alpha)$ and $(B,\beta)$ be separable, graded $C^*$-algebras. A Kasparov $A$--$B$ module $(X,\phi,F,\alpha_X)$ is a quadruple consisting of a countably generated, graded $A$--$B$ correspondence $X$ with left action $\phi$ by $A$ that is graded by $\alpha_X$, and an odd adjointable operator $F\in\LL(X)$ such that the three operators
	\begin{equation}\label{KK}
	(F-F^*)\phi(a),\qquad (F^2-1)\phi(a),\qquad [F,\phi(a)]^{\gr}
	\end{equation}
	are in $\KK(X)$ for every $a\in A$. We say that a Kasparov module is degenerate if each operator in Equation \eqref{KK} is the zero operator. We say that a Kasparov module $(X,\phi,F,\alpha_X)$ is non-degenerate if $\phi$ is non-degenerate in the sense of Remark~\ref{Sam}.
\end{definition}
\begin{remark}
	In \cite{Kasnotes} the notation $(A,X_B,F)$ is used to denote the Kasparov $A$--$B$ module $(X,\phi,F,\alpha_X)$, with the left action and grading operators implicit. We have chosen our notation to match that of \cite{ThePaper} since we will be focusing heavily on the grading operators and left actions in what is to follow. The operator $F$ we often call the \textit{Fredholm} operator for the Kasparov module $(X,\phi,F,\alpha_X)$ for historical reasons.
\end{remark}
\begin{example}\label{Trivial Kas Module}
	If $\phi:A\to B$ is a graded homomorphism then $_{\phi}B_B$ is a graded $A$--$B$ correspondence graded by $\beta$. We know from Example~\ref{Fang} that $\phi(a)\in \KK(B_B)$ for all $a\in A$, so for any odd operator $F$ the conditions of Equation \eqref{KK} are automatically satisfied since $\KK(B_B)$ is an ideal. In particular, $(B_B,\phi,0,\beta)$ is a Kasparov module.
\end{example}
\begin{example}\label{Deg Eg}
	If $(X,\phi,F,\alpha_X)$ and $(Y,\psi,G,\alpha_Y)$ are Kasparov modules then their direct sum $(X\oplus Y,\phi\oplus \psi,F\oplus G,\alpha_X\oplus\alpha_Y)$ is also a Kasparov module. In particular, $(X\oplus X,\phi\oplus(\phi\circ\alpha_A),\begin{psmallmatrix}0&1\\1&0\\\end{psmallmatrix},\alpha_X\oplus-\alpha_X)$ is a degenerate Kasparov module. We see that $\begin{psmallmatrix}0&1\\1&0\\\end{psmallmatrix}$ is odd since
	\begin{equation*}
	\begin{pmatrix}\alpha_x&0\\0&-\alpha_x\\\end{pmatrix}\begin{pmatrix}0&1\\1&0\\\end{pmatrix}\begin{pmatrix}\alpha_x&0\\0&-\alpha_x\\\end{pmatrix}=\begin{pmatrix}
	0&-\alpha^2_x\\
	-\alpha^2_x&0\\
	\end{pmatrix}
	=-\begin{pmatrix}
	0&1\\
	1&0\\	
	\end{pmatrix}.
	\end{equation*}
	Then by Equation \eqref{Def} we have
	\begin{align*}
	\Bigg[\begin{pmatrix}0&1\\1&0\\\end{pmatrix},\phi(a)\oplus \phi(\alpha_A(a))\Bigg]^{\gr}&=\begin{pmatrix}0&1\\1&0\\\end{pmatrix}\begin{pmatrix}\phi(a)&0\\0&\phi(\alpha_A(a))\\\end{pmatrix}-\begin{pmatrix}\phi(\alpha_A(a))&0\\0&\phi(\alpha^2_A(a))\\\end{pmatrix}\begin{pmatrix}0&1\\1&0\\\end{pmatrix}\\
	&=\begin{pmatrix}
	0&\phi(\alpha_A(a))\\
	\phi(a)&0\\
	\end{pmatrix}-
	\begin{pmatrix}
	0&\phi(\alpha_A(a))\\
	\phi(a)&0\\
	\end{pmatrix}\\
	&=0.
	\end{align*}
	Since $\begin{psmallmatrix}
	0&1\\1&0\\
	\end{psmallmatrix}$
	is self adjoint and square one we see that we have a degenerate Kasparov module.
\end{example}
\begin{definition}
	We say that two Kasparov modules $(X,\phi,F,\alpha_X)$ and $(Y,\psi,G,\alpha_Y)$ are unitarily equivalent if there exists an even (in other words graded) unitary map $U:X\to Y$ such that $UF=GU$ and $U\phi(a)=\psi(a)U$ for all $a\in A$. We write $(X,\phi,F,\alpha_X)\cong (Y,\psi,G,\alpha_Y)$.
\end{definition}
\begin{example}
	Given a non-degenerate Kasparov $A$--$B$ module $(X,\phi,F,\alpha_X)$ we have the unitary equivalence $(X,\phi,F,\alpha_X)\cong (X\otimes_{\iota}B,\phi\otimes 1,F\otimes 1,\alpha_X\otimes\beta)$. We have seen in Example~\ref{left and right} that $U:X\otimes_{\iota}\to X$, $U(x\otimes b)=xb$ defines a unitary equivalence of $A$--$B$ correspondences, so we need only check that $U$ intertwines $F$ and $F\otimes 1$, and that $U$ is even. We compute
	\begin{equation*}
	U\big((F\otimes 1)(x\otimes b)\big)=U(Fx\otimes b)=F(xb)=(Fx)b=FU(x\otimes b)
	\end{equation*}
	and
	\begin{equation*}
	U\Big((\alpha_X\otimes\beta)(x\otimes b)\Big)=U(\alpha_X(x)\otimes\beta(b))=\alpha_X(x)\beta(b)=\alpha_X(xb)=\alpha_X(U(x\otimes b)),
	\end{equation*}
	so $U$ indeed defines a unitary equivalence of Kasparov modules. We saw in Example~\ref{left and right} that the map $V:A\otimes_{\phi}X\to X$ given by $V(a\otimes x)=\phi(a)x$ is a unitary equivalence of correspondences when $\phi$ is non-degenerate. It is not the case however that $(A\otimes X,L\otimes 1,1\otimes F,\alpha_A\otimes\alpha_X)$ is unitarily equivalent to $(X,\phi,F,\alpha_X)$ under $V$ from Example~\ref{left and right} where $L(a)b=ab$. This is because the operator $1\otimes F$ cannot even be properly defined on $A\otimes X$ unless it commutes with all of $\phi(A)$. If $U:A\otimes X\to X$ is the unitary observing the unitary equivalence of Example~\ref{left and right}, then we can of course say that $(A\otimes X,L\otimes 1,U^*FU,\alpha_A\otimes\alpha_X)$ is unitarily equivalent to $(X,\phi,F,\alpha_X)$.
\end{example}
\begin{definition}
	Let $(A,\alpha)$ and $(B,\beta)$ be graded $C^*$-algebras. A homotopy between two Kasparov $A$--$B$ modules $(X_0,\phi_0,F_0,\alpha_{X_0})$ and $(X_1,\phi_1,F_1,\alpha_{X_1})$ is a Kasparov $A$-$C([0,1];B)$ module $(X,\phi,F,\alpha_X)$ such that if $\ep_t:C([0,1];B)\to B$ is the evaluation map then
	\begin{equation*}
	(X\otimes_{\ep_t}B_B,\phi\otimes 1,F\otimes 1,\alpha_X\otimes\beta)
	\end{equation*}
	is unitarily equivalent to $(X_t,\phi_t,F_t,\alpha_{X_t})$ for each $t=0,1$.
\end{definition}
\begin{remark}
	It is shown in \cite[Proposition~3.4.7]{JonosThesis} that homotopy is an equivalence relation.
\end{remark}
\begin{lemma}\label{Deg Lem}
	Every Degenerate Kasparov $A$--$B$ module $(X_B,\phi,F,\alpha_X)$ is homotopic to the zero module.
\end{lemma}
\begin{proof}
	The following homotopy is an unpacking of \cite[17.2.3]{Blackadar}. We claim that
	\begin{equation}
	\left(C_0([0,1);X),\widetilde{\phi},\widetilde{F},\widetilde{\alpha}_X\right)\label{homotopy}
	\end{equation}
	is the desired homotopy, where
	\begin{equation*}
	C_0([0,1);X)=\overline{C_c([0,1);X)}=\{f\in C([0,1];X):f(1)=0\}
	\end{equation*}
	and the maps $\widetilde{\phi}$, $\widetilde{F}$ and $\widetilde{\alpha}_X$ are defined by
	\begin{align*}
	(\widetilde{\phi}(a))(f)(t)&=\phi(a)(f(t))\\
	(\widetilde{F})(f)(t)&=F(f(t))\\
	(\widetilde{\alpha}_X)(f)(t)&=\alpha_X(f(t)).\\
	\end{align*}
	To show this we need to show that the $A$--$B$ module
	\begin{equation}
	\left(C_0([0,1);X)\Otimes _{\ep_t}B,\widetilde{\phi}\Otimes 1, \widetilde{F}\Otimes 1,\widetilde{\alpha}_X\Otimes \alpha_B\right)\label{the thing}		
	\end{equation}
	is a Kasparov module and is unitarily equivalent to $[X,\phi,F,\alpha_X]$ at $t=0$ and unitarily equivalent to the zero module at $t=1$. We are viewing $C_0([0,1);X)$ as an $A$--$C([0,1];B)$ correspondence with $C([0,1];B)$-valued inner-product
	\begin{equation*}
	\IP{f,g}_{C([0,1];B)}(t)=\IP{f(t),g(t)}_B
	\end{equation*}
	where the right-hand inner-products in both scenarios are taking place in $X$. Recall that $B_B$ has right multiplication $b\cdot c=bc$ and $B$-valued inner-product
	\begin{equation*}
	\IP{b,c}_B=b^*c.
	\end{equation*} 
	Before showing unitary equivalence we must show that our proposed homotopy is indeed a Kasparov module. That is we need to show that
	\begin{equation}
	1-(\widetilde{F})^2,\quad[\widetilde{F},\widetilde{\phi}(a)]^{\gr}\;\;\text{      and      }\;\; (\widetilde{F})-(\widetilde{F})^*\label{a}
	\end{equation}
	are all compact in $\LL(C_0([0,1);X))$ and that $\widetilde{F}$ is an odd operator with respect to the grading $\widetilde{\alpha}_X$. We start with oddness of $\widetilde{F}$
	\begin{align*}
	\widetilde{\alpha}_X \widetilde{F} \widetilde{\alpha}_X(f)(t)&=\alpha_XF\alpha_X(f(t))\\
	&=-F(f(t))\qquad\text{since $F$ is odd in $\LL(X)$}\\
	&=-\widetilde{F}(f)(t)\\
	\end{align*}
	and so $\widetilde{F}$ is an odd operator. Now we check the first operator in \eqref{a} is compact
	\begin{align*}
	(1-(\widetilde{F})^2)(f)(t)&=f(t)-F^2(f(t))\\
	&=(1-F^2)(f(t))\\
	&=0
	\end{align*}
	where we have used that $1-F^2=0$ on $X$, so $1-(\widetilde{F})^2= 0$ which is compact. Now for the second operator in \eqref{a} we use equation \eqref{Neat} and that $\widetilde{F}$ is odd to simplify the graded commutator
	\begin{align*}
	&[\widetilde{F},\widetilde{\phi}(a)]^{\gr}(f)(t)\\
	&=(\widetilde{F})(\widetilde{\phi}(a))(f)(t)-(\widetilde{\alpha}_X)(\widetilde{\phi}(a))(\widetilde{\alpha}_X)(\widetilde{F})(f)(t)\\
	&=F\phi(a)(f(t))-\alpha_X\phi(a)\alpha_X F(f(t))\\
	&=[F,\phi(a)]^{\gr}(f(t))\\
	&=0.\\
	\end{align*}
	Lastly we see, since
	\begin{align*}
	\IP{\widetilde{F}f,g}_{C_0([0,1);B)}(t)&=\IP{F(f(t)),g(t)}_B\\
	&=\IP{f(t),F^*(g(t))}_B\\
	&=\IP{f,\widetilde{F^*}g}(t),\\
	\end{align*}
	that $(\widetilde{F})^*=(\widetilde{F^*})$ which allows us to compute
	\begin{align*}
	\widetilde{F}(f)(t)-\widetilde{F}^*(f)(t)&=F(f(t))-F^*(f(t))=0.\\
	\end{align*}
	So \eqref{homotopy} is indeed a Kasparov module.\\
	The next step in our unpacking is to show that \eqref{the thing} is unitarily equivalent to the zero module at $t=1$. Recall that
	\begin{equation*}
	C_0([0,1);X)\Otimes _{\ep_t}B
	\end{equation*}
	is the completion of the quotient of the algebraic tensor product by the ideal $I$ generated by element $f\otimes b$ such that
	\begin{equation*}
	\IP{f\otimes b,f\otimes b}_B=0
	\end{equation*}
	where the $B$-valued inner-product on simple tensors in $C_0([0,1);X)\odot B$ is defined by
	\begin{equation*}
	\IP{f\otimes b,g\otimes c}_B=\IP{b,\ep_t(\IP{f,g}_{C_0([0,1);B)})c}_B=b^*\IP{f(t),g(t)}_Bc.
	\end{equation*}
	Since $f(1)=0$ for all $f\in C_0([0,1);X)$ the inner-product $\IP{f\otimes b,f\otimes b}_B$ on the algebraic tensor product is zero everywhere for $t=1$. The ideal $I$ is then all of the algebraic tensor product. Thus when we quotient by $I$ we get the trivial algebra and so
	\begin{equation*}
	C_0([0,1);X)\Otimes _{\ep_1}B=0.
	\end{equation*}
	Unitary equivalence in this situation is trivially observed by the identity map.\\
	Now we need to show that for $t=0$ there is a degree zero unitary
	\begin{equation*}
	U:C_0([0,1);X)\Otimes _{\ep_0}B\to X
	\end{equation*}
	such that $FU=U\widetilde{F}\Otimes 1$ and $\phi(a)U=U\widetilde{\phi}(a)\Otimes 1$. 
	We compute
	\begin{align*}
	\IP{f(0)b,g(0)c}_B&=b^*\IP{f(0),g(0)}_Bc=\IP{b,\ep_0(\IP{f,g})c}_B=\IP{f\otimes b,g\otimes c}_{C_0([0,1);X)\Otimes_{\ep_0}B},\\
	\end{align*}
	so by Lemma~\ref{lemma: Well defined unitaries on tensor product} there is a well defined map $U:C_0([0,1);X)\Otimes_{\ep_0}B\to X$ such that $U(f\otimes b)=f(0)b$. Cohen factorisation states that any $x\in X$ is of the form $x=y\IP{y,y}$ for some $y\in X$. Thus if we define $f:[0,1]\to X$ by $f(t)=(1-t)y$, then $f\in C_0([0,1);X)$ and we have
	\begin{equation*}
		U(f\otimes \IP{y,y})=f(0)\IP{y,y}=y\IP{y,y}=x.
	\end{equation*}
	Hence $U$ is surjective and Lemma~\ref{lemma: Well defined unitaries on tensor product} shows that $U$ is unitary.
	We also need $U$ to be of degree zero. Recall that the grading operator on the tensor product is given by
	\begin{equation*}
	\widetilde{\alpha}_X\otimes\alpha_B(f(\cdot)\otimes b)=\alpha_X(f(\cdot))\otimes\alpha_B(b)
	\end{equation*}
	For $U$ to be of degree one we want $U$ to intertwine the two gradings $\widetilde{\alpha}_X\otimes\alpha_B$ and $\alpha_X$ so that 
	\begin{equation*}
	U \left(\widetilde{\alpha}_X\otimes\alpha_B \right) =\alpha_XU.
	\end{equation*}
	We compute
	\begin{align*}
	U \widetilde{\alpha}_X\otimes\alpha_B (f(\cdot)\otimes b)&= U(\alpha_X(f(\cdot))\otimes\alpha_B(b))\\
	&=\alpha_X(f(0))\alpha_B(b)\\
	&=\alpha_X(f(0)b)\\
	&=\alpha_X U(f\otimes b)\\
	\end{align*}
	and so $U$ is of degree one.
	Now the last thing we have to check is that $U$ intertwines both $\widetilde{F}\Otimes 1$ and $F$, and $\widetilde{\phi}(a)\Otimes 1$ and $\phi(a)$. We compute
	\begin{align*}
	U\widetilde{F}\Otimes 1(f(\cdot)\otimes b)&=U(F(f(\cdot))\otimes b)\\
	&=F(f(0))\cdot b\\
	&=F(f(0)b)\\
	&=FU(f\otimes b)\\
	\end{align*}
	so $U$ intertwines the first two maps. Lastly
	\begin{align*}
	U\widetilde{\phi}(a)(f(\cdot)\otimes b)&=U(\phi(a)(f(\cdot))\otimes b)\\
	&=(\phi(a)(f(0))\cdot b\\
	&=\phi(a)(f(0)b)\\
	&=\phi(a)U(f\otimes b)\\
	\end{align*}
	and so $U$ intertwines the last two maps. So $C_0([0,1);X)\Otimes _{\ep_0}B$ and $X$ are both unitarily equivalent. Hence \eqref{homotopy} is a homotopy of $X$ to the zero module.
\end{proof}
\begin{remark}
	The homotopy we have used in this proof is itself a degenerate Kasparov $A$--$C([0,1];B)$ module. One part of this proof that relied on the original module $(X,\phi,F,\alpha_X)$ being degenerate was that $1-(\widetilde{F})^2$ is compact. If $1-F^2$ were not zero, then $1-(\widetilde{F})^2=1\otimes(1-F^2)$ will not usually be compact. This is mostly due to the fact that the constant function 1 does not vanish at zero.
\end{remark}
\begin{remark}
	Of particular interest to us will be a special case of homotopy called an operator homotopy. If $_{\phi}X$ is a graded $A$--$B$ correspondence and $t\mapsto F_t$ is a norm continuous homotopy in $\LL(X)$ such that $(X,\phi,F_t,\alpha_X)$ is a Kasparov module for every $t\in[0,1]$ then $(X,\phi,F_0,\alpha_X)$ and $(X,\phi,F_1,\alpha_X)$ are called \textit{operator homotopic}. An operator homotopy is a special case of a homotopy as follows. The set $C([0,1];X)$ is a graded $A$-$C([0,1];B)$ correspondence with point-wise right action
	\begin{equation*}
	(f\cdot g)(t)=f(t)\cdot g(t),\qquad f\in C([0,1];X),g\in C([0,1];B),
	\end{equation*}
	point-wise left action $\overline{\phi}$
	\begin{equation*}
	(\overline{\phi}(a)f)(t)=\phi(a)(f(t))
	\end{equation*}
	and point-wise grading $\overline{\alpha_X}$
	\begin{equation*}
	(\overline{\alpha_X}f)(t)=\alpha_X(f(t)).
	\end{equation*}
	Defining $\overline{F}\in\LL(C([0,1];X))$ by $(\overline{F}f)(t)=F_t(f(t))$ we obtain a Kasparov $A$-$C([0,1];B)$ module
	\begin{equation*}
	(C([0,1];X),\overline{\phi},\overline{F},\overline{\alpha_X})
	\end{equation*}
	which is a homotopy joining $(X,\phi,F_0,\alpha_X)$ and $(X,\phi,F_1,\alpha_X)$.
	If $(X,\phi,F,\alpha)$ is a Kasparov module and $K$ is an odd compact operator, then the straight line $F+tK$ from $F$ to $F+K$ is an operator homotopy to the Kasparov module $(X,\phi,F+K,\alpha)$. In particular, if $(X,\phi,F,\alpha)$ and $(X,\phi,G,\alpha)$ are two Kasparov modules with $F-G$ compact, then they are homotopic. The exact same construction shows that there is a norm continuous path of left actions $\phi_t$ for which $(X,\phi_t,F,\alpha_X)$ is a Kasparov module for all $t$, then $(X,\phi_0,F,\alpha_X)$ and $(X,\phi_t,F,\alpha_X)$ are homotopic. 
\end{remark}

\begin{definition}
	Let $A$ and $B$ be $C^*$-algebras. We define $KK(A,B)$ to be the set of homotopy equivalence classes of Kasparov $A$-$B$ modules.
\end{definition}
\begin{lemma}\label{essential}
	Suppose that $(_AX_B,\phi,F,\alpha)$ is a Kasparov module, that $X$ decomposes as a graded direct sum $X=Y^{\perp}\oplus Y$ and that $\phi(A)X\subseteq Y$. Let $P$ be the orthogonal projection onto $Y$. Then $\phi$ decomposes as $\phi_Y\oplus 0$ on this direct sum and 
	\begin{equation}\label{Kasp}
	(Y,\phi_Y,PFP,\alpha|_{Y}) 
	\end{equation}
	is a Kasparov module representing the same class as $(X,\phi,F,\alpha)$.
\end{lemma}
\begin{proof}
	We first wish to show that $\phi$ decomposes as $\phi_Y\oplus 0$ on $Y\oplus Y^{\perp}$. Fix $y\in Y$, $\hat{y}\in Y^{\perp}$ and $a\in A$. Since $\phi(A)X\subseteq Y$, we have $\phi(a^*)y\in Y$. Thus
	\begin{equation*}
	0=\IP{\phi(a^*)y,\hat{y}}=\IP{y,\phi(a)\hat{y}},
	\end{equation*}
	and so $\phi(a)\hat{y}\in Y^{\perp}$, which shows $\phi(A)Y^{\perp}\subseteq Y^{\perp}$. Since $\phi(A)Y^{\perp}\subseteq \phi(A)X\subseteq Y$, we must have $\phi(A)Y^{\perp}\in Y\cap Y^{\perp}=\{0\}$, giving $\phi(a)\hat{y}=0$ for all $a\in A$ and $\hat{y}\in Y^{\perp}$. Thus
	\begin{equation*}
	\phi=\phi_Y\oplus 0
	\end{equation*}
	on $Y\oplus Y^{\perp}$ where $\phi_Y=\phi|_Y$.\\
	To see that $(Y,\phi_Y,PFP,\alpha\big|_{Y})$  is a Kasparov module we need to check that the operator $PFP$ satisfies the three Kasparov module conditions. That $PFP$ is odd follows from $F$ and $P$ being respectively odd and even operators. The operators $\phi(a)$ and $P$ commute since they both act as zero on $Y^{\perp}$ and $P$ is the identity on $Y$, and since $(1-P)$ and $\phi(a)$ have orthogonal ranges, $(1-P)\phi(a)=0$. Note that the projection $P$ is the identity on $Y$. Thus we compute
	\begin{align*}
	\left(PFP-(PFP)^*\right)\phi(a)&=P\left(F-F^*\right)P\phi(a)\\
	&=P\left(F-F^*\right)\phi(a)P,\\
	\end{align*}
	\begin{align*}
	[PFP,\phi(a)]^{\gr}&=PFP\phi(a)-\phi(\alpha(a))PFP\\
	&=P\left(F\phi(a)-\phi(\alpha(a))F\right)P\\
	&=P[F,\phi(a)]^{\gr}P,\\
	\end{align*}
	\begin{align*}
	\left((PFP)^2-1\right)\phi(a)&=\left(PFPFP-P\right)\phi(a)\\
	&=\left(PF^2P-P\right)\phi(a)+PF(1-P)FP\phi(a)\\
	&=P\left(F^2-1\right)P\phi(a)+PF(1-P)F\phi(a)P\\
	&=P\left(F^2-1\right)\phi(a)P+PF(1-P)\phi(\alpha(a))FP+PF(1-P)[F,\phi(a)]^{\gr}P\\
	&=P\left(F^2-1\right)\phi(a)P+PF(1-P)[F,\phi(a)]^{\gr}P
	\end{align*}
	all of which are compact because $(X,\phi,F,\alpha)$ is a Kasparov module. Thus $(Y,\phi_Y,PFP,\alpha|_Y)$ is a Kasparov module.\\
	Lastly we wish to show that \eqref{Kasp} represents the same class as $(X,\phi,F,\alpha)$. We have
	\begin{equation*}
	[X,\phi,F,\alpha]=[Y,\phi_Y,PFP,\alpha|_Y]\oplus[Y^{\perp},0,(1-P)F(1-P),\alpha|_{Y^{\perp}}].
	\end{equation*}
	Since the left action on the right-hand module is zero, $(Y^{\perp},0,(1-P)F(1-P),\alpha|_{Y^{\perp}})$ is a degenerate Kasparov module, and so represents the zero class in $KK(A,B)$. So the difference between \eqref{Kasp} and $[X,\phi,F,\alpha]$ is zero, whence they represent that same class.
\end{proof}
\begin{remark}
	In \cite{ThePaper} a non-degenerate left action is called \textit{essential}. We will borrow some of their terminology and call the sub-module $\overline{\phi(A)X}\subset X$ the \textit{essential subspace} of $X$.
\end{remark}

\begin{theorem}
	Let $A$ and $B$ be $C^*$-algebras. The set $KK(A,B)$ is an abelian group under the operation of direct sums.
\end{theorem}
\begin{proof}
	This proof is an expansion of \cite[17.3.3]{Blackadar} and \cite[Theorem 3.3.10]{Kasnotes}. Clearly the class of the zero module acts as an identity on $KK(A,B)$, and associativity is left as an exercise, so we need only show the existence of inverses. We claim that given a Kasparov class with representative $(X,\phi,F,\alpha_Y)$, its inverse is given by the class of the module\\ $(X,\phi\circ\alpha_A,-F,-\alpha_X)$. That is, we wish to show that
	\begin{equation}\label{Greg}
	(X\oplus X,\phi\oplus(\phi\circ\alpha_A),F\oplus-F,\alpha_X\oplus-\alpha_X)
	\end{equation}
	is homotopic to the zero module.
	According to Lemma~\ref{Deg Lem} we need only find a homotopy from $\eqref{Greg}$ a degenerate module. Recalling Example~\ref{Deg Eg}, the module
	\begin{equation}\label{Bird}
	(X\oplus X,\phi\oplus(\phi\circ\alpha_A),\begin{psmallmatrix}
	0&1\\
	1&0\\
	\end{psmallmatrix},\alpha_X\oplus-\alpha_X)
	\end{equation}
	is degenerate, so we will show that
	\begin{equation}\label{nathan}
	(X\oplus X,\phi\oplus(\phi\circ\alpha_A),\begin{psmallmatrix}
	F\cos t&\sin t\\
	\sin t&-F\cos t\\
	\end{psmallmatrix},\alpha_X\oplus-\alpha_X)
	\end{equation}
	is an operator homotopy from \eqref{Greg} to \eqref{Bird}. Clearly $\begin{psmallmatrix}
	F\cos t&\sin t\\
	\sin t&-F\cos t\\
	\end{psmallmatrix}$ is a norm continuous map, so we need only show that \eqref{nathan} is a Kasparov module for each $t$. The calculation
	\begin{align*}
	\begin{pmatrix}
	\alpha_X&0\\
	0&-\alpha_X\\
	\end{pmatrix}
	\begin{pmatrix}
	F\cos t&\sin t\\
	\sin t&-F\cos t\\
	\end{pmatrix}
	\begin{pmatrix}
	\alpha_x&0\\
	0&-\alpha_x\\
	\end{pmatrix}&=
	\begin{pmatrix}
	\cos t \alpha_X F \alpha_X&-\sin t\alpha_X^2\\
	-\sin t\alpha_X^2&-\cos t \alpha_X F\alpha_X\\
	\end{pmatrix}\\
	&-\begin{pmatrix}
	F\cos t&\sin t\\
	\sin t&-F\cos t\\
	\end{pmatrix}
	\end{align*}
	shows that $\begin{psmallmatrix}
	F\cos t&\sin t\\
	\sin t&-F\cos t\\
	\end{psmallmatrix}$ is odd. We also have
	\begin{align*}
	&\left[\begin{pmatrix}
	F\cos t&\sin t\\
	\sin t&-F\cos t\\
	\end{pmatrix}-\begin{pmatrix}
	F^*\cos t&\sin t\\
	\sin t&-F^*\cos t\\
	\end{pmatrix}\right]
	\begin{pmatrix}
	\phi(a)&0\\
	0&\phi(\alpha_A(a))\\
	\end{pmatrix}\\
	&=\begin{pmatrix}
	\cos t(F-F^*)\phi(a)&0\\
	0&-\cos t(F-F^*)\phi(\alpha_A(a))\\
	\end{pmatrix}
	\end{align*}
	which is compact because $(F-F^*)\phi(a)$ is compact. We see that
	\begin{align*}
	&\left[\begin{pmatrix}
	F\cos t&\sin t\\
	\sin t&-F\cos t\\
	\end{pmatrix}^2-
	\begin{pmatrix}
	1&0\\
	0&1\\
	\end{pmatrix}\right]\begin{pmatrix}
	\phi(a)&0\\
	0&\phi(\alpha_A(a))\\
	\end{pmatrix}&\\
	&=\begin{pmatrix}
	\cos^2tF^2+\sin^2t-1&0\\
	0&\cos^2tF^2+\sin^2t-1\\
	\end{pmatrix}
	\begin{pmatrix}
	\phi(a)&0\\
	0&\phi(\alpha_A(a))\\
	\end{pmatrix}\\
	&=\begin{pmatrix}
	\cos^2t(F^2-1)\phi(a)&0\\
	0&\cos^2t(F^2-1)\phi(\alpha_A(a))\\
	\end{pmatrix}
	\end{align*}
	is compact because $(F^2-1)\phi(a)$ is compact, and lastly we see
	\begin{align*}
	&\left[\begin{pmatrix}
	F\cos t&\sin t\\
	\sin t&-F\cos t\\
	\end{pmatrix},\begin{pmatrix}
	\phi(a)&0\\
	0&\phi(\alpha_A(a))\\
	\end{pmatrix}\right]^{\gr}\\
	&=\begin{pmatrix}
	\cos t F\phi(a)&\sin t\phi(\alpha_A(a))\\
	\sin t\phi(a)&-\cos t F\phi(\alpha_A(a))\\
	\end{pmatrix}-\begin{pmatrix}
	\cos t \phi(\alpha_A(a))F&\sin t\phi(\alpha_A(a))\\
	\sin t\phi(a)&-\cos t \phi(a)F\\
	\end{pmatrix}\\
	&=\begin{pmatrix}
	\cos t[F,\phi(a)]^{\gr}&0\\
	0&\cos t[F,\phi(a)]^{\gr}\\
	\end{pmatrix}
	\end{align*}
	is compact since $[F,\phi(a)]^{\gr}$ is compact. Thus \eqref{nathan} defines an operator homotopy between \eqref{Greg} and \eqref{Bird}, and so \eqref{Greg} is a degenerate Kasparov module whence we deduce that
	\begin{equation*}
	-[X,\phi,F,\alpha_X]=[X,\phi\circ\alpha,-F,\alpha_X].\qedhere
	\end{equation*}
\end{proof}
\begin{definition}
	Let $A$ and $B$ be graded $C^*$-algebras. We define $KK_0(A,B)=KK(A,B)$ and $KK_n(A,B)=KK(A,B\Otimes\Cliff_n)$. We write $KK_*(A,B)$ for the pair $(KK_0(A,B),KK_1(A,B))$. 
\end{definition}
\begin{remark}
	If $_{\phi}X$ is an ungraded $A$--$B$ correspondence and $F\in\LL(X)$ is an operator satisfying
	\begin{equation}\label{NEED}
		(F-F^*)\phi(a)\in\KK(X),\quad (1-F^2)\phi(a)\in\LL(X),\qquad\text{ and }\qquad [F,\phi(a)]\in\KK(X)
	\end{equation}
	for all $a\in A$, then in \cite{Kasnotes} the triple $(X,\phi,F)$ is called an \emph{odd} Kasparov module. The module $X$ always carries the trivial grading, but the only odd operator under the trivial grading is zero, so we cannot make $(X,\phi,F)$ into a Kasparov module under the trivial grading. However, the information $(X,\phi,F)$ does define a Kasparov $A\Otimes\Cliff_1$--$B$ module. Suppose that $A$ and $B$ are trivially graded. Then graded tensor product $A\Otimes\Cliff_1$ is equal to the ungraded tensor product which is independent of tensor norm since $\Cliff_1$ is nuclear. If $e$ denotes the generator of $\Cliff_1$ then $A\Otimes \Cliff_1$ is graded by the operator $\alpha_{A\Otimes\Cliff_1}(a\otimes 1+b\otimes e)=a\otimes 1-b\otimes e$. The direct sum $X\oplus X$ becomes an $A\Otimes\Cliff_1$--$B$ correspondence when equipped with diagonal right action of $B$ and left action
	\begin{equation*}
		\widetilde{\phi}\Big(a\otimes 1+b\otimes e\Big)=\begin{pmatrix}
			\phi(a)&\phi(b)\\
			\phi(b)&\phi(a)\\
		\end{pmatrix}
	\end{equation*}
	of $A\Otimes\Cliff_1$. We claim that $(X\oplus X,\widetilde{\phi},\begin{psmallmatrix}
		0&-iF\\
		iF&0\\
	\end{psmallmatrix},\begin{psmallmatrix}
		1&0\\
		0&-1\\
	\end{psmallmatrix})$ is a Kasparov $A\Otimes\Cliff_1$--$B$ module. The Fredholm operator $\begin{psmallmatrix}
	0&-iF\\
	iF&0\\
	\end{psmallmatrix}$ is odd because
	\begin{equation*}
		\begin{pmatrix}
		1&0\\
		0&-1\\
		\end{pmatrix}\begin{pmatrix}
		0&-iF\\
		iF&0\\
		\end{pmatrix}\begin{pmatrix}
		1&0\\
		0&-1\\
		\end{pmatrix}=-
		\begin{pmatrix}
		0&-iF\\
		iF&0\\
		\end{pmatrix}.
	\end{equation*}
	Using Equation \eqref{NEED} we see for all $a\otimes 1+b\otimes e\in A\Otimes\Cliff_1$ that both
	\begin{equation*}
		\Bigg(\begin{pmatrix}
		1&0\\
		0&1\\
		\end{pmatrix}-\begin{pmatrix}
		0&-iF\\
		iF&0\\
		\end{pmatrix}^2\Bigg)\widetilde{\phi}(a\otimes 1+b\otimes e)=
		\begin{pmatrix}
			(1-F^2)\phi(a)&(1-F^2)\phi(b)&\\
			(1-F^2)\phi(b)&(1-F^2)\phi(a)&\\
		\end{pmatrix}
	\end{equation*}
	and
	\begin{equation*}
		\Bigg(\begin{pmatrix}
		0&-iF\\
		iF&0\\
		\end{pmatrix}-\begin{pmatrix}
		0&-iF\\
		iF&0\\
		\end{pmatrix}^*\Bigg)\widetilde{\phi}(a\otimes 1+b\otimes e)=
		i\begin{pmatrix}
		-(F-F^*)\phi(b)&-(F-F^*)\phi(a)&\\
		(F-F^*)\phi(a)&(F-F^*)\phi(b)&\\
		\end{pmatrix}
	\end{equation*}
	are compact. We compute the graded commutator
	\begin{align*}
		\Big[\begin{pmatrix}
		0&-iF\\
		iF&0\\
		\end{pmatrix},\widetilde{\phi}(a\otimes 1+b\otimes e)\Big]&=\begin{pmatrix}
		0&-iF\\
		iF&0\\
		\end{pmatrix}\begin{pmatrix}
		\phi(a)&\phi(b)\\
		\phi(b)&\phi(a)\\
		\end{pmatrix}-\begin{pmatrix}
		\phi(a)&\phi(b)\\
		\phi(b)&\phi(a)\\
		\end{pmatrix}\begin{pmatrix}
		0&-iF\\
		iF&0\\
		\end{pmatrix}\\
		&=i\begin{pmatrix}
			-[F,\phi(b)]&-[F,\phi(a)]\\
			[F,\phi(a)]&[F,\phi(b)]\\
		\end{pmatrix}
	\end{align*}
	which is compact by Equation \eqref{NEED}. Hence $(X\oplus X,\widetilde{\phi},\begin{psmallmatrix}
	0&-iF\\
	iF&0\\
	\end{psmallmatrix},\begin{psmallmatrix}
	1&0\\
	0&-1\\
	\end{psmallmatrix})$ is a Kasparov $A\Otimes\Cliff_1$--$B$ module. Taking the homotopy class of this module we see that the information $(X,\phi,F)$ defines a class in $KK_1(A,B)$. It can be shown that all classes in $KK_1(A,B)$ arise this way, so whilst $KK_0(A,B)$ is generated by classes of Kasparov modules, we may think of $KK_1(A,B)$ as being generated by \emph{odd} Kasparov modules.
\end{remark}

\section{The Kasparov product}
We have seen that $KK(A,B)$ is an abelian group under direct sums of Kasparov modules. There is also a multiplication called the Kasparov product that turns the group $KK(A,A)$ into a ring. For our purposes, the Kasparov product is a map $\Otimes _B:KK(A,B)\times KK(B,C)\to KK(A,C)$ such that for all $X,Y\in KK(A,B)$ and $U,V\in KK(B,C)$ we have 
\begin{equation*}
	(X\oplus Y)\Otimes U=X\Otimes U \oplus Y\Otimes U\qquad\text{ and }\qquad X\Otimes(U\oplus V)=X\Otimes U\oplus X\Otimes V. 
\end{equation*}
If $A=B=C$, then this gives $KK(A,A)$ a ring structure with identity class $[A_A,\iota,0,\alpha_A]$ where $\iota:A\to\LL(A_A)$ is left multiplication. The true form of the Kasparov product is much more general than what we will be describing: the full detail can be found in \cite{Kasnotes} and \cite{Blackadar}. In general, the Kasparov product is of the form
\begin{equation*}
[X,\phi,F,\alpha_X]\Otimes _B[Y,\psi,G,\alpha_Y]=[X\otimes_{\psi}Y,\phi\otimes 1,H,\alpha_X\Otimes \alpha_Y],
\end{equation*}
and is usually quite difficult to compute, the main problem being what the operator $H$ should be. In the case that one of the modules is the class of a homomorphism, the product is quite easy to compute. Given a graded homomorphism $\phi:(A,\alpha_A)\to(B,\alpha_B)$, the quadruple $(B_B,\phi,0,\alpha_B)$ is a Kasparov $A$--$B$ module. If $(X,\psi,F,\alpha_X)$ is a Kasparov $B$--$C$ module, then we may form a Kasparov $A$-$C$ module $(X,\psi\circ\phi,F,\alpha_X)$ whose class in $KK(A,C)$ we denote by $\phi^*[X,\psi,F,\alpha_X]$. We then have
\begin{equation}\label{triv right}
[B_B,\phi,0,\alpha_B]\Otimes _B[X,\psi,F,\alpha_X]=\phi^*[X,\psi\,F,\alpha_X].
\end{equation}
Similarly, if $(Y,\psi,G,\alpha_Y)$ is a Kasparov $C$-$A$-module the we may form the Kasparov module $(Y\Otimes _{\phi} B_B,\psi\Otimes 1,G\Otimes 1,\alpha_Y\Otimes \alpha_B)$ whose class in $KK(C,B)$ we denote $\phi_*[Y,\psi,G,\alpha_Y]$, and then
\begin{equation}\label{triv left}
[Y,\psi,G,\alpha_Y]\Otimes _A[B_B,\phi,0,\alpha_B]=\phi_*[Y,\psi,G,\alpha_Y].
\end{equation}
One more form of the Kasparov product that we will need is when neither module is trivial, but the left actions on both modules are by compacts. Let $(X,\phi,0,\alpha_X)$ be a Kasparov $A$--$B$ module and let $(Y,\psi,0,\alpha_Y)$ be a Kasparov $B$--$C$ module. Since the Fredholm operator is $0$ and each quadruple defines a Kasparov module, the condition $(1-0^2)\phi(a)\in\KK(X)$ for all $a\in A$ demands that $\phi(a)\in\KK(X)$ for all $a\in A$. Similarly, $\psi(b)\in\KK(Y)$ for all $b\in B$. The form of the Kasparov product for these two modules is then
\begin{equation}\label{Fredholm zero Kasp product}
	[X,\phi,0,\alpha_X]\Otimes [Y,\psi,0,\alpha_Y]=[X\Otimes_{\psi}Y,\phi\Otimes 1,0,\alpha_X\Otimes \alpha_Y].
\end{equation}
A fact that we will need is that recalling the definition of the additive inverse, we have
\begin{equation}\label{minus id}
[A_A,\alpha_A,0,-\alpha_A]=-[\text{id}_A]
\end{equation}
in $KK(A,A)$.
\begin{definition}\label{Def: KK equiv}
	Let $A$ and $B$ be $C^*$-algebras. We say that $A$ and $B$ are $KK$-equivalent if there exists $KK$-classes $[X]\in KK(A,B)$ and $[Y]\in KK(B,A)$ such that $[X]\Otimes_B [Y]=[\text{id}_A]$ and $[Y]\Otimes_A [X]=[\text{id}_B]$.
\end{definition}
\begin{remark}
	If $A$ and $B$ are $KK$-equivalent then Definition~\ref{Def: KK equiv} tells us that for any $C^*$-algebra $C$, there group isomorphisms $\Otimes [X]:KK(C,A)\to KK(C,B)$ and $[Y]\Otimes :KK(A,C)\to KK(B,C)$. As such, $A$ and $B$ are `indistinguishable' in $KK$-theory.
\end{remark}
As the name may suggest, $KK$-theory generalises $K$-theory in the following sense.
\begin{theorem}\label{theorem: triv graded KK is K}[{\cite[Proposition~17.5.5,17.5.6]{Blackadar}}]
	Let $A$ be a trivially graded, $\sigma$-unital $C^*$-algebra. There are isomorphisms $KK_0(\C,A)\cong K_0(A)$ and $KK_1(\C,A)\cong K_1(A)$.
\end{theorem}

\section{Exact sequences in $KK$-theory}
 In this short section we state some results about exact sequences and $KK$-theory without proof.
\begin{definition}
	Let $(A,\alpha)$ be a graded $C^*$-algebra with $J\triangleleft$ an ideal such that $\alpha(J)\subseteq J$. A short exact sequence
	\begin{equation*}
		0\to J\to A\to A/J\to 0
	\end{equation*}
	is \emph{semi-split} if there exists a completely positive, norm-decreasing map $\lambda:A/J\to A$ that commutes with $\alpha$. We call $J$ a semi-split ideal in $A$.
\end{definition}
\begin{remark}\label{remark:nuclear implies semi-split}
	In \cite[Example~19.5.2]{Blackadar} it is noted that if $A$ is a nuclear $C^*$-algebra then applying \cite[Theorem~15.8.3]{Blackadar} with $A=D/J$ and $\psi:D/J\to D/J$ the identity map, every ideal in $A$ is semi-split.==
	In \cite[1.2~Remarks(b)]{Skandalis} it is noted that given a completely positive norm-decreasing map $\lambda:A/J\to A$ which does not commute with $\alpha$, if $a_0$ and $a_1$ denote the odd and even components $a$, one can define a completely positive, norm decreasing map $\widetilde{\lambda}:A/J\to A$ which does commute with $\alpha$ by the formula
	\begin{equation*}
		\widetilde{\lambda}(a)=\lambda(a_0)_0+\lambda(a_1)_1.
	\end{equation*}
\end{remark}
\begin{theorem}\label{theorem: exact sequences in KK}
	(cf. \cite[Theorem~1.1]{Skandalis}, \cite[Theorem~19.5.6]{Blackadar}) Suppose that $A$ is a graded $C^*$-algebra, and suppose that
	\begin{equation*}
		0\to J\xrightarrow{j} A\xrightarrow{q} A/J\to 0
	\end{equation*}
	is a semi-split exact sequence. Then for every separable graded $C^*$-algebra $D$ there exist homomorphisms $\delta:KK_i(D,A/J)\to KK_{i+1}(D,J)$ with subscripts mod 2 such that the following six-term sequence is exact.
	\begin{equation*}
	\parbox[c]{0.8\textwidth}{\hfill
		\begin{tikzpicture}[yscale=0.8, >=stealth]
		\node (00) at (0,0) {$KK_1(B, A/J)$};
		\node (40) at (4,0) {$KK_1(B, A)$};
		\node (80) at (8,0) {$KK_1(B, J)$};
		\node (82) at (8,2) {$KK_0(B, A/J)$};
		\node (42) at (4,2) {$KK_0(B, A)$};
		\node (02) at (0,2) {$KK_0(B, J)$};
		\draw[->] (02)-- node[above] {${\scriptstyle{q_*}}$} (42);
		\draw[->] (42)-- node[above] {${\scriptstyle j_*}$} (82);
		\draw[<-] (82)--(80)-- node[right] {${\scriptstyle \delta}$} (82);
		\draw[->] (80)-- node[above] {${\scriptstyle{q_*}}$} (40);
		\draw[->] (40)-- node[above] {${\scriptstyle j_*}$} (00);
		\draw[<-] (00)--(02)-- node[left] {${\scriptstyle \delta}$} (00);
		\end{tikzpicture}\hfill\hfill}
	\end{equation*}
	Similarly, if $A$ is separable, then for any $\sigma$-unital graded $C^*$-algebra $D$ there exist homomorphisms $\delta:KK_i(J,D)\to KK_{i+1}(A/J,D)$ with subscripts mod 2 such that the following six-term sequence is exact.
	\begin{equation*}
	\parbox[c]{0.8\textwidth}{\hfill
		\begin{tikzpicture}[yscale=0.8, >=stealth]
		\node (00) at (0,0) {$KK_1(A/J, D)$};
		\node (40) at (4,0) {$KK_1(A, D)$};
		\node (80) at (8,0) {$KK_1(J, D)$};
		\node (82) at (8,2) {$KK_0(A/J, D)$};
		\node (42) at (4,2) {$KK_0(A, D)$};
		\node (02) at (0,2) {$KK_0(J, D)$};
		\draw[<-] (02)-- node[above] {${\scriptstyle q^*}$} (42);
		\draw[<-] (42)-- node[above] {${\scriptstyle j^*}$} (82);
		\draw[->] (82)--(80)-- node[right] {${\scriptstyle \delta}$} (82);
		\draw[<-] (80)-- node[above] {${\scriptstyle q^*}$} (40);
		\draw[<-] (40)-- node[above] {${\scriptstyle j^*}$} (00);
		\draw[->] (00)--(02)-- node[left] {${\scriptstyle \delta}$} (00);
		\end{tikzpicture}\hfill\hfill}
	\end{equation*}
\end{theorem}

\section{Graded $K$-theory}
We have finally built enough machinery to define graded $K$-theory. In this section we provide definitions and examples that will be of use to us later of graded $K$-theory.
\begin{definition}
	Let $(A,\alpha)$ be a graded $C^*$-algebra. We define the \emph{graded $K$-groups} $K^{\gr}_0(A)$ and $K_1^{\gr}(A)$ associated to $A$ to be the abelian groups $K_*^{\gr}(A)=KK_*(\C,A)$. 
\end{definition}
\begin{remark}
	Note that if $A$ is trivially graded then Theorem~\ref{theorem: triv graded KK is K} tells us that the graded $K$-groups of $A$ are the regular $K$ groups $K_*(A)$. By definition we have $K_0^{\gr}(A\Otimes\Cliff_1)=KK(\C,A\Otimes\Cliff_1)=K_1^{\gr}(A)$. Using Corollary~\ref{cor: Bott periodicity} we obtain
	\begin{equation*}
		K_1^{\gr}(A\Otimes\Cliff_1)=KK_1(\C,A\Otimes\Cliff_1)=KK_2(\C,A)\cong KK(\C,A)=K^{\gr}_0(A).
	\end{equation*}
	In particular, for $A=\Cliff_1$ we have
	\begin{equation*}
		K_i^{\gr}(\Cliff_1)=K_{i+1}^{\gr}(\C)=\begin{cases}
			0&i=0\\
			\Z&i=1.\\
		\end{cases}
	\end{equation*}
	For $\Cliff_2$ obtain
	\begin{equation*}
		K_i^{\gr}(\Cliff_2)=K_{i}^{\gr}(\C)=\begin{cases}
			\Z&i=0\\
			0&i=1.\\
		\end{cases}
	\end{equation*}
	Hence, we obtain $K^{\gr}_*(\Cliff_n)=(\Z,0)$ if $n$ is even and $K^{\gr}_*(\Cliff_n)=(0,\Z)$ if $n$ is odd.
\end{remark}
\begin{example}
	This Example comes from \cite[Example~3.10]{ThePaper}.
	Consider $C(\T)$ with the grading $\alpha_{C(\T)}(f)(x)=f(\overline{x})$. Let $C_0(\R)$ and $\C\oplus\C$ be trivially graded and define $\iota:C_0(\R)\Otimes\Cliff_1\to\C(\T)$ by
	\begin{equation*}
		\iota(f\otimes (1-e)/2+g\otimes(1+e)/2)=e^{i\theta}\mapsto \begin{cases}
			f(\tan(\theta-\pi/2))&\theta\in [0,\pi]\\
			g(\tan(\theta-\pi/2))&\theta\in [-\pi,0]\\
		\end{cases}
	\end{equation*}
	where $e$ is the odd generator of $\Cliff_1$.
	Since $f$ and $g$ vanish at infinity, $\iota(f\otimes(1-e)/2 +g\otimes(1+e)/2)$ is continuous. For all $f,g\in C_0(\R)$ we have
	\begin{align*}
		\iota\Big((f\otimes(1-e)/2+g\otimes(1+e)/2)^*\Big)&=\iota(\overline{g}\otimes (1-e)/2+\overline{f}\otimes(1+e)/2)\\
		&=e^{i\theta}\mapsto \begin{cases}
			\overline{g}(\tan(\theta-\pi/2))&\theta\in [0,\pi]\\
			\overline{f}(\tan(\theta-\pi/2))&\theta\in [-\pi,0]\\
		\end{cases}\\
		&=\overline{\left(e^{i\theta}\mapsto \begin{cases}
		f(\tan(\theta-\pi/2))&\theta\in [0,\pi]\\
		g(\tan(\theta-\pi/2))&\theta\in [-\pi,0]\\
		\end{cases}\right)}\\
		&=\iota(f\otimes(1-e)/2+g\otimes(1+e)/2)^*,\\
	\end{align*}
	so $\iota$ is $*$-preserving. Since $(1-e)/2$ and $(1+e)/2$ are orthogonal we compute for $f,g,a,b\in C_0(\R)$
	\begin{align*}
		\iota\Big((&f\otimes(1-e)/2+g\otimes(1+e)/2)(a\otimes(1-e)/2+b\otimes(1+e)/2\Big)\\&=\iota\Big(fa\otimes (1-e)/2+gb\otimes(1+e)/2\Big)\\
		&=e^{i\theta}\mapsto \begin{cases}
		a(\tan(\theta-\pi/2))f(\tan(\theta-\pi/2))&\theta\in [0,\pi]\\
		b(\tan(\theta-\pi/2))g(\tan(\theta-\pi/2))&\theta\in [-\pi,0]\\
		\end{cases}\\
		&=\iota(f\otimes(1-e)/2+g\otimes(1+e)/2)\iota(a\otimes(1-e)/2+b\otimes(1+e)/2)\\
	\end{align*}
	so $\iota$ is multiplicative and hence a homomorphism. Recall that $\Cliff_1$ is graded such that the odd subspace is spanned by $e$ and the even subspace by $1$. If $\alpha_{C_0(\R)\otimes\Cliff}$ denotes the grading on $C_0(\R)\otimes\Cliff_1$ with $C_0(\R)$ trivially graded then we compute
	\begin{align*}
		\iota(\alpha_{\Cliff}(f\otimes(1-e)/2+g\otimes(1+e)/2))&=\iota(f\otimes(1+e)/2+g\otimes(1-e)/2)\\
		&=e^{i\theta}\mapsto \begin{cases}
		g(\tan(\theta-\pi/2))&\theta\in [0,\pi]\\
		f(\tan(\theta-\pi/2))&\theta\in [-\pi,0]\\
		\end{cases}\\
		&=\alpha_{C(\T)}\left(e^{i\theta}\mapsto \begin{cases}
		f(\tan(\theta-\pi/2))&\theta\in [0,\pi]\\
		g(\tan(\theta-\pi/2))&\theta\in [-\pi,0]\\
		\end{cases}\right)\\
		&=\alpha_{C(\T)}(\iota(f\otimes(1-e)/2+g\otimes(1+e)/2)
	\end{align*}
	Thus $\iota$ is a graded homomorphism.
	Define $\ep:\C(\T)\to\C\oplus\C$ by $\ep(f)=(f(1),f(-1))$. Since $f(1)$ and $f(-1)$ are fixed by the grading $\alpha_{C(\T)}$ we see that $\ep$ is a graded homomorphism for $\C\oplus\C$ with the trivial grading. The image of $\iota$ is the ideal in $C(\T)$ consisting of all functions which vanish at $1$ and $-1$, which is precisely the kernel of $\ep$. Since $\iota$ is clearly injective, we obtain a graded short exact sequence
	\begin{equation}\label{exact}
		0\to C_0(\R)\Otimes\Cliff_1\xrightarrow{\iota}C(\T)\xrightarrow{\ep}\C\oplus\C\to 0.
	\end{equation}
	Hence by Theorem~\ref{theorem: exact sequences in KK} we obtain an exact sequence in $KK$-theory
	\begin{equation*}
	\parbox[c]{0.8\textwidth}{\hfill
		\begin{tikzpicture}[yscale=0.8, >=stealth]
		\node (00) at (0,0) {$KK_1(\C, \C\oplus\C)$};
		\node (40) at (4,0) {$KK_1(\C, C(\T)$};
		\node (80) at (8,0) {$KK_1(\C, C_0(\R)\Otimes\Cliff_1)$};
		\node (82) at (8,2) {$KK_0(\C, \C\oplus\C)$};
		\node (42) at (4,2) {$KK_0(\C, C(\T))$};
		\node (02) at (0,2) {$KK_0(\C, C_0(\R)\Otimes\Cliff_1)$};
		\draw[->] (02)-- node[above] {$\quad{\scriptstyle{\iota_*}}\quad$} (42);
		\draw[->] (42)-- node[above] {$\quad{\scriptstyle \ep_*}\quad$} (82);
		\draw[<-] (82)--(80)-- node[right] {${\scriptstyle \delta}$} (82);
		\draw[->] (80)-- node[above] {$\quad{\scriptstyle{\iota_*}}\quad$} (40);
		\draw[->] (40)-- node[above] {$\quad{\scriptstyle \ep_*}\quad$} (00);
		\draw[<-] (00)--(02)-- node[left] {${\scriptstyle \delta}$} (00);
		\end{tikzpicture}\hfill\hfill}
	\end{equation*}
	By Theorem~\ref{theorem: triv graded KK is K} since $\C\oplus\C$ is trivially graded we have $KK_*(\C,\C\oplus\C)=K_*(\C^2)=(\Z^2,0)$ where we have used Example~\ref{K theory of C} and Example~\ref{K1 of C}. Similarly, since 
	\begin{equation*}
		KK_i(\C,C_0(\R)\Otimes\Cliff_1)=KK_{i+1}(\C,C_0(\R))
	\end{equation*}
	and since $C_0(\R)$ is trivially graded, using Theorem~\ref{theorem: triv graded KK is K} we have 
	\begin{equation*}
		KK_i(\C,C_0(\R)\Otimes\Cliff_1)=K_{i+1}(C_0(\R))
	\end{equation*}
	which is $(\Z,0)$ by \cite[Page~234]{BlueBook}. Thus, writing $KK_i(\C,C(\T))=K_i^{\gr}(C(\T))$, our exact sequence becomes
	\begin{equation*}
	\parbox[c]{0.8\textwidth}{\hfill
		\begin{tikzpicture}[yscale=0.8, >=stealth]
		\node (00) at (0,0) {$0$};
		\node (40) at (4,0) {$K_1^{\gr}(C(\T))$};
		\node (80) at (8,0) {$0$};
		\node (82) at (8,2) {$\Z^2$};
		\node (42) at (4,2) {$K_0^{\gr}(C(\T))$};
		\node (02) at (0,2) {$\Z$};
		\draw[->] (02)-- node[above] {$\quad{\scriptstyle{\iota_*^0}}\quad$} (42);
		\draw[->] (42)-- node[above] {$\quad{\scriptstyle \ep_*^0}\quad$} (82);
		\draw[<-] (82)--(80)-- node[right] {${\scriptstyle \delta}$} (82);
		\draw[->] (80)-- node[above] {$\quad{\scriptstyle{\iota_*^1}}\quad$} (40);
		\draw[->] (40)-- node[above] {$\quad{\scriptstyle \ep_*^1}\quad$} (00);
		\draw[<-] (00)--(02)-- node[left] {${\scriptstyle \delta}$} (00);
		\end{tikzpicture}\hfill\hfill}
	\end{equation*}
\end{example}
The first isomorphisms theorem for groups implies that $K_0^{\gr}(C(\T))/\ker(\ep_*^0)\cong\Z^2$ and $K_1^{\gr}(C(\T))/\ker(\ep_*^1)\cong 0$. Exactness gives $\ker(\ep_*^1)=\Image(\iota_*^1)=0$, so we deduce that we deduce $K_1^{\gr}(C(\T))=0$. Applying exactness to $\ep_*^0$ gives $\ker(\ep_*^0)=\Image(\iota_*^0)$ and the first isomorphism theorem again gives $\Image(\iota_*^0)\cong\Z/\ker(\iota_*^0)$. Exactness once more gives $\ker(\iota_*^0)=\Image(\delta)=0$, hence we have $\ker(\ep_*^0)=\Image(\iota_*^0)=\Z$. Thus $K^{\gr}_0(C(\T))\cong\Z^3$ and we deduce that $K^{\gr}_*(C(\T))=(\Z^3,0)$ because extensions of free abelian groups split. 

Note that if we instead give $C(\T)$ the trivial grading then the exact sequence \eqref{exact} is no longer graded, since despite $C_0(\R)$ having the trivial grading, $C_0(\R)\Otimes\Cliff_1$ is non-trivially graded. Hence this particular Example cannot be applied to compute the regular $K$-theory of $C(\T)$. 
Indeed, by \cite[Page~234]{BlueBook} we have $K_*(C(\T))\cong (\Z,\Z)$ which is quite different from $K_*^{\gr}(C(\T))$.

\section{A dash of Morita equivalence}
In this section we develop a small amount of theory behind Morita equivalence to be used very briefly in Theorem~\ref{MAINTHEOREM}.
\begin{definition}
	Let $(A,\alpha_A)$ and $(B,\alpha_B)$ be graded $C^*$-algebras. We say that $A$ and $B$ are Morita equivalent if there exists a full graded Hilbert $B$ module $X$ and a graded isomorphism $\phi:A\to\KK(X)$.
\end{definition}
\begin{remark}
	In \cite[13.7]{Blackadar} this definition is referred to as \emph{Strongly} Morita equivalent. 
\end{remark}
\begin{proposition}\label{Morita is equivalence}
	Morita equivalence is an equivalence relation.
\end{proposition}
\begin{proof}
	This proof is an expansion of \cite[Proposition~3.18]{RaeburnWilliams}.
	Let $(A,\alpha_A)$ be a graded $C^*$-algebra. Then the Hilbert $A$ module $A_A$ is full and Example~\ref{Graded compact isomorphism} shows that $\KK(A_A)$ and $A$ are graded isomorphic, so we have reflexivity. Let $(B,\alpha_B)$ and $(C,\alpha_C)$ be graded $C^*$-algebras, let $X$ be a full Hilbert $B$ module, $Y$ a full Hilbert $C$ module and $\phi:A\to\KK(X)$ and $\psi:B\to\KK(Y)$ be isomorphisms. Then $X$ is an $A$--$B$ correspondence with left action $\phi$, $Y$ is a $B$--$C$ correspondence with left action $\psi$ and $X\otimes_{\psi}Y$ is a full $C$ module. This is because by Example~\ref{Fish} the left action $\psi$ on $Y$ is non-degenerate, so
	\begin{align*}
		\overline{\Span}\{\IP{z,w}:z,w\in X\otimes_{\psi}Y\}&=\overline{\Span}\{\IP{y,\psi(\IP{x,u})v}:x,u\in X,u,v\in Y\}\\
		&=\overline{\Span}\{\IP{y,v}:u,v\in Y\}=C\\
	\end{align*}
	by fullness of $Y$. Since our left action $\psi:B\to\KK(X)$ is injective and by compacts, Lemma~\ref{compact otimes 1} and Lemma~\ref{T} tell us that that map $\theta_{x,y}\mapsto \theta_{x,y}\otimes 1$ has range in the compact operators on $X\otimes_BY$ and is injective. Hence if we can show that this map is surjective we will have $\KK(X\otimes_BY)\cong \KK(X)\cong A$. We compute for $x,w,u\in X$ and $y,z,v\in Y$
	\begin{align*}
		\Theta_{x \otimes y, w \otimes z}(u \otimes v)
		&= (x \otimes y) \cdot \langle w \otimes z, u \otimes v\rangle_C\\
		&= x \otimes (y \cdot \langle z, \psi(\langle w, u\rangle_B) v\rangle_C)\\
		&= x \otimes \Theta_{y,z}(\psi(\langle w, u\rangle_B) \cdot v)\\
		&= x  \psi^{-1}(\Theta_{y,z}) \otimes \psi(\langle w, u\rangle_B) \cdot v\\
		&= (x \cdot \psi^{-1}(\Theta_{y,z})) \cdot \langle w, u\rangle_B \otimes v\\
		&= \Theta_{x \cdot \psi^{-1}(\theta_{y,z}), w}(u) \otimes v\\
		&= (\Theta_{x \cdot \psi^{-1}(\theta_{y,z}), w} \otimes 1)(u \otimes v).\\
	\end{align*}
	Hence $\theta_{x,y}\mapsto \theta_{x,y}\otimes 1$ is surjective and we deduce that $\KK(X\otimes_BY)\cong \KK(X)\cong A$. Now for symmetry we introduce the concept of the dual module $X^*$. Given our Hilbert $B$-module $X$ the dual module $X^*$ is a Hilbert $\KK(X)$-module. As an additive set we define $X^*=X$. To avoid confusion, matching notation with \cite{RaeburnWilliams} we write $\flat:X\to X^*$ for the identity map and we write elements in $X^*$ as $\flat(x)$ where $x\in X$. We give $X^*$ the scalar multiplication $\lambda\flat(x)=\flat(\overline{\lambda}x)$ for $\lambda\in\C$ and $x\in X^*$. We define a right action by $\KK(X)$ and inner product taking values in $\KK(X)$ on $X^*$ by
	\begin{equation*}
		\flat(x)\cdot K=\flat(K^*(x))\qquad\text{and}\qquad \IP{\flat(x),\flat(y)}_{\KK(X)}=\Theta_{x,y}
	\end{equation*}
	for all $x,y\in X^*$ and $K\in\KK(X)$. We compute
	\begin{align*}
		\normof{\flat(x)}&=\normof{\IP{\flat(x),\flat(x)}}^{1/2}\\
		&=\normof{\Theta_{x,x}}^{1/2}\\
		&=\sup_{\normof{z}\leq 1}\normof{x\IP{x,z}}^{1/2}\\
		&=\normof{x\normof{x}}^{1/2}\\
		&=\normof{x}.\\
	\end{align*}
	The dual module $X^*$ is full because
	\begin{equation*}
		\overline{\Span}\{\IP{\flat(x),\flat(y)}:\flat(x),\flat(y)\in X^*\}=\overline{\Span}\{\Theta_{x,y}:x,y\in X\}=\KK(X).
	\end{equation*}
	For a finite sum $\sum_n\Theta_{\flat(x_n),\flat(y_n)}$,
	 \begin{align*}
	 	\Bignormof{\sum_n\Theta_{\flat(x_n),\flat(y_n)}}&=\sup_{\normof{\flat(z)}\leq 1}\Bignormof{\sum_n\flat(x_n)\IP{\flat(y_n),\flat(z)}}\\
	 	&=\sup_{\normof{z}\leq 1}\Bignormof{\flat(\sum_n\Theta_{z,y_n}x_n)}=\sup_{\normof{z}\leq 1}\Bignormof{\sum_n\Theta_{z,y_n}x_n}\\
	 	&=\sup_{\normof{z}\leq 1}\Bignormof{\sum_nz\IP{y_n,x_n}}=\Bignormof{\sum_n\IP{y_n,x_n}}\\
	 	&=\Bignormof{\Big(\sum_n\IP{y_n,x_n}\Big)^*}=\Bignormof{\sum_n\IP{x_n,y_n}}.\\
	 \end{align*}
	 So there is a well defined linear map $\psi:\Span\{\Theta_{\flat(x),\flat(y)}:x,y\in X\}\to B$ such that
	 \begin{equation*}
		\psi\Big(\sum_n\Theta_{\flat(x_n),\flat(y_n)}\Big)=\sum_n\IP{x_n,y_n}. 
	 \end{equation*} 
	 Our calculation shows that $\psi$ is norm preserving, hence uniformly continuous so Lemma~\ref{Sang} shows that there is a norm preserving extension of $\psi$ to $\KK(X^*)$ such that if $K_n$ is a convergent sequence of finite rank operators then $\psi(\lim_{n\to\infty}K_n)=\lim_{n\to\infty}\psi(K_n)$. 
	 For $x,y,u,v\in X$, we compute
	 \begin{align*}
		 \Theta_{\flat(x),\flat(y)}\Theta_{\flat(u),\flat(v)}&=\flat(x)\IP{\flat(y),\flat(u)\IP{\flat(v),\cdot}}=\flat(x)\IP{\flat(y),\flat(u)}\IP{\flat(v),\cdot}\\
		 &=\Theta_{\flat(x)\IP{\flat(y),\flat(u)},\flat(v)}=\Theta_{\flat(\Theta_{u,y}(x)),\flat(v)}\\
		 &=\Theta_{\flat(u\IP{y,x}),\flat(v)}.\\
	 \end{align*}
	 Thus
	 \begin{align*}
		 \psi(\Theta_{\flat(x),\flat(y)}\Theta_{\flat(u),\flat(v)})&=\psi(\Theta_{\flat(u\IP{y,x}),\flat(v)})=\IP{u\IP{y,x},v}\\
		 &=\IP{x,y}\IP{u,v}=\psi(\Theta_{\flat(x),\flat(y)})\psi(\Theta_{\flat(u),\flat(v)}).\\
	 \end{align*} 
	 We also have
	 \begin{equation*}
	 \psi(\Theta_{\flat(x),\flat(y)}^*)=\psi(\Theta_{\flat(y),\flat(x)})=\IP{x,y}=\psi(\Theta_{\flat(x),\flat(y)})^*.
	 \end{equation*}
	 Since the rank one operators span a dense subset of $\KK(X^*)$ and $\psi$ is continuous, we deduce that $\psi(K_1K_2)=\psi(K_q)\psi(K_2)$ and $\psi(K_1^*)=\psi(K_1)^*$ for all $K_1,K_2\in\KK(X^*)$. Thus $\psi$ is a homomorphism.	 
	 Since $X$ is full, $\psi$ is surjective, and since $\psi$ is isometric it is injective, hence we have an isomorphism $\KK(X^*)\cong B$. We compute
	 \begin{equation*}
		 \psi(\widetilde{\alpha}_X(\Theta_{\flat(x),\flat(y)}))=\psi(\Theta_{\flat(\alpha_X(x)),\flat(\alpha_X(y))})=\IP{\alpha_X(x),\alpha_X(y)}=\alpha_B(\IP{x,y}),
	 \end{equation*}
	 so since the rank one operators span a dense subset in $\KK(X^*)$ and $\psi$ is continuous we deduce that $\psi$ is a graded isomorphism.
\end{proof}
The main result we will need about Morita equivalence is that Morita equivalent $C^*$-algebras are $KK$-equivalent. Let $(A,\alpha_A)$ and $(B,\alpha_B)$ be Morita equivalent graded $C^*$-algebras, so that there is a graded Hilbert $B$-module $X$ with $\KK(X)\cong A$. If $\phi:A\to\KK(X)$ and $\psi^{-1}:B\to\KK(X^*)$ implement the (graded) isomorphisms between $A$ and $\KK(X)$ and $B$ and $\KK(X^*)$, with $\psi$ defined by $\psi(\Theta_{\flat(x),\flat(y)})=\IP{x,y}$ as in the proof of Proposition~\ref{Morita is equivalence}, then $(X,\phi,0,\alpha_X)$ and $(X^*,\psi^{-1},0,\alpha_X)$ are Kasparov $A$--$B$ and $B$--$A$ modules. We use this set up for the following Theorem.
\begin{theorem}\label{Morita equiv KK}(cf. \cite[Proposition 1.5.7]{Kasnotes})
	 With notation as above, we have
	\begin{equation*}
		[X,\phi,0,\alpha_X]\Otimes_B[X^*,\psi^{-1},0,\alpha_X]=[\text{id}_A]\quad\text{and}\quad [X^*,\psi^{-1},0,\alpha_X]\Otimes_A[X,\phi,0,\alpha_X]=[\text{id}_B].
	\end{equation*}
	That is, $A$ and $B$ are $KK$-equivalent.
\end{theorem}
\begin{proof}
	Using~\ref{Fredholm zero Kasp product} we compute
	\begin{equation*}
		[X,\phi,0,\alpha_X]\Otimes_B[X^*,\psi^{-1},0,\alpha_X]=[X\Otimes_{\psi^{-1}} X^*,\phi\Otimes 1_{X^*},0,\alpha_X\Otimes_{A}\alpha_X].
	\end{equation*}
	We wish to show that there is a well defined unitary map $U:X\otimes X^*\to\KK(X)_{\KK(X)}$ such that $U(x\otimes \flat(y))=\Theta_{x,y}$. We compute
	\begin{align*}
		\IP{U(x\otimes\flat(y)),U(u\otimes\flat(v))}&=\IP{\Theta_{x,y},\Theta_{u,v}}\\
		&=\Theta_{x,y}^*\Theta_{u,v}\\
		&=\Theta_{y,x}\Theta_{u,v}\\
		&=\Theta_{\Theta_{y,x}u,v}\\
		&=\Theta_{y\IP{x,u},v}\\
		&=y\IP{x,u}\IP{v,\cdot}\\
		&=y\IP{v\IP{u,x},\cdot}\\
		&=\Theta_{y,v\IP{u,x}}\\
	\end{align*}
	Thus, we have
	\begin{align*}
		\IP{x\otimes\flat(y),u\otimes\flat(v)}&=\IP{\flat(y),\psi^{-1}(\IP{x,u})\flat(v)}\\
		&=\IP{\flat(y),\Theta_{\flat(x),\flat(u)}\flat(v)}\\
		&=\IP{\flat(y),\flat(x)\IP{\flat(u),\flat(v)}}\\
		&=\IP{\flat(y),\flat(v\IP{u,x})}=\IP{\flat(y),\flat(\Theta_{v,u}x)}\\
		&=\Theta_{y,v\IP{u,x}}\\
		&=\IP{U(x\otimes\flat(y)),U(u\otimes\flat(v))}.\\
	\end{align*}
	Hence $U$ is unitary. For any $a\in A$ and $x,y\in X$,
	\begin{equation*}
		U(\phi(a)\Otimes 1_{X^*}(x\otimes\flat(y)))=U(\phi(a)x\otimes\flat(y))=\Theta_{\phi(a)x,y}=\phi(a)\Theta_{x,y}.
	\end{equation*}
	Thus
	\begin{align*}
		U(\alpha_X\Otimes_{A}\alpha_X(x\otimes\flat(y)))&=U(\alpha_X(x),\flat(\alpha_X(y))=\Theta_{\alpha_X(x),\alpha_X(y)}\\
		&=\widetilde{\alpha_X}(\Theta_{x,y})=\widetilde{\alpha_X}(U(x\otimes\flat(y))).
	\end{align*}
	Let $\iota:A\to\LL(A_A)$ be left multiplication.
	Since $U0=0=0U$, we deduce that
	\begin{equation*}
		(X\Otimes_{\psi^{-1}} X^*,\phi\Otimes 1_{X^*},0,\alpha_X\Otimes_{A}\alpha_X)\cong (\KK(X)_{\KK(X)},\phi,0,\widetilde{\alpha_X}).
	\end{equation*}
	Hence we have
	\begin{align*}
		[X,\phi,0,\alpha_X]\Otimes_B[X^*,\psi^{-1},0,\alpha_X]&=[X\Otimes_{\psi^{-1}} X^*,\phi\Otimes 1_{X^*},0,\alpha_X\Otimes_{A}\alpha_X]\\
		&=[\KK(X)_{\KK(X)},\phi,0,\widetilde{\alpha_X}]\\
		&=[A_{A},\iota,0,\alpha_A]=[\text{id}_A].\\
	\end{align*}
	A symmetric argument shows that
	\begin{equation*}
		[X^*,\psi^{-1},0,\alpha_X]\Otimes_A[X,\phi,0,\alpha_X]=[\text{id}_B].
	\end{equation*}
	Hence $A$ and $B$ are $KK$-equivalent.
\end{proof}

\begin{corollary}\label{cor: Matrices are morita equiv}
	Let $A$ and $B$ be graded $C^*$-algebras and let $\HH$ be an infinite dimensional Hilbert space. Let $\KK$ denote $\KK(\HH)$. There are isomorphisms
	\begin{equation}\label{first}
		KK(A,B)\cong KK(A\otimes \KK,B\otimes\KK)\cong KK(A,B\otimes \KK)\cong KK(A\otimes \KK,B)	
	\end{equation}
	and
	\begin{equation}\label{second}
		KK(A,B)\cong KK(A\otimes M_n(\C),B\otimes M_n(\C))\cong KK(A,B\otimes M_n(\C))\cong KK(A\otimes M_n(\C),B)	
	\end{equation}
	for all $n$.
\end{corollary}
\begin{proof}
	Let $\{e_i\}$ be an orthonormal basis for $\HH$.
	There is an isomorphism $\KK(\HH_A)\cong A\otimes\KK$ is observed by the linear map $\psi:\KK(\HH_A)\to A\otimes\KK$ determined by
	\begin{equation*}
		\psi(\Theta_{(a),(b)})=\sum_{i,j=1}^{\infty}(a_ib_j^*)\otimes \Theta_{e_i,e_j}.
	\end{equation*}
	Hence $\HH_A$ defines a Morita equivalence of $A$ and $A\otimes\KK$ where $\KK$ has the trivial grading giving Equation \eqref{first}. Similarly, for any $n$ we have an isomorphism $\varphi:A\otimes M_n(\C)\to \KK(A^n)$ determined by
	\begin{equation*}
		\varphi(\Theta_{(a),(b)})=\sum_{i,j=1}^{n}(a_ib_j^*)\otimes E_{i,j}
	\end{equation*}
	where $E_{ij}$ is the matrix unit with zeroes everywhere and a one in the $i$-$j$tj entry. Hence $A^n$ defines a Morita equivalence of $A\otimes M_n(\C)$ and $A$, giving equation \eqref{second}.
\end{proof}

\begin{corollary}\label{cor: Bott periodicity}
	For all $C^*$-algebras $A$ and $B$ we have $KK_{i+2}(A,B)\cong KK_i(A,B)$.
\end{corollary}
\begin{proof}
	We have using Equation~\ref{Cliff tensor 2 period}
	\begin{equation*}
		KK_{i+2}(A,B)=KK_0(A,B\Otimes\Cliff_{i+2})=KK_i(A,B\Otimes\Cliff_{2}\Otimes\Cliff_i)=KK_i(A,B\Otimes\Cliff_{2}).
	\end{equation*}
	We computed in Example~\ref{Eg: Cliff 2} that $\Cliff_2$ is isomorphic to $M_2(\C)$ with the standard grading with even and odd subspaces given by diagonal and off-diagonal matrices. In Remark~\ref{remark: graded and ungraded tensor product} we saw that under this grading we have $B\Otimes M_2(\C)\cong B\otimes M_2(\C)$ where the right hand tensor product is ungraded. Thus by Corollary~\ref{cor: Matrices are morita equiv} we have
	\begin{equation*}
		KK_{i+1}(A,B)\cong KK(A,B\Otimes\Cliff_2)\cong KK(A,B\otimes M_2(\C))\cong KK(A,B).
	\end{equation*}
\end{proof}

%% file: chapter7.tex
\chapter{Pimsner's six-term exact sequence in $KK$-theory}
\label{The Exact Sequence!}

This chapter aims to expand on the proof of \cite[Theorem 4.2 and Theorem 4.4]{ThePaper} and simultaneously to remove the hypothesis used there that the left action of $A$ on $X$ be by compact operators. We will be making heavy reference to \cite{FreeProbTheory} as is done in \cite{ThePaper}.
Throughout this chapter we will be working with $C^*$ correspondences $_AX_A$ with the same algebra $A$ on the left and right. We assume that the left action $\varphi:A\to \LL(X)$ is injective, but do not assume that $\varphi$ is by compacts as in \cite[Chapter 4]{ThePaper}.

Recall that if $X$ is a graded correspondence then $\OO_X$ and $\TT_X$ are graded $C^*$-algebras with gradings $\alpha_{\TT}$ and $\alpha_{\OO}$ such that
\begin{equation*}
	\alpha_{\TT}(T_x)=T_{\alpha(x)}\qquad\text{ and }\alpha_{\OO}(S_x)=S_{\alpha(x)}.
\end{equation*}
Throughout this chapter $I$ will indicate the ideal $\varphi^{-1}(\KK(X))$ from Lemma~\ref{compact otimes 1}.

We wish to show that under some general conditions, $A$ and $\TT_A$ are $KK$ equivalent. That is, we wish to find two mutually inverse Kasparov classes in $KK(A,\TT_X)$ and $KK(\TT_X,A)$. We set $[i_A]$ to be the natural class in $KK(A,\TT_X)$ of the inclusion $i_A:A\hookrightarrow\TT_X$ of Lemma~\ref{Trivial Kas Module}. That is,
\begin{equation*}
[i_A]=[\TT_X,i_A,0,\alpha_{\TT}].
\end{equation*}
For the other class in $KK(\TT_X,A)$ we construct an action of $\TT_X$ on $\FF_X$ as in \cite[Chapter~4]{FreeProbTheory} as follows. Let $\pi_0:\TT_X\to \LL(\FF_X)$ be the natural action of $\TT_X$ on $\FF_X$. Define a linear map $t:X\to\LL(\FF_X)$ by
\begin{equation*}
t(x)(\xi)=\begin{cases}
\pi_0(x)\xi& \text{if }\xi\in X^{\otimes n},\; n\geq 1\\
0&\text{if }\xi\in A=X^{\otimes 0}.\\
\end{cases}
\end{equation*}
Then each $t(x)$ is adjointable with
\begin{equation*}
t(x)^*(\xi)=\begin{cases}
\pi_0(x)^*\xi& \text{if }\xi\in X^{\otimes n},\; n\geq 2\\
0&\text{if }\xi\in X^{\otimes n},\; n<2.\\
\end{cases}
\end{equation*}
Together with the restriction of the diagonal left action $\varphi^{\infty}$ of $A$ on $\FF_X$ to the sub-module $\bigoplus_{i=1}^{\infty}X^{\otimes n}$, the pair $(t,\varphi^{\infty})$ is a representation of $_AX_A$ on $\LL(\FF_X)$. Hence by the universal property~\ref{theorem: universal property of Toeplitz} of $\TT_X$ there exists a homomorphism $\pi_1:\TT_X\to \LL(\FF_X)$ such that for $x\in X$ and $a\in A$, 
\begin{equation*}
\pi_1(x)=t(x)\qquad  \text{ and } \qquad\pi_1(a)=\varphi^{\infty}(a).
\end{equation*}
We aim to show that the quadruple
\begin{equation}\label{M}
M=\Big(\FF_X\oplus\FF_X,\pi_0\oplus(\pi_1\circ\alpha_{\TT}),\begin{psmallmatrix}
0&1\\
1&0\\
\end{psmallmatrix},\begin{psmallmatrix}
\alpha_{\infty}&0\\
0&-\alpha_{\infty}\\
\end{psmallmatrix}\Big)
\end{equation}
defines a Kasparov module in $KK(\TT_X,A)$. Since we have given the second copy of $\FF_X$ the opposite grading to the first copy, the operator $\begin{psmallmatrix}
0&1\\
1&0\\
\end{psmallmatrix}$
is odd. This operator is also clearly self adjoint and square one. We can then compute the graded commutator using Equation \eqref{Neat}:
\begin{align}
\label{Comm}\Big[\begin{pmatrix}
0&1\\
1&0\\
\end{pmatrix},\pi_0\oplus(\pi_1\circ\alpha_{\TT})(T)\Big]^{\gr}&=\begin{pmatrix}
0&\pi_1(\alpha_{\TT}(T))\\
\pi_0(T)&0\\
\end{pmatrix}-
\begin{pmatrix}
0&\pi_0(\alpha_{\TT}(T))\\
\pi_1(T)&0\\
\end{pmatrix}\\
&=\begin{pmatrix}
0&\pi_1(\alpha_{\TT}(T))-\pi_0(\alpha_{\TT}(T))\\
\pi_0(T)-\pi_1(T)&0\\
\end{pmatrix}.\nonumber\\\nonumber
\end{align}
We wish to show the operator $\pi_0(T)-\pi_1(T)$ is compact. It suffices to establish this for $T=T_x$. Let $\{e_i\}$ be an approximate identity for $A$. We claim that for $x\in X^{\otimes n}$, the operator $\pi_0(T_x)-\pi_1(T_x)$ is the norm limit of the rank one operators $\Theta_{x,e_i}$. Fix $\ep>0$. Choose $i$ large enough such that
\begin{equation*}
\normof{\IP{x,x}-\IP{x,x}e_i}<\ep/2.
\end{equation*}
Note that the operators $\pi_0(T_x)-\pi_1(T_x)$ and $\Theta_{x,e_i}$ act non-trivially only on the sub-module $A=X^{\otimes 0}\subset\FF_X$ where we have for $b\in A$
\begin{equation*}
\left(\pi_0(T_x)-\pi_1(T_x)\right)b=xb\quad\text{ and }\quad \Theta_{x,e_i}b=x(e_ib).
\end{equation*}
Since both operators act non-trivially only on $A$ we have
\begin{equation*}
\normof{\pi_0(T_x)-\pi_1(T_x)-\Theta_{x,e_i}}=\sup_{\normof{b}\leq 1, b\in A\subset \FF_X}\normof{\left(\pi_0(T_x)-\pi_1(T_x)-\Theta_{x,e_i}\right)b}.
\end{equation*}
This gives
\begin{align*}
\normof{\pi_0(T_x)-\pi_1(T_x)-\Theta_{x,e_i}}^2&=\sup_{\normof{b}\leq 1, b\in A}=\normof{\left(\pi_0(T_x)-\pi_1(T_x)-\Theta_{x,e_i}\right)b}^2\\
&=\sup_{\normof{b}\leq 1, b\in A}\normof{xb-xe_ib}^2\\
&\leq\sup_{\normof{b}\leq 1, b\in A}\normof{x-xe_i}^2\normof{b}^2\\
&\leq \normof{x-xe_i}^2\\
&=\normof{\IP{x-xe_i,x-xe_i}}\\
&=\normof{\IP{x,x}-\IP{x,x}e_i-e_i\IP{x,x}+e_i\IP{x,x}e_i}\\
&\leq \normof{\IP{x,x}-\IP{x,x}e_i}+\normof{e_i}\normof{\IP{x,x}-\IP{x,x}e_i}\\
&<\ep,\\
\end{align*}
and so we do indeed have $\pi_0(T_x)-\pi_1(T_x)\in\KK(\FF_X)$. Since $\KK(\FF_X)$ is a $C^*$-sub-algebra of $\LL(\FF_X)$, and $\pi_i$ is a homomorphism for $i=0,1$, and $T_x$ generate $\TT_X$, we deduce that $\pi_0(T)-\pi_1(T)$ is compact for all $T\in\TT_X$. Thus equation \eqref{Comm} shows that 
\begin{equation*}
[\begin{psmallmatrix}
0&1\\
1&0\\
\end{psmallmatrix},\pi_0\oplus(\pi_1\circ\alpha)(a)]^{\gr}\in\KK(\FF_X).
\end{equation*}
Hence $M$ in Equation \eqref{M} defines a Kasparov module. Using the notation of \cite{ThePaper}, let $\FF_X\ominus A$ be the sub-module $\oplus_{n=1}^{\infty}X^{\otimes n}\subset\FF_X$. Since the essential subspace of the module $M$ is $\FF_X\oplus(\FF_X\ominus A)$, Lemma~\ref{essential} tells us that if $P$ if the projection onto $\FF_X\ominus A$, then the class $[M]$ is represented by the Kasparov module
\begin{equation}\label{M essential}
[M]=\left[\FF_X\oplus(\FF_X\ominus A),\pi_0\oplus(\pi_1\circ\alpha_{\TT}),\begin{psmallmatrix}
0&1\\
P&0\\
\end{psmallmatrix},\begin{psmallmatrix}
\alpha_{\infty}&0\\
0&-\alpha_{\infty}\\
\end{psmallmatrix}\right].
\end{equation}
Note that our classes $[i_A]$ and $[M]$ give the classes $\alpha$ and $\beta$ in \cite{FreeProbTheory} when $A$ is trivially graded.
\begin{theorem}[{cf.\cite[Theorem~4.2]{ThePaper}, \cite[Theorem~4.4]{FreeProbTheory}}]\label{KK equivalence}
	With notation as above, the pair $[M]$ and $[i_A]$ are mutually inverse. That is, $[i_a]\otimes_{\TT_X}[M]=[\text{id}_A]$ and $[M]\otimes_A[i_A]=[\text{id}_{\TT_X}]$. In particular, $A$ and $\TT_X$ are $KK$-equivalent as graded $C^*$-algebras.
\end{theorem}
\begin{proof}
	Since $[i_A]$ is the class of a homomorphism, Equations \eqref{triv left} and \eqref{triv right} show us how to compute the Kasparov products $[i_A]\Otimes _{\TT_X}[M]$ and $[M]\Otimes _{A}[i_A]$. Using the representative \eqref{M essential} of $[M]$, we obtain
	\begin{equation}\label{Brian}
	[i_A]\Otimes _{\TT_X}[M]=\left[\FF_X\oplus(\FF_X\ominus A),(\pi_0\oplus\pi_1\circ\alpha_{\TT})\circ i_A,\begin{psmallmatrix}
	0&1\\
	P&0\\
	\end{psmallmatrix},\begin{psmallmatrix}
	\alpha_{\infty}&0\\
	0&-\alpha_{\infty}\\
	\end{psmallmatrix}\right].
	\end{equation}
	To show that \eqref{Brian} is the class of the identity on $A$, we will add $-[\text{id}]$ and show that we are left with the class of a degenerate module. Recalling Equation \eqref{minus id} we obtain
	\begin{align*}
	\big(\FF_X\oplus(\FF_X\ominus A),(\pi_0\oplus&\pi_1\circ\alpha_{\TT})\circ i_A,\begin{psmallmatrix}
	0&1\\
	P&0\\
	\end{psmallmatrix},\begin{psmallmatrix}
	\alpha_{\infty}&0\\
	0&-\alpha_{\infty}\\
	\end{psmallmatrix} \big)\oplus(A,\alpha,0,-\alpha_X)\\
	&\cong
	\left(\FF_X\oplus\FF_X,(\pi_0\oplus\pi_0\circ\alpha_{\TT})\circ i_A,\begin{psmallmatrix}
	0&P\\
	P&0\\
	\end{psmallmatrix},\begin{psmallmatrix}
	\alpha_{\infty}&0\\
	0&-\alpha_{\infty}\\
	\end{psmallmatrix}\right)\\
	\end{align*}
	via the unitary $U:\FF_X\oplus(\FF_X\ominus A)\oplus A\to\FF_X\oplus\FF_X$ given by $U(x,y,a)=(x,y-a)$. Note that there is a mistake in the proof of \cite[Theorem~4.2]{ThePaper} where the Fredholm operator on the right is written as $\begin{psmallmatrix}
	0&1\\
	1&0\\
	\end{psmallmatrix}$. Now taking the straight line operator homotopy 
	$\begin{psmallmatrix}
	0&P+t(1-P)\\
	P+t(1-P)&0\\
	\end{psmallmatrix}$
	from
	$\begin{psmallmatrix}
	0&P\\
	P&0\\
	\end{psmallmatrix}$
	to
	$\begin{psmallmatrix}
	0&1\\
	1&0\\
	\end{psmallmatrix}$
	shows that
	\begin{align*}
	\big[\FF_X\oplus&\FF_X,(\pi_0\oplus\pi_0\circ\alpha_{\TT})\circ i_A,\begin{psmallmatrix}
	0&P\\
	P&0\\
	\end{psmallmatrix},\begin{psmallmatrix}
	\alpha_{\infty}&0\\
	0&-\alpha_{\infty}\\
	\end{psmallmatrix}\big]\\
	&=\left[\FF_X\oplus\FF_X,(\pi_0\oplus\pi_0\circ\alpha_{\TT})\circ i_A,\begin{psmallmatrix}
	0&1\\
	1&0\\
	\end{psmallmatrix},\begin{psmallmatrix}
	\alpha_{\infty}&0\\
	0&-\alpha_{\infty}\\
	\end{psmallmatrix}\right].\\
	\end{align*}
	Thus,
	\begin{equation}\label{Deng}
	[i_A]\Otimes _{\TT_X}[M]=\left[\FF_X\oplus\FF_X,(\pi_0\oplus\pi_0\circ\alpha_{\TT})\circ i_A,\begin{psmallmatrix}
	0&1\\
	1&0\\
	\end{psmallmatrix},\begin{psmallmatrix}
	\alpha_{\infty}&0\\
	0&-\alpha_{\infty}\\
	\end{psmallmatrix}\right].
	\end{equation}
	Since $\begin{psmallmatrix}
	0&1\\
	1&0\\
	\end{psmallmatrix}$
	commutes with $(\pi_0\oplus\pi_0\circ\alpha_{\TT})\circ i_A$, the right-hand module in Equation \eqref{Deng} is degenerate and so represents the zero class. We deduce that $[i_A]\Otimes _{\TT_X}[M]=[\text{id}_A]$.\\
	For the reverse composition, \eqref{triv right} shows us that
	\begin{equation*}
	[M]\Otimes _{\TT_X}[i_A]=\Big[(\FF_X\oplus\FF_X)\Otimes \TT_X,(\pi_0\oplus\pi_1\circ\alpha_{\TT})\Otimes 1_{\TT_X},\begin{psmallmatrix}
	0&1\\
	1&0\\
	\end{psmallmatrix},\begin{psmallmatrix}
	\alpha_X^{\infty}\Otimes \alpha_{\TT}&0\\
	0&-\alpha_X^{\infty}\Otimes \alpha_{\TT}\\
	\end{psmallmatrix}\Big]
	\end{equation*}
	where we have identified $(\FF_X\oplus\FF_X)\Otimes \TT_X$ with $(\FF_X\Otimes \TT_X)\oplus(\FF_X\Otimes \TT_X)$ in defining our Fredholm operator. Continuing this identification, we write $\pi_0'$ and $\pi_1'$ for the left actions $\pi_0\Otimes 1_{\TT_X}$ and $\pi_1\circ\alpha_{\TT}\Otimes 1_{\TT_X}$ on $\FF_X\Otimes \TT_X$. Since the left action $\varphi$ of $A$ on $X$ is non-degenerate, so too is the natural left action $i_A$ of $A$ on $\TT_X$. Thus Example~\ref{left and right} shows that the balanced tensor product $A\Otimes _A \TT_X$ over this left action is isomorphic to $\TT_X$ via the unitary $U:A\otimes\TT_X\to\TT_X$ determined by 
	\begin{equation}\label{U}
	U(a\otimes T)=\varphi^{\infty}(a)T. 
	\end{equation}
	This allows us to define a left action of $\TT_X$ on $A\Otimes \TT_X$ by 
	\begin{equation*}
	S\cdot(a\otimes T)=U^{-1}\alpha_{\TT}(S)U(a\otimes T).
	\end{equation*}
	This action is compact since $\TT_X$ is acting on itself as in Example~\ref{Fang}.	We can then extend this to a compact action $\tau$ of $\TT_X$ on all of $\FF_X\Otimes \TT_X$ by defining the action to be trivial on $X^{\otimes n}\Otimes \TT_X$ when $n\geq 1$. Since $\tau(S)$ is compact for all $S\in \TT_X$, we obtain a Kasparov module
	\begin{equation*}			\left((\FF_X\oplus\FF_X)\Otimes _A\TT_X,(0\oplus \tau),\begin{psmallmatrix}
	0&1\\
	1&0
	\end{psmallmatrix},\begin{psmallmatrix}
	\alpha_X^{\infty}\Otimes \alpha_{\TT}&0\\
	0&-\alpha_X^{\infty}\Otimes \alpha_{\TT}\\
	\end{psmallmatrix}\right).
	\end{equation*}
	By construction the essential subspace of this module is the sub-module $(0\oplus A)\Otimes \TT_X$ in $(\FF_X\oplus\FF_X)\Otimes \TT_X$. The projection of $(\FF_X\oplus\FF_X)\Otimes \TT_X$ onto this essential subspace is $\begin{psmallmatrix}
	0&0\\
	0&1-P\\
	\end{psmallmatrix}$, so by Lemma~\ref{essential} we obtain 
	\begin{align*}
	\Big[(\FF_X\oplus\FF_X)\Otimes _A\TT_X,&(0\oplus \tau),\begin{psmallmatrix}
	0&1\\
	1&0\\
	\end{psmallmatrix},\begin{psmallmatrix}
	\alpha_X^{\infty}\Otimes \alpha_{\TT}&0\\
	0&-\alpha_X^{\infty}\Otimes \alpha_{\TT}\\
	\end{psmallmatrix} \Big]\\
	&=\left[(0\oplus A)\Otimes \TT_X,(0\oplus\tau),\begin{psmallmatrix}
	0&0\\
	0&1-P\\
	\end{psmallmatrix}\begin{psmallmatrix}
	0&1\\
	1&0\\
	\end{psmallmatrix}\begin{psmallmatrix}
	0&0\\
	0&1-P\\
	\end{psmallmatrix},
	\begin{psmallmatrix}
	\alpha_X^{\infty}\Otimes \alpha_{\TT}&0\\
	0&-\alpha_X^{\infty}\Otimes \alpha_{\TT}\\
	\end{psmallmatrix}\right]\\
	&=\left[A\Otimes \TT_X,\tau,0,\alpha_A\Otimes \alpha_{\TT}\right].\\
	\end{align*}
	The unitary $U$ defined in Equation \eqref{U} defines a unitary equivalence
	\begin{equation*}
	\left[A\Otimes \TT_X,\tau,0,\alpha_A\Otimes \alpha_{\TT}\right]=[\TT_X,\alpha_{\TT},0,-\alpha_{\TT}]=-[\text{id}_{\TT_X}],
	\end{equation*}
	where we have used \eqref{minus id} in the last equality. Putting this together we deduce that
	\begin{equation*}
	\Big[(\FF_X\oplus\FF_X)\Otimes _A\TT_X,(0\oplus \tau),\begin{psmallmatrix}
	0&1\\
	1&0\\
	\end{psmallmatrix},\begin{psmallmatrix}
	\alpha_X^{\infty}\Otimes \alpha_{\TT}&0\\
	0&-\alpha_X^{\infty}\Otimes \alpha_{\TT}\\
	\end{psmallmatrix} \Big]=-[\text{id}_{\TT_X}].
	\end{equation*}
	Since $\pi_0'\oplus\pi_1'$ and $0\oplus \tau$ have complimentary essential subspaces in $(\FF_X\oplus\FF_X)\Otimes \TT_X$ we obtain,
	\begin{align}
	\nonumber[M]\Otimes _A[i_A]-[\text{id}_{\TT_X}]&=\Big[(\FF_X\oplus\FF_X)\Otimes \TT_X,(\pi_0\oplus\pi_1\circ\alpha_{\TT})\Otimes 1_{\TT_X},\begin{psmallmatrix}
	0&1\\
	1&0\\
	\end{psmallmatrix},\begin{psmallmatrix}
	\alpha_X^{\infty}\Otimes \alpha_{\TT}&0\\
	0&-\alpha_X^{\infty}\Otimes \alpha_{\TT}\\
	\end{psmallmatrix}\Big]\\\nonumber
	&\qquad\quad+\Big[(\FF_X\oplus\FF_X)\Otimes _A\TT_X,(0\oplus \tau),\begin{psmallmatrix}
	0&1\\
	1&0\\
	\end{psmallmatrix},\begin{psmallmatrix}
	\alpha_X^{\infty}\Otimes \alpha_{\TT}&0\\
	0&-\alpha_X^{\infty}\Otimes \alpha_{\TT}\\
	\end{psmallmatrix} \Big]\\
	&=\Big[(\FF_X\oplus\FF_X)\Otimes _A\TT_X,(\pi'_0\oplus \pi'_1+\tau),\begin{psmallmatrix}
	0&1\\
	1&0\\
	\end{psmallmatrix},\begin{psmallmatrix}
	\alpha_X^{\infty}\Otimes \alpha_{\TT}&0\\
	0&-\alpha_X^{\infty}\Otimes \alpha_{\TT}\\
	\end{psmallmatrix} \Big].\label{zero class}\\\nonumber
	\end{align}
	We claim that \eqref{zero class} is the zero class in $KK(\TT_X,\TT_X)$, and aim to prove it by finding a homotopy of graded homomorphisms $\pi'_t:\TT_X\to \LL(\FF_X\Otimes _A\TT_X)$ from $\pi_0'\circ\alpha_{\TT}$ to $\pi_1'+\tau$ such that for each $t\in[0,1]$,
	\begin{equation}\label{KAS}
	\left((\FF_X\oplus\FF_X)\Otimes _A\TT_X,\pi_0'\oplus\pi'_t,\begin{psmallmatrix}
	0&1\\
	1&0\\
	\end{psmallmatrix},\begin{psmallmatrix}
	\alpha_X^{\infty}\Otimes \alpha_{\TT}&0\\
	0&-\alpha_X^{\infty}\Otimes \alpha_{\TT}\\
	\end{psmallmatrix}\right)
	\end{equation}
	is a Kasparov module. To obtain $\pi'_t$ we construct a covariant representation $(\tilde{\varphi}^{\infty}\circ \alpha_A,\psi_t)$ of $X$ in $\LL(\FF_X\Otimes _A\TT_X)$ and exploit the universal property Theorem~\ref{theorem: universal property of Toeplitz} of $\TT_X$ to extend this representation to the desired homotopy $\pi_t'$. For each $t$ define $\psi_t:X\to\LL(\FF_X\Otimes _A\TT_X)$ by
	\begin{equation*}
	\psi_t(\xi)=\left(\cos(\pi t/2)(\pi_0'(\alpha_{\TT}(T_{\xi}))-\pi_1'(T_{\xi}))+\sin(\pi t/2)\tau(T_{\xi})\right)+\pi'_1(T_{\xi}).
	\end{equation*}
	Recall that we denoted the diagonal left action action of $A$ on $\FF_X$ by $\varphi^{\infty}$. We define $\tilde{\varphi}^{\infty}$ to be $\varphi^{\infty}\Otimes 1_{\TT_X}$. Since $\pi_0$ and $\pi_1$ together with $\varphi^{\infty}$ form covariant representations of $X$ in $\TT_X$ we have the equalities
	\begin{align*}
	\pi'_1(T_{\xi})\tilde{\varphi}^{\infty}(\alpha_A(a))&=(\pi_1(\alpha_{\TT}(T_{\xi}))\varphi^{\infty}(\alpha_A(a)))\Otimes 1_{\TT_X}\\
	&=(\pi_1((T_{\alpha_X(\xi)}))\varphi^{\infty}(\alpha_A(a)))\Otimes 1_{\TT_X}\\
	&=(\pi_1(T_{\alpha_X(\xi)\alpha_A(a)}))\Otimes 1_{\TT_X}\\
	&=(\pi_1(T_{\alpha_X(\xi a)}))\Otimes 1_{\TT_X}\\
	&=\pi'_1(T_{\xi a}).\\
	\end{align*}
	The same reasoning gives us 
	\begin{align*}
	\tilde{\varphi}^{\infty}(\alpha_A(a))\pi'_1(T_{\xi})&=\pi'_1(T_{\varphi(a)\xi}),\\
	\tilde{\varphi}^{\infty}(\alpha_A(a))\pi_0'(\alpha(T_{\xi}))&=\pi_0'(\alpha_{\TT}(T_{\varphi(a)\xi})),\text{ and }\\
	\pi_0'(\alpha_{\TT}(T_{\xi}))\tilde{\varphi}^{\infty}(\alpha_A(a))&=\pi_0'(\alpha_{\TT}(T_{\xi a})).\\
	\end{align*}
	We wish to show that 
	\begin{equation*}
	\tau(T_{\xi a})\tilde{\varphi}^{\infty}(\alpha_A(a))=\tau(T_{\xi a})
	\end{equation*}
	and
	\begin{equation*}
	\tilde{\varphi}^{\infty}(\alpha_A(a))\tau(T_{\xi a})=\tau(T_{\varphi(a)\xi}).
	\end{equation*}
	This is trivially true on $(\FF_X\oplus\FF_X\ominus A)\Otimes \TT_X$ since $\tau$ acts as zero on this sub-module, so we check that for $b\otimes T\in A\Otimes \TT_X$
	\begin{align*}
	\tau(T_{\xi})\tilde{\varphi}^{\infty}(\alpha_A(a))(b\otimes T)&=U^{-1}\alpha_{\TT}(T_{\xi})U\tilde{\varphi}^{\infty}(\alpha_A(a))(b\otimes T)\\
	&=U^{-1}\alpha_{\TT}(T_{\xi})U(\alpha_A(a)b\otimes T)\\
	&=U^{-1}\alpha_{\TT}(T_{\xi})\varphi^{\infty}(\alpha_A(a)b)T\\
	&=U^{-1}\alpha_{\TT}(T_{\xi a})\varphi^{\infty}(b)T\\
	&=\tau(T_{\xi a})(b\otimes T).\\
	\end{align*}
	A similar argument then gives
	\begin{equation*}
	\tilde{\varphi}^{\infty}(\alpha_A(a))\tau(T_{\xi})=\tau(T_{\varphi(a)\xi}).
	\end{equation*}
	Altogether we then obtain
	\begin{equation*}
	\tilde{\varphi}^{\infty}(\alpha_A(a))\psi_t(\xi)=\psi_t(\varphi(a)\xi)\qquad\text{ and }\qquad \psi_t(\xi)\tilde{\varphi}^{\infty}(\alpha_A(a))=\psi_t(\xi a).
	\end{equation*}
	Next we check that $\psi_t$ is compatible with the inner-product. Note that for all $\xi,\zeta\in X$ and $T\in\TT_X$ we have
	\begin{align*}
	\Image(\pi_0'(\alpha_{\TT}(T_{\xi}))-\pi_1'(T_{\xi}))&\subseteq X\Otimes\TT_X,\\
	\Image(\tau(T))&\subseteq A\Otimes \TT_X,\text{ and }\\
	\Ker(\pi'_1(T_{\zeta}^*))&\subseteq (\FF_X\ominus X\ominus A)\Otimes\TT_X\\
	\end{align*}
	which are all mutually orthogonal. Thus, since $\pi'_1(T_{\xi})^*=\pi'_1(T^*_{\xi})$ we have
	\begin{align*}
	\psi_t(\xi)^*\psi_t(\eta)&=\big(\left(\cos(\pi t/2)(\pi_0'(\alpha_{\TT}(T_{\xi}))-\pi_1'(T_{\xi}))+\sin(\pi t/2)\tau(T_{\xi})\right)+\pi'_1(T_{\xi})\big)^*\\
	&\quad\big(\left(\cos(\pi t/2)(\pi_0'(\alpha_{\TT}(T_{\eta}))-\pi_1'(T_{\eta}))+\sin(\pi t/2)\tau(T_{\eta})\right)+\pi'_1(T_{\eta})\big)\\
	&=\big(\left(\cos(\pi t/2)(\pi_0'(\alpha_{\TT}(T_{\xi}))-\pi_1'(T_{\xi}))+\sin(\pi t/2)\tau(T_{\xi})\right)\big)^*\\
	&\quad\big(\cos(\pi t/2)(\pi_0'(\alpha_{\TT}(T_{\eta}))-\pi_1'(T_{\eta}))+\sin(\pi t/2)\tau(T_{\eta})\big)+\pi'_1(T_{\xi})^*\pi'_1(T_{\eta}).\\
	\end{align*}
	Writing $\tilde{P}$ for the projection onto $(\FF_X\ominus A)\Otimes\TT_X$, we note that $\pi'_1$ can be written as $$\pi'_1=\tilde{P}(\pi_0'\circ\alpha_{\TT})\tilde{P}.$$ Since $\pi'_1$ is a homomorphism we then have
	\begin{equation*}
	\pi_1'(T_{\xi})^*\pi_1'(T_{\eta})=\pi_1'(T_{\xi}^*T_{\eta})=\pi'_1(\IP{\xi,\eta})=\tilde{P}\tilde{\varphi}^{\infty}(\alpha_A(\IP{\xi,\eta}))\tilde{P},
	\end{equation*}
	and similarly
	\begin{equation*}
	\pi'_0(T_{\xi})^*\pi'_0(T_{\eta})=\tilde{\varphi}^{\infty}(\IP{\xi,\eta}).
	\end{equation*}
	For $\zeta\in X$ we have
	\begin{equation*}
	\pi'_1(T_{\zeta})-\pi'_0(\alpha_{\TT}(T_{\zeta}))=\pi_0'(\alpha_{\TT}(T_{\zeta}))(1-\tilde{P}).
	\end{equation*}
	We compute that for any $a,b\in A$ and $T\in \TT_X$
	\begin{align*}
	\tau(\varphi^{\infty}(a))(b\otimes T)&=U^{-1}\alpha_{\TT}(\varphi^{\infty}(a))U(b\otimes T)\\
	&=U^{-1}\varphi^{\infty}(\alpha_A(a))\varphi^{\infty}(b)T\\
	&=U^{-1}\varphi^{\infty}(\alpha_A(a)b)T\\
	&=\alpha_A(a)b\otimes T.\\
	\end{align*}
	Since $\tau$ acts trivially on $(\FF_X\ominus A)\Otimes \TT_X$ we
	deduce that $\tau(\varphi^{\infty}(a))=(1-\tilde{P})\varphi^{\infty}(\alpha_A(a))(1-\tilde{P})$.
	We resume our computation of $\psi_t(\xi)^*\psi_t(\eta)$:
	\begin{align*}
	&\big(\left(\cos(\pi t/2)(\pi_0'(\alpha_{\TT}(T_{\xi}))-\pi_1'(T_{\xi}))+\sin(\pi t/2)\tau(T_{\xi})\right)\big)^*\\
	&\quad\big(\cos(\pi t/2)(\pi_0'(\alpha_{\TT}(T_{\eta}))-\pi_1'(T_{\eta}))+\sin(\pi t/2)\tau(T_{\eta})\big)\\
	&=\cos^2(\pi t/2)\big(\pi_0'(\alpha_{\TT}(T_{\xi}))(1-\tilde{P})\big)^*\pi_0'(\alpha_{\TT}(T_{\eta}))(1-\tilde{P})+\sin^2(\pi t/2)\tau(T_{\xi})^*\tau(T_{\eta})\\
	&=\cos^2(\pi t/2)(1-\tilde{P})\pi_0'(\alpha_{\TT}(T_{\xi}))^*\pi_0'(\alpha_{\TT}(T_{\eta}))(1-\tilde{P})+\sin^2(\pi t/2)\tau(\varphi^{\infty}(\IP{\xi,\eta}))\\
	&=(1-\tilde{P})\Big(\cos^2(\pi t/2)\tilde{\varphi}^{\infty}(\alpha_{A}(\IP{\xi,\eta}))+\sin^2(\pi t/2)\tilde{\varphi}^{\infty}(\alpha_A(\IP{\xi,\eta}))\Big)(1-\tilde{P})\\
	&=(1-\tilde{P})\tilde{\varphi}^{\infty}(\alpha_{A}(\IP{\xi,\eta}))(1-\tilde{P}).\\
	\end{align*}
	Since $\tilde{\varphi}^{\infty}\tilde{P}=\tilde{P}\tilde{\varphi}^{\infty}$, we deduce that
	\begin{align*}
	\psi_t(\xi)^*\psi_t(\eta)&=(1-\tilde{P})\tilde{\varphi}^{\infty}(\alpha_{A}(\IP{\xi,\eta}))(1-\tilde{P})+\tilde{P}\tilde{\varphi}^{\infty}(\alpha_{A}(\IP{\xi,\eta}))\tilde{P}\\
	&=\tilde{\varphi}^{\infty}(\alpha_{A}(\IP{\xi,\eta}))(1-\tilde{P})^2+\tilde{\varphi}^{\infty}(\alpha_{A}(\IP{\xi,\eta}))\tilde{P}^2\\
	&=\tilde{\varphi}^{\infty}(\alpha_{A}(\IP{\xi,\eta})),\\
	\end{align*}
	so we deduce that for each $t\in[0,1]$, $(\tilde{\varphi}^{\infty}\circ\alpha_A,\psi_t)$ is a covariant representation of $X$ in $\LL(\FF_X\Otimes\TT_X)$. Thus the universal property Theorem~\ref{theorem: universal property of Toeplitz} gives us a homomorphism $\pi'_t:\TT_X\to\LL(\FF_X\Otimes\TT_X)$ such that $\pi'_t(T_{\xi})=\psi_t(\xi)$ for $\xi\in X$ and $\pi'_t(a)=\tilde{\varphi}^{\infty}(\alpha_A(a))$ for $a\in A$.\\
	\quad We wish to show that this operator is compact, so that $\pi'_1$ is a compact perturbation of $\pi'_t$. For $t\in[0,1]$ and $\xi\in X$ we have
	\begin{equation*}
	\pi'_t(T_{\xi})-\pi'_1(T_{\xi})=\cos(\pi t/2)(\pi_0'(\alpha_{\TT}(T_{\xi}))-\pi_1'(T_{\xi}))+(\sin(\pi t/2)-1)\tau(T_{\xi}).
	\end{equation*}
	This operator vanishes on $\tilde{P}(\FF_X\Otimes\TT_X)$, so to see that $\pi'_t-\pi_1'$ is compact we need only show it on $A\Otimes\TT_X$. Recall from Lemma~\ref{compact otimes 1} that an operator $k\Otimes 1$ on $A\Otimes \TT_X$ is compact if and only if $k$ is a compact operator on $A\cdot I$ where $I$ is the ideal $(\varphi^{\infty})^{-1}(A)$ of elements in $A$ which act compactly by the left action $\varphi^{\infty}$ of $A$ on $\TT_X$. Since $\varphi^{\infty}$ is an inclusion of $A$ in $\TT_X$ so that the left action is simply left multiplication, Example~\ref{Fang} tells us that the left action is entirely by compacts. Thus, $I=A$ and so $A\cdot I=A$ and we deduce that $k\otimes 1$ is compact on $A\Otimes\TT_X$ if and only if $k$ is compact on $A$. Example \eqref{Fang} again tells us that the left action of multiplication by $A$ on $A$, which is the restriction of $\varphi^{\infty}$ to $A\subset\FF_X$ is compact. Thus, when restricted to $A\Otimes\TT_X$, $\tilde{\varphi}^{\infty}=\varphi^{\infty}\Otimes 1_{\TT_X}$ is compact. By Cohen Factorisation there exists $y\in X$ such that $\xi=y\IP{y,y}$, and we deduce that on $A\Otimes\TT$,
	\begin{equation*}
	\pi'_t(T_{\xi})-\pi'_1(T_{\xi})=\pi'_t(T_{y\IP{y,y}})-\pi'_1(T_{y\IP{y,y}})=\left(\pi'_t(T_{y})-\pi'_1(T_{y})\right)\tilde{\varphi}^{\infty}(\IP{y,y})
	\end{equation*}
	is compact. Since the elements $T_{\xi}$ generate $\TT_X$, we see that $\pi'_t(a)-\pi'_1(a)$ is a compact operator for every $a\in \TT_X$. Thus $\pi'_t$ is a compact perturbation of $\pi'_1$, and we see that \eqref{KAS} is a Kasparov module for all $t\in[0,1]$. Thus we have a homotopy of left actions and we deduce that \eqref{zero class} is represented by the class
	\begin{equation*}
	\left((\FF_X\oplus\FF_X)\Otimes _A\TT_X,\pi_0'\oplus\pi'_0,\begin{psmallmatrix}
	0&1\\
	1&0\\
	\end{psmallmatrix},\begin{psmallmatrix}
	\alpha_X^{\infty}\Otimes \alpha_{\TT}&0\\
	0&-\alpha_X^{\infty}\Otimes \alpha_{\TT}\\
	\end{psmallmatrix}\right).
	\end{equation*}
	This class is degenerate and so \eqref{zero class} is in fact the zero class. Hence $[M]\Otimes[i_A]=[\text{id}_{\TT_X}]$.
\end{proof}

Let $I$ be the ideal $I=\varphi^{-1}(\KK(X))$ of Lemma~\ref{compact otimes 1}. We have seen in Lemma~\ref{comapcts inclusion in Toeplitz} that there is an embedding
\begin{equation*}
	j:\KK(\FF_{X,I})\hookrightarrow \TT_X
\end{equation*}
where $\FF_{X,I}$ denotes the Fock space as a right $I$ module. The embedding $j$ defines a Kasparov $\KK(\FF_{X,I})$-$\TT_X$ module $(\TT_X,j,0,\alpha_{\TT})$ whose class in $KK(\KK(\FF_{X,I}),\TT_X)$ we denote $[j]$.
If 
\begin{equation*}
	\iota:\KK(\FF_{X,I})\to\LL(\FF_{X,I})
\end{equation*}
is the inclusion map then $\iota$ also defines a class
\begin{equation*}
	[\FF_{X,I},\iota,0,\alpha_X^{\infty}]\in KK(\KK(\FF_{X,I}),I).
\end{equation*} 
Let $[X]=[X,\varphi,0,\alpha_X]\in KK(I,A)$, and let $\iota_I:I\hookrightarrow A$ be the inclusion of $I$ in $A$. The inclusion $\iota_I$ also defines a Kasparov $I$-$A$ module $(A_A,\iota_I,0,\alpha_A)$ whose class in $KK(I,A)$ we denote by $[\iota_I]$.

\begin{lemma}\label{Hard work happened}
	With notation as above we have
	\begin{equation*}
		[j]\Otimes_{\TT_X}[M]=[\FF_{X,I},\iota,0,\alpha_X^{\infty}]\Otimes_I([\iota_I]-[X]).
	\end{equation*}
\end{lemma}
\begin{proof}
	Since $[j]$ is the class of a trivial module we may use \eqref{triv right} to compute
	\begin{equation*}
		[j]\Otimes_{\TT_X}[M]=\Big[\FF_X\oplus(\FF_X\ominus A),(\pi_0\oplus \pi_1\circ\alpha_{\TT})\circ j,\begin{psmallmatrix}
		0&1\\
		P&0\\
		\end{psmallmatrix},\begin{psmallmatrix}
		\alpha_X^{\infty}&0\\
		0&-\alpha_X^{\infty}\\
		\end{psmallmatrix}\Big].
	\end{equation*}
	Since $\pi_0\circ j$ and $\pi_1\circ\alpha_A\circ j\subseteq \KK(\FF_X)$, the straight line joining $\begin{psmallmatrix}
		0&1\\
		P&0\\
	\end{psmallmatrix}$ to $0$ is an operator homotopy and we may write
	\begin{align}
		\nonumber[j]\Otimes_{\TT_X}[M]&=\Big[\FF_X\oplus(\FF_X\ominus A),(\pi_0\oplus \pi_1\circ\alpha_{\TT})\circ j,0,\begin{psmallmatrix}
		\alpha_X^{\infty}&0\\
		\nonumber0&-\alpha_X^{\infty}\\
		\end{psmallmatrix}\Big]\\
		\nonumber&=\left[\FF_X,\pi_0\circ j,0,\alpha_X^{\infty}\right]+\left[\FF_X\ominus A,\pi_1\circ\alpha_{\TT}\circ j,0,-\alpha_X^{\infty}\right]\\
		\label{Eqn}&=\left[\FF_X,\iota,0,\alpha_X^{\infty}\right]+\left[\FF_X\ominus A,\pi_1\circ j\circ \alpha_{\KK},0,-\alpha_X^{\infty}\right].\\\nonumber
	\end{align}
	We have $\overline{\KK(\FF_{X,I})\FF_X}=\FF_{X,I}$ and similarly $\overline{\KK(\FF_{X,I})\FF_X\ominus A}=\FF_{X,I}\ominus A$. Thus we may replace each representative module in \eqref{Eqn} with its essential subspace to obtain
	\begin{equation*}
		[j]\Otimes_{\TT_X}[M]=\left[\FF_{X,I},\iota,0,\alpha_X^{\infty}\right]+\left[\FF_{X,I}\ominus A,\pi_1\circ j\circ \alpha_{\KK},0,-\alpha_X^{\infty}\right].
	\end{equation*}
	The map 
	\begin{equation*}
		U:\FF_{X,I}\Otimes_I A_A\to \FF_{X,I}, \quad U(\xi\otimes a)=\xi\cdot a
	\end{equation*} 
	defines an unitary equivalence of Kasparov modules 
	\begin{equation*}
		(\FF_{X,I}\Otimes_I A_A,\iota\Otimes 1,0,\alpha_X^{\infty}\Otimes \alpha_A)\cong (\FF_{X,I},\iota,0,\alpha_X^{\infty}).
	\end{equation*}
	Thus,
	\begin{equation*}
		[\FF_{X,I},\iota,0,\alpha_X^{\infty}]\Otimes_I [\iota_I]=\left[\FF_{X,I},\iota,0,\alpha_X^{\infty}\right].
	\end{equation*}
	There is a well defined map $V:\FF_{X,I}\Otimes_I X\to \FF_{X,I}\ominus A$ such that if $c_i\in \C$ and $\xi_i\in X^{\otimes i}$ then 
	\begin{equation*}
		V\Big(\Big(\sum_i c_i\xi_i\Big)\otimes a\Big)=\sum_i c_i\xi_i\otimes a. 
	\end{equation*}
	The map $V$ defines an isomorphism $\FF_{X,I}\Otimes_I X\cong \FF_{X,I}\ominus A$ which carries $\pi_0\Otimes 1_X$ to $\pi_1$ and hence $(\pi_0\circ j)\otimes 1_X$ to $\pi_1\circ j$. The isomorphism $V$ also carries $-\alpha^{\infty}_X$ to $-\alpha^{\infty}_X\Otimes \alpha_X$, so we deduce that
	\begin{align*}
		\left[\FF_{X,I}\ominus A,\pi_1\circ j\circ \alpha_{\KK},0,-\alpha_X^{\infty}\right]&=\left[\FF_{X,I}\otimes_I X, (\pi_0\circ j\circ\alpha_{\KK})\Otimes 1_X,0,-\alpha_X^{\infty}\Otimes\alpha_X\right]\\
		&=[\FF_{X,I},\pi_0\circ j\circ\alpha_{\KK},0,-\alpha_X^{\infty}]\Otimes_I [X]\\
		&=-[\FF_{X,I},\pi_0\circ j,0,\alpha_X^{\infty}]\Otimes_I [X],\\
	\end{align*}
	where we have used \eqref{minus id} in the last line. We conclude that
	\begin{equation*}
		[j]\Otimes_{\TT_X}[M]=\left[\FF_{X,I},\iota,0,\alpha_X^{\infty}\right]-\left[\FF_{X,I},\pi_0\circ j,0,\alpha_X^{\infty}\right]=[\FF_X,\iota,0,\alpha_X]\Otimes_I([\iota_I]-[X]).\qedhere
	\end{equation*}
\end{proof}
We finally arrive at the main Theorem of this thesis. By comparison with\\ \cite[Theorem~4.4]{ThePaper}, we have removed the hypothesis that the left action of $A$ on $X$ be by compacts, and we obtain a slightly different exact sequence.
\begin{theorem}\label{MAINTHEOREM}(cf. \cite[Theorem~4.9]{FreeProbTheory}, \cite[Theorem 4.4]{ThePaper})
	Let $(A,\alpha_A)$ and $(B,\alpha_B)$ be graded, separable $C^*$-algebras and let $A$ be nuclear. Let $X$ be a full, graded $A$--$A$ correspondence with injective left action $\phi$, and let $I=\phi^{-1}(\KK(X))$. Then with notation as above we have two exact six-term sequences as follows.\footnote{This diagram is borrowed from the course code for \cite{ThePaper}.}
	\begin{equation}\label{eq:exact1}
	\parbox[c]{0.8\textwidth}{\hfill
		\begin{tikzpicture}[yscale=0.8, >=stealth]
		\node (00) at (0,0) {$KK_1(B, \OO_X)$};
		\node (40) at (4,0) {$KK_1(B, I)$};
		\node (80) at (8,0) {$KK_1(B, A)$};
		\node (82) at (8,2) {$KK_0(B, \OO_X)$};
		\node (42) at (4,2) {$KK_0(B, A)$};
		\node (02) at (0,2) {$KK_0(B, I)$};
		\draw[->] (02)-- node[above] {${\scriptstyle{\Otimes_A ([\iota_I] - [X])}}$} (42);
		\draw[->] (42)-- node[above] {${\scriptstyle i_*}$} (82);
		\draw[->] (82)--(80);
		\draw[->] (80)-- node[above] {${\scriptstyle{\Otimes_A ([\iota_I] - [X])}}$} (40);
		\draw[->] (40)-- node[above] {${\scriptstyle i_*}$} (00);
		\draw[->] (00)--(02);
		\end{tikzpicture}\hfill\hfill}
	\end{equation}
	\begin{equation}\label{eq:exact2}
	\parbox[c]{0.8\textwidth}{\hfill
		\begin{tikzpicture}[yscale=0.8, >=stealth]
		\node (00) at (0,0) {$KK_1(\OO_X, B)$};
		\node (40) at (4,0) {$KK_1(A, B)$};
		\node (80) at (8,0) {$KK_1(I, B)$};
		\node (82) at (8,2) {$KK_0(\OO_X, B)$};
		\node (42) at (4,2) {$KK_0(A, B)$};
		\node (02) at (0,2) {$KK_0(I, B)$};
		\draw[<-] (02)-- node[above] {${\scriptstyle{([\iota_I] - [X]) \Otimes_A}}$} (42);
		\draw[<-] (42)-- node[above] {${\scriptstyle i^*}$} (82);
		\draw[<-] (82)--(80);
		\draw[<-] (80)-- node[above] {${\scriptstyle{([\iota_I] - [X]) \Otimes_A}}$} (40);
		\draw[<-] (40)-- node[above] {${\scriptstyle i^*}$} (00);
		\draw[<-] (00)--(02);
		\end{tikzpicture}\hfill\hfill}
	\end{equation}
\end{theorem}
\begin{proof}
	As in \cite[Theorem~4.4]{ThePaper}, we just prove exactness of the first diagram since the second follows from an almost identical argument. Since $A$ is nuclear so is $\TT_X$ by \cite[Theorem~6.3]{ToeplitzNuclear}. The groupoid $\GG$ of \cite[Theorem~6.3]{ToeplitzNuclear} is in this instance the Deaconu-Renault groupoid (see \cite[Example~2.1.16]{Groupoids}) for the action of $\N$ on $\Z \cup \{\infty\}$ by addition, and is therefore amenable by \cite[Lemma~3.5]{GroupoidAmenable}. By Remark~\ref{remark:nuclear implies semi-split} the quotient map $q:\TT_X\to\OO_X$ admits a completely positive splitting. Hence \cite[Theorem~1.1]{Skandalis} applied to the graded short exact sequence
	\begin{equation*}
		0\rightarrow\KK(\FF_{X,I})\xrightarrow{\ \ j\ \ }\TT_X\xrightarrow{\ \ q\ \ }\OO_X\rightarrow 0
	\end{equation*}
	of Theorem~\ref{Cuntz-Pimsner exact sequence} gives connecting homomorphisms $\delta:KK_i(B,\OO_X)\to KK_{i+1}(B,\KK(\FF_{X,I}))$ for which the following six-term exact sequence is exact.\footnote{This diagram too is borrowed from the source code for \cite{ThePaper}.}
	\[
	\begin{tikzpicture}[yscale=0.8]
	\node (00) at (0,0) {$KK_1(B, \OO_X)$};
	\node (40) at (4,0) {$KK_1(B, \TT_X)$};
	\node (80) at (8,0) {$KK_1(B, \KK(\FF_{X,I}) )$.};
	\node (82) at (8,2) {$KK_0(B, \OO_X)$};
	\node (42) at (4,2) {$KK_0(B, \TT_X)$};
	\node (02) at (0,2) {$KK_0(B, \KK(\FF_{X,I}))$};
	\draw[-stealth] (02)-- node[above] {${\scriptstyle j_*}$} (42);
	\draw[-stealth] (42)-- node[above] {${\scriptstyle q_*}$} (82);
	\draw[-stealth] (82)-- node[right] {${\scriptstyle \delta}$} (80);
	\draw[-stealth] (80)-- node[above] {${\scriptstyle j_*}$} (40);
	\draw[-stealth] (40)-- node[above] {${\scriptstyle q_*}$} (00);
	\draw[-stealth] (00)-- node[left] {${\scriptstyle \delta}$}(02);
	\end{tikzpicture}
	\]
	Note that the maps $j_*$ and $q_*$ are the induced maps as in \eqref{triv left}. We define $\delta':KK_*(B,\OO_X)\to KK_{*+1}(B,A)$ by $\delta'=(\cdot \Otimes[\FF_{X,I},\iota,0,\alpha_X^{\infty}])\circ\delta$ and, let $i:A\to\OO_X$ be the inclusion of $A$ into $\OO_X$. Consider the following diagram.\footnote{This diagram is also lovingly borrowed from the source code for \cite{ThePaper}.}
	
	\[
	\begin{tikzpicture}[yscale=0.8]
	\node (00) at (0,0) {$KK_1(B, \OO_X)$};
	\node (40) at (4,0) {$KK_1(B, \TT_X)$};
	\node (80) at (8,0) {$KK_1(B, \KK(\FF_{X,I}) )$};
	\node (82) at (8,2) {$KK_0(B, \OO_X)$};
	\node (42) at (4,2) {$KK_0(B, \TT_X)$};
	\node (02) at (0,2) {$KK_0(B, \KK(\FF_{X,I}))$};
	\draw[-stealth] (02)-- node[above] {${\scriptstyle j_*}$} (42);
	\draw[-stealth] (42)-- node[above] {${\scriptstyle q_*}$} (82);
	\draw[-stealth] (82)-- node[right] {${\scriptstyle \delta}$} (80);
	\draw[-stealth] (80)-- node[below] {${\scriptstyle j_*}$} (40);
	\draw[-stealth] (40)-- node[below] {${\scriptstyle q_*}$} (00);
	\draw[-stealth] (00)-- node[left] {${\scriptstyle \delta}$} (02);
	\node (00') at (-2,-2) {$KK_1(B, \OO_X)$};
	\node (40') at (4,-2) {$KK_1(B, A)$};
	\node (80') at (10,-2) {$KK_1(B, I)$};
	\node (02') at (-2,4) {$KK_0(B, I)$};
	\node (42') at (4,4) {$KK_0(B, A)$};
	\node (82') at (10,4) {$KK_0(B, \OO_X)$};
	\draw[-stealth] (02')-- node[above] {${\scriptstyle \Otimes ([\iota_I] - [X])}$} (42');
	\draw[-stealth] (42')-- node[above] {${\scriptstyle i_*}$} (82');
	\draw[-stealth] (82')-- node[right] {${\scriptstyle \delta'}$} (80');
	\draw[-stealth] (80')-- node[below] {${\scriptstyle \Otimes ([\iota_I] - [X])}$} (40');
	\draw[-stealth] (40')-- node[below] {${\scriptstyle i_*}$} (00');
	\draw[-stealth] (00')-- node[left] {${\scriptstyle \delta'}$} (02');
	\draw[-stealth] (02)--(02') node[pos=0.25, anchor=south west, inner sep=0pt] {$\scriptstyle \Otimes [\FF_{X,I}, \iota, 0, \alpha_X^\infty]$};
	\draw[-stealth, out=75, in=285] (42) to node[pos=0.5, right] {${\scriptstyle\Otimes[M]}$} (42');
	\draw[-stealth, out=255, in=105] (42') to node[pos=0.5, left] {${\scriptstyle (i_A)_*}$} (42);
	\draw[-stealth] (82') to node[anchor=south east, inner sep=1pt] {$\scriptstyle\text{id}$} (82);
	\draw[-stealth] (80)--(80') node[pos=0.25,  anchor=north east, inner sep=0pt] {$\scriptstyle \Otimes [\FF_{X,I}, \iota, 0, \alpha_X^\infty]$};
	\draw[-stealth, out=255, in=105] (40) to node[pos=0.5, left] {${\scriptstyle\Otimes[M]}$} (40');
	\draw[-stealth, out=75, in=285] (40') to node[pos=0.5, right] {${\scriptstyle (i_A)_*}$} (40);
	\draw[-stealth] (00') to node[anchor=north west, inner sep=1pt] {$\scriptstyle\text{id}$} (00);
	\end{tikzpicture}
	\]
	We aim to show that the inside and outside diagrams commute. By definition of the maps $\delta'$ the left-hand and right-hand squares commute. Lemma~\ref{Hard work happened} implies that the top left and bottom right squares commute. By definition we have $q\circ i_A=i$ as homomorphisms, so the induced maps $q_*\circ (i_A)_*$ and $i_*$ are equal as well. So the top right and bottom left squares commute as well. By Theorem~\ref{KK equivalence} the maps $(i_A)_*$ and $\Otimes [M]$ are mutually inverse. Since the class $[\FF_{X,I},\iota,0,\alpha_X^{\infty}]$ defines a Morita equivalence between $\KK(\FF_{X,I})$ and $I$, by Theorem~\ref{Morita equiv KK} we deduce that $\Otimes [\FF_{X,I},\iota,0,\alpha_X^{\infty}]$ is an isomorphism. Thus every map between the inner and outer square is an isomorphism, and it follows that the outer rectangle is exact as required.
\end{proof}

%% file: chapter8.tex
\chapter{Graph algebras}
\label{Graph algebras}

In this Chapter we compute the graded $K$-theory of graph algebras using Theorem~\ref{MAINTHEOREM} as in \cite{ThePaper}. Since~\ref{MAINTHEOREM} does not require compact left actions, we are able to compute graded $K$-theory for graph algebras of row infinite graphs.\\
\begin{lemma}\label{lemma: trivially graded A gives direct sum}
	Let $A$ be a trivially graded $C^*$-algebra, $(_{\phi}X,\alpha_X)$ be a graded $A$-$A$ correspondence and let $I=\phi^{-1}(\KK(X))$. Then $\alpha_X\in\LL(X)$ and it is an even, self-adjoint unitary with respect to $\widetilde{\alpha}_X$. Let
	\begin{align*}
		X_0&=\overline{\Span}\{x+\alpha_X(x):x\in A\}\qquad\text{ and }\\
		X_1&=\overline{\Span}\{x-\alpha_X(x):x\in A\}
	\end{align*}
	be the even and odd subspaces of $X$. Then $X\cong X_0\oplus X_1$ as $A$-$A$ correspondences, and in $KK(I,A)$, we have $[X,\phi,0,\alpha_X]=[X_0,\phi|_{X_0},0,\text{id}]-[X_1,\phi|_{X_1},0,\text{id}]$.
\end{lemma}
\begin{proof}
	Since $A$ is trivially graded, for $x,y\in X$ we have
	\begin{equation*}
		\IP{\alpha_X(x),y}=\IP{\alpha_X(x),\alpha^2_X(y)}=\alpha_A(\IP{x,\alpha_X(y)})=\IP{x,\alpha_X(y)},
	\end{equation*}
	so we see that $\alpha_X$ is adjoint with $\alpha_X^*=\alpha_X$. Since $\alpha_X$ is self inverse we deduce that it is unitary. The grading operator $\alpha_X$ is even with respect to $\widetilde{\alpha}_X$ because
	\begin{equation*}
		\widetilde{\alpha}_X(\alpha_X)=\alpha_X\circ\alpha_X\circ\alpha_X=\alpha_X.
	\end{equation*} 
	We have $X=X_0\oplus X_1$ as vector spaces, so to show that we have a direct sum of correspondences we need only show that $X_0\perp X_1$ and $\phi(A)X_i\subseteq X_i$ by Definition~\ref{def:isomorphic and direct sum as correspondences}. Fix $x\in X_0$ and $y\in X_1$. Then we have
	\begin{equation*}
		\IP{x,y}=\alpha_A(\IP{x,y})=\IP{\alpha_X(x),\alpha_X(y)}=-\IP{x,y},
	\end{equation*}
	so we deduce that $\IP{x,y}=0$ and so $X_0\perp X_1$. For $x\in X_i$ and $a\in A$ we have
	\begin{equation*}
		\alpha_X(\phi(a)x)=\phi(\alpha_A(a))\alpha(x)=(-1)^i\phi(a)x,
	\end{equation*}
	so we see that $\phi(A)X_i\subseteq X_i$. Thus we have $X=X_0\oplus X_1$ as $A$-$A$ correspondences, and so $\phi$ decomposes as $\phi_{X_0}\oplus\phi_{X_1}$. In particular, restricting the left action $\phi$ to the ideal $I\triangleleft A$ we have $X=X_0\oplus X_1$ as $I$-$A$ correspondences. On $X_i$, $\alpha_X$ acts as $(-1)^i$, so $\alpha_X$ restricts to the trivial grading on $X_0$ and multiplication by $-1$ on $X_1$. This allows us to write 
	\begin{align*}
		[X,\phi,0,\alpha_X]&=[X_0\oplus X_1,\phi_{X_0}\oplus\phi_{X_1},0,\alpha_{X}]\\
		&=[X_0,\phi_{X_0},0,\text{id}]+[X_1,\phi_{X_1},0,-\text{id}]\\
		&=[X_0,\phi_{X_0},0,\text{id}]+[X_1,\phi_{X_1}\circ\alpha_A,-0,-\text{id}]\\
		&=[X_0,\phi_{X_0},0,\text{id}]-[X_1,\phi_{X_1},0,\text{id}].\qedhere\\
	\end{align*}	
\end{proof}
Let $(E^0,E^1,r,s)$ be a directed graph with no sinks, and $\delta:E^1\to \Z_2$ be a function, where $Z_2$ is viewed multiplicatively. Give $E^0$ the discrete topology and consider the $C_0(E)$-$C_0(E)$ correspondence $X$ of Example~\ref{Marns}. As shown in Lemma~\ref{Graph compact left action}, the ideal $I=\phi^{-1}(\KK(X))$ is equal to the ideal 
\begin{equation*}
	I=\{f\in C_0(E^0):f(v)=0\text{ whenever }|s^{-1}(v)|<\infty\}.
\end{equation*}
Giving $C_0(E^0)$ the trivial grading, we obtain a grading $\alpha_X$ of $X$ by defining $\alpha(1_e)=\delta(e)1_e$ and extending to $X$ by linearity and continuity. Proposition 12 of \cite{CuntzKreigerIsCuntzPimsner} says that $\OO_X$ is isomorphic to the Cuntz-Krieger algebra $C^*(E)$ since $E$ has no sinks. Since $E^0$ is countable and $C_0(E^0)$ is trivially graded, we may write $C_0(E^0)=\bigoplus_{v\in E^0}C_0(v)=\bigoplus_{v\in E^0}\C$. Since $C_0(E^0)$ is trivially graded we have $K^{\gr}_*(C_0(E^0))=K_*(C_0(E^0))$. Thus Lemma~\ref{lemma: K is continuous} gives
\begin{equation*}
K^{\gr}_*(C_0(E^0))=K_*(C_0(E^0))=K_*(\oplus_{v\in E^0} \C)=\oplus_{v\in E^0}KK_*(\C)=(\Z^{E^0},0)
\end{equation*}
where we have used the result of Examples~\ref{K theory of C}~and~\ref{K1 of C} in the last equality. Now let $F$ be the subset
\begin{equation*}
	F=\{v\in E^0:|s^{-1}(v)|<\infty\}
\end{equation*}
of $E^0$. Lemma~\ref{Graph compact left action} says that the ideal $I=\phi^{-1}(\KK(X))$ is equal to $C_0(F)$. By the same argument as above, we see that $K^{\gr}_*(C_0(F))=(\Z^{F},0)$.
The exact sequence of Theorem~\ref{MAINTHEOREM} then reduces to
\begin{equation*}
0 \rightarrow K^{\gr}_1(C^*(E)) \xrightarrow{\ \partial \ } \Z^{F} \xrightarrow{\dis\Otimes _{C_0(E^0)}([\iota]-[X])}  \Z^{E^0} \xrightarrow{\ \iota_*\ } K^{\gr}_0(C^*(E)) \rightarrow 0.
\end{equation*}
Exactness tells us that $\partial$ is injective and $\iota_*$ is surjective, so the first isomorphism theorem tells us that
\begin{equation*}
K_1^{\gr}(C^*(E))=\Image(\partial)\qquad\text{and}\qquad K_0^{\gr}(C^*(E))=\Z^{E^0}/\Ker(\iota_*).
\end{equation*}
Letting $f=\dis\Otimes _{C_0(E^0)}(1-[X])$ and applying exactness once more we obtain
\begin{equation*}
K_1^{\gr}(C^*(E))=\Ker(f)\qquad\text{and}\qquad K_0^{\gr}(C^*(E))=\Z^{E^0}/\Image(f)=\text{Coker}(f).
\end{equation*}
Now we wish to analyse the map $\bigOtimes _{C_0(E^0)}([\iota_I]-[X])$. Following the work of \cite[Chapter 8]{ThePaper}, we construct the generators of $K_0(C_0(E^0))$ as follows. Let $\C_v=\{z\delta_v:z\in \C\}$ be a copy of the complex numbers, considered as being embedded in $C_0(E^0)$. We may make $\C_v$ into a Hilbert $C_0(E^0)$-module by defining a right action $(z\delta_v)a=a(v)z\delta_v$ for $a\in C_0(E^0)$, and inner-product $\IP{z\delta_v,w\delta_v}=\overline{z}w\delta_v$. We define a homomorphism $\varphi_v:\C\to C_0(E^0)$ by $\varphi_v(z)=z\delta_v$. Left multiplication by $\varphi_v(\cdot)$ in $C_0(E^0)$ then defines a left action of $\C$ on $\C_v$ by $\varphi_v(z)(w\delta_v)=z\delta_vw\delta_v=zw\delta_v$ which is just left multiplication by the complex number $z$. This makes $\C_v$ into a $\C$-$C_0(E^0)$ correspondence. The $KK$ classes $\{[\C_v]\}_{v\in E^0}$ then generate $KK_0(\C,C_0(E^0))$, and there is a group isomorphism $\theta_{E^0}:KK_0(\C,C_0(E^0))\to \Z^{E^0}$ such that $\theta_{E^0}([\C_v])=\delta_v$ where $1_v\in \Z^{E^0}$ is the group element whose $v$th component is 1 and all other components are zero. Similarly, the $KK$ classes $\{[\C_v]\}_{v\in F}$ generate $KK_0(\C,C_0(F))$, and there is an isomorphism $\theta_F:KK_0(\C,C_0(F))\to \Z^{F}$ such that $\theta_F([\C_v])=\delta_v$. Let $X_0$ and $X_1$ denote the even and odd subspaces of $X$. Define the graded adjacency matrix $A^{\delta}_{F,E^0} \in M_{F\times E^0}(\C)$ by
\begin{equation*}
(A^{\delta}_{F,E^0})_{v,w}=\sum_{f\in vE^1w}\delta(f)
\end{equation*}
for $v\in F$ and $w\in E^0$. Since all vertices in $F$ emit finitely many edges we see that this matrix is well defined.
We will show that the map $\bigOtimes_{C_0(E^0)}([\iota_I]-[X])$ can be computed in terms of $A^{\delta}_{F,E^0}$.
\begin{lemma}
	Let $X$, $E$, $\theta_{E^0}$ and $\theta_F$ be as above. Then the following diagram commutes
	\begin{equation*}
	\xymatrix{
		K^{\gr}_0(C_0(F)) \ar[r]^{\ \ \ \  \theta_F}\ar[d]_{\Otimes [X]}&\Z^F\ar[d]^{(A^{\delta}_{F,E^0})^T}\\
		K^{\gr}_0(C_0(E^0)) \ar[r]^{\ \ \ \  \theta_{E^0}}&\Z^{E^0}\\
	}
	\end{equation*}
\end{lemma}
\begin{proof}
	
	Let $E^1_i$ be the partition of $E^1$ into the sets
	\begin{equation*}
	E^1_0=\{e\in E^1:\delta(e)=1\},\qquad E^1_1=\{e\in E^1:\delta(e)=-1\}.
	\end{equation*}
	It suffices to check that $(A^{\delta}_{F,E^0})^{T}\theta_F([\C_v])=\theta_{E^0}([\C_v]\Otimes[X])$. We have
	\begin{align*}
	\left((A^{\delta}_{F,E^0})^T\theta_F([\C_v])\right)_{u}&=\left((A^{\delta}_{F,E^0})^T\delta_v\right)=\sum_{w\in F}(A^{\delta}_{F,E^0})_{w,u}(\delta_v)_w\\
	&=(A^{\delta}_{F,E^0})_{v,u}=\sum_{f\in vE^1u}\delta(f)\\
	\end{align*}
	Thus
	\begin{equation}\label{Brilliant}
	(A^{\delta}_{F,E^0})^T\theta_F([\C_v])=\sum_{u\in E^0, f\in vE^1u}\delta(f)\delta_u=\sum_{u\in E^0,f\in vE^1_0u}\delta_u-\sum_{u\in E^0,f\in vE^1_1u}\delta_u.
	\end{equation}
	Now we wish to compute $\theta_{E^0}([\C_v]\Otimes[X])$. Using Lemma~\ref{lemma: trivially graded A gives direct sum} and using~\ref{triv right} we compute on generators $[\C_v]$
	\begin{align}
	\theta_{E^0}\left([\C_v]\Otimes _{C_0(E^0)}[X,\phi,0,\alpha_X]\right)&=\theta_{E^0}\left([\C_v]\Otimes _{C_0(E^0)}[X_0,\phi,0,\text{id}]-[\C_v]\Otimes _{C_0(E^0)}[X_1,\phi,0,\text{id}]\right)\nonumber\\
	&=\theta_{E^0}\left([X_0,\phi\circ\varphi_v,0,\text{id}]-[X_1,\phi\circ\varphi_v,0,\text{id}]\right).\label{8}\\\nonumber 
	\end{align}
	Fix $v\in E^0$. To understand the class $[X_i,\phi\circ\varphi_v,0,\text{d}]$ we aim to apply Lemma~\ref{essential} to replace $X_i$ with the essential subspace for $\phi\circ\varphi_v$. For this we claim that if we write $s^{-1}_i(g)=E^1_i\cap s^{-1}(g)$ then the essential subspace is given by
	\begin{equation*}
		\phi\circ\varphi_v(C_0(E^0))X_i=\overline{\Span}\{\delta_g:g\in s^{-1}(v)\}.
	\end{equation*}
	We fix a finitely supported function $c:E^1_i:\C$ and $f\in E^1_i$ and calculate
	\begin{align*}
	\Big(\phi\circ\varphi_v(z)\sum_{e\in E^1_i}c_e1_e\Big)(f)&=z\delta_v(s(f))\sum_{e\in E^1_i}c_e1_e(s(f))\\
	&=zc_f\delta_v(s(f))\in \overline{\Span}\{\delta_g:g\in s^{-1}(v)\}.\\
	\end{align*}
	Also if $s(g)=v$ with $g\in E^1_i$ then
	\begin{equation*}
		\delta_g=\phi\circ\varphi_v(1)\delta_g\in \phi\circ\varphi_v(C_0(E^0))X_i.
	\end{equation*}
	Hence by Lemma~\ref{essential} we may write
	\begin{equation*}
	[X_i,\phi\circ\varphi_v,0,\text{id}]=[\overline{\Span}\{\delta_g:g\in s^{-1}(v)\},\phi\circ\varphi_v,0,\text{id}].
	\end{equation*}
	Elements of the module $\overline{\Span}\{\delta_g:g\in s^{-1}(v)\}$ may be written in the form $\sum_{e\in s^{-1}_i(v)}c_e1_e$. There is a unitary map $U:\overline{\Span}\{\delta_g:g\in s^{-1}(v)\}\to \bigoplus_{e\in vE^1_iw}[\C_w]$ such that 
	\begin{equation*}
	U\Big(\sum_{e\in s^{-1}_i(v)}c_e1_e\Big)=\sum_{e\in vE^1_iw}c_e\delta_{r(e)}.
	\end{equation*}
	Since our gradings are trivial $U$ is even. Since our Fredholm operator is zero we have $U0=0=0U$. By construction $U$ intertwines the left actions, giving us a unitary equivalence. Hence
	\begin{equation*}
	[X_i,\phi\circ\varphi_v,0,\text{id}]=\sum_{\substack{w\in E^0_i,\\ f\in vE^1w}}[\C_w].
	\end{equation*}
	Finally we resume our computation \eqref{8} of the left-hand side of the diagram
	\begin{align*}
	\theta_{E^0}\left([\C_v]\Otimes [X,\phi,0,\alpha_X]\right)&=\theta_{E^0}\left([X_0,\phi\circ\varphi_v,0,\text{id}]-[X_1,\phi\circ\varphi_v,0,\text{id}]\right)\\
	&=\theta_{E^0}\Big(\dis\sum_{w\in E^0, f\in vE^1_0w}[\C_w]-\sum_{w\in E^0, f\in vE^1_1w}[\C_w]\Big)\\
	&=\sum_{w\in E^0, f\in vE^1_0w}\delta_w-\sum_{w\in E^0, f\in vE^1_1w}\delta_w\\
	\end{align*}
	which is exactly $(A^{\delta}_{F,E^0})^T\theta_F([\C_v])$ as we calculated in \eqref{Brilliant}.
\end{proof}
\begin{corollary}\label{Cor: Graph K theory}
	Let $(E^0,E^1,r,s)$ be a directed graph with no sinks, with $\delta:E^0\to\Z_2$ inducing a grading on $C^*(E)$, and let $A^{\delta}_{F,E^0}$ be the graded adjacency matrix. Suppose that $A^{\delta}_{F,E^0}$ has block matrix form
	\begin{equation*}
		A^{\delta}=\begin{pmatrix}
			A^{\delta}_{F,F}&A^{\delta}_{F,E^0\bs F}\\
		\end{pmatrix}
	\end{equation*}
	where $A^{\delta}_{F,F}\in M_{F\times F}(\C)$ and $A^{\delta}_{F,E^0\bs F}\in M_{F,E^0\bs F}(\C)$ are the matrices with entries
	\begin{equation*}
		(A^{\delta}_{F,F})_{v,w}=\sum_{e\in vE^1w}\delta(f)\qquad (A^{\delta}_{F,E^0\bs F})_{v,w}=\sum_{e\in vE^1w}\delta(f).
	\end{equation*}
	Let $\iota:\Z^F\to\Z^{E^0}$ be the $E^0\times F$ given in block form by
	\begin{equation*}
		\iota=\begin{pmatrix}
			1_{F\times F}\\
			0_{F\times E^0}\\
		\end{pmatrix}.
	\end{equation*}
	Then the graded $K$-theory of $C^*(E)$ is given by
	\begin{equation*}
	K_0(C^*(E))=\Coker\left(\iota-(A^{\delta}_{F,E^0})^T\right)\qquad\qquad K_1(C^*(E))=\Ker\left(\iota-(A^{\delta}_{F,E^0})^T\right).
	\end{equation*}
\end{corollary}
\begin{example}
	To give an idea about how the grading and the subset $F$ affect the graded $K$-theory of $C^*(E)$ we give the specific example of the following graph. Let $a,b,r,s,p,q\in\N$ be finite, and let $E$ be the graph with two vertices $E^0=\{v,w\}$, edges
	\begin{align*}
		E^1=\{e^1_j:j\leq q\}&\sqcup\{e^{-1}_j:j\leq p\}\sqcup\{f_j:j\in\N\}\sqcup\{g^1_j:j\leq a\}\\
		&\sqcup\{g^{-1}_j:j\leq b\}\sqcup\{h^1_j:j\leq r\}\sqcup\{h^{-1}_j:j\leq s\}
	\end{align*}
	and range and source maps
	\begin{equation*}
		r(e^i_j)=r(g^i_j)=s(g_j^i)=s(f_j)=v\qquad\text{ and }\qquad s(e^i_j)=r(f_j)=s(h^i_j)=r(h^i_j)=w
	\end{equation*}
	for all $i,j$ (see~\ref{fig:graph}).
	\begin{figure}
		\centering	
		\[
	\tikzset{->-/.default=0.5, ->-/.style={decoration={
				markings,
				mark=at position #1 with {\arrow{stealth}}},postaction={decorate}}}
	\begin{tikzpicture}
	\node[circle, inner sep=0.5pt] (u) at (-4,0) {$v$};
	\node[circle, inner sep=0.5pt] (v) at (4,0) {$w$};
	\draw[out=30, in=150, ->-] (u) to node[above, pos=0.5] {$\{f_j : j \in \mathbb{N}\}$} (v);
	\draw[out=200, in =340, ->-] (v) to node[below, pos=0.5] {$\{e^{-1}_j : j \le p\}$} (u);
	\draw[out=225, in =315, ->-] (v) to node[below, pos=0.5] {$\{e^{1}_j : j \le q\}$} (u);
	\draw[->-] (u) .. controls +(-2, 2) and +(-2, -2) .. (u) node[pos=0.5, left, inner sep=1pt] (1) {};
	\draw[->-] (u) .. controls +(-3, 3.5) and +(-3, -3.5) .. (u) node[pos=0.5, left] (2) {$\{g^{-1}_j : j \le b\}$};
	\draw[->-] (v) .. controls +(2, 2) and +(2, -2) .. (v) node[pos=0.5, right, inner sep=1pt] (3) { };
	\draw[->-] (v) .. controls +(3, 3.5) and +(3, -3.5) .. (v) node[pos=0.5, right] (4) {$\{h^{-1}_j : j \le s\}$};
	\node[anchor=south, inner sep=1pt] (5) at (2.north) {$\{g^1_j : j \le a\}$};
	\node[anchor=south, inner sep=1pt] (6) at (4.north) {$\{h^1_j : j \le r\}$};
	\draw (5.east)--(1.west);
	\draw (6.west)--(3.east);
	\end{tikzpicture}
	\]
	\caption{E} \label{fig:graph}
\end{figure}

	Let $\delta:E^1\to\Z$ be any function such that $\delta(\alpha_j^i)=i$ for any symbol $\alpha\in\{e,g,h\}$ (we allow any values for $\delta(f_j)$).
	For this graph the subset $F$ consisting of vertices which emit finitely many edges is $F=\{w\}$. We compute the graded adjacency matrix
	\begin{equation*}
		A^{\delta}_{F,E^0}=(r-s,q-p),
	\end{equation*}
	hence we have
	\begin{equation*}
		\iota-(A^{\delta}_{F,E^0})^T=\begin{pmatrix}
			1-s+r\\
			p-q\\
		\end{pmatrix}.
	\end{equation*}
	Thus we have the graded $K$-theory
	\begin{equation*}
		K^{\gr}_*(C^*(E))=\begin{cases}
			(\Z\oplus \Z,\Z)&1-r+s=0,p=q\\
			(\Z_{1-r+s}\oplus\Z,0)&1-r+s\neq 0,p=q\\
			(\Z\oplus\Z_{p-q},0)&1-r+s=0,p\neq q\\
			(\Z_{1-r+s}\oplus\Z_{p-q},0)&1-r+s\neq 0,p\neq q\\
		\end{cases}.
	\end{equation*}
	In particular we note that $K^{\gr}_*(C^*(E))$ is independent of the values of $a$ and $b$. In fact, if we had let one of $a$ or $b$ be infinite with finitely many edges $g^i_j$ such that $s(f^i_j)=v$ and $r(f^i_j)=w$ then we would obtain the same results. As such, if $E$ is any graph with no sinks and $\delta_1:E^1\to\Z_2$ and $\delta_2:E^1\to\Z_2$ are two functions such that $\delta_1(e)=\delta(e)$ whenever $e\in\{f\in E^1:|s^{-1}(f)|<\infty\}$ then we deduce that $K_*^{\gr}(C^*(E),\delta_1)=K_*^{\gr}(C^*(E),\delta_2)$ where we have identified $\delta_i$ with the induced grading on $C^*(E)$.
\end{example}
\begin{example}
	Consider the graph $B_{\infty}$ with one vertex and infinitely many edges. Fix $\delta:B_{\infty}^1\to\Z_2$. Let $X$ be the $\C$--$\C$ correspondence of Example~\ref{Marns} with the grading induced by $\delta$. By Lemma~\ref{Graph compact left action}, the ideal $I$ consisting of elements of $\C$ for which the left action on $X$ is compact is zero. Hence by Theorem~\ref{Cuntz-Pimsner exact sequence} we have $\TT_X\cong\OO_X$. Since $B_n$ has no sinks, Proposition \cite[12]{CuntzKreigerIsCuntzPimsner} implies the Cuntz-Krieger algebra $\OO_{\infty}\coloneqq C^*(B_{\infty})$ is isomorphic to the Cuntz-Pimsner algebra $\OO_X$. Theorem~\ref{KK equivalence} then gives
	\begin{equation*}
		\KK^{\gr}_*(\OO_{\infty})=\KK^{\gr}_*(\OO_X)=\KK^{\gr}_*(\TT_X)=\KK^{\gr}_*(\C)=(\Z,0).
	\end{equation*}
	In particular we note that $K^{\gr}_*(\OO_X)$ is independent of the grading on $X$.
\end{example}

%% file: chapter9.tex
\chapter{Closing remarks}
\label{ch: closing remarks}

	Let $E$ be a graph with no infinite emitters and consider the $C_0(E^0)$--$C_0(E^0)$ correspondence $X$ of Example~\ref{Marns}. By Lemma~\ref{Graph compact left action} the left action on $X$ is by compacts, and so the set 
	\begin{equation*}
		F=\{v\in E^0:|s^{-1}(v)|<\infty\}
	\end{equation*} 
	is all of $E^0$. Corollary~\ref{Cor: Graph K theory} then reduces to \cite[Corollary~8.3]{ThePaper}.
	Corollary~\ref{Cor: Graph K theory} corresponds in a very natural way to existing results about the regular $K$-theory of graph algebras. Let $E$ be any graph (possibly with sinks). As before, let 
	\begin{equation*}
	F=\{v\in E^0:|s^{-1}(v)|<\infty\},
	\end{equation*}
	and now let
	\begin{equation*}
	G=\{v\in E^0:s^{-1}(v)\neq\emptyset\}.
	\end{equation*}
	That is, $F$ is the set of infinite emitters in $E^0$ and $G$ is the set of sinks in $E^0$.
	Let $A_{F\cap  G,E^0}\in M_{|F\cap G|\times |E^0|}(\C)$ be the (ungraded) adjacency matrix with entries
	\begin{equation*}
	(A_{F\cap G,E^0})_{v,w}=|vE^1w|
	\end{equation*}
	for $v\in F\cap G$ and $w\in E^0$. Let $\iota$ be the $|F\cap G|$ by $|E^0|$ matrix given in block form
	\begin{equation*}
	\iota=\begin{pmatrix}
	1_{F\times F}\\
	0_{F\times E^0}\\
	\end{pmatrix}.
	\end{equation*}
	In \cite[Theorem~6.1]{ArbitraryGraphKTheory} it is shown that there are isomorphisms
	\begin{equation}\label{special K}
	K_0(C^*(E))\cong \Coker(\iota-A_{F\cap G,E^0}^T)\qquad\text{ and }\qquad K_1(C^*(E))\cong \Ker(\iota-A_{F\cap G,E^0}^T).
	\end{equation}
	In the case when $E$ has no sources (so $G=\emptyset$), Corollary~\ref{Cor: Graph K theory} is a direct analogue of this result, replacing the adjacency matrix $A_{F,E^0}$ with the graded adjacency matrix $A^{\delta}_{F,E^0}$. This also suggests that if $E$ has sinks (so the left action on $X$ is not injective) then we should be able to compute the graded $K$-theory of $E$ using~\ref{special K} replacing $A_{F\cap G,E^0}$ with $A^{\delta}_{F\cap G,E^0}$. A problem arises however in that when the left action is not injective: we have the competing definitions of Pimsner and Katsura for $\OO_X$. In \cite[Proposition~3.10]{Katsura}, Katsura showed that the Cuntz-Pimsner algebra $\OO_X$ according to his definition is isomorphic to the Cuntz-Krieger algebra $C^*(E)$ for the $C_0(E^0)$--$C_0(E^0)$ correspondence $X$ of Example~\ref{Marns}. Hence for the purpose of computing the graded $K$-theory of graph algebras with sinks, we would be interested in proving an analogous statement to Theorem~\ref{MAINTHEOREM} for $\OO_X$ according to Katsura's definition with $X$ a correspondence with non-injective left action. In \cite{AddingTails} it is shown that given an $A$--$A$ correspondence $X$, it is possible to construct a $C^*$-algebra $T$ and an $A\oplus T$--$A\oplus T$ correspondence $Y$ with injective left action such that (for Katsura's definition) $\OO_X$ is isomorphic to a full corner of $\OO_Y$ (\cite[Theorem~4.3]{AddingTails}). By \cite[Example~3.6]{RaeburnWilliams}, $\OO_X$ and $\OO_Y$ are Morita equivalent, and hence by Theorem~\ref{Morita equiv KK} $\OO_X$ and $\OO_Y$ are $KK$-equivalent. Hence, employing Theorem~\ref{MAINTHEOREM}, we obtain an exact sequence for $\OO_X$ in terms of $A\oplus T$. Such a process may be useful for computing the graded $K$-theory of graph algebras for graphs with no sources. For this process to be useful, we would need to know about what role the $C^*$-algebra $T$ plays with gradings.

%% file: Appendix.tex
\chapter{Uniformly continuous extensions}
\label{Appendix}

\begin{theorem}[Uniformly continuous extension theorem]\label{Sang}
	Let $X,Y$ be metric spaces and $Y$ be complete. Let $f:X\rightarrow Y$ be uniformly continuous and denote by $\overline{X}$ the completion of $X$. Then there exists a unique continuous function $\tilde{f}:\overline{X}\rightarrow Y$ such that $\tilde{f}|_X=f$. For any $x\in\overline{X}\bs X$ and any sequence $(x_n)\in X$ with $x_n\to x$, we have
	\begin{equation*}
	\tilde{f}(x)=\lim_{n\to\infty}f(x_n).
	\end{equation*}
	In addition to being continuous, $\tilde{f}$ is uniformly continuous. If $X$ and $Y$ are topological vector spaces with respect to the metric topology and $f$ is linear then so too is $\tilde{f}$. 
\end{theorem}
\begin{proof}
	Define $\tilde{f}=f$ on $X$. For $x\in\overline{X}\bs X$ choose by density of $X\subseteq\overline{X}$ a sequence $x_n\in X$ with $x_n\rightarrow x$. We wish to define $\tilde{f}(x)=\lim_{n\rightarrow\infty}f(x_n)$.\\ We first3 check that $f(x_n)$ converges to some element of $Y$, so fix $\ep>0$. Since $x_n\in X$ for all $n$ we may find by uniform continuity a $\delta>0$ such that
	\begin{equation*}
	d(x_n,x_m)<\delta\implies d(f(x_n),f(x_m))<\ep.
	\end{equation*}
	Since $(x_n)$ is Cauchy there exists $N$ such that for $n,m\geq N$,
	\begin{equation*}
	d(x_n,x_m)<\delta.
	\end{equation*}
	Hence $f(x_n)$ is a Cauchy sequence, and by completeness converges to some element of $Y$.\\ To see that $f$ is well defined we must show that $\lim_{n\rightarrow\infty}f(x_n)$ is independent of choice of $(x_n)$. Let $(x_n)$, $(y_n)$ be two sequences in $X$ that converge to $x$. We wish to show that
	$d(f(x_n),f(y_n))\rightarrow 0$,
	so fix $\ep>0$. Since $x_n,y_n\in X$, by uniform continuity there exists $\delta >0$ such that 
	\begin{equation*}
	d(x_n,y_n)<\delta\implies d(f(x_n),f(y_n))<\ep.
	\end{equation*}
	Now choose $N$ such that for $n\geq N$,
	\begin{equation*}
	d(x_n,x)<\delta/2\;\;\text{ and }\;\;d(y_n,x)<\delta/2.
	\end{equation*}
	Then for $n\geq N$ we have
	\begin{equation*}
	d(x_n-y_n)\leq d(x_n,x)+d(x,y_n)<\delta.
	\end{equation*}
	Hence $d(f(x_n),f(y_n))<\ep$ and so $d(f(x_n),f(y_n))\rightarrow 0$. Since $f(x_n)$ and $f(y_n)$ are both themselves convergent, we deduce that $\lim_{n\rightarrow\infty}f(x_n)=\lim_{n\rightarrow\infty}f(y_n)$ and so $f(x)$ is is well defined.\\ Now we must show that $\tilde{f}$ is uniformly continuous, so fix $\ep>0$. By uniform continuity, choose $\delta>0$ such that
	\begin{equation*}
	d(x,y)<\delta\implies d(f(x),f(y))<\ep/3
	\end{equation*}
	for all $x,y\in X$. Let $z,w\in\overline{X}$ with $d(z,w)<\delta/3$. Clearly by construction $f$ is continuous on all of $\overline{X}$, and so we may pick $\delta/3>\sigma>0$ such that for all $x\in\overline{X}$,
	\begin{equation*}
	d(z,x)<\sigma\implies d(f(z),f(x))<\ep/3
	\end{equation*}
	and for all $y\in\overline{X}$, 
	\begin{equation*}
	d(y,w)<\sigma\implies d(f(y),f(w))<\ep/3.
	\end{equation*}
	By density of $X$, there exist $x,y\in X$ such that $d(z,x),d(y,w)<\sigma$. Then we have
	\begin{equation*}
	d(x,y)\leq d(x,z)+d(z,w)+d(w,y)<2\sigma/3 +\delta/3<\delta.
	\end{equation*}
	Hence $d(f(x),f(y))<\ep/3$ by uniform continuity on $X$. Thus
	\begin{equation*}
	d(f(z),f(w))\leq d(f(z),f(x))+d(f(x),f(y))+d(f(y),f(w))<\ep.
	\end{equation*}
	Thus $\tilde{f}$ is uniformly continuous on $\overline{X}$. If $X$ and $Y$ are vector spaces and $f$ is linear then by the algebra of limits $\tilde{f}$ is linear too. The function $\tilde{f}$ is unique since if $g$ is any other continuous extension of $f$ to $\overline{X}$ then $\tilde{f}$ and $g$ agree on $X$ and for $x\in\overline{X}$ and any sequence $x_n\in X$ converging to $x$, we have
	\begin{align*}
	g(x)&=\lim_{n\rightarrow\infty}g(x_n)\qquad\text{ by contintuity}\\
	&=\lim_{n\rightarrow\infty}f(x_n)\qquad\text{ since $g$ and $f$ agree on $X$}\\
	&=\tilde{f}(x).\qedhere\\
	\end{align*}
\end{proof}

%% file: AppendixB.tex
\chapter{Infinite direct sums}
\label{Direct sums}

The following content combines definitions and discussion from \cite[Section~9]{CStarNotes} and \cite[Section~7]{TopoNotes}.
When taking finite direct sums of vector spaces, the direct sum is exactly the Cartesian product. When taking infinite direct sums however, the Cartesian product and direct sums of vector spaces become quite different. Recall that if $S$ is an indexing set, and for each $s\in S$ we have a set $X_s$, the Cartesian product $\prod_{s\in S}X_s$ is defined to be the set of all functions $x:S\to \bigcup_{s\in S}X_s$ such that $x(s)\in X_s$. We write $x_s$ for $x(s)$ and refer to $x_s$ as the `$s$th component of $x$'.
\begin{definition}
	Let $S$ be an indexing set, and suppose that for each $s\in S$ we have a vector space $V_s$ over the same field $K$. The \textit{direct sum} $\bigoplus_{s\in S}V_s$ is defined to be the set of all $v\in \prod_{s\in S}V_s$ such that $v_s\neq 0$ for only finitely many $s\in S$. The direct sum is a vector space with
	\begin{equation*}
	(v+w)_s=v_s+w_s,\qquad \text{ and }\qquad (\lambda v)_s=\lambda v_s
	\end{equation*}
	for $v,w\in\bigoplus_{s\in S}V_s$ and $\lambda\in K$.
\end{definition}
\begin{remark}
	If the indexing set $S$ is given the discrete topology, then the direct sum $\bigoplus_{s\in S}V_s$ can be thought of as compactly supported continuous functions $C_c(S;\bigcup_{s\in S}V_s)$ such that $v_s\in V_s$. Since a subset of $S$ is compact if and only if it is finite, it becomes clear why the Cartesian product and direct sum are the same when the set $S$ is finite, and different when $S$ is infinite. For a collection of Hilbert modules $\{X_s\}_{s\in S}$ over a $C^*$-algebra $A$ we can make sense of right multiplication and an inner-product on the vector space direct sum $\bigoplus_{s\in S}X_s$ by the formulas
	\begin{equation*}
	(xa)_s=x_sa\qquad\text{ and }\qquad \IP{x,y}=\sum_{s\in S}\IP{x_s,y_s}.
	\end{equation*}
	This sum is well defined since only finitely many terms are non-zero. However for infinite $S$, the vector space direct sum $\bigoplus_{s\in S}X_s$ is not complete in the norm induced by the inner product. What we would really like out of a direct sum of Hilbert modules, is something akin to `$C_0(S;\bigcup_{s\in S}V_s)$'. We will first need to establish a some facts about summation over indexing sets.
\end{remark}
\begin{lemma}\label{absolute}
	Let $A$ be a $C^*$-algebra and $a_n$ be a sequence of positive elements. If $\sum_na_n$ converges then it does so unconditionally: if $\sigma:\N\to\N$ is a bijection, then $\sum_n a_{\sigma(n)}$ exists and is equal to the sum $\sum_{n}a_n$.
\end{lemma}
\begin{sketch}
	We know that this is true for the case $A=\C$, because the positive elements of $\C$ are positive real numbers, and we know that series of positive numbers converge unconditionally. Let $A$ be any $C^*$-algebra and let $\pi:A\to\BB(\HH)$ be a faithful representation. Let $a_n$ be a sequence of positive elements in $A$ such that $\sum_na_n$ converges. Then for any $\xi\in\HH$ the sum $\sum_n(\pi(a_n)\xi,\xi)$ is a convergent sequence of positive numbers. Hence if $\sigma$ is a permutation of $\N$, $\sum_n(\pi(a_n)\xi,\xi)$ converges to $\sum_n(\pi(a_{\sigma(n)})\xi,\xi)$, and hence $\sum_n a_{\sigma(n)}$ converges to $\sum_n a_n$.
\end{sketch}
\begin{definition}
	Let $S$ be an indexing set, $A$ be a $C^*$-algebra and $\{a_s\}_{s\in S}$ be a set of positive elements in $A$. 
	If $a_s\neq 0$ for only countably many $s\in S$ and if for some listing $\{s_n\}_{n\in\N}$ of the countable set $\{s\in S:a_s\neq 0\}$ the sum
	\begin{equation*}
	\sum_{n=1}^{\infty}a_{s_n}
	\end{equation*}
	converges, then we define
	\begin{equation*}
	\sum_{s\in S}a_s\coloneqq \sum_{n=1}^{\infty}a_{s_n}.
	\end{equation*}
	Lemma~\ref{absolute} shows that this sum is independent of choice of listing $\{s_n\}_{n\in\N}$.
	If either the set $\{s\in S:a_s\neq 0\}$ is uncountable, or if the set $\{s\in S:a_s\neq 0\}$ is countable but the sum $\sum_{n=1}^{\infty}\normof{a_{s_n}}$ does not converge for any listing $\{s_n\}_{n\in\N}$, then we say that the sum $\sum_{s\in S}a_s$ diverges.
\end{definition}

\begin{definition}
	Let $S$ be an indexing set, and suppose that for each $s\in S$ we have a right Hilbert module $X_s$ over a $C^*$-algebra $A$. We define the Hilbert module direct sum to be
	\begin{equation*}
	\bigoplus_{s\in S}X_s=\{x\in\prod_{s\in S}X_s:\sum_{s\in S}\IP{x_s,x_s} \text{ converges in }A\}.
	\end{equation*}
	The direct sum is a right Hilbert $A$-module with right action and inner-product given by the formulas
	\begin{equation*}
	(xa)_s=x_sa\qquad\text{ and }\IP{x,y}=\sum_{s\in S}\IP{x_s,y_s}.
	\end{equation*}
\end{definition}
\begin{remark}
	If $x\in\bigoplus_{s\in S}$ then $\sum_{s\in S}\IP{x_s,x_s}$ converges and one can see that
	\begin{equation*}
	\sum_{s\in S}\IP{x_sa,x_sa}=a^*\Big(\sum_{s\in S}\IP{x_s,x_s}\Big)a
	\end{equation*}
	also converges because multiplication in $A$ is continuous, so that $xa\in\bigoplus_{s\in S}X_s$. If $x,y\in\bigoplus_{s\in S}X_s$ and $F\subset S$ is finite then using the Cauchy-Schwartz inequality \eqref{C-S} and H\"{o}lder's inequality for real numbers we see that
	\begin{align*}
	\normof{\sum_{s\notin F}\IP{x_s,y_s}}&\leq\sum_{s\notin F}\normof{\IP{x_s,y_s}^*\IP{x_s,y_s}}^{1/2}\\
	&\leq \sum_{s\notin F}\normof{x_s}^{1/2}\normof{y_s}^{1/2}\leq \Big(\sum_{s\notin F}\normof{x_s}\Big)^{1/2}\Big(\sum_{s\notin F}\normof{y_s}\Big)^{1/2},\\
	\end{align*}
	which which can be made arbitrarily small by choosing $F$ large enough since $x,y\in\bigoplus_{s\in S}X_s$. So we see that the inner-product is well defined. We also would like to know about adjointable operators on the direct sum:
\end{remark}
\begin{lemma}\label{direct sum of operators}
	Suppose that $S$ is an indexing set, and for each $s\in S$ we have a Hilbert $A$-module $X_s$ and an adjointable operator $T_s\in\LL(X_s)$. Then the formula
	\begin{equation}\label{oplus}
	\left(\left(\oplus T_s\right)x\right)_s=T_sx_s
	\end{equation}
	defines an adjointable operator if and only if $\sup_{s\in S}\normof{T_s}$ is finite. We then have $\normof{\oplus T_s}=\sup_{s\in S}\normof{T_s}$.
\end{lemma}
\begin{proof}
	Suppose that $\sup_{s\in S}\normof{T_s}$ is finite.
	Since $T_s 0=0$ for each $T_s$, if $x\in\bigoplus_{s\in S} X_s$ then $x_s\neq 0$ for only countably many $s\in S$, and so $T_sx_s\neq 0$ for only countably many $s\in S$. Thus to show that $\oplus T_s$ maps $\bigoplus_{s\in S}X_s$ into itself, we need only show that $\sum_{s\in S}\IP{\oplus T_sx_s,T_sx_s}$ converges in $A$.
	Since $\sup_{s\in S}\normof{T_s}$ is finite, using \cite[Corollary 1.25]{Crossedandunitizations} which says that $\IP{Tx,Tx}\leq\normof{T}^2\IP{x,x}$, we see for $x\in \bigoplus_{s\in S} X_s$ that
	\begin{align*}
	\sum_{s\in S}\IP{((\oplus T_s)x)_s,((\oplus T_s)x)_s}&=\sum_{s\in S}\IP{T_sx_s,T_sx_s}\\
	&\leq\sum_{s\in S}\normof{T_s}^2\IP{x_s,x_s}\\
	&\leq\sup_{s\in S}\normof{T_s}^2\sum_{s\in S}\IP{x_s,x_s}.\\
	\end{align*}
	Since $x\in\bigoplus_{s\in S} X_s$ and $\sup_{s\in S}\normof{T_s}$ is finite, this sum is convergent and we deduce that $\oplus T_s$ maps $\bigoplus_{s\in S}X_s$ into itself.
	To see then that \eqref{oplus} defines an adjointable operator we need only show that $\oplus T_s$ has an adjoint. We compute
	\begin{align*}
	\IP{(\oplus T_s)x,y}&=\sum_{s\in S}\IP{((\oplus T_s)x)_s,y_s}\\
	&=\sum_{s\in S}\IP{T_sx_s,y_s}=\sum_{s\in S}\IP{x_s,T^*_sy_s}\\
	&=\sum_{s\in S}\IP{x_s,((\oplus T^*_s)y)_s}=\IP{x,(\oplus T^*)y}\\
	\end{align*}
	so $\oplus T_s$ is adjointable with adjoint $\oplus T^*_s$. 
	Now suppose that the formula
	\begin{equation*}
	(\oplus T_sx)_s=T_sx_s
	\end{equation*}
	defines an adjointable operator on $\bigoplus_{s\in S}X_s$. Then by Lemma~\ref{Rang} $\oplus T_s$ is bounded. Fix $s\in S$. For any $x\in \bigoplus_{t\in S}X_t$ we have
	\begin{equation*}
	\IP{T_sx_s,T_sx_s}\leq \sum_{t\in S}\IP{T_tx_t,T_tx_t},
	\end{equation*}
	so taking norms gives
	\begin{equation*}
	\normof{T_s}^2=\sup_{\normof{x_s}\leq 1}\normof{\IP{T_sx_s,T_sx_s}}\leq\sup_{\normof{x}\leq 1}\Bignormof{\sum_{t\in S}\IP{T_tx_t,T_tx_t}}=\normof{\oplus T_t}.
	\end{equation*}
	Hence $\normof{T_s}\leq\normof{\oplus T_t}$. Thus the set $\{\normof{T_s}:s\in S\}$ is bounded above, whence $\sup_{s\in S}\normof{T_s}$ is finite.
	
	Now we wish to show that $\normof{\oplus T_s}=\sup_{s\in S}\normof{T_s}$. First we show that $\oplus T_s$ is invertible if and only if each $T_s$ is invertible and $\inf_{s\in S}\normof{T_s}>0$. 
	Suppose that each $T_s$ is invertible with $\inf_{s\in S}\normof{T_s}>0$, and suppose that $(\oplus T_s)x=(\oplus T_s)y$. We claim that $\oplus T^{-1}_s$ is an inverse for $\oplus T_s$. We first need to check that $\oplus T^{-1}_s$ is well defined by showing $\sup_{s\in S}\normof{T^{-1}_s}$ is finite. Since $\inf_{s\in S}\normof{T_s}>0$, we have
	\begin{equation*}
	\sup_{s\in S}\normof{T_s^{-1}}=\sup_{s\in S}\normof{T_s}^{-1}=(\inf_{s\in S}\normof{T_s})^{-1}<\infty.
	\end{equation*}
	Thus $\oplus T_s^{-1}$ is a well defined adjointable operator. We compute that
	\begin{equation*}
	\left((\oplus T_s^{-1})(\oplus T_s)x\right)_s=T_s^{-1}T_sx_s=x_s=T_sT_s^{-1}=\left((\oplus T_s\oplus T_s^{-1})x\right)_s
	\end{equation*}
	and so deduce that $\oplus T_s$ is invertible with inverse $\oplus T_s^{-1}$. Now suppose that $\oplus T_s$ is invertible. We wish to show that each $T_s$ is bijective. Fix $s\in S$ and suppose for some $x_s,y_s\in X_s$ that $T_sx_s=T_sy_s$. Let $x$ and $y$ be the elements of $\bigoplus_{t\in S}X_t$ with $x_t=y_t=0$ if $t\neq s$ and $x_t=x_s$, $y_t=y_s$ if $t=s$. Since $\oplus T_t$ is injective, and since 
	\begin{equation*}
	(\oplus T_t x)_s=T_sx_s=T_sy_s=(\oplus T_ty)_s,
	\end{equation*}
	we deduce that $x_s=y_s$ and so $T_s$ is injective. Similarly, if $x_s\in X_s$, then letting $x$ be as above, by surjectivity of $\oplus T_t$ we may find some $z\in\bigoplus_{t\in S} X_t$ such that $\oplus T_sz=x$, whence
	\begin{equation*}
	(\oplus T_tz)_s=T_sz_s=x_s,
	\end{equation*}
	so $T_s$ is surjective. Lastly, we wish to show that $\inf_{s\in S}\normof{T_s}>0$. We see that for all $x\in\bigoplus_{s\in S}X_s$
	\begin{equation*}
	T_s((\oplus T_s)^{-1}x)_s=(\oplus T_s(\oplus T_s)^{-1}x)_s=x_s,
	\end{equation*}
	so applying $T_s^{-1}$ to both sides we have
	\begin{equation*}
	((\oplus T_s)^{-1}x)_s=T^{-1}_sx_s.
	\end{equation*}
	This shows that $(\oplus T_s)^{-1}=\oplus T_s^{-1}$. Since we know that $\oplus T_s^{-1}$ is an adjointable operator, we must have
	\begin{equation*}
	\sup_{s\in S}\normof{T_s^{-1}}<\infty\implies \inf_{s\in S}\normof{T_s}>0.
	\end{equation*}
	Now we wish to compute the spectral radius of a direct sum of operators. Let $T_s\in\LL(X_s)$ be self adjoint operators with $\sup_s\normof{T_s}$ finite. If $\lambda\in\sigma(T_s)$ for some $s$ then $T_s-\lambda$ is not invertible. Hence $\oplus T_s-\lambda=\oplus (T_s-\lambda)$ is not invertible because and so $\lambda\in\sigma(\oplus T_s)$. We deduce that $\bigcup_{s\in S}\sigma(T_s)\subseteq \sigma(\oplus T_s)$, and so
	\begin{equation*}
		r(\oplus T_s)\geq \sup_sr(T_s)=\sup_s\normof{T_s}
	\end{equation*}
	since each $T_s$ is self-adjoint.
	Now suppose that $\lambda\in\sigma(\oplus T_s)$. If $\lambda\in\sigma(T_s)$ for some $s$ then
	\begin{equation*}
		|\lambda|\leq r(T_s)=\normof{T_s}\leq\sup_s\normof{T_s}.
	\end{equation*}
	Suppose that $\lambda\notin\bigcup_{s\in S}\sigma(T_s)$. Then $T_s-\lambda$ is invertible for all $s$, so we must have $\inf_s\normof{T_s-\lambda}=0$. Hence there is a sequence $T_{s_n}$ converging in norm to $\lambda$. In particular we have $\normof{T_{s_n}}\to\normof{\lambda}=|\lambda|$, and so $|\lambda|\leq \sup_s\normof{T_s}$. Hence we have
	\begin{equation*}
		r(T)=\sup_{\lambda\in\sigma(\oplus T_s)}|\lambda|\leq \sup_s\normof{T_s}
	\end{equation*}
	and so we deduce that $r(T)=\sup_s\normof{T_s}$. Now suppose that $T_s\in\LL(X_s)$ are any adjointable operators with $\sup_s\normof{T_s}$ finite. Then we have
	\begin{align*}
		\normof{\oplus T_s}^2&=\normof{\oplus T_s^*T_s}=r(\oplus T_s^*T_s)\\
		&=\sup_s\normof{T_s^*T_s}=\sup_s\normof{T_s}^2.\\
	\end{align*}
	Hence $\normof{\oplus T_s}=\sup_s\normof{T_s}$.
\end{proof}